\documentclass{article}
\usepackage{amsthm,fullpage,ddag,mathrsfs}
\usepackage{amsmath}
\usepackage{amssymb}
\usepackage{amsfonts}
\usepackage[all]{xy}
\usepackage[hypertex]{hyperref}

\newtheorem{prop}{Proposition}[subsection]

\newtheorem{coro}[prop]{Corollary}

\newtheorem{lem}[prop]{Lemma}
\newtheorem*{lemm*}{Lemma}

\makeatletter
 
 \@addtoreset{equation}{prop}
 \def\tagform@#1{\maketag@@@{
 \ignorespaces#1\unskip\@@italiccorr}}

 \newcommand{\nsection}{\@startsection{section}{1}{\z@}%
     {-5ex}
     {1ex}
     {\reset@font\center\large\sc}}

 \renewenvironment{thebibliography}[1]
 {\nsection*{\refname\@mkboth{\refname}{\refname}}%
   \list{\@biblabel{\@arabic\c@enumiv}}%
        {\settowidth
	\labelwidth{\@biblabel{#1}}%
         \leftmargin
	 \labelwidth
         \advance
	 \leftmargin
	 \labelsep
         \@openbib@code
         \usecounter{enumiv}%
         \let\p@enumiv\@empty
	 \parskip=0pt
	 \itemsep=1pt
	 \parsep=1pt
	 \itemindent=\z@
         \renewcommand\theenumiv{\@arabic\c@enumiv}}%
   \sloppy
   \clubpenalty4000
   \@clubpenalty\clubpenalty
   \widowpenalty4000%
   \footnotesize
   \sfcode`\.\@m}
  {\def\@noitemerr
    {\@latex@warning{Empty `thebibliography' environment}}%
   \endlist}
\makeatother

\makeatletter

\newenvironment{enumlatin}[1][\roman]{%
\begin{enumerate}}
{\end{enumerate}}
\makeatother

\theoremstyle{definition}

\newtheorem{empt}[prop]{}
\newtheorem{dfn}[prop]{Definition}
\newtheorem{rem}[prop]{Remark} 
\newtheorem{ntn}[prop]{Notation} 
\newtheorem*{rem*}{Remark}

\theoremstyle{thm}
\newtheorem{thm}[prop]{Theorem}
\newtheorem*{thm*}{Theorem}
\newtheorem*{lem*}{Lemma}
\newtheorem{cor}[prop]{Corollary}
\newtheorem*{cor*}{Corollary}
\newtheorem*{prop*}{Proposition}
\newtheorem*{Mthm}{Main Theorem}

\theoremstyle{dfn}
\newtheorem*{dfn*}{Definition}
\newtheorem*{cl}{Claim}

\newcommand{\riso}{\xrightarrow{\sim}}
\newcommand{\liso}{\xleftarrow{\sim}}

\newcommand{\Spec}{\mathrm{Spec}\,}
\newcommand{\Spf}{\mathrm{Spf}\,}

\renewcommand{\sp}{\mathrm{sp}}
\renewcommand{\det}{\mathrm{det}}
\newcommand{\coker}{\mathrm{coker}\,}
\newcommand{\gr}{\mathrm{gr}}

\newcommand{\RR}{{\mathcal{R}}}
\newcommand{\FF}{{\mathcal{F}}}

\newcommand{\E}{{\mathcal{E}}}
\newcommand{\G}{{\mathcal{G}}}
\renewcommand{\H}{{\mathcal{H}}}
\newcommand{\M}{{\mathcal{M}}}
\newcommand{\NN}{{\mathcal{N}}}
\newcommand{\D}{{\mathcal{D}}}

\newcommand{\PP}{{\mathcal{P}}}
\newcommand{\QQ}{{\mathcal{Q}}}
\renewcommand{\O}{{\mathcal{O}}}

\newcommand{\V}{\mathcal{V}}
\newcommand{\T}{{\mathcal{T}}}
\newcommand{\Y}{\mathcal{Y}}
\newcommand{\ZZ}{\mathcal{Z}}
\newcommand{\X}{\mathcal{X}}

\newcommand{\U}{\mathcal{U}}

\newcommand{\A}{\mathbb{A}}
\renewcommand{\P}{\mathbb{P}}

\newcommand{\DD}{\mathbb{D}}
\renewcommand{\L}{\mathbb{L}}
\newcommand{\R}{\mathbb{R}}
\newcommand{\Q}{\mathbb{Q}}
\newcommand{\Z}{\mathbb{Z}}
\newcommand{\N}{\mathbb{N}}

\newcommand{\hdag}{  \phantom{}{^{\dag} }    }

\DeclareMathOperator{\coh}{Coh}
\DeclareMathOperator{\Mod}{Mod}
\DeclareMathOperator{\id}{id}

\newcommand{\SQ}[1]{\text{\normalfont(SQ)}_{#1}}
\newcommand{\Da}{\mr{Da}}
\newcommand{\shom}{\mc{H}\mr{om}}
\newcommand{\HH}{\mathcal{H}_{\mr{t}}}

\usepackage{color}
\newenvironment{meta}{
\noindent \color{red}
\sffamily[}{\upshape]}

\newcommand{\ff}[1]{(\!(#1)\!)}

\begin{document}

\title{Theory of weights in $p$-adic cohomology}
\author{Tomoyuki Abe and Daniel Caro}
\date{}

\maketitle

\begin{abstract}
 Let $k$ be a finite field of characteristic $p>0$. We construct a
 theory of weights for overholonomic complexes of arithmetic
 $\D$-modules with Frobenius structure on varieties over $k$. The
 notion of weight behave like Deligne's one in the $\ell$-adic
 framework: first, the six operations preserve weights, and secondly,
 the intermediate extension of an immersion preserves pure complexes and
 weights.
\end{abstract}

\tableofcontents

\section*{Introduction}
Using Grothendieck's theory of $\ell$-adic \'{e}tale cohomology of
varieties over a field of characteristic $p\neq \ell$, 
Deligne 
proved Weil's conjectures on the numbers of points of algebraic
varieties over finite fields. Moreover, he built a theory
of mixedness and weights for constructible $\ell$-adic sheaves which is compatible with
the ``six operations formalism'' of the $\ell$-adic cohomology, namely the mixedness and weight are 
stable under the operations $f _!, ~ f _*, f ^*, f ^!,
\otimes$ and $\shom$ (see \cite{deligne-weil-II}). Later, using the theory of perverse sheaves,  in
\cite{BBD} Gabber proved the stability of purity and mixedness under
intermediate extensions, with which the theory of weights for
$\ell$-adic cohomology may be regarded as complete. However, the problem
of obtaining similar results  within a $p$-adic cohomological framework
remained opened. After Dwork's $p$-adic proof of the rationality of zeta
functions, a part of Weil's conjectures, it seems natural to expect
better computability of zeta functions with a $p$-adic approach. 
In this paper, we build a $p$-adic theory of weights.

The first attempt to calculate the weights of some $p$-adic cohomology was
made by Katz and Messing in their famous paper \cite{KatzMessing}.
Using Deligne's deep
results on weights, especially {\em ``Le th\'eor\`eme du pgcd''},
they showed that, for projective smooth varieties, the weight of crystalline
cohomology is the same as that of $\ell$-adic one. It is
reasonable to hope that the coefficient theory of weights parallel to
$\ell$-adic cohomology should exist in the spirit of the {\em petit
camarade conjecture} \cite[1.2.10]{deligne-weil-II}, even though there
were many obstacles that prevented the construction of such a theory.
After the work of Katz-Messing, efforts were made until Kedlaya 
finally obtained in \cite{kedlaya-weilII} the expected estimation of weights of rigid cohomology, 
a $p$-adic cohomology constructed by Berthelot. 
We do not go into more
details of the history, and recommend the reader to consult the
excellent explanation in the introduction of \cite{kedlaya-weilII}.

But the context of rigid cohomology was still not completely satisfactory, in the sense that this is not ``functorial
enough'', namely we do not have six operations formalism, specially push-forwards as pointed out
in \cite[5.3.3]{kedlaya-weilII}. 
In this paper, we use systematically  
Berthelot's arithmetic $\D$-modules to complete
the program of establishing a $p$-adic theory of weights stable under six operations. 
In many applications, such a
theory should play as important roles as suggested by the
classical situations; for example,  the theory of intersection cohomology and its purity, 
the theory of Springer
representations, Lafforgue's proof of Langlands correspondence, etc. In
the final part of the paper, we show one such application that the
Hasse-Weil $L$-function for a function field defined by means of
$\ell$-adic methods and $p$-adic methods coincide.
\bigskip

Now let us explain our results in more details.
Let $\mathcal{V}$ be a complete discretely valued ring of mixed
characteristic $(0,p)$, $K$ its field of fractions, and $k$ its residue
field which is assumed to be a finite field. 
We suppose that there exists a lifting 
of the Frobenius of $k$ to $\V$.
In order to obtain a
$p$-adic cohomology on algebraic varieties over $k$ stable under the six
operations
whose coefficients contain
the category of overconvergent isocrystals, Berthelot constructed the
theory of arithmetic $\D$-modules (cf.\ \cite{Beintro2}). Arithmetic
$\D$-modules may be seen as a $p$-adic analogue of modules over the ring
of differential operators over complex varieties or analytic
varieties. Berthelot's theory is inspired by his construction of
rigid or crystalline cohomology, and we consider not only differential
operators of finite order, but also of infinite order subject to a
convergence condition. Even though the objects are defined, the
preservation of finiteness had been a difficult question. For this, the
notion of overholonomicity was introduced by the second author in
\cite{caro_surholonome}. Thanks to Kedlaya's semistable reduction
theorem of \cite{kedlaya-semistableIV}, the second author together with Tsuzuki, finally in
\cite{caro-Tsuzuki}, proved that overholonomicity of $\D$-modules with
Frobenius structure is preserved under various cohomological functors.

Using this foundation, in this paper, we go a step further in the
development of a good $p$-adic cohomology, and build a theory of weights
for overholonomic complexes with Frobenius structure of arithmetic
$\D$-modules. More precisely, the main result is the following (cf.\
Theorem \ref{sumupWeilII}, Corollary \ref{purityinterme}):
\begin{Mthm}
 Let $X$ be a (realizable) variety over $k$. 
 We construct the full subcategory
 $F\text{-}D^{\mr{b}}_{\mr{m}}(X/K)$ of
 $F\text{-}D^{\mr{b}}_{\mr{ovhol}}(X/K)$ of $\iota$-mixed complexes. The
 following properties hold:

 (i) Let $f\colon X\rightarrow
 Y$ be a morphism of varieties. The $\iota$-mixedness is stable under
 functors $f_+$, $f_!$, $f^+$, $f^!$, $\DD_X$,
 $\widetilde{\otimes}$. More precisely, $f_+$ and $f^!$ send
 $\iota$-mixed $F$-complexes of weight $\geq w$ to those of weight $\geq
 w$. Similarly, $f_!$ and $f^+$ send $\iota$-mixed $F$-complexes of
 weight $\leq w$ to those of weight $\leq w$, $\DD_X$ exchanges
 $\iota$-mixed $F$-complexes of weight $\leq w$ to $\geq -w$, and
 $\widetilde{\otimes}$ send $\iota$-mixed $F$-complexes of weight $(\geq
 w,\geq w')$ to $\geq w+w'$.

 (ii) Intermediate extension of an immersion preserves pure $F$-complexes
 and weights.
\end{Mthm}
Concerning part (i), 
for the convenience of the reader, we recall that, 
in the very special case where $Y = \Spec \, k$, 
the stability under $f _+$ was already 
checked by d\'{e}vissage in overconvergent $F$-isocrystals from Kedlaya's stability of weights
(see \cite[8.3.4]{caro_devissge_surcoh}
\footnote{In fact, the compatibility with Frobenius of the comparison between push-forward and rigid cohomology was implicitly used. 
A proof of this compatibility is given in \cite{Abe-Frob-Poincare-dual}}). 
The relative version
had been missing:
as Kedlaya remarked in \cite[5.3.3]{kedlaya-weilII}, one had at least to work with a theory admitting Grothendieck's six operations,
which is only possible in the theory of arithmetic $\D$-modules. 
To check part (i), 
we follow some ideas appearing in Deligne's proof of the stability under push-forwards of \cite{deligne-weil-II}, 
which leads us to the relative $1$ dimensional case. 
To tackle this relative $1$ dimensional case, specially concerning the stability of mixedness,
one needs to study thoroughly the monodromy filtration associated to a log convergent isocrystal (this is section $3$ of our work). 
One also notices that any proofs here do not use \cite{kedlaya-weilII}. 
Moreover, the part (ii) is completely new 
in $p$-adic cohomology theory.
\bigskip

Let us describe the contents of the paper. We restrict
our attention to the category of {\em realizable varieties}, the
full subcategory of the category of varieties over $k$ consisting of $Y$
such that there exists an immersion into a proper and smooth formal
scheme over $\V$. We recall that quasi-projective varieties are
realizable. In this introduction, let us simply call these
objects, varieties.

In \S1, we construct the coefficient theory using the foundational
results of overholonomic modules. Even though the main ideas had already
been provided, 
with the definition of a new t-structure here, 
the main theme of the first section is to extend some properties of the
theory of arithmetic $\D$-modules and, at the same time, to make it in a more usable form. Let $Y$
be a variety. First, we recall the
construction of the category $F\text{-}D ^\mathrm{b}_\mathrm{ovhol}
(Y/K)$ of overholonomic $F$-complexes over $Y/K$ and the formalism of
the six operations for these coefficients. 
We then introduce
a new t-structure on this category whose heart is denoted by
$F\text{-}\mathrm{Ovhol} (Y/K)$ and which may be seen as the category
of arithmetic $\D$-modules over $Y/K$ (even though they are complexes
and not modules). When $Y$ is smooth, we construct a fully faithful
subcategory $F\text{-}\mathrm{Isoc} ^{\dag \dag} (Y/K)$ of
$F\text{-}\mathrm{Ovhol} (Y/K)$ which is equivalent to the category of
overconvergent $F$-isocrystals over $Y/K$ (the latter are the coefficients of the rigid cohomology).
In the spirit of the Riemann-Hilbert
correspondence, $F\text{-}\mathrm{Ovhol} (Y/K)$ (resp.\
$F\text{-}\mathrm{Isoc} ^{\dag \dag} (Y/K)$) is a $p$-adic analogue of
the category of perverse sheaves (resp.\ smooth perverse sheaves) over
$Y$. We proceed to check several necessary properties for being six
operations formalism.
We conclude the section by defining the intermediate extension as in
the context of $\D$-modules over complex varieties. We also check that
the analogous properties concerning irreducibility hold in our context.
This enables us to define the $p$-adic intersection cohomology for
realizable varieties, which is an advantage of using arithmetic
$\D$-modules.

We fix an isomorphism $\iota\colon\overline{\mb{Q}}_p\cong\mb{C}$. In
\S2, we define the notion of $\iota$-mixedness as Deligne
does for the $\ell$-adic cohomology in \cite{deligne-weil-II}. More
precisely, in the spirit of the Riemann-Hilbert correspondence,
the $\iota$-purity of an object in $F\text{-}\mathrm{Isoc} ^{\dag \dag}
(Y/K)$ is by definition a pointwise property (with the natural
definition over a point) and the category of $\iota$-mixed objects of
$F\text{-}\mathrm{Isoc} ^{\dag \dag} (Y/K)$ is the full subcategory
generated by extensions of $\iota$-pure objects. Next, by d\'{e}vissage
in overconvergent $F$-isocrystals, we extend naturally the notion of
$\iota$-mixedness to $F\text{-}D ^\mathrm{b}_\mathrm{ovhol} (Y/K)$. At
the end of the section, we estimate the weight of the cohomology
on curves, using the methods developed in \cite{AM} by the first author
together with Marmora.

The aim of \S3 is to prove the following preliminary
results on the $\iota$-mixedness stability. Let $X$ be a proper smooth
variety, $Z$ be a strict normal crossing divisor of $X$ and $j\colon
Y:=X\setminus Z \to X$ be the open immersion. Let $\G$ be a
log-convergent $F$-isocrystal on $X$ with logarithmic poles along $Z$
which possesses nilpotent residues and let $\E$ be the object of
$F\text{-}\mathrm{Isoc} ^{\dag \dag} (Y/K)$ induced by restriction. If
$\E$ is $\iota$-mixed, then we show that  $j _! (\E)$ is $\iota$-mixed
as well. Moreover, when $Z$ is smooth and $\E$ is $\iota$-pure, we prove
that the weight of $j _! (\E)$ is less than or equal to that of
$\E$. The theorem follows from the study of the monodromy filtration
given by the nilpotent residue morphisms of $\G$.

In the last section \S4, we prove the stability of $\iota$-mixedness and
weights under the six operations. Roughly speaking, we reduce
to the case treated in the previous section by using Kedlaya's
semistable reduction theorem (see \cite{kedlaya-semistableIV}).
Since we need to take special care of the boundary in $p$-adic
cohomology theory compared to the case of $\ell$-adic theory, the proof
of this theorem is slight involved.
Finally, by using Noot-Huyghe's Fourier transform
(\cite{Noot-Huyghe-fourierI}), we establish the stability of the
$\iota$-purity under intermediate extension, which is a $p$-adic
analogue of Gabber's purity theorem. We conclude the section with a few
applications: the weight filtration for an $\iota$-mixed object of
$F\text{-}\mathrm{Ovhol}(Y/K)$, the semi-simplicity of $\iota$-mixed
$F$-complexes after forgetting Frobenius structures, and some
$\ell$-independence type results.

\subsection*{Acknowledgment}
The first author (T.A.) would like to thank F. Trihan for asking a
question whose answer is nothing but Corollary \ref{anstoTri}.
He would also like to thank O. Gabber for sending the original proof of
his purity theorem, and giving some opinion of his against what
(T.A.) was thinking. Finally (T.A.) would like to thank T. Saito for his
continuous encouragement and numerable suggestions. This work was
supported by Grant-in-Aid for Research Activity Start-up 23840006.

The second author (D.C.) is grateful to A. P\'{a}l.
P\'{a}l noticed that Fujiwara-Gabber's $\ell$-independence in the form
of Theorem \ref{FujGabbind} was a consequence of a $p$-adic version of
Gabber's purity theorem. His remark gave one motivation of the second
author to undertake this work on weight theory.

\subsection*{Notation and convention}
Throughout this paper, $\V$ is a complete discrete valuation ring with
mixed characteristic $(0,p)$, $K$ is its field of fractions, $k$ is its
residue field which is assumed to be perfect.
We suppose that there exists a lifting 
of the Frobenius of $k$ to $\V$.
In principle, we denote formal schemes by using script
fonts (e.g.\ $\X$), and the special fiber is denoted by the
corresponding capital letter (e.g.\ $X$). A $k$-variety is a separated
scheme of finite type over $k$.
We remark that without loss of generality, one could assume varieties over $k$ are reduced if needed.
We often use $\star$ to abbreviate when it is obvious what should be
put and do not want to introduce too many notations.
If there is no risk of confusion, we sometimes use $(-)$ by meaning
respectively. For example ``$p^{(\prime)}\cong q^{(\prime)}$''
means ``$p\cong q$ (resp.\ $p'\cong q'$)''.

In this paper, we consider one of the following situations:
\begin{itemize}
 \item[(A)] Let $q= p ^s$ be a power of $p$, and fix a lifting $\sigma$
	    of the $s$-th Frobenius $F$ of $k$ to $\V$ and
	    $K$. (\S\ref{firstsection})

 \item[(B)] We are in situation (A), and moreover, we assume that $k$ is
	    a finite field with $q$ elements, and $F$ is the $s$-th
	    Frobenius. We fix an isomorphism
	    $\iota\colon\overline{\mb{Q}}_p\cong\mb{C}$. (Except for
	    \S\ref{firstsection})
\end{itemize}

\section{Preliminaries on arithmetic $\D$-modules}
\label{firstsection}
Throughout this section, we consider situation (A) in Notation and
convention.
\subsection{Arithmetic $\D$-modules over frames and couples}
In this subsection, we fix terminologies, and recall the definitions of
six functors in the theory of arithmetic $\D$-modules. Fundamental
additional properties are established from the next subsection.

\begin{dfn}
 We define the following categories:

 (i) A {\em frame} $(Y,X,\PP)$ is the data consisting of a separated and
 smooth formal scheme  $\PP$ over $\V$, a closed subvariety $X$ of $P$,
 an open subscheme $Y$ of $X$. A morphism of frames $u = (b,a,f)\colon
 (Y,X,\PP) \to (Y',X',\PP')$ is the data consisting of morphisms  $b
 \colon  Y' \to Y$, $a\colon X' \to X$, $f\colon \PP' \to \PP$ such that
 $f$ induces the other ones.

 (ii) An {\em l.p.\ frame}\footnote{Abbreviation of ``locally proper
 frame''.} $(Y,X,\PP,\QQ)$ is the data consisting of a proper and smooth
 formal scheme  $\QQ$ over $\V$, an open formal subscheme $\PP$ of
 $\QQ$, a closed subvariety $X$ of $P$, an open subscheme $Y$ of $X$. A
 morphism of frames $u = (b,a,g,f)\colon (Y,X,\PP,\QQ) \to
 (Y',X',\PP',\QQ')$ is the data consisting of morphisms  $b \colon  Y'
 \to Y$, $a\colon X' \to X$, $g\colon \PP' \to \PP$, $f\colon \QQ' \to
 \QQ$ such that $f$ induces the other ones. A morphism of l.p.\ frames
 is said to be {\em complete} if $a$ is proper.
\end{dfn}

\begin{dfn}
 \label{defi-(d)plongprop}
 We define the category of {\em couples}\footnote{In
 \cite{caro:formal6}, couples are called ``properly realizable
 couples''. Contrary to {\it ibid.}, we only use this type of couples,
 so we simplify the name.}
 as follows:
 \begin{itemize}
  \item An object $(Y,X)$ is the data consisting of a $k$-variety $X$
	and an open subscheme $Y$ of $X$ such that there exists an l.p.\
	frame of the form $(Y,X, \PP, \QQ)$.

  \item A morphism $(Y',X')\to (Y,X)$ of couples is a
	data $(b,a)$ consisting morphisms of $k$-varieties $a \colon
	X'\to X$ and $b\colon Y' \to Y$ such that $b$ is induced by
	$a$.
	The morphism $u$ is said to be
	{\em complete} if $a$ is proper. Let $\mathsf{P}$ be a
	property of morphisms of schemes. We say\footnote{ The ``c-''
	stands for ``complete''.}
	that $u$ is {\em c-$\mathsf{P}$} if $u$ is complete and $b$
	satisfies $\mathsf{P}$. (e.g.\ $u$ is a c-immersion if $u$ is
	proper and $b$ is an immersion.)
 \end{itemize}
 For a couple $(Y,X)$, we sometimes denote by $\mb{Y}$ using the bold
 font corresponding to the first data of the couple.
\end{dfn}

Let $u =\colon \mb{Y} \to \mb{Y}'$ be a morphism of couples.
By definition, there exist 
some  l.p.\ frames of the form
$(Y,X, \PP,\QQ)$ and $(Y',X',\PP', \QQ')$ such that 
$ \mb{Y}=  (Y,X)$ and $\mb{Y'}=(Y',X')$.
By replacing $\QQ$ by $\QQ \times \QQ'$ (resp. $\PP$ by $\PP \times \PP'$)
if necessary, we check that 
there exists a morphism of l.p.\ frames
$(b,a,g,f)\colon(Y,X, \PP,\QQ) \to  (Y',X',\PP', \QQ')$ of $u$, namely a
morphism of l.p.\ frames such that $u=(b,a)$. Moreover, we may take $g$
and $f$ to be smooth. When $u$ is complete, we can even choose $g$,
$f$ to be proper.

\begin{dfn}
 The category of {\em realizable varieties} is the full subcategory of
 the category of varieties over $k$ consisting of $X$ such that there
 exists an immersion into a proper and smooth formal scheme over $\V$.
\end{dfn}

\begin{empt}
 \label{Frobformstr}
 Let $\mc{A}$ be an additive category endowed with an additive
 endofunctor $F^*\colon\mc{A}\rightarrow\mc{A}$. We define the category
 $F\text{-}\mc{A}$ of $F$-objects of $\mc{A}$ as follows: $F$-objects
 consist of $(\E,\Phi)$ where $\E$ is an object of $\mc{A}$, and
 $\Phi\colon\E\xrightarrow{\sim} F^*\E$ is an isomorphism of
 $\E$. Later, we take $F$ to be the Frobenius pull-back, and $\Phi$ is
 called the {\em Frobenius structure}.
 A morphism of $F$-objects is a morphism in
 $\mc{A}$ commuting with Frobenius structures. We have a {\em faithful}
 additive functor $\varrho\colon F\text{-} \mc{A}\rightarrow \mc{A}$
 defined by forgetting Frobenius structures. If $\mc{A}$ is abelian,
 $F\text{-}\mc{A}$ is abelian as well.

 Now, let $\mc{T}$ be a triangulated category.
 Assume given an additive endofunctor
 $F^*$ of $\mc{T}$. 
 The category
 $F\text{-}\mc{T}$ has a shift functor in an obvious way, and a triangle
 in $F\text{-}\mc{T}$ is said to be {\em distinguished} if it is
 distinguished triangle after we forget Frobenius structures. We simply
 say triangle instead of distinguished triangle for the rest of the
 paper. We caution, however, that $F\text{-}\mc{T}$ is {\em not}
 triangulated in general. More precisely, it only satisfies
 (TR1), (TR2) of \cite{HaRD}.
\end{empt}

\begin{empt}
 Before defining some notions, the notion of {\em overholonomic
 complexes} is defined in \cite[3.1]{caro_surholonome}, which we do not
 recall here. The derived category of overholonomic
 $\D^\dag_{\X,\Q}$-complexes is denoted by
 $D^{\mathrm{b}}_{\mathrm{ovhol}}(\X)$.
 Let $\X$ be a smooth formal scheme over $\V$, and $Z$ be a
 closed subscheme of $X$. Then the {\em localization functor}
 $\R\underline{\Gamma}_Z^\dag\colon
 D^{\mathrm{b}}_{\mathrm{ovhol}}(\X)\rightarrow
 D^{\mathrm{b}}_{\mathrm{ovhol}}(\X)$ is defined in
 \cite[2.2]{caro_surcoherent} (see also
 \cite[1.15]{caro_surholonome}). We note
 that Berthelot defined a similar functor in \cite{Beintro2}, but these
 functors are not known to coincide unless $Z$ is a divisor of $X$.
 There exists a {\em functorial cone} for the homomorphism
 $\R\underline{\Gamma}_Z^\dag\rightarrow\mr{id}$ by
 \cite[4.4.3]{caro-stab-sys-ind-surcoh}, which is
 denoted by $(^\dag Z)$. When we are given two closed subschemes $Z_1$,
 $Z_2$ of $X$, we have a canonical isomorphism
 $\R\underline{\Gamma}_{Z_1}^\dag\circ
 \R\underline{\Gamma}_{Z_2}^\dag\cong \R\underline{\Gamma}_{Z_1\cap
 Z_2}^\dag$ by \cite[2.2.8]{caro_surcoherent}. 
 Other important properties of the functor are that when we
 have a closed immersion of smooth formal schemes $i\colon
 \ZZ\hookrightarrow\X$, then canonically
 $i_+\circ i^!\cong\R\underline{\Gamma}_Z^\dag$ (cf.\
 \cite[3.1.12]{caro_surcoherent}), and if a coherent module $\E$ is
 supported on $Z$, then the canonical homomorphism
 $\R\underline{\Gamma}_Z^\dag(\E)\rightarrow\E$ is an
 isomorphism (cf.\ \cite[2.2.9]{caro_surcoherent}).

 \begin{dfn*}
 \label{prop-nota-Dsurhol}  
  (i) Let $(Y,X,\PP)$ be a frame. We define the category
  $D^{\mathrm{b}}_{\mathrm{ovhol}} (Y, \PP/K)$ to be the full
  subcategory of $D ^{\mathrm{b}} _{\mathrm{ovhol}} (\D ^\dag_{\PP,\Q})$
  consisting of complexes $\E$ such that there exists an isomorphism of the form $\E
  \riso\R \underline{\Gamma} ^\dag _{Y} (\E ):= \R
  \underline{\Gamma}^\dag _{X}\circ(\hdag X \setminus Y) (\E)$  
  (cf.\ \cite[3.2.1]{caro-2006-surcoh-surcv}). The category
  is triangulated.

  (ii) Now, let $\mb{Y}$ be a couple (cf.\
  \ref{defi-(d)plongprop}). Take an l.p.\ frame $(Y,X,\PP,\QQ)$ of
  $\mb{Y}$. By using Lemma \cite[2.5]{caro:formal6}, we can check that
  the category $D^\mathrm{b} _\mathrm{ovhol} (Y,\PP/K)$ does not depend,
  up to a canonical equivalence, on the choices of $\PP$ and $\QQ$ (i.e. this only depends on $\mb{Y}=(Y,X)$ which is fixed). 
  Hence, we can denote $D ^\mathrm{b} _\mathrm{ovhol}  (Y,\PP/K)$ by $D ^\mathrm{b}
  _\mathrm{ovhol} (\mb{Y}/K)$. Its objects are called
  {\em overholonomic complexes of arithmetic $\D$-modules on
  $\mb{Y}$}, or simply {\em overholonomic complexes on $\mb{Y}$}.

  (iii) Finally, let us consider Frobenius structure. On the category $D
  ^{\mathrm{b}}_{\mathrm{ovhol}} (Y,\PP/K)$, we have the Frobenius
  pull-back functor. Thus the results of paragraph \ref{Frobformstr}
  can be applied to get the category
  $F\text{-}D^{\mathrm{b}}_{\mathrm{ovhol}} (Y,\PP/K)$. We also have
  $F\text{-}D ^\mathrm{b} _\mathrm{ovhol} (\mb{Y}/K)$.
 \end{dfn*}
\end{empt}

\begin{rem*}
 Our category $F\text{-}D^{\mathrm{b}}_{\mathrm{ovhol}} (Y, \PP/K)$ is
 an analog of the category of ``Weil complexes''
 \cite[VI.18]{KW-Weilconjecture}, so even though it is not triangulated,
 it gives us a suitable formalism for the theory of weights.
\end{rem*}

\begin{empt}
 Let $\mb{Y}$ be a couple. Choose an l.p.\ frame $(Y,X,\PP,\QQ)$ of
 $\mb{Y}$.

 (i) We have the dual functor
 \begin{equation*}
  \DD _{Y, \PP}\colon F\text{-}D ^\mathrm{b} _\mathrm{ovhol}  (Y,\PP/K)
   \to F\text{-}D ^\mathrm{b} _\mathrm{ovhol}  (Y,\PP/K).
 \end{equation*}
 This is defined as follows: we denote by $\DD _{\PP}$ the $\D ^\dag
 _{\PP,\Q}$-linear dual (e.g.\ see \cite{virrion} or
 \cite{Beintro2}). We define $\DD_{Y,\PP}$ by $\DD_{Y, \PP}( \E) :=\R
 \underline{\Gamma} ^\dag _{Y} \circ \DD _{\PP}(\E)$ for any $\E
 \in  F\text{-}D ^\mathrm{b} _\mathrm{ovhol}  (Y,\PP/K)$. The functor
 $\DD_{Y,\PP}$ only depends on the couple $\mb{Y}$ (more precisely, 
 by using\footnote{We may check
 that this is not a vicious circle.}
 Lemma \ref{t-gen-coh-PXTindtPsurhol}, 
 we can check that the functor  $\DD_{Y,\PP}$ is compatible with the 
 canonical equivalences of categories of independence of 
\ref {prop-nota-Dsurhol}).
This functor, called the {\em
 dual functor},
 is denoted by
 \begin{equation*}
  \DD _{\mb{Y}}\colon
   F\text{-}D ^\mathrm{b} _\mathrm{ovhol} (\mb{Y}/K)
   \to
   F\text{-}D ^\mathrm{b} _\mathrm{ovhol} (\mb{Y}/K).
 \end{equation*}

 (ii) The tensor product $\smash{\overset{\L}{\otimes}} ^{\dag}_{\O
 _{\PP,\Q}}$ is defined on $F\text{-}D ^\mathrm{b}
 _\mathrm{ovhol} (\D ^\dag _{\PP,\Q})$.
 By \cite{caro-stab-prod-tens},
 the bifunctor $\smash{\overset{\L}{\otimes}} ^{\dag}_{\O
 _{\PP,\Q}}[-\dim(\PP)]$ is independent of the choice of $\PP$ and is
 denoted by
 \begin{equation*}
  \widetilde{\otimes}_{\mb{Y}}
   \colon 
   F\text{-}D ^\mathrm{b} _\mathrm{ovhol}  (\mb{Y}/K)
   \times
   F\text{-}D ^\mathrm{b} _\mathrm{ovhol}  (\mb{Y}/K)
   \to 
   F\text{-}D ^\mathrm{b} _\mathrm{ovhol}  (\mb{Y}/K).
 \end{equation*}
 This is called the {\em twisted tensor product functor}.
 For simplicity, we sometimes abbreviate $\widetilde{\otimes}_{\mb{Y}}$
 by $\widetilde{\otimes}$.
 
 (iii) We define the {\em tensor product} by
 \begin{equation*}
  (-)\otimes_{\mb{Y}}(-):=\DD_{\mb{Y}}\,\bigl(\DD_{\mb{Y}}(-)
   \widetilde{\otimes}_{\mb{Y}}\,\DD_{\mb{Y}}(-)\bigr)\colon
   D^{\mr{b}}_{\mr{ovhol}}(\mb{Y}/K)\times
   D^{\mr{b}}_{\mr{ovhol}}(\mb{Y}/K)\rightarrow
   D^{\mr{b}}_{\mr{ovhol}}(\mb{Y}/K).
 \end{equation*}
\end{empt}

\begin{rem*}
 In the spirit of Riemann-Hilbert correspondence, our
 $-\widetilde{\otimes}-$ corresponds to
 $\DD\bigl(\DD(-)\otimes\DD(-)\bigr)$ in the theory of constructible
 sheaves, which justifies our notation. 
\end{rem*}

\begin{empt}
 \label{operation-cohomoYXsurhol}
 Let $u \colon\mb{Y}\to\mb{Y}'$ be a morphism of couples. We have the
 following cohomological functors.

 (i) We have {\em ordinary} and {\em extraordinary pull-back by $u$}:
 \begin{equation*}
   u ^+ , \, u  ^{!} 
   \colon  
   F\text{-}D ^\mathrm{b} _\mathrm{ovhol} (\mb{Y}'/K)
   \to 
   F\text{-}D ^\mathrm{b} _\mathrm{ovhol} (\mb{Y}/K).
 \end{equation*}
 They are constructed as follows:
 choose a morphism of l.p.\ frames $\widetilde{u}=(\star,\star,g,\star)
 \colon (Y, X, \PP, \QQ) \to(Y', X', \PP', \QQ')$ of $u$.
 Then $\widetilde{u}^!  := \R \underline{\Gamma} ^\dag _{Y}  \circ g
 ^{!}$ does not depend on the choices of l.p.\ frames (for instance, see Lemma \cite[2.5]{caro:formal6}). 
 We define $u^!$
 to be $\widetilde{u}^!$. Using this, we define
 $\widetilde{u}^+:=\DD_{Y,\PP}\circ \widetilde{u}^!\circ \DD_{Y',\PP'}$,
 and $u ^+  := \DD _{\mb{Y}} \circ u ^{!}\circ \DD _{\mb{Y}'}$.

 (ii) If $u $ is complete (cf.\ Definition \ref{defi-(d)plongprop}),
 we have the {\em ordinary} and {\em extraordinary direct image by $u$}:
 \begin{equation*}
   u _{!}, \,u _{+}\colon  F\text{-}D ^\mathrm{b}
   _\mathrm{ovhol} (\mb{Y}/K)
   \to
   F\text{-}D ^\mathrm{b} _\mathrm{ovhol} (\mb{Y}'/K).
 \end{equation*}
 There are constructed as follows:
 choose a morphism of l.p.\ frames $\widetilde{u}=(\star,\star,g,\star)
 \colon (Y, X, \PP, \QQ) \to(Y', X', \PP', \QQ')$  for $u$.
 Then we put $\widetilde{u} _+ := g _+$, which does not depend on the
 choices and define $u_+:=\widetilde{u}_+$. As usual, we put
 $\widetilde{u}_!:=\DD_{Y',\PP'}\circ u_+\circ \DD_{Y,\PP}$, and $u _!
 := \DD_{\mb{Y}'} \circ u _{+} \circ \DD _{\mb{Y}}$.
\end{empt}

\begin{empt}
 \label{extrafunctordfn}
 Let us define some extra functors.

 (i) Let $\mb{Y}=(Y,X)$ and $\mb{Y}'=(Y',X')$ be couples.
 Let $\mb{Y}'':=(Y\times Y',X\times X')$, and denote by
 $p^{(\prime)}\colon\mb{Y}''\rightarrow\mb{Y}^{(\prime)}$ the
 projections. We define the {\em exterior tensor product} by
 \begin{equation*}
  (-)\boxtimes_K(-):=p^!(-)\widetilde{\otimes}_{\mb{Y}''}\,p'^!(-)
   \colon
   D^{\mr{b}}_{\mr{ovhol}}(\mb{Y}/K)\times
   D^{\mr{b}}_{\mr{ovhol}}(\mb{Y}'/K)\rightarrow
   D^{\mr{b}}_{\mr{ovhol}}(\mb{Y}''/K).
 \end{equation*}

 (ii) Let $\mb{Y}=(Y,X)$ be a couple. Let $U$ be a subscheme of $Y$,
 $\overline{U}$ be the closure of $U$ in $X$, $\mathbb{U}:=
 (U,\overline{U})$ and $u\colon \mb{U} \rightarrow \mb{Y}$
 be the canonical c-immersion. We put
 \begin{equation*}
  \R\underline{\Gamma}^\dag_U:=u_+\circ u^!
   \colon
   F\text{-}D ^\mathrm{b}_\mathrm{ovhol} (\mb{Y}/K)\rightarrow
   F\text{-}D ^\mathrm{b}_\mathrm{ovhol} (\mb{Y}/K).
 \end{equation*}
 Moreover, we put
 $(^\dag U) :=   \R\underline{\Gamma}^\dag _{ Y\backslash U}$.
 If  $U _1 $ and $U _2$ are subschemes of $Y$, we have
 $\R\underline{\Gamma}^\dag_{U _1} \circ \R\underline{\Gamma}^\dag_{U
 _2}\riso \R\underline{\Gamma}^\dag_{U _1 \cap U _2}$ and
 $(\hdag U _1) \circ (\hdag U _2)\riso(\hdag U _1 \cup U _2)$.
 When $Z$ is closed subscheme of $Y$, we have the following
 {\em localization triangle} of functors:
 \begin{equation*}
   \R\underline{\Gamma}^\dag_Z\rightarrow\mr{id}
   \rightarrow(^\dag Z)
   \xrightarrow{+1}.
 \end{equation*}

 (iii) Finally, let $(Y,X,\PP)$ be a frame, and $\E\in
 D^{\mr{b}}_{\mr{ovhol}}(Y,\PP/K)$. For a subscheme $Y'$ of $Y$, let
 $X'$ be the closure of $Y'$ in $X$ and
 $j=(\star,\star,\mr{id})\colon(Y',X',\PP)\rightarrow(Y,X,\PP)$ be the
 morphism of frames. We define $\E\Vert_{Y'}$ to be $j^!(\E)$.
\end{empt}

\begin{empt}
\label{relativedual}
 We recall the following useful properties:

 (i) Let $u \colon \mb{Y} \to \mb{Y}'$ be a morphism of couples. From
 \cite[2.1.9]{caro-stab-prod-tens},
 for any $\E,\FF \in F\text{-}D ^\mathrm{b} _\mathrm{ovhol}
 (\mb{Y}'/K)$, we have the canonical isomorphism in
 $F\text{-}D^\mathrm{b} _\mathrm{ovhol} (\mb{Y}/K)$:
 \begin{equation}
  \label{otimes-comm-u!}
   u ^{!}\bigl( \E \widetilde{\otimes}_{\mb{Y}'} \FF\bigr)\riso
   u ^{!} ( \E )\widetilde{\otimes}_{\mb{Y}}\,u ^{!} ( \FF ).
 \end{equation}
 
 (ii) Let $u \colon \mb{Y} \to \mb{Y}'$ be a c-proper morphism of
 couples. From the relative duality isomorphism together with
 \cite[4.14]{Abe-Frob-Poincare-dual}, we get the isomorphism $u _{!}
 \riso u _{+}$.
\end{empt}

\begin{lem}
\label{adju+u+couple}
Let $u \colon\mb{Y}\to\mb{Y}'$ be a complete morphism of couples.
We have adjoint pairs $(u^+,u_+)$ and $(u_!,u^!)$.
\end{lem}

\begin{proof}
 By transitivity with the composition, we reduce to the case where
 $u$ is a c-open immersion and $u$ is c-proper.
 In the first case, we may check
 easily that we have adjoint pairs $(u^+,u_+)$ (and then  $(u_!,u^!)$ by
 duality). In the second case, from the relative duality isomorphism, we
 need to check that we have adjoint pairs $(u_+,u^!)$, which follows
 from \cite{Vir04}.
\end{proof}

\subsection{t-structures}
In this subsection, we introduce a t-structure on the triangulated
category $F\text{-}D^{\mr{b}}_{\mr{ovhol}}(\mb{Y}/K)$ for a couple
$\mb{Y}$. The heart of this t-structure is an analogue of the category
of holonomic $\D$-modules for the theory of algebraic $\D$-modules.
\medskip

 Let $\mathbb{Y}= (Y,X)$ be a couple (cf.\ Definition
 \ref{defi-(d)plongprop}). Choose
 an l.p.\ frame $(Y,X,\PP,\QQ)$ of $\mb{Y}$. Let $Z$ be a closed
 subvariety of $P$ (the special fiber of $\PP$ following our convention) 
 so that $Z := X \setminus Y$. We set $\U := \PP
 \setminus Z$. Let us denote by $D ^{\leq 0} (\D^\dag_{\U,\Q})$ (resp.\
 $D ^{\geq 0} (\D^\dag_{\U,\Q})$)
 the strictly full subcategory of $D^{\mathrm{b}}(\D^\dag_{\U,\Q})$
 consisting of complexes $\E$ such that,
 for any $j \geq 1$ (resp.\ for any $j \leq -1$), we have $\H^{j}
 (\E)=0$. We denote by $\tau _{\leq 0}\colon D ^{\mathrm{b}}
 (\D^\dag_{\U,\Q}) \to D ^{\leq 0} (\D^\dag_{\U,\Q})$
 and  $\tau _{\geq 0}\colon D ^{\mathrm{b}} (\D^\dag_{\U,\Q})
 \to D ^{\geq 0} (\D^\dag_{\U,\Q})$ the usual truncation functors.

\begin{dfn}
 \label{nota-t-stru}
 (i) Let $D ^{\leq 0} (Y, \PP/K)$ (resp.\ $D ^{\geq 0} (Y, \PP/K)$) be
 the strictly full subcategory of $D ^{\mathrm{b}}_{\mathrm{ovhol}}
 (Y,\PP/K)$ of complexes $\E$ such that $\E |_\U \in D ^{\leq 0}
 (\D^\dag_{\U,\Q})$ (resp.\ $\E |_\U \in D ^{\geq 0}
 (\D^\dag_{\U,\Q})$).

 (ii) Let $\tau ^{(Y,\PP)} _{\leq 0}\colon D
 ^{\mathrm{b}}_{\mathrm{ovhol}}  (Y, \PP/K) \to D
 ^{\mathrm{b}}_{\mathrm{ovhol}}  (Y, \PP/K)$ and $\tau ^{(Y,\PP)}
 _{\geq 0}\colon D ^{\mathrm{b}}_{\mathrm{ovhol}} (Y, \PP/K) \to D
 ^{\mathrm{b}}_{\mathrm{ovhol}} (Y, \PP/K)$ be the functors defined by
 $\tau ^{(Y,\PP)} _{\leq 0}:= (\hdag Z) \circ \tau _{\leq 0}$, $\tau
 ^{(Y,\PP)} _{\geq 0}:= (\hdag Z) \circ \tau _{\geq 0}$.

 (iii) For any $n\in \Z$, we put $D ^{\leq n} (Y, \PP/K):= D ^{\leq 0}
 (Y, \PP/K)[-n]$, $D ^{\geq n} (Y, \PP/K):= D ^{\geq 0} (Y, \PP/K)[-n]$,
 $\tau ^{(Y,\PP)} _{\leq n}:= [-n] \circ \tau ^{(Y,\PP)} _{\leq 0} \circ
 [n]$ and $\tau ^{(Y,\PP)} _{\geq n}:= [-n] \circ \tau ^{(Y,\PP)} _{\geq
 0} \circ  [n]$.
\end{dfn}

\begin{rem}
 \label{rem-t-structure}
 (i) When we can choose $Z$ to be the support of a divisor of $P$, the
 definitions of \ref{nota-t-stru} are not really new.
 Indeed, in that case we get 
 $D ^{\leq 0} (Y, \PP/K) =D ^{\leq 0} (\D^\dag_{\PP} (\hdag Z)
 _{\Q})$ (resp.\ $D ^{\geq 0} (Y, \PP/K)= 
 D ^{\geq 0} (\D^\dag_{\PP} (\hdag Z) _{\Q})$.
 Moreover, in that case, we simply have
 $\tau ^{(Y,\PP)} _{\leq 0}=\tau _{\leq 0}$ 
 and
 $\tau ^{(Y,\PP)} _{\geq 0}=\tau _{\geq 0}$ (i.e.\ the functor $
 (\hdag Z)$ is useless).

 (ii) Now, suppose that $Z$ is a rational closed point of $\P ^{2} _k$,
 that $\PP= \QQ= \widehat{\mathbb{P}} ^{2} _{\V}$,
 $X= \P ^{2} _k$, and recall that $Y:=X\setminus Z$.
 Since $(\hdag Z) (\O _{\PP, \Q} ) | \U \riso \O _{\U, \Q}$, we get 
 $(\hdag Z) (\O _{\PP, \Q} ) \in D ^{\leq 0} (Y, \PP/K) \cap D ^{\geq 0} (Y, \PP/K)$. 
 However, by using the localization triangle with respect to $Z$
 of $\O _{\PP, \Q}$, we have $\tau _{\leq 0} ((\hdag Z) (\O _{\PP,
 \Q} ) )\riso \O _{\PP, \Q} $, which is not in $D
 ^{\mathrm{b}}_{\mathrm{ovhol}}  (Y, \PP/K)$. This means that the
 functor $(\hdag Z)$ is needed in the definition of $\tau
 ^{(Y,\PP)} _{\geq 0}$.
 Now, let $u$ be a lifting of the closed immersion $Z
 \hookrightarrow P$. We have $\R\underline{\Gamma}^\dag_Z(\O
 _{\PP, \Q})\cong u_+u^!(\O _{\PP, \Q})\cong u_+(K)[-2]$.
 Using the localization triangle again, we have $\tau
 _{\geq 1} ((\hdag Z) (\O _{\PP, \Q} )) \riso u _{+} (K)[-1]$.
 Hence, the object $(\hdag Z) (\O _{\PP, \Q} )$ is not a module in the
 usual sense.
\end{rem}

\begin{lem}
\label{iso-onU}
 Let $\phi \colon \E\to \FF$ be a homomorphism of 
 $D ^{\mathrm{b}}_{\mathrm{ovhol}}  (Y, \PP/K) $. Then, $\phi$ is an
 isomorphism if and only if $\phi |_\U$ is.
\end{lem}
\begin{proof}
 Let $\G$ be a mapping cone of $\phi$, is an object of $D
 ^{\mathrm{b}}_{\mathrm{ovhol}}  (Y, \PP/K) $. Since $\phi|_{\U}$ is an
 isomorphic, $\G$ is supported in $Z$. Thus, we have
 $\G\xleftarrow{\sim}\R\underline{\Gamma}^\dag_Z(\G)$.
 But since $\G$ is an object of $D ^{\mathrm{b}}_{\mathrm{ovhol}}  (Y,
 \PP/K) $ we get $\G\riso(\hdag Z ) (\G)$. Thus, we have
 $\G\xrightarrow{\sim}(\hdag Z)(\G)\xleftarrow{\sim}(\hdag Z)
 \bigl(\R\underline{\Gamma}^\dag_Z(\G)\bigr)\cong
 \bigl((\hdag Z)\circ\R\underline{\Gamma}^\dag_Z \bigr)(\G)
 =0$ where the last equality follows by the localization triangle
 \ref{extrafunctordfn} (ii).
\end{proof}

\begin{lem}
  \label{lem1-appendix}
 Let $\E \in D ^{\leq n} (Y, \PP/K)$ (resp.\ $\E \in D ^{\geq n} (Y,
 \PP/K)$). The canonical homomorphism $\tau ^{(Y,\PP)} _{\leq n} (\E)
 \to \E$ (resp.\ $\E \to  \tau ^{(Y,\PP)} _{\geq n} (\E)$) is an
 isomorphism. In particular, the essential image of $\tau ^{(Y,\PP)}
 _{\leq n}$ (resp.\ $\tau ^{(Y,\PP)} _{\geq n}$) is $D ^{\leq n} (Y,
 \PP/K)$ (resp. $D ^{\geq n} (Y, \PP/K)$).
\end{lem}
\begin{proof}
 This follows from Lemma \ref{iso-onU}.
\end{proof}

\begin{prop}
 The functors $\tau^{(Y,\PP)}_{\star}$ define a t-structure, called the
 {\em canonical t-structure}, on $D
 ^{\mathrm{b}}_{\mathrm{ovhol}}  (Y, \PP/K) $:

 (i) For any homomorphism $f \colon \E \to \FF$ such that $\E \in D
 ^{\leq 0} (Y, \PP/K)$, $\FF \in D ^{\geq 1} (Y, \PP/K)$, we have
 $f=0$.

  (ii) We have the inclusions $D ^{\geq 1} (Y, \PP/K) \subset D ^{\geq
  0} (Y, \PP/K)$, $D ^{\leq 0} (Y, \PP/K) \subset D ^{\leq 1} (Y,
  \PP/K)$.

  (iii) For any $\E\in D ^{\mathrm{b}}_{\mathrm{ovhol}}  (Y, \PP/K) $,
  we have the distinguished triangle
  \begin{equation*}
   \tau ^{(Y,\PP)} _{\leq 0} (\E) \to \E \to \tau ^{(Y,\PP)} _{\geq 1}
    (\E) \xrightarrow{+1}.
  \end{equation*}
\end{prop}

\begin{proof}
 Since $\tau ^{(Y,\PP)} _{\leq 0} (\FF)$ has its support in $Z$, $\tau
 ^{(Y,\PP)} _{\leq 0} (\FF)=0$. From Lemma \ref{lem1-appendix}, we get
 that $f$ factorizes through $\tau ^{(Y,\PP)} _{\leq 0} (f)=0$. Hence we
 get the first assertion.  The other ones are obvious.
\end{proof}

\begin{dfn}
 We denote by $\mathrm{Ovhol} (Y, \PP/K) $ the heart of the canonical
 t-structure on $D ^{\mathrm{b}}_{\mathrm{ovhol}}  (Y, \PP/K) $. We
 define for any integer $n$, the $n$-th cohomology functor $\HH^{n}
 \colon D^{\mathrm{b}}_{\mathrm{ovhol}}  (Y, \PP/K)  \to \mathrm{Ovhol}
 (Y, \PP/K) $ by putting $\HH ^{n} (\E): = \tau ^{(Y,\PP)} _{\leq 0}
 \tau ^{(Y,\PP)} _{\geq 0} (\E [n])$ for any $\E \in D
 ^{\mathrm{b}}_{\mathrm{ovhol}}  (Y, \PP/K) $.
\end{dfn}

Let $(Y,X,\PP,\QQ)$ be an l.p.\ frame. An important example of object
of $\mr{Ovhol}(Y,\PP/K)$ is $(^\dag Z)\O_{\PP,\Q}$ where $Z:=X\setminus
Y$.

\begin{rem}
 \label{rem-scrHcalH}
 (i) Let $\E \in D ^{\mathrm{b}}_{\mathrm{ovhol}}  (Y, \PP/K) $. 
 We get a natural morphism $\H ^{n} (\E) \to \HH ^{n} (\E)$, where $\H
 ^{n} (\E)$ is the usual $n$-th cohomology functor. 
 When $Z$ can be chosen to be the support of a divisor of $P$, this morphism is an isomorphism. 
 But in general this is not the case (e.g.\ see the example given in the
 second remark \ref{rem-t-structure}). However, the induced morphism
 \begin{equation}
  \label{calHscrH}
   (\hdag Z) \circ \H ^{n} (\E) \to \HH ^{n} (\E),
 \end{equation}
 is an isomorphism since this is the case outside $Z$.

 (ii) 
 For any $n \in \Z$, $\E \in D ^{\mathrm{b}}_{\mathrm{ovhol}}  (Y,
 \PP/K) $, we have $\HH ^n (\E) |_\U \riso \mathcal{H} ^n (\E
 |_\U)$. Hence, from Lemma \ref{iso-onU}, $\HH ^n (\E) =0$ if
 and only if $\H ^n (\E |_\U)=0$.

 (iii) When $Z$ can be chosen to be the support of a divisor of $P$, 
 the category $\mathrm{Ovhol} (Y, \PP/K) $ is simply the category of 
 overholonomic $\D^\dag_{\PP} (\hdag Z) _{\Q}$-modules, i.e. 
 the (usual) heart of the category 
 $D ^{\mathrm{b}} _{\mathrm{ovhol}} (\D^\dag_{\PP} (\hdag Z) _{\Q})$.

 (iv) 
 We remark that the category $\mathrm{Ovhol} (Y, \PP/K) $ is also the
 strictly full subcategory of $D
 ^{\mathrm{b}}_{\mathrm{ovhol}}(Y, \PP/K) $ consisting of complexes
 $\E$ such that $\mathcal{H} ^i (\E|_\U)=0$ for $i\neq0$. By Lemma
 \ref{lem1-appendix}, the category $\mathrm{Ovhol} (Y, \PP/K) $ is
 also the strictly full subcategory of
 $D^{\mathrm{b}}_{\mathrm{ovhol}}(Y, \PP/K) $ of complexes $\E$ such
 that $\HH^i(\E)=0$ for $i\neq0$. By \ref{calHscrH}, this implies the
 inclusion $\mathrm{Ovhol} (Y, \PP/K)\subset D ^{\geq 0}
 (\D^\dag_{\PP,\Q})$
 (remark that with the example given in the second remark
 \ref{rem-t-structure}, the objects of $\mathrm{Ovhol}(Y, \PP/K)$ are
 not in general modules).
 Indeed, by using Mayer-Vietoris exact triangles and an induction on
 the number of divisors whose intersection is $Z$, we reduce to the
 case where $Z$ is a divisor.
\end{rem}

Next, we complete \cite[4.2.3]{caro-image-directe} with the following
lemma:
\begin{lem}
 \label{t-gen-coh-PXTindtPsurhol}
 Let $u =(\mr{id},\star,\star,\star)\colon (Y, X', \PP', \QQ')\to (Y, X,
 \PP, \QQ)$ be a complete morphism of l.p.\ frames.
\begin{enumerate}
 \item \label{t-gen-coh-PXTindtPsurhol-i} 
       For any $\E \in \mathrm{Ovhol} (Y,\PP/K)$, $\E '\in
       \mathrm{Ovhol}(Y,\PP'/K)$, for any $n \in \Z\setminus \{0\}$,
       \begin{equation*}
	 \HH ^n \bigl( u ^! (\E') \bigr) =0,\qquad
	 \HH ^n \bigl(u _+(\E)\bigr) =0.
       \end{equation*}

 \item \label{t-gen-coh-PXTindtPsurhol-ii} 
       The functors $\HH ^0 u ^! $ and $\HH ^0 u _+$ (resp.\ $u ^! $ and
       $u _+$) induce equivalence of categories between
       $\mathrm{Ovhol} (Y,\PP/K)$ and $\mathrm{Ovhol} (Y,\PP'/K)$
       (resp.\ between $F\text{-}D ^\mathrm{b}_\mathrm{ovhol}
       (Y,\PP/K)$ and $F\text{-}D ^\mathrm{b}_\mathrm{ovhol}
       (Y,\PP'/K)$).

\item \label{t-gen-coh-PXTindtPsurhol-iii}
      We have canonical isomorphisms $u _! \riso u _+$ and $u ^{!} \riso
      u ^{+}$.
\end{enumerate}
\end{lem}

\begin{proof}
 By Remark \ref{rem-scrHcalH}, the two first statements follows by
 \cite[4.2.3]{caro-image-directe}. Concerning the last one, we remark
 that the isomorphism $u _! \riso u _+$ is equivalent to the other one
 $u ^{!} \riso u ^{+}$. Moreover, by transitivity, we come down to the
 case where $g$ is proper and is an open immersion. In the proper
 case we have $u_!\xrightarrow{\sim}u_+$, and in the open immersion
 case, we have $ u ^{!} \riso u ^{+}$.
\end{proof}

\begin{dfn}
 \label{nota-129}
 (i)  Let $\mb{Y}$ be a couple. Take an l.p.\ frame $(Y,X,\PP,\QQ)$ of
 $\mb{Y}$. By Lemma \ref{t-gen-coh-PXTindtPsurhol}, the t-structure of
 $D ^\mathrm{b} _\mathrm{ovhol}  (Y,\PP/K)$ is compatible with canonical
 equivalence of categories, and independent on the choices of $\PP$ and
 $\QQ$ (cf.\ Definition \ref{prop-nota-Dsurhol}). This makes $D
 ^\mathrm{b}_\mathrm{ovhol} (\mb{Y}/K)$ a triangulated category with
 t-structure. Its heart is denoted by $\mr{Ovhol}(\mb{Y}/K)$.
 The objects in $\mr{Ovhol}(\mb{Y}/K)$ are called {\em overholonomic
 arithmetic $\D$-modules on $\mb{Y}$}, or {\em overholonomic $\D
 ^{\dag} _{\mb{Y}}$-modules}. For simplicity, we often call them {\em
 $\D ^{\dag} _{\mb{Y}}$-modules}, or {\em modules on $\mb{Y}$}.

 (ii)  Let $u\colon \mathbb{Y} '\to \mathbb{Y}$ be a morphism of
 couples. We define two functors $\mathrm{Ovhol} (\mathbb{Y}/K)  \to
 \mathrm{Ovhol} (\mathbb{Y}'/K) $ by $ u ^{!0}:= \HH ^{0} \circ u ^{!}$
 and $ u ^{+0}:= \HH ^{0} \circ u ^{+}$. When $u$ is complete, we define
 two functors $\mathrm{Ovhol} (\mathbb{Y}'/K)  \to \mathrm{Ovhol}
 (\mathbb{Y}/K) $ by setting $u _{+}  ^{0}:= \HH ^{0} \circ u _{+}$ and
 $u _{!}  ^{0}:= \HH ^{0} \circ u _{!}$.

 (iii)  Let $j\colon (Y',X') \to (Y,X)$ be a morphism of couples such
 that $X ' \to X$ is an immersion. We denote $j^{!0}$ by $|_{(Y',X')}$.
\end{dfn}

\begin{rem}
\label{rem-faithful}
Let $j\colon (Y,X') \to (Y,X)$ be a morphism of couples
 such that $X ' \to X$ is an open immersion.
Then a sequence $\E ' \to \E \to \E''$ in $\mathrm{Ovhol} ((Y,X)/K)$ is
 t-exact if and only if so is the sequence $\E ' |_{(Y,X')}\to \E
 |_{(Y,X')}\to \E''|_{(Y,X')}$.
 Moreover, an object $\E \in \mathrm{Ovhol}(\mathbb{Y}/K)$ is $0$ if and
 only if $\E |_{(Y,X')}$ is $0$.
 In particular, the functor $|_{(Y,X')} \colon \mathrm{Ovhol} ((Y,X)/K)
 \to \mathrm{Ovhol} ((Y,X')/K)$ is faithful.
\end{rem}

\begin{dfn}
  \label{t-exactness}
 Let $\mathbb{Y} , ~ \mathbb{Y}'$ be two couples and $\phi \colon D
 ^\mathrm{b} _\mathrm{ovhol} (\mb{Y}/K) \to D ^\mathrm{b}
 _\mathrm{ovhol} (\mb{Y}'/K)$ be a functor.
 We say that $\phi$ is {\em left t-exact} (resp.\ {\em right t-exact},
 resp.\  {\em t-exact}) if, for any $\E \in \mathrm{Ovhol} (Y,\PP/K)$
 and any integer $n\in \Z$ such that $n \leq -1$ (resp.\ $n \geq 1$,
 resp.\ $n \not = 0$), we have $\HH ^n \circ \phi (\E')  =0$.
\end{dfn}

\begin{empt}
 \label{glueing-Ovhol}
 Let $(Y^{(\prime)},X^{(\prime)})\rightarrow(Y'',X'')$ be morphisms of
 couples. We denote $(Y,X)\times_{(Y'',X'')}(Y',X'):=
 (Y\times_{Y''}Y',X\times_{X''}X')$.
 Let $\bigl\{\mathbb{Y} _i\bigr\} _{i \in I}$ be a c-open covering of
 $\mathbb{Y}$, 
 namely, we have the c-open immersions
 $(Y _i,X _i)=\mb{Y}_i\rightarrow\mb{Y}$ such that $\{Y_i\} _{i \in I}$ is an open covering
 of $Y$. 
 We put $ \mathbb{Y} _{ij}:= \mathbb{Y} _{i} \times
 _{\mathbb{Y}}\mathbb{Y} _{j}$ and $\mathbb{Y}
 _{ijk}:= \mathbb{Y} _{ij} \times _{\mathbb{Y}}\mathbb{Y} _{k}$. For any
 $i,j,k \in I$, we denote by $u _i \colon \mathbb{Y} _i \to \mathbb{Y}
 $, $u _{ij} \colon \mathbb{Y} _{ij} \to \mathbb{Y} _i$, $v _{ij}:= u _i
 \circ u _{ij}$, $u _{ijk} \colon \mathbb{Y} _{ijk} \to \mathbb{Y}
 _{ij}$ the induced c-open immersions. 

 Now, we define $\mathrm{Ovhol}\bigl (\{\mathbb{Y} _i\} _{i \in
 I}/K\bigr) $ to be the category whose objects are the data of objects
 $\E _i \in \mathrm{Ovhol} (\mathbb{Y} _i/K) $ endowed with isomorphisms
 of the form $\theta _{ji} \colon u ^{!0} _{ij} ( \E _i) \riso u ^{!0}
 _{ji} ( \E _j)$ which satisfy the cocycle condition $u ^{!0} _{ijk}
 (\theta_{ki} ) = u ^{!0} _{jki} (\theta _{kj} ) \circ u ^{!0} _{ijk}
 (\theta_{ij} )$, for any $i,j,k\in I$. A morphism $(\E _i , \theta
 _{ij}) \to (\E '_i , \theta '_{ij}) $ in $\mathrm{Ovhol} \bigl(
 \{\mathbb{Y}_i\} _{i \in I}/K\bigr) $ is a collection of morphisms $f
 _i \colon \E _i \to \E ' _i$ in $ \mathrm{Ovhol} (\mathbb{Y} _i/K) $
 such that $u ^{!0} _{ji} (f _i ) \circ \theta _{ij} = \theta '_{ij}
 \circ u ^{!0} _{ij} (f _{j})$, for any $i,j\in I$.

 \begin{lem*}[Gluing]
  The canonical functor $\mathrm{Ovhol}
  (\mathbb{Y}/K)  \to \mathrm{Ovhol}\bigl(\{\mb{Y} _i\} _{i \in
  I}/K\bigr) $ is an equivalence of categories.
  \end{lem*}
 \begin{proof}
  By the quasi-compactness of $Y$, we may assume $I$ to be a finite
  set. We can construct a canonical quasi-inverse functor as follows:
  let $\{\E _i , \theta _{ij}\}\in \mathrm{Ovhol} ((\mathbb{Y}_i) _{i
  \in I}/K)$. We denote by $d _1, d _2\colon \prod _{i \in I} u
 ^{0} _{i +} \E _i \to \prod _{i ,j\in I} v ^{0} _{ij +} u ^{!0} _{ij}
 ( \E _i)$ the morphisms such that $d _1$ is induced by $u ^{0} _{i +}
 \E _i \to v ^{0} _{ij +} u ^{!0} _{ij}  ( \E _i)$ and $d _2$ is induced
 by $u ^{0} _{j +} \E _j \to v ^{0} _{ji +} u ^{!0} _{ji}  ( \E
 _j)\xrightarrow[\theta _{ij}]{\sim} v ^{0} _{ij +} u ^{!0} _{ij}  ( \E
 _i)$. Then the canonical quasi-inverse functor is by definition the
  kernel of $d _1 -d _2$. 
 \end{proof}
\end{empt}

\begin{empt}
 \label{Frobformtstr}
 Recall the situation of paragraph \ref{Frobformstr}.
 Let $\mc{T}$ be a triangulated category with t-structure
 $\tau_*$. Assume given an additive endofunctor
 $F^*$ of $\mc{T}$ which is assumed to be t-exact. 
 The truncation functor $\tau_{*}$
 lifts to a functor from $F\text{-}\mc{T}$ to itself since $F^*$ is
 assumed to be t-exact. By construction, $\tau_*$ commutes with
 $\varrho$. As usual, we put $\H^n:=\tau_{\leq n}\circ\tau_{\geq n}$,
 and are able to define the ``heart'' of $F\text{-}\mc{T}$, in an
 obvious manner. This heart is nothing but
 $F\text{-}\mr{Heart}(\mc{T})$, and in particular, it is abelian.

 Now, the Frobenius pull-back functor on the category
 $D^{\mathrm{b}}_{\mathrm{ovhol}} (Y,\PP/K)$ is t-exact. Hence, we may
 apply the abstract non-sense above, and get that the heart of
 $F\text{-}D^{\mathrm{b}}_{\mathrm{ovhol}} (Y,\PP/K)$ (resp.\
 $F\text{-}D ^\mathrm{b} _\mathrm{ovhol} (\mb{Y}/K)$) is $F
 \text{-}\mathrm{Ovhol} (Y,\PP/K)$ (resp.\ $F \text{-}\mathrm{Ovhol}
 (\mathbb{Y}/K)$).
 The category $F \text{-}\mathrm{Ovhol} (\mathbb{Y}/K)$ is noetherian
 and artinian. Indeed, we reduce to the case where $Y=X=P$ in which case
 it follows by \cite[5.4.3]{Beintro2}.
\end{empt}

\begin{empt}
 \label{isoc-nota}
 Let $\mb{Y}=(Y,X)$ be a couple such that $Y$ is smooth. Let
 $Z:=X\setminus Y$. When there exists a divisor $W$ of $P$ such that
 $Z=W\cap X$, a strictly full subcategory $\mr{Isoc}^{\dag \dag}
 (Y,\PP/K) $ of the category of coherent $\D ^\dag_{\PP} (\hdag Z)
 _{\Q}$-modules with support in $X$ is defined in
 \cite{caro-pleine-fidelite}. Moreover, we have the equivalence
 \begin{equation*}
   \sp _{+}\colon\mr{Isoc}^{\dag}(Y,X/K)\xrightarrow{\sim}
   \mr{Isoc}^{\dag \dag}(Y,\PP/K).
 \end{equation*}
 We may endow with Frobenius structure: $F\text{-}\mr{Isoc}^{\dag \dag}
 (Y,\PP/K)$ is a full subcategory of $F\text{-}\mathrm{Ovhol}
 (Y,\PP/K)$ (cf.\ \cite{caro-Tsuzuki}), and the above equivalence commutes with Frobenius structures , i.e. induces
 the equivalence
 $\sp _{+}\colon F\text{-}\mr{Isoc}^{\dag}(Y,X/K)\xrightarrow{\sim}
   F\text{-}\mr{Isoc}^{\dag \dag}(Y,\PP/K)$.

 Let us generalize this latter equivalence with Frobenius structures
 (because in our context we need to stay with overholonomic complexes)
 to arbitrary $Z$.
 We define the
 category $F\text{-}\mr{Isoc}^{\dag
 \dag}  (Y,\PP/K)$ to be the strictly full subcategory of
 $F\text{-}\mathrm{Ovhol} (Y,\PP/K)$ consisting of objects $\E$ such
 that $\E |_\U \in F\text{-}\mr{Isoc}^{\dag \dag}  (Y,\U/K) $. The
 category $\mathrm{Isoc} ^{\dag \dag}(Y,\PP/K)$  only depends on
 $\mb{Y}$ and $K$ and can be denoted by $\mathrm{Isoc}^{\dag\dag}
 (\mb{Y}/K)$. Moreover, we denote by $D ^{\mathrm{b}} _{\mathrm{isoc}}
 (Y, \PP/K)$ the full subcategory of $D ^{\mathrm{b}}_{\mathrm{ovhol}}
 (Y,\PP/K)$ consisting of complexes $\E$ such that $\HH ^{j}(\E) \in
 \mr{Isoc}^{\dag \dag}(Y,\PP/K)$ for any integer $j$.
 The category $D^{\mathrm{b}}_{\mr{isoc}} (Y,\PP/K)$ depends only
 on $\mb{Y}$ and $K$, and can be denoted by
 $D^{\mr{b}}_{\mr{isoc}}(\mb{Y}/K)$.

 (i) We have the equivalence of categories: 
 \begin{equation*}
  F\text{-}\mr{Isoc}^{\dag}  (\mathbb{Y}/K)\cong
   F\text{-}\mr{Isoc}^{\dag \dag}  (\mathbb{Y}/K).
 \end{equation*}
 Indeed, by gluing lemma \ref{glueing-Ovhol} and the gluing for
 overconvergent isocrystals, we can reduce to the case where $Z$ is the
 intersection of a divisor on $P$ with $X$, in which case we have
 already recalled above.

 (ii) Let $E \in F\text{-}\mr{Isoc}^{\dag}(\mb{Y}/K)$, $y$ be a closed
 point of $Y$. Let $d$ be the dimension of $Y$ locally around $y$. By
 using \cite[5.6]{Abe-Frob-Poincare-dual}, we have the isomorphisms
 \begin{equation}
  \label{!et*-isoc}
   i _y ^{+}\bigl(\sp_+(E)\bigr)\riso
   i_y^!\bigl(\sp_+(E)\bigr)(d)[2d]
   \riso
   \sp_+\bigl(i _y ^{*}(E)\bigr)(d)[d],
 \end{equation}
 where $i_y\colon(\{y\},\{y\})\rightarrow\mb{Y}$ denotes the canonical
 morphism.
\end{empt}

\subsection{Properties of six functors for couples}
In this subsection, we prove some fundamental properties for six
functors defined in the previous subsection.

\begin{prop}
 \label{exactnessofdual}
 Let $\mb{Y}$ be a couple. Then the dual functor $\DD_\mb{Y}$ is t-exact
 (cf.\ Definition \ref{t-exactness}).
\end{prop}
\begin{proof}
 Let $(Y,X,\PP,\QQ)$ be an l.p.\ frame of $\mb{Y}$. 
 By Remark \ref{rem-scrHcalH} (ii),
 we reduce to the case where $Y=X$ and then from Berthelot-Kashiwara theorem (see \cite[5.3.3]{Beintro2}) 
 to the case where $Y=X=P$. The lemma
 follows from the exactness of $\DD_{\PP}$ (see \cite{virrion}).
\end{proof}

\begin{prop}
 \label{propofpullbacks}
 Let $u\colon\mb{Y}=(Y,X)\rightarrow\mb{Y}'=(Y',X')$ be a morphism of couples.

 (i) If $u$ is c-smooth of relative dimension $d$ such that the fibers
 are equidimensional, then $u ^+[d]$ and $u ^![-d]$ are t-exact.

(ii) If $u$ is a c-immersion, 
 then $u^!$ (resp.\ $u^+$) is left t-exact (resp.\
 right t-exact). Moreover, we have the canonical isomorphism
 $ u ^! \circ u _+ \riso id$.

 (iii) (Kashiwara's theorem)\
 Suppose $u$ is a c-closed immersion. Then $u _+$ is
 t-exact. Moreover, $u _+$ (resp.\ $u _+ ^0$) is fully
 faithful. The objects of the essential image of $u _+$ (resp.\ $u _+
 ^0$) are called ``with support in $\mathbb{Y}$''. Restricted to objects
 with support in $\mathbb{Y}$, the functor $u ^!$ is t-exact and  $u ^!$
 (resp.\ $u ^{!0}$) is canonically a quasi-inverse to $u _+$ (resp.\ $u
 _+ ^0$).
\end{prop}
\begin{proof}
 Let us show (i).  
 By Remark \ref{rem-scrHcalH} (ii), 
 we may assume that $u$ is of the form
 $u=(b,b)\colon(Y,Y)\rightarrow(Y',Y')$. The problem
 is local both on $Y$ and $Y'$, so we may assume that there exists a
 factorization $Y\rightarrow Y'\times\mb{A}^d\xrightarrow{p}Y'$ where
 the first morphism is \'{e}tale and the second is the projection. For
 the projection case, take a closed embedding $Y'\hookrightarrow\PP'$,
 and let
 $\widetilde{p}\colon\PP'\times\widehat{\mb{A}}^d\rightarrow\PP'$. Then
 $p^!$ is induced by $\widetilde{p}^!$, which is exact after shifting by
 the flatness of $\widetilde{p}$. Thus, we may assume that $b$ is
 \'{e}tale. Take a closed point $y\in Y$. It suffices to show the
 exactness around $y$. By the structure theorem of \'{e}tale morphism
 (cf.\ EGA IV, Theorem 18.4.6), locally around $y$ and $b(y)$, we can
 take the following cartesian diagram
 \begin{equation*}
  \xymatrix{
   Y\ar@{^{(}->}[r]\ar[d]_{b}\ar@{}[rd]|\square&
   \PP\ar[d]^{\widetilde{b}}\\
  Y'\ar@{^{(}->}[r]&\PP'
   }
 \end{equation*}
 where $\PP$ and $\PP'$ are smooth formal schemes, and $\widetilde{b}$
 is \'{e}tale. By the flatness of $\widetilde{b}$, the exactness
 follows.

 Let us show (ii) and (iii). By duality, we may concentrate on showing the $u^!$
 case. By using Lemma \ref{t-gen-coh-PXTindtPsurhol}, we may assume that
 $u$ is of the form $(b,\mr{id})\colon(Y,X)\rightarrow(Y',X)$, which reduce easily 
 to already known cases (e.g. see Kashiwara's theorem proven by
 Berthelot in \cite{Beintro2}).
\end{proof}

\begin{prop}
 (i) Using the notation of paragraph \ref{extrafunctordfn} (i),
 we have a canonical isomorphism $(-)\boxtimes_K(-)\cong
 p^+(-)\otimes_{\mb{Y}''}p'^+(-)$.

 (ii) Exterior tensor products are t-exact.
\end{prop}
\begin{proof}
 Let us show (i). Take l.p.\ frames $(Y,X,\PP,\QQ)$ and
 $(Y',X',\PP',\QQ')$. Let $\widetilde{p}^{(\prime)}\colon
 \PP\times\PP'\rightarrow\PP^{(\prime)}$. Let $\E^{(\prime)}\in
 D^{\mr{b}}_{\mr{ovhol}}(\D^\dag_{\PP^{(\prime)},\Q})$.
 Consider the homomorphism
 \begin{equation*}
  \R\underline{\Gamma}^\dag_{Y\times Y'}
   \bigl(\E\boxtimes_K\E'\bigr)
   \rightarrow\E\boxtimes_K\E'.
 \end{equation*}
 When $\E^{(\prime)}\in D^{\mr{b}}_{\mr{ovhol}}(Y^{(\prime)},X/K)$, 
 this is an isomorphism.
 Indeed, the right side is supported on $\overline{Y\times Y'}$, 
 the closure of $Y\times Y'$ in $P\times P'$. Putting
 $Z:=\overline{Y\times Y'}\setminus Y\times Y'$, it remains to show that
 $(^\dag Z)(\E\boxtimes\E')=0$. This follows by
 \ref{otimes-comm-u!}. Thus the proposition is reduced to showing that
 there exists an isomorphism
 $\DD_{\PP\times\PP'}\bigl(\E\boxtimes\E'\bigr)
 \cong\DD_{\PP}(\E)\boxtimes\DD_{\PP'}(\E')$ for $\E^{(\prime)}\in
 D_{\mr{perf}}(\D^\dag_{\PP^{(\prime)},\Q})$. Since
 $\widehat{\D}^{(m)}_{\PP}$ is of finite cohomological
 dimension, by \cite[I.4]{virrion}, it suffices to show that
 $\DD^{(m)}_{\PP\times\PP'}\bigl(\E^{(m)}\boxtimes\E'^{(m)}\bigr)
 \cong\DD^{(m)}_{\PP}(\E^{(m)})\boxtimes\DD^{(m)}_{\PP'}(\E'^{(m)})$ for
 $\E^{(\prime)(m)}\in D_{\mr{perf}}(\widehat{\D}^{(m)}
 _{\PP^{(\prime)}})$. By passing to the limit, we are reduced to the
 following fact, whose verification is easy: let $\mc{R}$ be a
 commutative ring on a topos, and $\mc{A}$, $\mc{B}$ be
 flat $R$-algebras such that $\mc{R}$ is the center of
 them. Let $\mc{M}\in D^{\mr{b}}_{\mr{perf}}(\mc{A})$, $\mc{N}\in
 D^{\mr{b}}_{\mr{perf}}(\mc{B})$. Putting
 $\mc{C}:=\mc{A}\otimes_{\mc{R}}\mc{B}$, which is a $\mc{R}$-algebra, we
 have an isomorphism
 \begin{equation*}
  \R\shom_{\mc{A}}(\mc{M},\mc{A})\otimes^{\mb{L}}_{\mc{R}}
   \R\shom_{\mc{B}}(\mc{N},\mc{B})\cong
   \R\shom_{\mc{C}}(\mc{M}\otimes^{\mb{L}}_{\mc{R}}\mc{N},\mc{C}).
 \end{equation*}
 To check the compatibility of Frobnius, recall the definition of the
 canonical isomorphism $F^*\circ\mb{D}^{(m)}\cong\mb{D}^{(m+1)}\circ F^*$
 defined in \cite[II.3.2]{virrion}. Going back to this definition, the
 verification is straightforward.

 Let us check (ii). For an immersion $u$ of couples, $u_+$ and
 $\widetilde{\otimes}$ commute by \ref{otimes-comm-u!}. Thus, by
 d\'{e}vissage, the lemma is reduced to the overconvergent $F$-isocrystal
 case. In this case, the verification is easy.
\end{proof}

\begin{empt}
 \label{constproptononprop}
 Let $u=(\star,a)\colon\mb{Y}=(Y,X)\rightarrow (Y',X') =\mb{Y}'$ be a
 complete morphism of couples. In this paragraph, we construct the
 canonical homomorphism $\theta_u\colon u_!\rightarrow u_+$.

 Let $\E \in D ^{\mathrm{b}}_{\mathrm{ovhol}} (\mb{Y}/K)$. The morphism
 $u$ factors as
 $\mb{Y}\xrightarrow{\iota}(U,X)=\mb{U}\xrightarrow{u'}\mb{Y}'$, where
 $U:=a^{-1}(Y')$. Let $Z := U \setminus Y$. Since $\DD_{\mb{Y}} = (\hdag
 Z) \circ \DD_{\mb{U}}$ and since $Z$ is closed in $Y$, we get the
 canonical homomorphism $\DD_{\mb{U}}\circ\DD_{\mb{Y}} (\E) \to
 \DD_{\mb{Y}}\circ\DD_{\mb{Y}}(\E)\cong\E$, which yields the functorial
 homomorphism $\iota_!(\E)\rightarrow \iota _+(\E)$. Moreover, since
 $u'$ is c-proper, we have the relative duality isomorphism $ u'_! \riso
 u ' _+$. We define 
 \begin{equation*}
  \theta_u\colon u_!\cong u'_!\circ\iota_!\rightarrow u'_!\circ\iota_+
   \cong u'_+\circ\iota_+\cong u_+.
 \end{equation*}
\end{empt}

In order to check the transitivity in Proposition \ref{transi-theta}, we
need the following lemmas.
\begin{lem}
 \label{elemcommdual}
 Let $(X,\mc{A})$ be a ringed space (where $\mc{A}$ is not necessary
 commutative), and $\mc{B}\rightarrow\mc{C}$ be a
 homomorphism in the derived category
 $D^{\mr{b}}(\mc{A}\otimes_{\Z}\mc{A})$. For any
 $\mc{A}$-complex $\E$, and
 $\star\in\bigl\{\mc{B},\mc{C}\bigr\}$, we put
 $\DD_{\star}(\E):=\R\shom_{\mc{A}}(\E,\star)$, where we used the
 first $\mc{A}$-module structure of $\star$ to take $\shom$, and
 consider $\DD_{\star}(\E)$ as a left $\mc{A}$-module using the second
 $\mc{A}$-module structure of $\star$. Let
 $\beta\colon\E\rightarrow\DD_{\mc{B}}\circ\DD_{\mc{B}}(\E)$ be
 the canonical homomorphism. Then the following diagram is commutative:
 \begin{equation*}
  \xymatrix@C=30pt{
   \DD_{\mc{B}}\circ\DD_{\mc{B}}\circ\DD_{\mc{C}}(\E)&
   \DD_{\mc{B}}\circ\DD_{\mc{B}}\circ\DD_{\mc{B}}(\E)
   \ar[l]\ar[r]^-{\DD_{\mc{B}}(\beta)}_-{}&
   \DD_{\mc{B}}(\E)\ar@{=}@/^2ex/[dl]\\
  \DD_{\mc{C}}(\E)\ar[u]^{\beta(\DD_\mc{C}(\E))}&
   \DD_{\mc{B}}(\E).\ar[l]\ar[u]^{\beta(\DD_\mc{B}(\E))}&
   }
 \end{equation*}
\end{lem}
\begin{proof}
 We only need to check the commutativity of the right triangle, whose
 verification is left to the reader.
\end{proof}

\begin{lem}
 \label{commpropnorm}
 Consider the following left commutative diagram of l.p.\ frames:
 \begin{equation*}
  \xymatrix@C=50pt{
   (Y,X,\PP,\QQ)\ar[d]_{\iota=(i,\mr{id},\mr{id},\mr{id})}
   \ar[r]^-{u=(b,\star,f,g)}&
   (Y''',X',\PP',\QQ')\ar[d]^{\iota'=(i',\mr{id},\mr{id},\mr{id})}\\
  (Y'',X,\PP,\QQ)\ar[r]_-{u'=(b',\star,f,g)}&
   (Y',X',\PP',\QQ'),}
   \qquad
   \xymatrix@C=50pt{
   \iota'_!\circ u_!\ar[r]^{\iota'_!(\theta_u)}\ar@{-}[d]_{\sim}
   &\iota'_!\circ u_+\ar[r]^{\theta_{\iota'}(u_+)}&
   \iota'_+\circ u_+\ar@{-}[d]^{\sim}\\
  u'_!\circ \iota_!\ar[r]_{u'_!(\theta_\iota)}&
   u'_!\circ \iota_+\ar[r]_{\theta_{u'}(\iota_+)}&
   u'_+\circ \iota_+
   }
 \end{equation*}
 Assume that all the morphisms in the diagram are complete, and $b$,
 $b'$ are proper. Then the right diagram above of functors
 from $(F\text{-})D^{\mr{b}}_{\mr{ovhol}}(Y,\PP/K)$ to
 $(F\text{-})D^{\mr{b}}_{\mr{ovhol}}(Y',\PP'/K)$ is commutative.
\end{lem}
\begin{proof}
 Consider the following diagram, where we omit subscripts $\PP$ and
 $\PP'$ from $\DD$:
 \begin{equation*}
  \def\objectstyle{\scriptstyle}
   \xymatrix@R=10pt@C=30pt{
  \DD\circ\DD_{Y'''}\circ\DD_{Y'''}\circ f_+\circ\DD_{Y}
  \ar[dd]_{\sim}\ar[rr]\ar@{}[rd]|{\star}&&
  \DD\circ\DD_{Y'''}\circ f_+\ar@/^2ex/[ddr]\ar[d]&\\
  &&\DD\circ\DD\circ f_+\ar[dr]&\\
  \DD\circ f_+\circ\DD_{Y}\ar[r]\ar[uurr]\ar[ddrr]&
   \DD\circ f_+\circ\DD\ar[ur]\ar@{=}[dr]&&
   f_+\\
  &&\DD\circ f_+\circ\DD\ar[ur]&\\
  \DD\circ f_+\circ\DD_{Y''}\circ\DD_{Y''}\circ\DD_{Y}
   \ar[uu]^{\sim}\ar[rr]\ar@{}[ru]|{\star}&&
   \DD\circ f_+\circ\DD_{Y''}.\ar[u]\ar@/_2ex/[uur]&
  }
 \end{equation*}
 By taking the functor $\R\underline{\Gamma}^\dag_{Y'}$, we get the
 desired diagram. Thus, it suffices to check the commutativity of the
 big diagram above. The commutativity of small diagrams except for the
 one with $\star$ are easy. To check the commutativity of the triangle
 diagrams marked with $\star$, we need to check the commutativity of the
 following diagram:
 \begin{equation*}
   \xymatrix{
   \DD_{Y''}\circ\DD_{Y''}\circ\DD_{Y}&
   \DD_{Y''}\ar[l]\ar[dl]\\
  \DD_Y.\ar[u]^{\sim}&
   }
 \end{equation*}
 Take a set of divisors $\bigl\{T_i\bigr\}_{1\leq i\leq r''}$ of $P$
 such that $X\backslash\bigl(\bigcap_{1\leq i\leq
 r^{(\prime\prime)}}T_i\bigr)=Y^{(\prime\prime)}$ for some $r''\leq
 r$. Consider the following \v{C}ech type complexes
 \begin{align*}
  \D^{(\prime\prime)}&=\\
  &\mathbf{s}\left[\vcenter{
  \xymatrix@C=10pt@R=10pt{
  \D^\dag_{\PP,\Q}
  \ar[r]\ar[d]&
  \bigoplus_{1\leq k\leq r^{(\prime\prime)}}
  \D^\dag_{\PP,\Q}(^\dag T_k)
  \ar[r]\ar[d]&
  \bigoplus_{1\leq k_1<k_2\leq r^{(\prime\prime)}}\D^\dag_{\PP,\Q}
  (^\dag T_{k_1}\cup T_{k_2})
  \ar[d]\ar[r]&
  \dots\ar[r]
  &\D^\dag_{\PP,\Q}(^\dag T_1\cup\dots\cup T_{r^{(\prime\prime)}})\\
  \D^\dag_{\PP,\Q}\ar[r]&0\ar[r]&\dots&&
  }}\right]
 \end{align*}
 where $\D^\dag_{\PP,\Q}$ of the first line is placed at degree $(0,0)$, and
 $\mathbf{s}$ denotes the simple complex associated with the double
 complex.
 We have the canonical homomorphism of complexes
 $\D\rightarrow\D^{\prime\prime}$. Recall that
 \begin{equation*}
  \DD_{Y^{(\prime\prime)}}(-)\cong\R\shom_{\D^\dag_{\PP,\Q}}
   \bigl(-,\D^{(\prime\prime)}\bigr)\otimes_{\mc{O}_{\PP}}
   \omega^{-1}_{\PP}\,[d_P].
 \end{equation*}
 The commutativity follows from a lemma on general non-sense
 Lemma \ref{elemcommdual} above.
\end{proof}

\begin{prop}
\label{transi-theta}
Let $u\colon\mb{Y}\rightarrow \mb{Y}'$, 
 $u'\colon\mb{Y}'\rightarrow\mb{Y}''$ be two complete morphisms of couples. 
We have a canonical isomorphism
	$\theta_{u'}(u_+)\circ u'_!(\theta_u)\cong\theta_{u'\circ u}$.
\end{prop}

\begin{proof}
This follows from Lemma \ref{commpropnorm}.
\end{proof}

\begin{lem}
 \label{lemm-theta-diag}
 Let $u\colon\mb{Y}\rightarrow\mb{Y}'$ be a complete morphism of couples.
 Then the following diagram of functors from
 $F\text{-}D^{\mr{b}}_{\mr{ovhol}}(\mb{Y}/K)$ to
 $F\text{-}D^{\mr{b}}_{\mr{ovhol}}(\mb{Y}'/K)$ is commutative.
 \begin{equation*}
  \xymatrix@R=20pt@C=60pt {
   { u _!}\ar[r]^-{\theta _{u}}
   \ar@{=}[d]
   &{ u _+}\ar[d] ^-{\sim}\\
   {\DD _{\mb{Y}'} u _+ \DD _{\mb{Y}}}
   \ar[r] _-{\DD _{\mb{Y}'} (\theta _{u})}
   &{ \DD _{\mb{Y}'} u _! \DD _{\mb{Y}}.}
   }
 \end{equation*}
\end{lem}
\begin{proof}
 Let $u=(\star,a)$, and consider the factorization
 $\mb{Y}=(Y,X)\xrightarrow{i}\mb{Y}'':=(a^{-1}(Y'),X)
 \xrightarrow{p}\mb{Y}'=(Y',X')$. 
 Let $\E \in F\text{-}D^{\mr{b}}_{\mr{ovhol}}(\mb{Y}/K)$.
 Consider the following diagram:
 \begin{equation*}
  \xymatrix@R=0,5cm{
   u_!\E\ar[r]\ar[d]&p_+\,i_!\E\ar[r]\ar[d]&
   p_+\DD_{\mb{Y}''}i_+\DD_{\mb{Y}}\E
   \ar[r]\ar[d]&\DD_{\mb{Y}'}u_+\DD_{\mb{Y}}\E\ar[d]\\
  u_+\E\ar[r]&p_+\,i_+\E\ar[r]&
   p_+\DD_{\mb{Y}''}i_!\DD_{\mb{Y}}\E\ar[r]&
   \DD_{\mb{Y}'}u_!\DD_{\mb{Y}}\E.
   }
 \end{equation*}
 Since the outer two squares are commutative, it suffices to treat the
 $u=i$ case, and to show that the diagram
 \begin{equation}
  \label{theta-diag}
   \tag{$(\star)$}
   \xymatrix @R=0,5cm @C=70pt {
   { i _! (\E)}\ar[r] ^-{\theta _{i}(\E)}\ar@{=}[d] ^-{}
   &{ i _+ (\E)}\ar@{=}[r] ^-{}\ar[d] ^-{\sim}
   &{\E}\ar[d]^-{\sim}\\
   {\DD _{\mb{Y}''} i _+ \DD _{\mb{Y}} (\E)}
   \ar[r]^-{\DD _{\mb{Y}''}(\theta _{i}(\DD _{\mb{Y}} (\E)))}
   &{ \DD _{\mb{Y}''} i _! \DD _{\mb{Y}} (\E)}\ar@{=}[r] ^-{}
   &{\DD _{\mb{Y}''} \DD _{\mb{Y}''} \DD _{\mb{Y}}  \DD _{\mb{Y}} (\E)}
   }
 \end{equation}
 is commutative. We assume $u=i$, namely $X=X'$ and $Y\subset Y'\subset
 X$, in the following. Take an l.p.\ frame $(Y',X,\PP,\QQ)$, and put
 $T:=X\setminus Y$.
 We denote $\DD_{Y^{(\prime)},\PP}$
 by $\DD_{Y^{(\prime)}}$. 
 The homomorphism $\alpha$ is the one induced by $\mr{id} \to (\hdag T)$
 and $\beta$ is the biduality isomorphisms $\mr{id} \riso \DD _{Y}  \DD
 _{Y}=(\hdag T)  \DD _{Y'}  \DD _{Y}$.
 Let us consider the following diagram of
 $F$-$\D^\dag_{\PP,\Q}$-complexes:
 \begin{equation}
  \label{theta-diag-proof1}
   \tag{$(\star\star)$}
   \xymatrix@R=0,5cm @C=1,5cm{
   {\E}\ar@{=}[d]&
   {\DD _{Y'}\DD _Y (\E)}\ar[l]_-{\theta(\E)}
   \ar[d] ^-{\alpha}&
   {\DD _{Y'} (\hdag T) \DD _{Y'} \DD _{Y}  \DD _Y (\E)}
   \ar[l] _-{\DD _{Y'} (\beta )}
   \ar[r] ^-{\DD _{Y'} (\alpha )}
   \ar[d] ^-{\alpha}& 
   {\DD _{Y'} \DD _{Y'} \DD _{Y}  \DD _Y (\E)} 
   \ar[d] ^-{\alpha} _-{\sim}\\ 
  {\E}\ar[r] _-{\beta}^-{\sim}&
   {(\hdag T)  \DD _{Y'}  \DD _{Y} (\E) }&
   {(\hdag T)  \DD _{Y'}  (\hdag T) \DD _{Y'}  \DD _{Y}  \DD _Y (\E)}
   \ar[l] ^-{(\hdag T) \DD _{Y'} (\beta )} _-{\sim}
   \ar[r] _-{(\hdag T) \DD _{Y'} (\alpha )}^-{\sim}
   &{(\hdag T)   \DD _{Y'}  \DD _{Y'}  \DD _{Y}  \DD _Y (\E).}
   }
 \end{equation}
 This diagram is commutative by definition and by
 functoriality. Composing the morphisms of the bottom and next the right
 vertical morphism (i.e.\ $\alpha ^{-1}$) of \ref{theta-diag-proof1}, we get an
 isomorphism $\gamma\colon \E \riso \DD _{Y'} \DD _{Y'} \DD
 _{Y}  \DD _{Y} (\E)$. Since the diagram \ref{theta-diag-proof1} is
 commutative, it is sufficient to check that this isomorphism
 $\gamma$ is equal to the right arrow of \ref{theta-diag}.
 Let us consider the following diagram:
 \begin{equation}
  \label{theta-diag-proof2}
   \tag{$(\star\!\star\!\star)$}
   \xymatrix@C=50pt{
  {\DD_{Y'}  \DD_{Y} (\E)}&
   {\E.}\ar[l] _-{\beta} ^-{\sim}\ar[d] ^{\beta \circ \beta} _-{\sim}\\
   {\DD_{Y'}  \DD_{Y} (\E)}\ar@{=}[u] ^-{}& 
   {\DD_{Y'} \DD_{Y'} \DD_{Y}   \DD_{Y} (\E)}\ar[l]
   ^-{\DD_{Y'} (\beta)}_-{\sim}
   }
 \end{equation}
 We remark that, in this case, the right vertical morphism of
 \ref{theta-diag-proof2} is the right morphism of \ref{theta-diag} and 
 that $\gamma = \DD_{\PP} (\beta) \circ \beta ^{-1}$. Thus, it remains
 to check the commutativity of \ref{theta-diag-proof2}. For this, we
 argue as in the last part of the proof of Lemma \ref{commpropnorm}.
\end{proof}

\begin{lem}
 \label{Lammaforbasechangeframe}
 Consider the following cartesian diagram $D$ on the left of l.p.\
 frames (i.e.\ the four underlying squares are
 cartesian):
 \begin{equation*}
  \xymatrix{
   (Y''',X''',\PP'''\QQ''')\ar[r]^-{v'}\ar[d]_{u'}
   \ar@{}[dr]|\square&
   (Y,X,\PP,\QQ)\ar[d]^u\\
  (Y'',X'',\PP'',\QQ'')\ar[r]_v&(Y',X',\PP',\QQ'),
   }\qquad
   \xymatrix{
   \star\ar[r]^{v_2'}\ar[d]_{u''}\ar@{}[rd]|\square&
   \star\ar[r]^-{v_1'}\ar[d]_{u'}
   \ar@{}[dr]|\square&
   \star\ar[d]^u\\
  \star\ar[r]_{v_2}&\star\ar[r]_{v_1}&\star,
   }
 \end{equation*}
 such that $u$ is complete (which implies the completeness of $u'$ as
 well). Then we have a canonical base change isomorphism $b_D\colon
 v^!u_+\xrightarrow{\sim}u'_+v'^!$. This base change isomorphism is
 compatible with compositions with respect to $v$ in the following
 sense: consider the diagram on the right above. The stars denote some
 l.p.\ frames, and $u$ is assumed to be complete. Its right
 (resp.\ left) square is called $D$ (resp.\ $D'$) and the outer big
 diagram is called $E$. Then the composition
 \begin{equation*}
  (v_1\circ v_2)^!\circ u_+\xrightarrow[\sim]{b_D}
   v_2^!\circ u'_+\circ v_1'^!\xrightarrow[\sim]{b_{D'}}
   u'_+\circ(v_1'\circ v_2')^!
 \end{equation*}
 is equal to $b_E$. Similarly the base change isomorphism is compatible
 with composition with respect to $u$ as well.
\end{lem}
\begin{proof}
 Let $v\colon(Y,X,\PP,\QQ)\rightarrow(Y',X',\PP',\QQ')$ be a morphism
 of l.p.\ frames. The morphism $v$ is said to
 be {\em cartesian} if $\PP=\QQ\times_{\QQ'}\PP'$, $X=Q\times_{Q'}X'$,
 $Y=Q\times_{Q'}Y'$, {\em of type I} if $g$ is an immersion, and {\em
 of type II} if it is cartesian and $g$ is smooth. 
 We notice that $v$ factors
 canonically as $(Y,X,\PP,\QQ)\xrightarrow{\iota}
 (Q\times Y',Q\times X',\QQ\times\PP',\QQ\times\QQ')\xrightarrow{v'}
 (Y',X',\PP',\QQ')$ where $\iota$ is of type I and $v'$ is II. We
 can define the base change isomorphism for morphisms of type I and
 II individually. Indeed, for type I, the base change isomorphism is
 nothing but the transitivity of the local cohomology functor 
 and for type II, this
 is  \cite[10.7]{Abe-Frob-Poincare-dual}.
 These isomorphisms are compatible with the
 composition of $u$. The base change isomorphism for $v$ is defined to
 be the composition.

 It remains to show that the isomorphism is compatible with composition
 with respect to $v$. To construct such a natural transform, we may
 reduce to the following commutative and cartesian diagrams
 \begin{equation*}
  \xymatrix{
   \star\ar[r]^-{v}\ar[dr]_{\iota'}&
   \star\ar[d]^{\iota}\\
  &\star,
   }\qquad\qquad
  \xymatrix{
   \star\ar[r]^-{\iota}\ar[dr]_{v'}&
   \star\ar[d]^v\\
  &\star,
   }\qquad\qquad
   \xymatrix{
   \star\ar[r]^-{\iota'}
   \ar[d]_{v'}\ar@{}[dr]|\square&
   \star\ar[d]^v\\
  \star\ar[r]_{\iota}&\star,
   }
 \end{equation*}
 where $\iota$ and $\iota'$ are morphisms of type I and $v$ and $v'$
 are of type II, and construct a natural isomorphism between the
 composition of the base change isomorphism of $\iota$, $v'$ and
 $\iota'$, $v$. The first one easily follows using the equivalence in
 Lemma \ref{t-gen-coh-PXTindtPsurhol}, and the last one is similar to
 this. Let us check the second one. Let
 $v'\colon(Y',X',\PP',\QQ')\rightarrow(Y,X,\PP,\QQ)$ and
 $v\colon(Y'',X'',\PP'',\QQ'')\rightarrow(Y,X,\PP,\QQ)$. Let $i$ be the
 immersion $\QQ'\hookrightarrow\QQ''$ defining $\iota$, and
 $\widetilde{v}\colon\QQ'\rightarrow\QQ$ be the one defining $v$. Since
 $v$ and $v'$ are assumed to be cartesian, $\iota$ is cartesian as well,
 and $\iota^!$ is nothing but $i^!$. Thus the base change theorem is
 reduced to showing the following claim, whose verification is left to
 the reader.
 \begin{cl}
  Keeping the notation, consider the following cartesian diagram of
  smooth formal schemes:
  \begin{equation*}
   \xymatrix{
    \PP'\ar[r]^{i'}\ar[d]_{f'}\ar@{}[rd]|\square&
    \PP''\ar[r]^{\widetilde{v}'}\ar[d]\ar@{}[rd]|\square&
    \PP\ar[d]^f\\
   \QQ'\ar[r]_i&\QQ''\ar[r]_{\widetilde{v}}&\QQ.
    }
  \end{equation*}
  Then the two canonical base change isomorphisms between functors
  \begin{equation*}
   (\widetilde{v}\circ i)^!\circ f_+\cong
    f'_+\circ(\widetilde{v}'\circ i')^!
    \colon D^{\mr{b}}_{\mr{coh}}(\D^\dag_{\PP,\Q})\rightarrow
    D^{\mr{b}}(\D^\dag_{\QQ',\Q})
  \end{equation*}
  are identical.
 \end{cl}
\end{proof}

\begin{prop}[Base change]
 \label{basechangeisom}
 Consider the following diagram of couples which induces an isomorphism
 $Y'''\cong Y\times_{Y'}Y''$:
 \begin{equation*}
  \xymatrix{
   \mb{Y}'''\ar[r]^-{v'}\ar[d]_{u'}&
   \mb{Y}\ar[d]^u\\
  \mb{Y}''\ar[r]_v&\mb{Y}',
   }
 \end{equation*}
 where $u$ and $u'$ are complete.
 Then we have a canonical isomorphism $v^!\circ u_+\cong u'_+\circ
 v'^!$. This isomorphism is compatible with composition with respect to
 both $u$ and $v$.
\end{prop}
\begin{proof}
 Let $\mb{Y}^{*}=(Y^{*},X^*)$ where
 $*\in\bigl\{\emptyset,\prime,\prime\prime,\prime\prime\prime\bigr\}$.
 The diagram can be supplemented as follows:
 \begin{equation*}
   \xymatrix{
   (Y''',X''')\ar[r]^-{\alpha}&
   (Y''',X''\times_{X'}X)\ar[r]^-{v''}\ar[d]_{u''}
   \ar@{}[dr]|\square&
   (Y,X)\ar[d]^u\\
  &(Y'',X'')\ar[r]_v&(Y',X').
   }
 \end{equation*}
 Let us construct the base change isomorphism for cartesian diagram of
 couples: $v^!\circ u_+\cong u''_+\circ v''^!$. To construct this
 isomorphism, take a following cartesian diagram
 \begin{equation*}
  \xymatrix{
   (Y''',X''\times_{X'}X,\PP\times_{\PP'}\PP'',\QQ\times_{\QQ'}\QQ'')
   \ar[r]^-{\widetilde{v}'}\ar[d]_{\widetilde{u}''}\ar@{}[dr]|\square&
   (Y,X,\PP,\QQ)\ar[d]^{\widetilde{u}}\\
  (Y'',X'',\PP'',\QQ'')\ar[r]_{\widetilde{v}}&(Y',X',\PP',\QQ')
   }
 \end{equation*}
 where $\widetilde{u}$ and $\widetilde{v}$ are morphisms of frames of
 $u$ and $v$ respectively such that the {\em morphism of formal schemes
 are smooth}. Then applying Lemma \ref{Lammaforbasechangeframe}, we have
 a base change isomorphism $\widetilde{v}^!\circ\widetilde{u}_+\cong
 \widetilde{u}''_+\circ\widetilde{v}''^!$. We need to check that the
 isomorphism does not depend on the choice of $\widetilde{u}$ and
 $\widetilde{v}$. This follows easily from the compatibility of the
 composition.

 Now, we define the desired base change isomorphism to be $v^!\circ
 u_+\cong u''_+\circ v''^!\cong u''_+\circ (\alpha_+\circ\alpha^!)\circ
 v''^!\cong u'_+\circ v'^!$. To check the compatibility, we need diagram
 chasing using Lemma \ref{Lammaforbasechangeframe}, which is tedious but
 not difficult.
\end{proof}

\begin{lem}
 \label{lem-ix0=0}
 Let $\mb{Y}$ be a couple, and
 $\E \in F\text{-}D ^{\mathrm{b}}_{\mathrm{ovhol}}(\mb{Y}/K)$. 
 Assume, for any closed point $y$ of $Y$, we
 have $i ^{!} _y (\E) =0$, where
 $i_y\colon(\{y\},\{y\})\rightarrow\mb{Y}$ is the canonical
 morphism. Then $\E = 0$.
\end{lem}
\begin{proof}
 From Lemma \ref{iso-onU}, one can suppose $\mb{Y}$ is of the form
 $(Y,Y)$. We proceed by induction on the dimension of the support of
 $\E$. We may assume that the support of $\E$ is $Y$. There exists a
 closed subscheme $Z\subset Y$ such that $U:=Y\setminus Z$ is dense and
 smooth, and $j^!\E$ is in $F\text{-}D
 ^{\mathrm{b}}_{\mathrm{isoc}}(U,Y/K)$ where $j \colon (U,Y) \to (Y,Y)$
 (cf.\ paragraph \ref{isoc-nota}).
 Since for $u \in U$ we have $i'^! _u j^!\E=0$ where
 $i_u\colon(\{u\},\{u\})\rightarrow (U, Y)$,
 we have $j^!\E=0$. Consider the localization triangle
 $\R\underline{\Gamma}^\dag_Z(\E)\rightarrow\E\rightarrow(^\dag Z)
 (\E)\xrightarrow{+}$. Since for any $y\in Z$,
 $i_y^!(\R\underline{\Gamma}^\dag_Z\E)=0$, 
 we have
 $\R\underline{\Gamma}^\dag_Z(\E)=0$ by induction hypothesis. We
 conclude since $(^\dag Z)(\E)=j _+j^!\E=0$.
\end{proof}

\begin{prop}
 \label{homeo-univ}
 Let $u\colon\mb{Y}'\to\mb{Y}$ be a morphism of couples. If $u$ is a
 c-universal homeomorphism (i.e.\ by EGA IV, Proposition
 8.11.6, finite, surjective and radicial), then the
 functors $u_{+}$ and $u ^{!}$ induce canonical equivalence of
 categories between $F\text{-}D
 ^{\mathrm{b}}_{\mathrm{ovhol}}(\mb{Y}/K)$ and $F\text{-}D
 ^{\mathrm{b}}_{\mathrm{ovhol}}(\mb{Y}'/K)$.
\end{prop}
\begin{proof}
 Let $\E \in F\text{-}D ^{\mathrm{b}}_{\mathrm{ovhol}}(\mb{Y}/K)$ and
 $\E' \in F\text{-}D ^{\mathrm{b}}_{\mathrm{ovhol}}(\mb{Y}'/K)$.
 It is sufficient to check that adjunction morphisms $u _+ \circ u^{!}
 (\E) \to \E$ and $\E ' \to u ^{!} \circ u _{+} (\E')$ are
 isomorphisms. Let $y$ be a closed point $Y$ and $y' := u^{-1} (y)$.
 We denote by $i _y \colon y \hookrightarrow Y$ (resp.\ $i _{y'} \colon
 y '\hookrightarrow Y'$) the canonical closed immersion, and by $u _y
 \colon y '\rightarrow y$ the induced morphism, which is in fact an
 isomorphism since $u$ is surjective, radicial, and $k$ is perfect.
 From the base change theorem \ref{basechangeisom}, we get $i ^{!} _{y}
 \circ u _+ \riso u _{y+} \circ i ^{!} _{y'}$.
 By applying $i ^{!} _{y}$ to the canonical morphism $u _+ \circ u ^{!}
 (\E) \to \E$ we get the adjunction morphism $u _{y+} \circ u ^{!} _{y}
 (i ^{!} _{y} \E) \to i ^{!} _{y} \E$, which is an isomorphism.
 From Lemma \ref{lem-ix0=0}, this implies that the morphism
 $u _+ \circ u ^{!} (\E) \to \E$ is an isomorphism.
 We may check similarly that $\E ' \to u ^{!} \circ u _{+} (\E')$ is an
 isomorphism as well.
\end{proof}

\begin{prop}
 \label{exactnessforsomemorph}
 Let $u\colon\mb{Y}\rightarrow\mb{Y}'$ be a complete morphism of
 couples.

 (i) If $u$ is c-affine, then $u_+$ (resp.\ $u_!$) is right t-exact
 (resp.\ left t-exact).

 (ii) If $u$ is c-quasi-finite, then $u_+$ (resp.\ $u_!$) is left t-exact
 (resp.\ right t-exact).
\end{prop}
\begin{proof}
 Let us show (i). Take a morphism of l.p.\ frames
 $(\star,\star,\star,g)\colon(Y,X,\PP,\QQ)\rightarrow(Y',X',\PP',\QQ')$
 of $u$. From Remark \ref{rem-scrHcalH}, we may assume $Y'=P'$. Since
 the claim is local, we may assume $\PP'$ to be affine. Since $Y$ is
 affine, we can take a closed immersion
 $Y\hookrightarrow\widehat{\mb{A}}^n_{\PP'}$ for some $n$. Since $g$ is
 assumed to be smooth, the canonical morphism
 $\widehat{\mb{P}}^n_{\PP'}\times_{\PP'}\PP\rightarrow\PP'$ is proper
 smooth, and we may consider the following commutative diagram of
 frames:
 \begin{equation*}
  \xymatrix@R=15pt{
   (Y,\overline{Y},\widehat{\mb{P}}^n_{\PP'})\ar[dr]_{u'}&
   (Y,\overline{Y},\widehat{\mb{P}}^n_{\PP'}\times_{\PP'}\PP)
   \ar[r]\ar[l]\ar[d]&
   (Y,X,\PP)\ar[dl]\\
  &(P',P',\PP').&
   }
 \end{equation*}
 Thus, we are reduced to showing the right exactness of $u'_+$. Let
 $D\subset\widehat{\mb{P}}^n_{\PP'}$ be the divisor at infinity. By
 construction, $Y\hookrightarrow\widehat{\mb{P}}^n_{\PP'}\setminus
 D=\widehat{\mb{A}}^n_{\PP'}$ is a closed immersion, and we may assume
 that $Y=\mb{A}_{P'}^n$ by Proposition \ref{propofpullbacks} (iii).
 Then the proposition follows by
 \cite[5.4.1]{Noot-Huyghe-affinite-proj}.

 Let us show (ii). By exactness of the dual functor, it suffices to show
 the $u_+$ case. First, consider the case where $u$ is a c-open
 immersion. By Lemma \ref{t-gen-coh-PXTindtPsurhol}, we can suppose
 $X=X'$. Then, the claim follows from Remark
 \ref{rem-scrHcalH} (iv). 
 The case where $u$ is a c-closed immersion is treated in Proposition
 \ref{propofpullbacks}.
 Finally, let us show the general case by the induction on the dimension
 of $X$. We may assume that the image of $b$ is dense in $Y'$. By using
 the localization triangle and the induction hypothesis, we may shrink
 $Y$, and assume that $Y$ is irreducible. For a morphism of frames
 $a\colon(\star,\star,\star,\mr{id})\colon
 (Y_{\mr{red}},\star,\star,\QQ)\rightarrow(Y,\star,\star,\QQ)$, $a_+$ is
 t-exact by definition.
 Thus we may replace $Y$ and $Y'$ by $Y_{\mr{red}}$ and $Y'_{\mr{red}}$,
 and assume that $Y$ and $Y'$ are integral.
 By using Zariski main theorem, there exists an open dense subscheme
 $U'\subset Y'$ such that $u^{-1}(U')\rightarrow U'$ is finite
 \'{e}tale. Using the proven immersion case and the induction
 hypothesis, we may shrink $Y'$, and reduced to showing the case where
 $Y'$ is smooth and $u$ is c-finite \'{e}tale. In this case, we see
 easily that $u_+$ is t-exact.
\end{proof}

\begin{empt}
 \label{summpropre}
 Let $Y$ be a realizable variety. 
 Take a couple $(Y,X)$ such that $X$ is proper. Then we define the
 category $(F\text{-})D^{\mr{b}}_{\mr{ovhol}}(Y/K):=
 (F\text{-})D^{\mr{b}}_{\mr{ovhol}}(Y,X/K)$. This does not depend on the
 choice of such a couple up to canonical equivalence of categories with
 t-structure. Indeed, take another couple $(Y,X')$ with $X'$ is
 proper. Let $X''$ be the closure of $Y$ in $X\times X'$. We obtain two
 morphisms $(Y,X'')\rightarrow (Y,X^{(\prime)})$. Then, using Lemma
 \ref{t-gen-coh-PXTindtPsurhol} the independence follows.

 With these categories with Frobenius structures, we get six functors formalism
 on the category of realizable varieties (without Frobenius structures, the definition of tensor
 products is problematic so far). Let $f\colon Y\rightarrow Y'$
 be a morphism of realizable varieties. We have the following
 properties:

 \begin{enumlatin}
  \item Dual functor on $Y$, denoted by $\DD_Y$, is defined thanks to
  Lemma
  \ref{t-gen-coh-PXTindtPsurhol}.\ref{t-gen-coh-PXTindtPsurhol-iii}.

  \item Tensor functor on $Y$, denoted by $\widetilde{\otimes}_Y$ or
  simply $\widetilde{\otimes}$.

  \item  Push-forward  for $f$, denoted by $f_+$ is defined by
  the transitivity of the push-forward for couples.

  \item  Extraordinary pull-back for $f$, denoted by $f^!$ is defined by
  the transitivity as well.

  \item Extraordinary push-forward and ordinary pull-back for $f$, denoted
  by $f_!$ and $f^+$, is defined by (i), (iii), (iv).

  \item  We have a canonical homomorphism $f_!\rightarrow f_+$ and this
  is an isomorphism if $f$ is proper (see paragraph
  \ref{constproptononprop}).

  \item  We have a base change isomorphism by Proposition
  \ref{basechangeisom}.

  \item We have adjoint pairs $(f^+,f_+)$ and $(f_!,f^!)$ by Lemma
  \ref{adju+u+couple}.
 \end{enumlatin}
\end{empt}

\subsection{Intermediate extensions}
In this subsection, we define the intermediate extension functor, which
plays an essential role in defining the intersection cohomology. Unless
otherwise stated, we let
$u=(b,a)\colon\mb{Y}=(Y,X)\rightarrow\mb{Y}'=(Y',X')$ be a c-immersion
of couples in this subsection.

\begin{dfn}
 For $\E\in F\text{-}\mr{Ovhol}(\mb{Y}/K)$, we have the homomorphism
 $\theta^0_{u,\E}:=\HH ^0(\theta_{u,\E})\colon u_! ^0\E\rightarrow
  u_+ ^0\E$ (see Definition \ref{nota-129} and paragraph
 \ref{constproptononprop} for the notation). We define
 \begin{equation*}
  u_{!+}(\E):=\mr{Im}\bigl(\theta^0_{u,\E}\colon u^0_!\E\rightarrow
 u^0_+\E\bigr).
 \end{equation*}
 This defines a functor $u_{!+}\colon F\text{-}\mr{Ovhol}(\mb{Y}/K)
 \rightarrow F\text{-}\mr{Ovhol}(\mb{Y}'/K)$, and it is called the {\em
 intermediate extension functor}.
\end{dfn}

\begin{rem}
 \label{rem-ker-coker-Im-dual}
 (i) It follows from Proposition \ref{exactnessforsomemorph} that the
 functor $u_{!+}$ preserves injections and surjections.

 (ii) Since the functor $ \DD _{\mb{Y}'}$ is t-exact on the category
 $F\text{-}\mathrm{Ovhol} (\mb{Y}'/K)$ by Proposition
 \ref{exactnessofdual}, for any $\E \in F\text{-}\mathrm{Ovhol}
 (\mb{Y}/K)$, we get 
 the canonical isomorphisms:
 \begin{gather*}
  \DD _{\mb{Y}'}\bigl(\ker (\theta ^0 _{u,\E})\bigr)  \riso
  \mathrm{coker}\bigl(\DD_{\mb{Y}'}(\theta ^0 _{u,\E})\bigr),\qquad
  \DD _{\mb{Y}'}\bigl(\coker (\theta ^0 _{u,\E})\bigr)  \riso
  \ker \bigl(\DD _{\mb{Y}'}(\theta ^0_{\E})\bigr),\\
  \DD _{\mb{Y}'} \bigl(\mathrm{Im} (\theta ^0 _{u,\E})\bigr)  \riso
  \mathrm{Im} \bigl(\DD _{\mb{Y}'}(\theta ^0 _{u,\E})\bigr).
 \end{gather*}
 Moreover, when $u$ is c-affine, both functors $u _+$, $u _!\colon
 F\text{-}\mathrm{Ovhol}(\mb{Y}/K)\to F\text{-}\mathrm{Ovhol}
 (\mb{Y}'/K)$ are t-exact by Proposition \ref{exactnessforsomemorph},
 and we do not need to take $\HH^0$ in the definition of $u_{!+}$.
\end{rem}

\begin{cor}
 \label{i!+bidual}
 Let $\E\in F\text{-}\mr{Ovhol}(\mb{Y}/K)$. We have the canonical
 isomorphisms:
 \begin{equation*}
  u _{!+} (\E) \riso  \DD _{\mb{Y}'} \circ u _{!+} \circ \DD _{\mb{Y}}
  (\E),\quad
  \ker(\theta ^0 _{u,\E}) \riso \DD _{\mb{Y}'}\bigl(\mathrm{coker} (\theta ^0
  _{u,\DD _{\mb{Y}} (\E)})\bigr),\quad
  \mathrm{coker} (\theta ^0 _{u,\E}) \riso \DD _{{\mb{Y}}'}\bigl(
  \ker(\theta ^0_{u,\DD _{\mb{Y}} (\E)})\bigr).
 \end{equation*}
\end{cor}
\begin{proof}
 This is a consequence of Remark \ref{rem-ker-coker-Im-dual} (i) and
 Lemma \ref{lemm-theta-diag}.
\end{proof}

\begin{rem}
 \label{dagTtheta}
 Let $\E\in F\text{-}\mr{Ovhol}(\mb{Y}/K)$. 
 Let $Y''$ be the closure of $Y$ in $Y'$ and $Z:=Y''\setminus Y$. 
 Since $\theta^0_{u,\E}$ is an isomorphism outside $Z$, 
 this implies that the $F$-modules $\mathrm{ker} (\theta^0_{u,\E}) $ and
 $\mr{coker} (\theta^0_{u,\E})$ have their support in $Z$ (cf.\
 Proposition \ref{propofpullbacks}).
 Since $u_+$ and $\R\underline{\Gamma}^\dag_Z$ are
 left t-exact by Propositions \ref{propofpullbacks} and
 \ref{exactnessforsomemorph}, we get
 $0=\HH^0\bigl(\R\underline{\Gamma}^\dag_Z\circ u_+(\E)\bigr)\cong
 \R^0\underline{\Gamma}^\dag_Z(u^0_+\E)$.
 Now applying the functor $\R \underline{\Gamma}^\dag_Z$ to the exact sequence $0 \to u _{!+}
 (\E) \to  u ^0 _+ (\E) \to \mathrm{coker} (\theta^0_{u,\E}) \to 0$ (which can be viewed as an exact triangle)
 yields another short exact sequence:
 \begin{equation*}
  0 \to \mathrm{coker} (\theta ^0 _{u,\E}) 
  \to \R^1\underline{\Gamma}^\dag _{Z}\bigl(u _{!+} (\E)\bigr)
  \to \R^1\underline{\Gamma}^\dag_{Z} \bigl(u ^0 _{+} (\E)\bigr)
  \to 0.
 \end{equation*}
 Suppose $u$ is affine (e.g.\ when $Z$ is a divisor of $Y''$).
 In this case, $\theta  _{u,\E} = \theta  ^0 _{u,\E}$
(i.e. $u _+$ and $u _!$ are t-exact).
Since  $\R\underline{\Gamma} ^\dag _{Z}\bigl(u _+ (\E)\bigr)=0$, we get (in that case)
the isomorphism:
 \begin{equation}
  \label{cokertheta-div}
   \mathrm{coker}(\theta  _{u,\E}) \riso
   \R \underline{\Gamma} ^\dag _{Z}\bigl(u _{!+} (\E)\bigr) [1].
 \end{equation}
\end{rem}

\begin{lem}
\label{transitivityi!+}
 Let $\mb{Y}\xrightarrow{u}\mb{Y}'\xrightarrow{u'}\mb{Y}''$
 be c-immersions of couples.

 (i) We have canonical equivalences of functors $u'^0_+ \circ
 u^0_+\cong(u'\circ u)^0_+$, and $u^{\prime 0}_!\circ
 u^0_!\cong(u^{\prime} \circ u)^0_!$. This induces an equivalence
 $u^{\prime}_{!+}\circ u_{!+}\cong(u^{\prime }\circ u)_{!+}$.

 (ii) We have canonical isomorphisms $u^{\prime !} \circ
 (u^{\prime}\circ u)^0 _{+}\riso  u^0 _{+}$, $u^{\prime !} \circ
 (u^{\prime}\circ u)^0 _{!} \riso u^0_{!}$, and $u^{\prime !}\circ (u
 ^{\prime}\circ u) _{!+} \riso  u _{!+}$.
\end{lem}

\begin{proof}
 The first two isomorphisms follow by Proposition
 \ref{exactnessforsomemorph}. Let $\E\in
 F\text{-}\mr{Ovhol}(\mb{Y}/K)$. Using this, we have the canonical
 surjections: $u^{\prime 0}_!\circ  u ^0_! (\E)
 \twoheadrightarrow u ^{\prime 0} _{!}\circ  u _{!+} (\E)
 \twoheadrightarrow u ^{\prime }_{!+}\circ  u _{!+} (\E)$. We also have the
 canonical inclusions: $u^{\prime }_{!+}\circ  u _{!+} (\E)
 \hookrightarrow u^{\prime 0} _+\circ  u _{!+} (\E)\hookrightarrow
 u^{\prime 0}_+\circ  u^0_+ (\E)$. By functoriality, the composition of
 these homomorphisms is the canonical morphism $u ^{\prime 0} _!\circ  u
 ^0 _! (\E) \to u ^{\prime 0} _+\circ  u ^0_+ (\E)$. Using the base
 change \ref{basechangeisom}, the second part of the lemma is obvious.
\end{proof}

Next proposition is a generalization of \cite[2.1]{Abe-Langlands}, and
is generalized further at Theorem \ref{intemedstructhem} below:
\begin{lem}
 \label{irreduciblerestpreserve}
 Let $u\colon\mb{Y}\rightarrow\mb{Y}'$ be a c-open immersion. 
Let  $\E'$ be an {\em irreducible} object of
 $F\text{-}\mathrm{Ovhol}(\mb{Y}'/K)$. 
 Then $u^{!0}\E'$ is either irreducible in 
 $F\text{-}\mathrm{Ovhol}(\mb{Y}/K)$ or $0$.
\end{lem}
\begin{proof}
 Let $\alpha\colon u^{!0}\E'\rightarrow\FF$ be a homomorphism with non-zero
 kernel. By adjointness \ref{adju+u+couple} and $u ^! = u ^+$ is
 t-exact, we get $\alpha'\colon\E'\rightarrow u ^0 _+\FF$. Since
 $u^!(\ker(\alpha'))\cong\ker(\alpha)\neq0$, we have
 $\ker(\alpha')\neq0$. Since $\E'$ is assumed to be irreducible, we have
 $\alpha'=0$, and thus $\alpha=u^{!0} (\alpha ')=0$. Thus $u^{!0}\E'$ is
 irreducible or $0$.
\end{proof}

\begin{prop}
 \label{BorelVII.10.5}
 Let $\E$ be an {\em irreducible} object of
 $F\text{-}\mathrm{Ovhol}(\mb{Y}/K)$. Then:

 (i) In the category $F\text{-}\mathrm{Ovhol} (\mb{Y}'/K)$, $ u _{!+}
 (\E)$ is the unique irreducible subobject of $ u^0 _+(\E)$.

 (ii) In the category $F\text{-}\mathrm{Ovhol} (\mb{Y}'/K)$, $ u _{!+}
 (\E)$ is the unique irreducible quotient of $ u^0 _!(\E)$.

 (iii)  Let $j\colon \mb{Y}'' \to\mb{Y}'$ be a c-open immersion. Then,
 in $F\text{-}\mathrm{Ovhol} (\mb{Y}''/K)$, either $j ^! (u^0 _+(\E))=0$
 or $j^! (u _{!+} (\E))$ is the unique irreducible subobject of $j^! (u^0
 _+(\E) )$.
\end{prop}
\begin{proof}
 The proof is essentially the same as that of
 \cite[VII.10.5]{borel}. Let us show (i). 
 By using Berthelot-Kashiwara theorem \ref{propofpullbacks}.(iii) and
 Lemma \ref{transitivityi!+}, we may assume that $u$ is a c-open
 immersion. We first claim that for any non-zero $F$-submodule $\FF$ of
 $u^0_+\E$, $u^!\FF$ is non-zero.
 Indeed, assume $u^!\FF=0$, which is equivalent to assuming
 $\R\underline{\Gamma}^\dag_Y(\FF)=0$ by Proposition
 \ref{propofpullbacks}.(ii). This means that $\FF$ has its
 support in $Z\,(:=\overline{Y}\setminus Y)$ and then
 $\R^0\underline{\Gamma}^\dag_Z(\FF)\xrightarrow{\sim}\FF$.
 This implies that
 $\FF\subset\R^0\underline{\Gamma}^\dag_Z\bigl(u^0_+(\E)\bigr)$.
 By Remark \ref{dagTtheta}, we have
 $\R^0\underline{\Gamma}^\dag_Z\bigl(u^0_+(\E)\bigr)=0$, thus $\FF=0$,
 which contradicts with the assumption.

 Let us return to the proof. Now, let $\FF$ be an irreducible $F$-submodule of $u^0_+\E$.
 By the left t-exactness of $u^!$, we have $0\neq u^!\FF\subset
 u^!u^0_+(\E)\cong\E$. Since $\E$ is assumed to be irreducible, we have
 $u^!\FF=\E$. Let $\FF'$ be another irreducible $F$-submodule of
 $u^0_+\E$ (resp.\ $\FF':=u_{!+}\E$). Assume $\FF\cap \FF'=0$. Left
 t-exactness of $u^!$ implies that $u^!\FF\cap u^!\FF'=0$ in $\E$, which
 is impossible. Thus $\FF=\FF'$ (resp.\ $\FF\subset u_{!+}\E$), and
 $\FF$ is the unique irreducible $F$-submodule of $u^0_+\E$.
 Let $\mc{Q}:= u_{!+}(\E) / \FF$ be the quotient, and assume that this
 is not zero. This is supported in $Z$. We have
 \begin{equation*}
  \DD_{\mb{Y}'}\mc{Q}\subset\DD_{\mb{Y}'}(u_{!+}\E)\cong
   u_{!+}(\DD_{\mb{Y}}\E)\subset u^0 _+(\DD_{\mb{Y}}\E).
 \end{equation*}
 The claim above shows that $u^!(\DD_{\mb{Y}'}\mc{Q})$ is non-zero,
 which is a contradiction with the fact that $\mc{Q}$ is supported in
 $Z$. Thus $\mc{Q}=0$, and $\FF$ is equal to $u_{!+}\E$, which completes
 the proof of (i). By duality, (ii) follows.

 To show (iii), use Lemma \ref{irreduciblerestpreserve}, and base change
 of couples of functors $(j^!,u_+)$ and $(j^!\cong j^+,u_{!+})$. For
 details, we refer to \cite{borel}.
\end{proof}

We can extend Proposition \ref{BorelVII.10.5} as follows:
\begin{prop}
 \label{fact-i_!+-i+}
 Let $\E \in F\text{-}\mathrm{Ovhol} (\mb{Y}/K)$ and $\E '$ be a
 subobject of $u _+ (\E)$ in $F\text{-}\mathrm{Ovhol} (\mb{Y}'/K)$.
 Assume that the inclusion $u^{!} (\E') \hookrightarrow u^{!}\circ
 u_{+}(\E) =\E$ is an isomorphism. Then we have the canonical
 factorization $u_{!+}(\E) \subset  \E'  \subset u_{+}(\E)$ of the
 inclusion $u_{!+}(\E) \subset u_{+}(\E)$.
\end{prop}
\begin{proof}
 The proof is analogue to that of Proposition \ref{BorelVII.10.5}.(i).
\end{proof}

\begin{thm}
 \label{intemedstructhem}
 (i) Let $\E$ be an {\em irreducible} object of
 $F\text{-}\mathrm{Ovhol}(Y,X/K)$. Then there exists an open dense
 smooth subscheme $Y'$ of $Y$ and $\E'\in  F\text{-}\mathrm{Isoc} ^{\dag
 \dag}(Y',X/K)$ such that $\E\cong u_{!+}(\E')$ where
 $u\colon(Y',X)\rightarrow (Y,X)$ is the canonical c-inclusion.

 (ii) Let $\mb{Y}$ be a couple, and $\{\mb{Y}_i\} _{i \in I}$ be an c-open covering
 of $\mb{Y}$ (cf.\ paragraph \ref{glueing-Ovhol}).
Let
 $u_i\colon\mb{Y}_i\rightarrow\mb{Y}$ be the 
 $c$-open immersions.
Then $\E$ is semi-simple (resp.\ irreducible or $0$) if and
 only if $u_i^!\E$ is semi-simple (resp.\ irreducible or $0$) for any
 $i$.
\end{thm}
\begin{proof}
 Copy the proof of \cite[VII.10.6]{borel}.
\end{proof}

\subsection{Local theory on formal disk}
This subsection is aimed to describe some cohomological functors for
formal disk in terms of solution data {\it \`{a} la} Crew.
These results clarify the relation between some computations of
rigid cohomology and that of arithmetic $\D$-modules.
In this subsection, we consider situation (A) in Notation and
convention, and moreover, we assume $k$ contains the field with $q$
elements. We remark that this condition is satisfied in situation (B).

We denote by $K^{\mr{ur}}$ the maximal unramified extension of $K$. We
put $\mc{K}:=k\ff{x}$, and let $G_{\mc{K}}$ be the absolute Galois group
of $k\ff{x}$. We often denote $G_{\mc{K}}$ by $G$, and the inertia group
of $G$ by $I$. For an integer $n$ and a $\varphi$-$K^{\mr{ur}}$-vector
space $V$, $V(n)$ denotes the $n$-th Tate twist of $V$, namely the
underlying vector space is $V$ itself and the endomorphism is multiplied
by $q^{-n}$ (cf.\ \cite[2.7]{Abe-Frob-Poincare-dual}).

\begin{dfn}[{\cite[3.1.3]{Marmora-fact-eps}}]
 A {\em Deligne module} $D$ is a finite dimensional
 $\varphi$-$K^{\mr{ur}}$-vector space $(V,\varphi)$ endowed with
 semi-linear $G_{\mc{K}}$-action $\sigma$ commuting with the Frobenius
 action such that $\sigma|_I$ factors through a finite index subgroup of
 $I$, and a linear homomorphism
 of $\varphi$-$K^{\mr{ur}}$-vector spaces $N\colon V(1)\rightarrow V$
 which is nilpotent and equivariant.
 A homomorphism between Deligne modules is a
 homomorphism compatible with other data in the obvious way. We denote
 such a Deligne module by $(V,\sigma,\varphi,N)$, and the category of
 Deligne modules is denoted by $\mr{Del}$.
 We sometimes denote $\sigma$ (resp.\ $\varphi$, $N$) by $\sigma_D$
 (resp.\ $\varphi_D$, $N_D$). For an integer $n$, we define the {\em
 $n$-th Tate twist} $D(n)$ of $D$ by the quadruple
 $(V,q^{-n}\cdot\varphi,\sigma,N)$.
\end{dfn}

\begin{empt}
 Let $(V,\varphi,\sigma,N)$ be a Deligne module. We will construct a
 canonical decomposition. The action of $I$ on $K^{\mr{ur}}$ is
 trivial by definition, and in particular, the action of $I$ on $V$ is
 linear. We denote by
 $V^{I=1}$ the subset of $V$ fixed by the action of $I$. Since the action
 of $I$ on $V$ is linear, $V^{I=1}$ is in fact a $K^{\mr{ur}}$-vector
 space. Now, by the definition of Deligne module, there exists a subgroup
 $I'\subset I$ of finite index such that the action of $I'$ on $V$ is
 trivial, and thus, this induces the linear action of $I/I'$ on
 $V$. For any $x\in V$, $\sum_{\sigma\in I/I'}\sigma(x)$ is in
 $V^{I=1}$. Thus, we get a $K^{\mr{ur}}$-linear homomorphism
 \begin{equation*}
 \pi:=(\#I/I')^{-1}\sum_{\sigma\in I/I'}\sigma\colon V\rightarrow
  V^{I=1}.
 \end{equation*}
 We see easily that this does not depend on the choice of $I'$.
 For any $\tau\in G$, we have $\tau\cdot I=I\cdot \tau$. Thus we get, for
 any $x\in V$, $\tau\bigl(\bigl(\sum_{\sigma\in
 I/I'}\sigma\bigr)(x)\bigr)=\bigl(\sum_{\sigma\in
 I/I'}\sigma\bigr)(\tau(x))$, which implies that
 $V^{I=1}$ possess a semi-linear action of $G$, which is nothing but the
 restriction of the $G$-action on $V$ to $V^{I=1}$, and $\pi$ commutes
 with $G$-actions. Since $\varphi$ and $N$ commute with $G$-action on
 $V$, these induce Frobenius and nilpotent operator on $V^{I=1}$, which
 is also denoted by $\varphi$ and $N$. Summing up, the quadruple
 $(V^{I=1},\varphi,\sigma,N)$ form a Deligne module, and $\pi$ defines a
 homomorphism of Deligne modules.

 By definition, the canonical inclusion $V^{I=1}\hookrightarrow V$
 induces a homomorphism of Deligne modules, and $\pi$ is a section of
 this inclusion. Thus, we obtain a canonical decomposition of Deligne
 modules
 \begin{equation*}
 V\xrightarrow[(\pi,c)]{\sim}V^{I=1}\oplus V/V^{I=1},
 \end{equation*}
 where $c\colon V\rightarrow V/V^{I=1}$ is the projection. We denote
 $V/V^{I=1}$ by $V^{I\neq1}$, and consider this as a submodule of $V$.
\end{empt}

\begin{dfn}[{\cite[6.1]{Crew-unit-disk}}]
 A {\em solution data} is a set $(\Psi,\Phi,c,v,\{v(\sigma)\}
 _{\sigma\in I})$, where $\Psi,\Phi$ are Deligne modules,
 $c\colon\Psi\rightarrow\Phi$ ({\em canonical homomorphism}),
 $v\colon\Phi(1)\rightarrow\Psi$ ({\em variation homomorphism}), and
 $v(\sigma)\colon\Phi\rightarrow\Psi$ ({\em Galois variation
 homomorphism}) are homomorphisms of Deligne modules such that
 \begin{enumerate}
  \item $N_{\Psi}=v\circ c$ and $N_{\Phi}=c\circ v$;
  \item $\sigma_{\Psi}=1+v(\sigma)\circ c=$ and
	$\sigma_{\Phi}=1+c\circ	v(\sigma)$.
 \end{enumerate}
 A homomorphism between solution data $(\Psi,\Phi,c,v,\{v(\sigma)\})$ and
 $(\Psi',\Phi',c',v',\{v'(\sigma)\})$ is a pair of homomorphisms
 $(f,f')\colon(\Psi,\Phi)\rightarrow(\Psi',\Phi')$ which is compatible
 with $c$, $c'$, $v$, $v'$, $v(\sigma)$, $v'(\sigma)$ in the obvious
 way. We denote the category, in fact an abelian category, of solution
 data by $\mr{Sol}$.
\end{dfn}

Let $V$ be a Deligne module such that $\sigma_V$ is trivial for any
$\sigma\in I$. Then $(0,V,0,0,\{0\})$ defines a solution data. This
solution data is denoted by $i_+(V)$.

\begin{empt}
 \label{defofneavancyc}
 One of the most important results of \cite{crew-arith-D-mod-curve} is that there
 exists an equivalence between the category of holonomic
 $F$-$\mc{D}^{\mr{an}}$-modules and that of solution data. Let
 $\mc{S}:=\mr{Spf}(R[\![x]\!])$, a formal disk. Crew defined the {\em
 ring of analytic differential operators}
 $\mc{D}^{\mr{an}}_{\mc{S},\mb{Q}}$ and that with poles
 $\mc{D}^{\mr{an}}_{\mc{S},\mb{Q}}(0)$. For simplicity, we
 denote these rings by $\mc{D}^{\mr{an}}$ and $\mc{D}^{\mr{an}}(0)$. He
 proved fundamental properties of these rings. For the details, see
 {\it ibid}.. Let $\mc{M}$ be a holonomic
 $F$-$\mc{D}^{\mr{an}}$-module. He defined\footnote{Actually, he defined
 functors $\mb{V}$ and $\mb{W}$. We modify the definition slightly after
 \cite{AM}.}
 two functors $\Psi$ and $\Phi$ by
 \begin{equation*}
 \Psi(\mc{M}):=\mr{Hom}_{\mc{D}^{\mr{an}}}(\mb{D}(\mc{M}),\mc{B}),\quad
  \Phi(\mc{M}):=\mr{Hom}_{\mc{D}^{\mr{an}}}(\mb{D}(\mc{M}),\mc{C}),
 \end{equation*}
 where $\mc{B}$ is the ring of hyperfunctions and $\mc{C}$ is the
 microfunction space (cf.\  [{\it ibid.}, 1.4, 6.1]).
 These are naturally Deligne modules, and the canonical homomorphism
 $c\colon\mc{B}\rightarrow\mc{C}$ and variation homomorphism
 $v\colon\mc{C}\rightarrow\mc{B}$ define canonical and variation
 homomorphisms between $\Psi(\mc{M})$ and $\Phi(\mc{M})$, which gives us
 a solution data $(\Psi(\mc{M}),\Phi(\mc{M}),\dots)$, and denoted by
 $\Da(\mc{M})$. For the precise construction, see [{\it ibid.}, \S6].

 Now, let us denote by $\mr{Hol}(\mc{D}^{\mr{an}}_{\mc{S},\mb{Q}}(0))$
 the category of holonomic
 $F$-$\mc{D}^{\mr{an}}_{\mc{S},\mb{Q}}(0)$-modules, which is a
 subcategory of the category of holonomic
 $F$-$\mc{D}^{\mr{an}}_{\mc{S},\mb{Q}}$-modules denoted by
 $\mr{Hol}(\mc{D}^{\mr{an}}_{\mc{S},\mb{Q}})$.
 A direct consequence of local monodromy theorem is that the functor
 \begin{equation*}
  \Psi\colon\mr{Hol}(\mc{D}^{\mr{an}}_{\mc{S},\mb{Q}}(0))\rightarrow
   \mr{Del}
 \end{equation*}
 induces an equivalence of categories. Now, one of the main theorems of
 the Crew's paper is the following:
 \begin{thm}[{\cite[7.1.1]{Crew-unit-disk}}]
  The functor
  \begin{equation*}
   \Da\colon\mr{Hol}(\mc{D}^{\mr{an}}_{\mc{S},\mb{Q}})\rightarrow
    \mr{Sol}
  \end{equation*}
  induces an equivalence of categories.
 \end{thm}

 A goal of this subsection is to describe $j_!\mc{M}$ and $j_+\mc{M}$ in
 terms of the solution data for a holonomic $F$-$\D^{\mr{an}}(0)$-module
 $\mc{M}$.
\end{empt}

\begin{empt}
 Let $\Psi:=(\Psi,\sigma,\varphi,N)$ be a Deligne module. Then the data
 $(\Psi,\Psi,\mr{id},N,\{\sigma-1\})$ defines a solution data, and we
 denote this by $j_!(\Psi)$.

 We are able to construct another solution data out of $\Psi$. We put
 $\Psi':=\Psi^{I=1}(-1)\oplus\Psi^{I\neq1}$. Let
 \begin{align*}
  c&\colon\Psi\cong\Psi^{I=1}\oplus\Psi^{I\neq1}\xrightarrow
  {N\oplus\mr{id}}\Psi^{I=1}(-1)\oplus\Psi^{I\neq1}
  \cong\Psi',\\
  v&\colon\Psi'(1)\cong\Psi^{I=1}\oplus\Psi^{I\neq1}(1)
  \xrightarrow{\mr{id}\oplus N}\Psi^{I=1}\oplus\Psi^{I\neq1}\cong\Psi\\
  v(\sigma)&\colon\Psi'\cong\Psi^{I=1}(-1)\oplus\Psi^{I\neq1}
  \xrightarrow{0\oplus(\sigma-1)}\Psi^{I=1}\oplus\Psi^{I\neq1}\cong
  \Psi.
 \end{align*}
 Then we can check that $(\Psi,\Psi',c,v,\{v(\sigma)\})$ defines a
 solution data. We denote this data by $j_+(\Psi)$. We can check that
 the homomorphisms of pairs
 $(\mr{id},c)\colon(\Psi,\Psi)\rightarrow(\Psi,\Psi')$ induces a
 canonical homomorphism of solution data $j_!(\Psi)\rightarrow
 j_+(\Psi)$.

 Let $\Psi^{I=1,N=0}:=\mr{Ker}(N\colon\Psi^{I=1}\rightarrow
 \Psi^{I=1}(-1))$, which is a Deligne module. We have the following
 exact sequence of solution data:
 \begin{equation*}
 0\rightarrow i_+\bigl(\Psi^{I=1,N=0}\bigr)\rightarrow j_!(\Psi)
  \rightarrow j_+(\Psi)\rightarrow i_+\bigl(\Psi^{I=1}(-1)/N
  \Psi^{I=1}\bigr)\rightarrow0.
 \end{equation*}
\end{empt}

\begin{prop}
 \label{adjuncsol}
 Let $D:=(\widetilde{\Psi},\widetilde{\Phi},c,v,\{v(\sigma)\})$ be a solution
 data. There exists a canonical isomorphism
 \begin{equation*}
  \mr{adj}\colon\mr{Hom}_{\mr{Del}}(\widetilde{\Psi},\Psi)\xrightarrow{\sim}
   \mr{Hom}_{\mr{Sol}}(D,j_+(\Psi)).
 \end{equation*}
\end{prop} 
\begin{proof}
 Let $c':=c|_{\widetilde{\Psi}^{I\neq1}}\colon\widetilde{\Psi}^{I\neq1}
 \rightarrow\widetilde{\Phi}^{I\neq1}$. Let us show that $c'$ is
 an isomorphism. First, let us see the injectivity. Assume it were not
 injective, and let $x\in\widetilde{\Psi}^{I\neq1}$ be a non-zero
 element in $\mr{Ker}(c')$. Then by assumption, there exists $\sigma\in
 I$ such that $\sigma(x)\neq x$. On the other hand,
 $\sigma(x)=(1+v(\sigma)\circ c')(x)=x$, which is a contradiction.
 Now, to show the claim, we construct the inverse $v'$ of $c'$.
 Let $I'$ be the finite index subgroup of $I$ which acts trivially on
 $\widetilde{\Phi}^{I\neq1}$. We put
 \begin{equation*}
  v':=(\#I/I')^{-1}\sum_{\sigma\in
   I/I'}v(\sigma)\colon\widetilde{\Phi}^{I\neq1}\rightarrow
   \widetilde{\Psi}^{I\neq1}.
 \end{equation*} 
 Then, we see that
 \begin{equation*}
  c'\circ v'(x)=(\#I/I')^{-1}\Bigl(\sum_{\sigma\in I/I'}x
   -\sigma(x)\Bigr)=x
 \end{equation*}
 where the last equality holds since there are no trivial
 subrepresentation of $I$ in $\widetilde{\Phi}^{I\neq1}$.
 Since $c'$ is an injection, $v'$ is the inverse of $c'$. Thus the
 claim follows.

 With these preparations, let us construct the homomorphism
 $\mr{adj}$. Assume given a homomorphism
 $f\colon\widetilde{\Psi}\rightarrow\Psi$. We put
 \begin{align*}
  v_1\colon\widetilde{\Phi}^{I=1}\xrightarrow{v}\widetilde{\Psi}
 ^{I=1}(-1)\xrightarrow{f}\Psi^{I=1}(-1),\qquad
  v_2\colon\widetilde{\Phi}^{I\neq1}\xrightarrow[\sim]
  {v'}
  \widetilde{\Psi}^{I\neq1}
  \xrightarrow{f}\Psi^{I\neq1}.
 \end{align*}
 These homomorphisms define a homomorphism
 \begin{equation*}
 f'\colon\widetilde{\Phi}\cong\widetilde{\Phi}^{I=1}\oplus
  \widetilde{\Phi}^{I\neq1}\xrightarrow{v_1\oplus v_2}\Psi^{I=1}(-1)
  \oplus\Psi^{I\neq1}=\Psi'.
 \end{equation*}
 This defines the following homomorphism of solution data:
 \begin{equation*}
  \xymatrix @R=0,3cm{
   D\ar[d]_{\mr{adj}(f)}&\widetilde{\Psi}\ar[d]_{f}
   \ar@<0.5ex>[rr]|{c}&&
   \widetilde{\Phi}\ar[d]^{f'}\ar@<0.5ex>[ll]|v
   \ar@/^0.5pc/@<0.5ex>[ll]^{v(\sigma)}\\
  j_+(\Psi)&\Psi\ar@<0.5ex>[rr]&&\Psi'.
   \ar@<0.5ex>[ll]\ar@/^0.5pc/@<0.5ex>[ll]
   }
 \end{equation*}
 The verification of the compatibilities between homomorphisms is
 straightforward. As a result, we have a homomorphism
 \begin{equation*}
  \mr{adj}\colon\mr{Hom}(\widetilde{\Psi},\Psi)
   \rightarrow\mr{Hom}(D,j_+(\Psi)).
 \end{equation*}
 It is easy to check that this defines an isomorphism, and the details
 are left to the reader.
\end{proof}

\begin{empt}
Let $\mc{R} _K$ or simply $\mc{R}$ be the Robba ring over $K$.
Let $\mc{M}$ be a finite free differential $\mc{R}$-module with
 Frobenius structure. Consider the following complex:
 \begin{equation*}
 0\rightarrow\mc{M}\xrightarrow{\nabla}\mc{M}\otimes\Omega^1_{\mc{R}}
  \rightarrow0
 \end{equation*}
 where $\mc{M}$ is placed at degree $0$. We denote the $i$-th cohomology
 group by $H^i_{\mr{loc}}(\mc{M})$. A result of Christol-Mebkhout
 \cite{christol-MebkhoutIV}
 shows that these cohomology groups are finite dimensional
 vector space over $K$ with the same dimension for
 $i=0,1$. Moreover, the canonical pairing
 \begin{equation*}
 H^0_{\mr{loc}}(\mc{M})\otimes H^1_{\mr{loc}}(\mc{M}^\vee)
  \rightarrow H^1_{\mr{loc}}(\mc{R})\cong K(-1)
 \end{equation*}
 is perfect from \cite[\S5]{crewfini}. Finally let us remark that
our $H^0_{\mr{loc}}(M\otimes A(x))$ is denoted by $H^0(M\otimes A(x))$
 by Crew in \cite[10.7]{crewfini}.
\end{empt}

\begin{thm}
 \label{compcrewloc}
 Let $\mc{M}$ be a holonomic $F$-$\ms{D}^{\mr{an}}(0)$-module.

 (i) We have $\Psi:=\Psi(\mc{M})\cong(\mc{B}\otimes_{\mc{R}}
 \mc{M})^{\partial=0}(1)$.

 (ii) We have
 \begin{equation*}
  \Da(j_!(\mc{M}))\cong j_!(\Psi(\mc{M})),\quad
   \Da(j_+(\mc{M}))\cong j_+(\Psi(\mc{M})).
 \end{equation*}

 (iii) We have the following isomorphisms of
 $\varphi$-$K^{\mr{ur}}$-vector spaces with semi-linear
 $\mr{Gal}(\overline{k}/k)$-action:
 \begin{gather*}
  H^0_{\mr{loc}}(\mc{M})(1)\otimes K^{\mr{ur}}\cong
  H^{-1}(i^+j_+(\mc{M}))\otimes K^{\mr{ur}}\cong\Psi^{I=1,N=0},\\
  H^1_{\mr{loc}}(\mc{M})(1)\otimes K^{\mr{ur}}\cong
  H^{0}(i^+j_+(\mc{M}))\otimes
  K^{\mr{ur}}\cong\Psi^{I=1}(-1)/N\Psi^{I=1}.
 \end{gather*}
\end{thm}

\begin{proof}
 Let us show (i). We have
 \begin{equation*}
  \Psi:=\mr{Hom}_{\mc{D}^{\mr{an}}}(\mb{D}(\mc{M}),\mc{B})\cong
   \mr{Hom}_{\mc{D}^{\mr{an}}(0)}(\mb{D}(\mc{M}),\mc{B}).
 \end{equation*}
 By using \cite[3.12]{Abe-Frob-Poincare-dual}, we see that
 $\mb{D}(\mc{M})\cong\mc{M}^\vee(-1)$. Thus, we get
 \begin{equation*}
  \mr{Hom}_{\mc{D}^{\mr{an}}(0)}(\mb{D}(\mc{M}),\mc{B})\cong
   \mr{Hom}_{\mc{D}^{\mr{an}}(0)}(\mc{M}^\vee,\mc{B})(1)\cong
   \mr{Hom}_{\mc{D}^{\mr{an}}(0)}(\mc{R},\mc{M}\otimes\mc{B})(1),
 \end{equation*}
 and the first claim follows.

 By \cite[6.1.2]{Crew-unit-disk},
 $\Da(j_!(\mc{M}))\cong j_!(\Psi)$. Moreover, by Proposition
 \ref{adjuncsol}, $\Da(j_+(\mc{M}))$ is canonically isomorphic to
 $j_+(\Psi)$, and (ii) follows. Now, we have
 \begin{equation*}
 \mr{Hom}(\mb{D}(j_+\mc{M}),\mc{O}^{\mr{an}}_{K^{\mr{ur}}})\cong
 H^0_{\mr{loc}}(\mc{M})(1),\qquad
 \mr{Ext}^1(\mb{D}(j_+\mc{M}),\mc{O}^{\mr{an}}_{K^{\mr{ur}}})\cong
 H^1_{\mr{loc}}(\mc{M})(1),
 \end{equation*}
 by \cite[3.1.10]{AM}. We have the following diagram of exact sequences:
 \begin{equation*}
  \xymatrix @R=0,3cm{
  0\ar[r]&\mr{Hom}(\mb{D}(\mc{M}),\mc{O}^{\mr{an}}_{K^{\mr{ur}}})
   \ar[r]&\Psi(j_+\mc{M})\ar[r]\ar[d]_\sim&
   \Phi(j_+\mc{M})\ar[r]\ar[d]^{\sim}&
   \mr{Ext}^1(\mb{D}(\mc{M}),\mc{O}^{\mr{an}}_{K^{\mr{ur}}})
   \ar[r]&0\\
  0\ar[r]&\Psi^{I=1,N=0}\ar[r]&\Psi\ar[r]^{c}&\Psi'\ar[r]&
   \Psi^{I=1}(-1)/N\Psi^{I=1}\ar[r]&0,
   }
 \end{equation*}
 where the first row is exact by \cite[6.1.1]{Crew-unit-disk}. These show
 that
 \begin{equation*}
 H^0_{\mr{loc}}(\mc{M})(1)\otimes K^{\mr{ur}}\cong\Psi^{I=1,N=0},
  \qquad H^1_{\mr{loc}}(\mc{M})(1)\otimes K^{\mr{ur}}\cong
 \Psi^{I=1}(-1)/N\Psi^{I=1}.
 \end{equation*}
 Finally, considering the following exact sequence, we conclude the
 proof.
 \begin{equation*}
  0\rightarrow i_+H^{-1}(i^+j_+(\mc{M}))\rightarrow j_!(\mc{M})
   \rightarrow j_+(\mc{M})\rightarrow i_+H^0(i^+j_+(\mc{M}))
   \rightarrow0.
 \end{equation*}
\end{proof}

\section{Mixed complexes}
In this section, we define ``mixed $F$-complexes'', and prove some basic
properties. We consider situation (B) in Notation and convention.
We define the notion of mixedness in the first subsection, followed by
definition of mixedness for any overholonomic $F$-complexes in the next
subsection. The last subsection is aimed to estimate some weights which
is used in the proof of our main theorem.

\subsection{Mixedness for $F$-isocrystals}
\begin{empt}
\label{nota-Vert}
 In this subsection, let $(Y,X,\PP)$ be a frame such that $Y$ is
 smooth. Let $y$ be a closed point of $Y$ with residue field $k(y)$. By abuse of notation, we
 denote by $y$ the frame $(\mr{Spec}(k(y) ),\mr{Spec}(k (y)),\PP)$, and by $i _{y}$ the
 canonical morphism of frames $y \to (Y,X,\PP)$. We recall the notation
 of paragraph \ref{extrafunctordfn} (iii).
\end{empt}

We start by recalling the following  definition of $\iota$-weight by
Deligne \cite{deligne-weil-II}:

\begin{dfn}
 (i) Let $w \in \R$. A number $\alpha$ in $\overline{\mb{Q}}_p$ is said
 to be of {\em $\iota$-weight $w$} if
 \begin{equation*}
  |\iota(\alpha)|=q^{w/2}.
 \end{equation*}

 (ii) Let $\varphi\text{-}\mr{Vect}_K$ be the category of pairs
 $(V,\varphi)$ where $V$ is a finite dimensional $K$-vector space and
 $\varphi$ is an automorphism of $V$.
 Let
 $(a,a, f)\colon
  (\mr{Spec}(k'),\mr{Spec}(k'), \PP ')
   \rightarrow(\mr{Spec}(k),\mr{Spec}(k), \mathrm{Spf} \, \V)$
 be a morphism of frames such that $a$ is finite and $f$ is the structural morphism of $\PP'$. 
We have the canonical functor
 $f _+\colon F\text{-}\mr{Ovhol}(\mr{Spec}(k'),\PP'/K)
 \to 
 \varphi\text{-}\mr{Vect}_K$. 
 Let $\E ' \in F\text{-}D ^{\mathrm{b}} _{\mathrm{ovhol}} (\mr{Spec}(k'),\PP '/K)$. We say that $\E '$ is  {\em
 $\iota$-pure} if there exists a real number $w$, the {\em weight} of
 $\E '$, such that any eigenvalue $\alpha$ of the automorphism on
 $f _+ \bigl(\mathcal{H}^j(\E ')\bigr)$ is of $\iota$-weight equal to
 $w+j$, for any integer $j$. 
 
 This definition is independent of the choice of $\PP '$, i.e.,
 if $\rho= (\mathrm{id}, \mathrm{id}, \star) \colon
 (\mr{Spec}(k'),\mr{Spec}(k'), \PP '')
 \rightarrow(\mr{Spec}(k'),\mr{Spec}(k'),\PP')$ is a morphism of frames,
 then $\rho ^!$ and $\rho _+$ induces quasi-inverse equivalences of
 categories between
 $\iota$-pure of weight $w$ objects of $F\text{-}D ^{\mathrm{b}}
 _{\mathrm{ovhol}} (\mr{Spec}(k'),\PP'/K)$ and 
 $\iota$-pure of weight $w$ objects of $F\text{-}D ^{\mathrm{b}}
 _{\mathrm{ovhol}} (\mr{Spec}(k'),\PP''/K)$.
\end{dfn}

\begin{dfn}
 \label{def-mixt}
 Let $\E \in F\text{-}\mr{Isoc}^{\dag \dag}(Y,\PP/K)$ (see notation \ref{isoc-nota}).
 \begin{enumerate}
  \item We say that $\E$ is {\em $\iota$-pure} if there exists a real
	number $w$, called the {\em weight of $\E$}, such that for any
	closed point $x$ of $Y$ the pull-back $i_x^+\E$ is $\iota$-pure
	of weight $w$.
  \item We say that $\E$ is {\em $\iota$-mixed} if all the constituents
	(i.e.\ irreducible subquotients) of $\E$ are $\iota$-pure.

  \item We say that $\E$ is {\em $\iota$-mixed of weight  $\geq w$}
	(resp.\ {\em  $\leq w$}) if $\E$ is $\iota$-mixed and if the
	weights of all the constituents of $\E$ are $\geq w $ (resp.\
	$\leq w$).
 \end{enumerate}
\end{dfn}

\begin{rem}
 \label{defpurerem}
 Let $E\in F\text{-}\mr{Isoc}^{\dag}(Y,X/K)$ be an overconvergent
 $F$-isocrystal and $\E:= \sp _{X \hookrightarrow \PP,+} (E)$ be the
 corresponding object of $F\text{-}\mr{Isoc}^{\dag \dag}(Y,\PP/K)$. We
 assume that $Y$ is pure of dimension $d _Y$. The $F$-isocrystal $\E$ is
 $\iota$-pure of weight\footnote{We notice that there is a typo in the
 definition \cite[8.3.2]{caro_devissge_surcoh}. Indeed, $w - d _{Y _i}$
 should be replaced by $w + d _{Y _i}$.}  $w$ if and only if $E$ is
 $\iota$-pure of weight $w + d_Y$ in the sense of Crew or Kedlaya (see
 \cite[5.1]{kedlaya-weilII}) by \ref{!et*-isoc}.
\end{rem}

\begin{rem}
 Let $\E \in F\text{-}\mr{Isoc}^{\dag \dag}(Y,\PP/K)$ and
 let $\E'$ be a constituent of $\E$ considered as an object
 of $\mr{Isoc}^{\dag \dag}(\PP,X,Y)$ (without Frobenius). Then $\E'$
 automatically possesses a Frobenius structure, namely we have an
 isomorphism $\E ' \riso F ^{*N} (\E')$ for some positive integer $N$.
 The argument is the same as \cite[6.0-15]{christol-MebkhoutIV}.
\end{rem}

\begin{lem}
 \label{lem-isco-stabwf!+}
 Let $\E \in F\text{-}\mr{Isoc}^{\dag \dag}(Y,\PP/K)$.
 \begin{enumerate}
  \item \label{lem-isco-stabwf!+(1)}
       The $F$-complex $\E$ is $\iota$-pure of weight $w$ if and only
	if, for every closed point $x$ of $Y$, $i ^{!} _{x} (\E)$ is
	$\iota$-pure of weight $w$.

  \item \label{lem-isco-stabwf!+(2)}
	Let $f\colon (Y',X',\PP')\to (Y,X,\PP)$ be a morphism of frames
	such that $Y'$ is smooth as well. If $\E$ is $\iota$-mixed of
	weight $\leq w$ (resp.\ $\geq w$) then so are $f^+\E$ and
	$f^!\E$.
 
  \item \label{lem-isco-stabwf!+(3)}
	If $\E$ is $\iota$-mixed of weight $\leq w$ (resp.\ $\geq w$) then
	 $\DD _{Y,\PP}(\E)$ is $\iota$-mixed of weight $\geq - w$ (resp.\ $\leq -w$).

  \item \label{lem-isco-stabwf!+(4)}
	The notion of $\iota$-mixedness is local in $Y$, namely $\E$
	is $\iota$-mixed if and only if there exists an open covering
	$\bigl\{Y_i\bigr\}$ of $Y$ such that $\E\Vert_{Y_i}$
	is $\iota$-mixed for any $i$.
 \end{enumerate}
\end{lem}
\begin{proof}
 The first statement follows from the isomorphism \ref{!et*-isoc}.
 Let us check the second claim. It suffices to show that if $\E$ is
 $\iota$-pure of weight $w$, so are $f^+\E$ and $f^!\E$. Let $x'$ be a
 closed point of $Y'$. Then to check the purity of $f^+\E$ (resp.\
 $f^!\E$), it suffices to check the purity for $i_{x'}^+f^+\E$ (resp.\
 $i_{x'}^!f^!\E$) by definition (resp.\ by the first claim). These
 follow by the transitivity of $(\cdot)^+$ and $(\cdot)^!$.
 The third claim follows by $i_x^+\circ\DD\cong \DD\circ i_x^!$
 combining with \ref{lem-isco-stabwf!+(1)}.
 The last one is a consequence of the Lemma
 \ref{irreduciblerestpreserve}.
\end{proof}

\begin{lem}
 \label{pre-sub-quot-mixed}
 Let $0\to \E' \to \E \to \E'' \to 0$ be an exact sequence in
 $F\text{-}\mr{Isoc}^{\dag \dag}(Y,\PP/K)$. The overconvergent
 $F$-isocrystal $\E$ is $\iota$-mixed (resp.\ $\iota$-pure of weight
 $w$, resp.\ $\iota$-mixed of weight $\leq w$, resp.\ $\iota$-mixed of
 weight $\geq w$) if and only if so are $\E'$ and $\E''$.
\end{lem}
\begin{proof}
 Let $\gr(\E)$ be the direct sum of the constituents of $\E$, which is
 defined up to (non-canonical) isomorphism. We can check easily that
 $\gr (\E)\riso \gr (\E') \oplus \gr (\E'')$.
\end{proof}

\begin{cor}
 \label{sub-quot-mixed}
 \begin{enumerate}
  \item\label{sub-quot-mixed-i}
       Let $\E \in F\text{-}\mr{Isoc}^{\dag \dag}(Y,\PP/K)$ and $\E'$
       be a subquotient of $\E$. If $\E$ is $\iota$-mixed
       (resp.\ $\iota$-pure of weight $w$, resp.\ $\iota$-mixed of
       weight $\leq w$, resp.\ $\iota$-mixed of weight $\geq w$) then so
       is $\E'$.

  \item\label{sub-quot-mixed-ii}
       Let $\E '\to \E \to \E''$ be an exact sequence in
       $F\text{-}\mr{Isoc}^{\dag \dag}(Y,\PP/K)$. If the
       overconvergent $F$-isocrystal $\E',\, \E ''$ are $\iota$-mixed
       (resp.\ $\iota$-pure of weight $w$, resp.\ $\iota$-mixed of
       weight $\leq w$, resp.\ $\iota$-mixed of weight $\geq w$) then so
       is $\E$.

 \end{enumerate}
\end{cor}
\begin{proof}
 These are consequences of Lemma \ref{pre-sub-quot-mixed}.
\end{proof}

\begin{dfn}
 \begin{enumerate}
  \item We say that $\E \in
	F\text{-}D ^{\mathrm{b}} _{\mathrm{isoc}} (Y, \PP/K)$
	 (see notation \ref{isoc-nota}) is
	{\em $\iota$-mixed} if $\HH^{j}(\E)$ is $\iota$-mixed
	for any integer $j$. We denote by $F\text{-}D
	^{\mathrm{b}} _{\mathrm{isoc}, \mathrm{m}}(Y, \PP/K)$ the full
	subcategory of $F\text{-}D ^{\mathrm{b}} _{\mathrm{isoc}} (Y, \PP/K)$
	consisting of $\iota$-mixed $F$-complexes.
 
  \item We say that $\E \in F\text{-}D ^{\mathrm{b}} _{\mathrm{isoc}}
	(Y, \PP/K)$ is {\em $\iota$-pure of weight $w$} (resp.\ {\em
	$\iota$-mixed of weight $\geq w$}, resp.\ {\em $\iota$-mixed of
	weight $\leq w$}) if $\HH ^{j}(\E)$ is $\iota$-pure of
	weight $w+j$ (resp.\ {\em $\iota$-mixed of weight $\geq w+j$},
	resp.\ {\em $\iota$-mixed of weight $\leq w+j$}) for any integer
	$j$.
 \end{enumerate}
\end{dfn}

\begin{rem}
\label{rem-m-out-div}
 Assume $Y$ is smooth, and let $\U$ be an open formal subscheme of $\PP$
 containing $Y$.
 Let $\E\in F\text{-}D ^{\mathrm{b}}_{\mathrm{ovhol}}(Y,\PP/K)$.
 The property ``$\E \in F\text{-}D ^{\mathrm{b}} _{\mathrm{isoc}}  (Y,
 \PP/K)$ and $\E $ is $\iota$-pure of weight $w$''
 is equivalent to the property ``$\E |\U \in F\text{-}D ^{\mathrm{b}}
 _{\mathrm{isoc}}  (Y, \U/K)$ and $\E|\U $ is $\iota$-pure of weight
 $w$''. However, even if the property $\E \in F\text{-}D ^{\mathrm{b}}
 _{\mathrm{isoc}, \mathrm{m}}  (Y, \PP/K)$ implies the property $\E |\U
 \in F\text{-}D ^{\mathrm{b}} _{\mathrm{isoc}, \mathrm{m}}  (Y, \U/K)$,
 we do not know if the converse is true (see \cite[Remark
 2.2]{Abe-Langlands}).
\end{rem}

\begin{rem}
\label{rem-proper-dense}
Suppose in this remark that $X$ is proper. Let $\E \in F\text{-}D ^{\mathrm{b}} _{\mathrm{isoc}}  (Y, \PP/K)$.
Let $U$ be a dense open subvariety of $Y$ and $\E \Vert U \in F\text{-}D ^{\mathrm{b}} _{\mathrm{isoc}}  (U, \PP/K)$
the restriction on $(U,\PP/K)$ of $\E$ (see notation \ref{extrafunctordfn}.iii)).
Then, it follows from \cite[10.8]{crewfini} (or we can follow the $p$-adic analogue of the proof
of the semicontinuity of weights of \cite[I.2.8]{KW-Weilconjecture}) that
$\E$ is $\iota$-pure of weight $w$ if and only if 
$\E \Vert U$ is $\iota$-pure of weight $w$.
The hypothesis that $X$ is proper is useful in this proof since we use the $p$-adic cohomological interpretation of the $L$-function. 
We do not know what is true when $X$ is not proper. 
\end{rem}

\begin{lem}
 \label{extensisocmix}
 Let $\E' \to \E \to \E''\xrightarrow{+}$ be an exact triangle in
 $F\text{-}D ^{\mathrm{b}} _{\mathrm{isoc}}  (Y, \PP/K)$. If the
 $F$-complexes $\E'$ and $\E''$ are $\iota$-mixed (resp.\ $\iota$-pure
 of weight $w$, resp.\ $\iota$-mixed of weight $\leq w$, resp.\
 $\iota$-mixed of weight $\geq w$), then so is $\E$.
\end{lem}
\begin{proof}
 This is a consequence of Lemma
 \ref{sub-quot-mixed}.\ref{sub-quot-mixed-ii}.
\end{proof}

\subsection{Mixedness for overholonomic $F$-complexes}
In this subsection, we define our one of the main players, $\iota$-mixed
$F$-complexes, and show some first properties. We continue to let
$(Y,X,\PP)$ be a frame.

\begin{empt}
 \label{recallofstrat}
 Recall from \cite[4.1.2]{caro-stab-prod-tens} that an ordered set of
 subschemes $\{Y_i\}_{i=1,\dots,r}$ of $Y$ is said to be a {\em smooth
 stratification} if the following
 holds: 1.\ $\{Y_i\}$ is a stratification, namely putting
 $Y_0:=\emptyset$, $Y_k$ is an open subscheme of
 $Y\setminus\bigcup_{i<k}Y_i$ and $Y=\bigcup_{1\leq i\leq r} Y_i$. 2.\
 $Y_i$ is a smooth variety. A stratification $\{Y'_j\}_{1\leq j\leq r'}$
 is said to be a {\em refinement of $\{Y_i\}$} if there exists an
 increasing function $\phi\colon[1,r]\cap \N \rightarrow[1,r']\cap \N $
 such that $\bigcup_{i\leq k}Y_i=\bigcup_{j\leq\phi(k)}Y'_j$
 for any $1\leq k\leq r$.
 The following facts are useful:
 \begin{quote}
  (*)\
  For any stratification $\{Y_i\}_{1\leq i\leq r}$ of $Y$, there is a
  smooth stratification $\{Y'_j\}_{1\leq j\leq r'}$, which is a
  refinement of $\{Y_i\}$. Moreover, given two stratifications $\{Y_i\}$
  and $\{Y'_j\}$, there exists a stratification $\{Y''_k\}$ which is a
  refinement of both $\{Y_i\}$ and $\{Y'_j\}$.
 \end{quote}
\end{empt}

\begin{dfn}
\label{mixedness-ovhol}
 Let $\E \in F\text{-}D ^{\mathrm{b}} _{\mathrm{ovhol}}(Y,\PP/K)$.
 \begin{enumerate}
  \item  We say that $\E$ is {\em $\iota$-mixed} if there exists a
	 smooth stratification $\{Y_i\}_{1\leq i\leq r}$ of $Y$
	 such that
	 $\E\Vert_{Y_i}$ is in $F\text{-}D
	 ^{\mathrm{b}}_{\mathrm{isoc},\mathrm{m}}(Y _i, \PP/K)$ for any
	 $i=1,\dots , r$. We denote by $F\text{-}D
	 ^{\mathrm{b}}_{\mathrm{m}}(Y,\PP/K)$ the full subcategory
	 of $F\text{-}D ^{\mathrm{b}}_{\mathrm{ovhol}}(Y,\PP/K)$
	 consisting of $\iota$-mixed $F$-complexes.

  \item We say that  $\E$ is {\em $\iota$-mixed of weight $\leq w$}
	if $\E$ is $\iota$-mixed and, for every closed point $x$ in $Y$,
	the $F$-complex $i _x ^+ ( \E ) $ is $\iota$-mixed of weight
	$\leq w$. We denote by $ F\text{-}D ^{\mathrm{b}} _{\leq
	w}(Y,\PP/K)$ the full subcategory of $F\text{-}D ^{\mathrm{b}}
	_{\mathrm{m}}(Y,\PP/K)$ consisting of $\iota$-mixed
	$F$-complexes of weight $\leq w$.

  \item We say that  $\E$ is {\em $\iota$-mixed of weight $\geq w$} if
	$\E$ is $\iota$-mixed and, for every closed point $x$ in $Y$,
	the $F$-complex $i _x ^! ( \E ) $ is $\iota$-mixed of weight
	$\geq w$. We denote by $ F\text{-}D ^{\mathrm{b}} _{\geq
	w}(Y,\PP/K)$ the full subcategory of $F\text{-}D ^{\mathrm{b}}
	_{\mathrm{m}}(Y,\PP/K)$ consisting of $\iota$-mixed
	$F$-complexes of weight $\geq w$.

  \item We say that $\E$ is {\em $\iota$-pure of weight $w$} if $\E$ is
	both $\iota$-mixed of weight $\leq w$ and of weight $\geq w$.
	The $\iota$-weight of $\E$ is denoted by $\mr{wt}(\E)$.

 \end{enumerate}
\end{dfn}

\begin{rem}
 \label{rem-m-geqw}
 (i) Let $\E \in F\text{-}D ^{\mathrm{b}}_{\mr{m}}(Y,\PP/K)$. Then there
 exist real numbers $w _1, w _2$ such that $\E$ is of weight both $\geq
 w_1$ and $\leq w_2$. Indeed, for an $\iota$-pure overconvergent
 $F$-isocrystal, we may take $w_1=w_2$ to be its $\iota$-weight. For
 an $\iota$-mixed overconvergent $F$-isocrystal, the constituents are
 $\iota$-pure by definition, so the claim holds. In the general case, we
 may reduce to the $\iota$-mixed overconvergent $F$-isocrystal case by
 d\'{e}vissage.

 (ii) Let $\E$ be an $F$-complex in
 $F\text{-}D^{\mr{b}}_{\mr{m}}(Y,\PP/K)$. {\em Assume that
 $\DD_{Y,\PP}(\E)$ is $\iota$-mixed as well}. Then by definition, $\E$
 is $\iota$-mixed of weight $\geq w$ (resp.\ $\leq w$) if and only if
 $\DD_{Y,\PP}(\E)$ is $\iota$-mixed of weight $\leq -w$ (resp.\ $\geq
 -w$). One of the main results of this paper is to show that, in fact,
 $\DD_{Y,\PP}(\E)$ is $\iota$-mixed if $\E$ is. See Theorem
 \ref{sumupWeilII} for the details.
 
 (iii) Suppose that $X$ is proper. Let $\E \in F\text{-}D ^{\mathrm{b}} _{\mathrm{isoc}}  (Y, \PP/K)$.
 Then, it follows from the remark \ref{rem-proper-dense} that 
the $F$-complex $\E$ is $\iota$-mixed as an object of 
$F\text{-}D ^{\mathrm{b}} _{\mathrm{isoc}}  (Y, \PP/K)$ if and only if 
$\E$ is $\iota$-mixed as an object of 
$F\text{-}D ^{\mathrm{b}} _{\mathrm{ovhol}}  (Y, \PP/K)$.
\end{rem}

\begin{prop}
 \label{pullbackstability}
 Let $u\colon (Y', X', \PP')\to (Y, X, \PP)$ be a morphism of
 frames, and let $\E$ be in $F\text{-}D ^{\mathrm{b}}_{\geq
 w}(Y,\PP/K)$. Then $u ^!(\E)$ is in $F\text{-}D ^{\mathrm{b}}_{\geq
 w}(Y',\PP'/K)$.
\end{prop}
\begin{proof}
 Let $Y =\sqcup _{i=1,\dots , r} Y_i$ be a smooth
 stratification in $P$ such that $\E\Vert_{Y_i} \in
 F\text{-}D^{\mathrm{b}}_{\mathrm{isoc},\mathrm{m}} (Y _i, \PP/K)$ for
 any $i=1,\dots , r$. Let $b\colon Y'\rightarrow Y$ be the
 morphism defined by $u$.  We get the stratification $Y' =\sqcup
 _{i=1,\dots , r} b ^{-1}(Y_i)$. By using \ref{recallofstrat} (*), there
 exists a smooth stratification of $Y'$ in $P'$ which is a refinement
 of $Y'=\sqcup _{i=1,\dots , r} b^{-1}(Y_i)$. With Lemma
 \ref{lem-isco-stabwf!+}.\ref{lem-isco-stabwf!+(2)}, we can conclude.
\end{proof}

\begin{lem}
 \label{lem-defi-mixXY}
 Let $u =(\mr{id},\star,\star,\star)\colon (Y, X', \PP', \QQ')\to (Y, X,
 \PP, \QQ)$ be a complete morphism of l.p.\ frames. Take $\E\in
 F\text{-}D ^{\mathrm{b}}_{\mathrm{ovhol}}(Y,\PP/K)$, $\E'\in F\text{-}D
 ^{\mathrm{b}}_{\mathrm{ovhol}}(Y,\PP'/K)$, and choose
 $*\in\{\mr{m},\geq w,\leq w\}$.
 \begin{enumerate}
  \item\label{lem-defi-mixXY1}
       Suppose $Y$ to be smooth. We get $\E \in F\text{-}D
	^{\mathrm{b}}_{\mathrm{isoc},*}(Y,\PP/K)$ if and only if
	$u ^{!}(\E )\in F\text{-}D ^{\mathrm{b}}_{\mathrm{isoc},
	*}(Y,\PP'/K)$ (resp.\ $\E '\in F\text{-}D
	^{\mathrm{b}}_{\mathrm{isoc},*}(Y,\PP'/K)$ if and only if
	 $u _{+}(\E' )\in F\text{-}D ^{\mathrm{b}}_{\mathrm{isoc},
	*}(Y,\PP/K)$).

  \item\label{lem-defi-mixXY2}
       We have $\E \in F\text{-}D ^{\mathrm{b}}_*(Y,\PP/K)$ if
	and only if $u ^{!}(\E )\in F\text{-}D
	^{\mathrm{b}}_*(Y,\PP'/K)$ (resp.\ $\E '\in F\text{-}D
	^{\mathrm{b}}_*(Y,\PP'/K)$ if and only if $u _{+}(\E'
	)\in F\text{-}D ^{\mathrm{b}}_*(Y,\PP/K)$).
 \end{enumerate}
\end{lem}
\begin{proof}
 From \cite[2.5]{caro:formal6}, the functors $u _+$ and $u
 ^{!}$ induce canonical equivalence of categories between
 $F\text{-}D ^{\mathrm{b}}_{\mathrm{ovhol}}(Y,\PP/K)$ and $F\text{-}D
 ^{\mathrm{b}}_{\mathrm{ovhol}}(Y,\PP'/K)$. Hence, we reduce to checking
 the part concerning the functor $u ^{!}$.
 From \cite[2.5]{caro:formal6}, we have the isomorphism $u
 ^{!}\riso u ^{+}$. Thus, by transitivity of the (extraordinary) inverse
 image, the part concerning ``$\leq w$'' or ``$\geq w$'' is a
 consequence of that concerning ``$\mr{m}$''.

 Let us check \ref{lem-defi-mixXY1}. It amounts to check that $\E$ is
 purely of weight $w$ if and only if so is $u^+(\E)$. Take a closed
 point $y$ of $Y$. We denote the morphism $y\rightarrow
 (Y,X',\PP',\QQ')$ by $i'_y$ to tell from the morphism $i_y\colon
 y\rightarrow(Y,X,\PP,\QQ)$. Then $u\circ i'_y=i_y$. Thus by
 transitivity of $(\cdot)^+$, the claim follows by definition.
 Let us check \ref{lem-defi-mixXY2}. The ``only if'' part follows from
 Proposition \ref{pullbackstability}, so let us check the ``if''
 part. Let $Y =\sqcup _{i=1,\dots , r} Y_i$ be a smooth
 stratification of $Y$ such that for any $i=1,\dots , r$,
 $u ^{!}(\E)\Vert_{Y _i} \in F\text{-}D
 ^{\mathrm{b}} _{\mathrm{isoc}, \mathrm{m}} (Y _i, \PP'/K)$.
 Since for any subvariety $Y'$ of $Y$ appearing in this
 stratification we have $\R\underline{\Gamma} ^\dag _{Y'} \bigl(u
 ^{!}(\E)\bigr) \riso u ^{!}\bigl(\R\underline{\Gamma} ^\dag _{Y '}
 (\E)\bigr) $, and we are reduced to \ref{lem-defi-mixXY1}.
\end{proof}

\begin{lem}
 \label{wj+}
 Let $u =(j,a,g,f)\colon (Y', X', \PP', \QQ')\to (Y, X, \PP, \QQ)$  
 be a complete morphism of l.p.\ frames such that $j$ is an
 immersion. If $\E '$ is an $F$-complex in
 $F\text{-}D ^{\mathrm{b}} _{\mathrm{m}}(Y',\PP/K)$ (resp.\ $F\text{-}D
 ^{\mathrm{b}} _{\geq w}(Y',\PP/K)$) then $u _+ (\E ')$ is an
 $F$-complex in $F\text{-}D ^{\mathrm{b}} _{\mathrm{m}}(Y,\PP/K)$
 (resp.\ $F\text{-}D^{\mathrm{b}} _{\geq w}(Y,\PP/K)$).
\end{lem}
\begin{proof}
 Let $X''$ be the closure of $Y'$ in $X$, $c \colon X' \to X''$ and
 $d\colon X'' \to X$ the canonical factorization of $a$. We have $u
 =(j,d,\mr{id},\mr{id}) \circ (\mr{id},c,g,f)$. By Lemma
 \ref{lem-defi-mixXY}, we reduce to the case where $u
 =(j,d,\mr{id},\mr{id})$. In that case $u _{+} (\E') =
 \E'$. Suppose that $Y'\subset Y$ is closed. The we get the
 stratification $Y = (Y \setminus Y ') \sqcup Y'$. Let $Y' =\sqcup
 _{i=1,\dots , r} Y '_i$ be a smooth stratification of $Y'$
 such that $\E '\Vert_{Y '_i} \in F\text{-}D
 ^{\mathrm{b}} _{\mathrm{isoc}, \mathrm{m}} (Y
 '_i, \PP/K)$ for any $i=1,\dots, r$. Applying \ref{recallofstrat} (*) to the trivial
 stratification of $Y\setminus Y'$, there exist a
 smooth stratification $Y \setminus Y' = \sqcup _{i=1,\dots , s} Y
 ''_i$. Thus, we get the smooth stratification
 $Y= (  \sqcup _{i=1,\dots , s} Y ''_i) \sqcup ( \sqcup _{i=1,\dots , r}
 Y '_i) $. Since $\E'\Vert_{Y ''_i} =0$,
 we get a desired stratification. Similarly, if $Y'\subset Y$ is open,
 then we get the stratification $Y =Y'  \sqcup (Y \setminus Y ')$ and we
 proceed as above by using \ref{recallofstrat} (*).
\end{proof}

\begin{prop}
 \label{lem-2.2.8}
 \begin{enumerate}
  \item \label{lem-2.2.3(4)}
	Let $\mathcal{F}'\rightarrow\mathcal{F}\rightarrow
	\mathcal{F}''\xrightarrow{+}$ be a distinguished
	triangle in $F\text{-}D ^{\mathrm{b}}_{\mathrm{ovhol}} (Y, \PP/K)$. 
	If $\mathcal{F}'$ and $\mathcal{F}''$ are
	$\iota$-mixed (resp.\ of weight $\geq w$, resp.\ of weight $\leq
	w$), then so is $\mathcal{F}$.
	
  \item  Let $\E \in F\text{-}D ^{\mathrm{b}}
	 _{\mathrm{ovhol}}(Y,\PP/K)$. Let  $Y =\sqcup _{i=1,\dots , r} Y
	 _i$ be a stratification of $Y$. The $F$-complex $\E$ is
	 $\iota$-mixed (resp.\ $\iota$-mixed of $\iota$-weight
	 $\geq w$) if and only if, for any $i = 1,\dots r$, so are
	 the $F$-complexes $\R \underline{\Gamma} ^\dag _{Y _i} (\E)$.

  \item \label{lem-2.2.8(3)}
	Let $\E \in F\text{-}D ^{\mathrm{b}} _{\mathrm{ovhol}}(Y,   \PP/K)$. 
	The notion of $\iota$-mixedness is local in $Y$, namely $\E$
	is $\iota$-mixed if and only if there exists an open covering
	$\bigl\{Y_i\bigr\}$ of $Y$ such that $\E\Vert_{Y_i}$
	is $\iota$-mixed for any $i$.
 \end{enumerate}
\end{prop}

\begin{proof}
 Let us show the first claim. By \ref{recallofstrat} (*),
 there exists a smooth stratification $\{Y_i\}_{1\leq i\leq r}$
 of $Y$ such that $\FF'\Vert_{Y_i}$,
 $\FF\Vert_{Y_i}$, $\FF''\Vert_{Y_i}$ are in $F\text{-}D ^{\mathrm{b}}
 _{\mathrm{isoc}}(Y _i, \PP/K)$ for any $i=1,\dots , r$. We conclude
 thanks to Lemma \ref{extensisocmix}. Now, let us show the second
 statement. By Proposition \ref{pullbackstability} and Lemma
 \ref{wj+}, if $\E $ is a $F$-complex in $F\text{-}D ^{\mathrm{b}}
 _{\mathrm{m}}(Y,\PP/K)$ (resp.\
 $F\text{-}D ^{\mathrm{b}} _{\geq w}(Y,\PP/K)$) then so is
 $\R\underline{\Gamma} ^\dag _{Y_i} (\E )= j _{i+} \circ j  _i ^{!}
 (\E)$ where $j _i = (\star, \mr{id},\mr{id},\mr{id}) \colon (Y _i, X
 _i, \PP, \QQ)\to (Y, X, \PP, \QQ)$ is the morphism of frames.
 Conversely, suppose that the $F$-complex $\R \underline{\Gamma} ^\dag
 _{Y _i} (\E )$ is in $F\text{-}D ^{\mathrm{b}} _{\mathrm{m}}(Y,\PP/K)$
  (resp.\ $F\text{-}D ^{\mathrm{b}} _{\geq w}(Y,\PP/K)$)
 for any $i = 1,\dots r$.
 Then the second claim follows since $\E$ possesses the same property by
 the following triangle:
 \begin{equation*}
   \R \underline{\Gamma} ^\dag _{\sqcup _{j=i+1,\dots , r} Y _i} (\E ) \to 
   \R \underline{\Gamma} ^\dag _{\sqcup _{j=i,\dots , r} Y _i} (\E ) 
   \to \R \underline{\Gamma} ^\dag _{Y _i} (\E )
   \xrightarrow{+},
 \end{equation*}
 and by using the first part of the proof.
 By using Lemma \ref{lem-isco-stabwf!+}.\ref{lem-isco-stabwf!+(4)}, the first
 and the second claim implies the third one.
\end{proof}


Now, we show that the notion of $\iota$-mixedness only depends on
couple $(Y,X)$ and not on the auxiliary choices.

\begin{prop}
 \label{defi-mixXY}
 Let $*\in\{\mr{m},\geq w,\leq w\}$.
 \begin{enumerate}
  \item  Let $\mb{Y}$ be a couple and choose an l.p.\ frame $(Y, X, \PP,
	 \QQ)$ of $\mb{Y}$. The category $F\text{-}D
	 ^{\mathrm{b}}_{*}(Y,\PP/K)$ does not depend on the choice of
	 the l.p.\ frame of $\mb{Y}$. We
	 denote by $F\text{-}D^{\mathrm{b}}_*(\mb{Y}/K)$ the
	 corresponding category.

  \item Let $Y$ be a realizable variety, and choose an l.p.\
	frame $(Y, X, \PP, \QQ)$ of $Y$ such that $X$ is
	proper. The category $F\text{-}D^{\mathrm{b}}_*(Y,\PP/K)$ does
	not depend on the choice of the l.p.\ frame. We denote by
	$F\text{-}D ^{\mathrm{b}}_*(Y/K)$ the corresponding
	category.
 \end{enumerate}
\end{prop}
\begin{proof}
 Let us only prove the second one, since the first part of the
 proposition can be checked similarly. By using fiber products, we
 reduce to the case where there exists a complete morphism of l.p.\
 frames $u\colon (Y, X', \PP', \QQ')\to (Y, X, \PP, \QQ)$. In this case,
 the proposition is a consequence of Lemma \ref{lem-defi-mixXY}.
\end{proof}


\begin{prop}
 \label{weightotimes}
 Let $\mb{Y}$ be a couple, and $\E ,\E'\in F\text{-}D ^{\mathrm{b}}
 _{\mathrm{ovhol}}(\mb{Y}/K)$. If $\E$ is $\iota$-mixed of weight $\geq
 w$ and $\E'$ is $\iota$-mixed of weight $\geq w'$ then
 $\E\widetilde{\otimes}_{\mb{Y}}\E'$ is $\iota$-mixed of weight $\geq
 w+w'$.
\end{prop}
\begin{proof}
 By paragraph \ref{relativedual} (i), $\widetilde{\otimes}$ commutes
 with $j_+$ where $j$ is a c-immersion. Thus, by d\'{e}vissage, we may
 reduce to the case where $Y$ is smooth and $\E, \E'\in F\text{-}D
 ^{\mathrm{b}}_{\mathrm{isoc}, \mathrm{m}}(Y, \PP/K)$.
 In this isocrystal case, use paragraph \ref{relativedual} (i) again to
 show that the $\iota$-mixedness is preserved and the estimation of
 weights.
\end{proof}

\begin{lem}
 \label{finiteextmixok}
 Let $\V\rightarrow\V'$ be a finite homomorphism of complete discrete
 valuation ring with mixed characteristic $(0,p)$, and  $k\rightarrow
 k'$ be the induced homomorphism of residue fields. Let $K'$ be the
 field of fractions of $\V'$, and we put
 $X:=X\otimes_kk'$, $Y':= Y \otimes_kk'$. Let $\E \in F\text{-}D
 ^{\mathrm{b}}_{\mathrm{ovhol}}(Y,X/K)$ and $\E'$ be the induced object
 of $F\text{-}D^{\mathrm{b}}_{\mathrm{ovhol}} (Y',X'/K')$. Then $\E$
 is $\iota$-mixed (resp.\ of weight $\geq w$, resp.\ of weight $\leq
 w$) if and only if so is $\E'$.
\end{lem}
\begin{proof}
 We leave the verification to the reader.
\end{proof}

\subsection{Estimation of weights in the curve case}
In this section, we estimate the weights for curves, which is used in
the proof of the main result. The proof is
similar to that of Laumon in \cite{Laumon-transf-Fourier}, by using the
methods developed in \cite{AM}.

\begin{dfn}
 \label{defweigh}
 Let $\mb{Y}=(Y,X)$ be a couple such that $Y$ is smooth, $\E$ be
 an object of $F\text{-}\mr{Isoc}^{\dag}  (\mathbb{Y}/K)$, and
 $i_y\colon(y,y)\rightarrow\mb{Y}$ be the canonical morphism for a
 closed point $y$ of $Y$.
 The object $\E$ is said to be {\em $\iota$-real}
 if for any closed point $y$ of $Y$, the characteristic polynomial
 $\iota\,\det\bigl(1-t\cdot\mr{Frob}_y;i^+_y(\ms{E})\bigr)$,
 which is {\it a priori} in $\mb{C}[t]$, is in $\mb{R}[t]$.
\end{dfn}

\begin{thm}[\cite{crewfini}, Theorem 10.5]
 \label{crewrealpure}
 If an $F$-isocrystal $\E$ on a smooth {\em curve} over $k$ is
 $\iota$-real, then any constituent of $\E$ is $\iota$-pure.
\end{thm}
\begin{proof}
 It suffices to note that we did not assume $\E$ to be quasi-unipotent
 contrary to the statement of \cite[10.5]{crewfini} since the assumption is
 used only to assure the finiteness of $\E$ and $\bigotimes^{2k}\E$ as
 written in {\it ibid.}, which is now known to be true if there is a
 Frobenius structure.
\end{proof}

\begin{empt}
 \label{crewfini}
 Let $\X$ be a smooth formal curve over $\V$, and $x$ be a closed point. For a
 holonomic module $\E$ on $\X$, we denote by $\Psi_x(\E)$ the nearby
 cycle $\Psi(\E|_{S_x})$ using the notation of \cite[2.1.8]{AM}
 (as well as \ref{defofneavancyc}).
 The theorem below follows directly from a result of Crew.
 \begin{thm*}[\cite{crewfini}, Theorem 10.8]
 Let $\mathcal{X}$ be a smooth formal curve over $\V$, $s$ be a closed point of
 $\mathcal{X}$, and we put $\mathcal{U}:=\mathcal{X}\setminus\{s\}$. Let
 $j\colon \mathcal{U}\hookrightarrow\mathcal{X}$. Let $\E$ be an
 overconvergent $F$-isocrystal $\iota$-pure of weight $w$ on
 $\mathcal{U}$.

 (i) Let $\alpha$ (resp.\ $\beta$) be an eigenvalue of the Frobenius
 acting on $\HH^{-1}(i^+j_+\E)$ (resp.\ $\HH^0(i^+j_+\E)$). Then
 we have the following estimations:
 \begin{equation*}
  |\iota(\alpha)|\leq q^{\deg(s)\,(w-1)/2},\qquad
   |\iota(\beta)|\geq q^{\deg(s)\,(w+1)/2}
 \end{equation*}

 (ii) 
 We put $\Psi:=\Psi_s(j_+\E)$. Assume
 that the action of $N$ on $\Psi/\Psi^{I=1,N=0}$ is trivial. Let
 $\alpha$ be the eigenvalue of the Frobenius acting on the latter
 quotient. Then $|\iota(\alpha)|=q^{\deg(s)\,w/2}$.
 \end{thm*}
 \begin{proof}
  Let $E$ be a differential $F$-module around $s$ associated to
  $\E$. The nilpotent operator $N$ on the vector space
  $\Psi(E)(\cong\Psi)$ induces the monodromy filtration $M$ (cf.\
  \cite[1.6.1]{deligne-weil-II}). By
  Theorem \ref{compcrewloc} (i), $(\Psi(-1),M)$ coincides with
  $(\E_\eta,\mc{M})$ using the notation of \cite[10.6]{crewfini}, thus,
  $\mr{gr}^M_i\Psi(-1)\cong\mr{gr}^{\mc{M}}_i(\E_\eta)$.
  The weight of $E$ is $w+1$ by Remark \ref{defpurerem}, and
  \cite[Theorem 10.8]{crewfini} implies that $\mr{gr}^M_i\Psi(-1)$ is
  purely of weight $w+1+i$.
  Thus, $\Psi^{I=1,N=0}(-1)\subset\Psi^{N=0}(-1)\subset
  M_0(\Psi)(-1)$ is $\iota$-weight $\leq w+1$. We get (i) by applying
  Theorem \ref{compcrewloc} (iii) and duality.
  Under the situation of (ii), $\Psi/\Psi^{I=1,N=0}$ is
  nothing but $\mr{gr}^M_1(\Psi)$, and we get what we wanted.
 \end{proof}
\end{empt}

\begin{empt}
\label{Fourier}
 Let us briefly review the geometric Fourier transform
 \cite{{Noot-Huyghe-fourierI}}, only in the affine space case.

 Fix $\pi\in\overline{\mb{Q}}_p$ such that $\pi^{p-1}=-p$. We assume
 that $\pi\in K$. We denote by $\mc{L}_\pi$ the Artin-Schreier
 isocrystal in $F\text{-}\mr{Isoc}^{\dag\dag}(\mb{A} _k ^1/K)$.
 Since $\mc{L}_\pi$ is a direct factor of $\mr{Art}_*\mc{O}_{\mb{A}
 _k^1}$ (cf.\ \cite[1.5]{Ber-CohomologieRigide-Dwork}),
 where $\mr{Art}\colon\mb{A} _k^1\rightarrow\mb{A} _k^1$ is the
 Artin-Schreier morphism, the weight of $\mc{L}_\pi [1]$ is $0$ since
 the weight is preserved under push-forward by finite \'{e}tale
 morphism.

 Now, consider the following diagram:
 $(\mb{A} _k ^n)'\xleftarrow{p_2}\mb{A} _k  ^{2n}\xrightarrow{p_1}\mb{A}
 _k^n$, where $(\mb{A}_k^n)'$ is the ``dual affine space'', which is
 nothing but $\mb{A}^n_k$.
 We denote by $\delta \colon \mb{A} _k ^n \times (\mb{A} _k ^n)' \to
 \mb{A} _k ^1$ the canonical duality bracket induced by $t\mapsto \sum
 _{i=1} ^{n} x _i y _i$. We put\footnote{There is a typo in the
 definition of $K_\pi$ in \cite[3.1.1]{Noot-Huyghe-fourierI}: the
 shift of $K_\pi$ is not $[2-2N]$ but $[2N-2]$.} $\mc{K}_\pi:=
 \delta ^{!} (\mc{L}_\pi [-1])$. For any $\E\in
 F\text{-}D^{\mr{b}}_{\mathrm{ovhol}}(\mb{A} _k^n)$, the geometric
 Fourier transform $\ms{F}_\pi(\E)$ is defined to
 be $p_{2+}\bigl(p_1^!\E \widetilde{\otimes}_{\mb{A}^{2n}}
 \mc{K}_\pi\bigr)$ (cf.\
 \cite[3.2.1]{Noot-Huyghe-fourierI}\footnote{Notice that our twisted
 tensor product and hers are the same.}).
 Important properties for us are that:
 \begin{enumerate}
  \item \label{Fourier1} if $\E$ is a $F$-module, $\ms{F} _{\pi}(\E)[2-n]$
	is a $F$-module as well (cf.\ \cite[Theorem
	5.3.1]{Noot-Huyghe-fourierI}).

  \item \label{Fourier2}
	$\ms{F}'_{\pi}\circ\ms{F}_{\pi}(\E)\,[4-2n]\cong\E(n)$ where
	$\ms{F}' _{\pi}$ denotes the geometric Fourier transform for
	dual affine space (cf.\ \cite[Remark 3.2.8]{AM}).
 \end{enumerate}
\end{empt}

\begin{empt}
 In this subsection, we use the Fourier transform for $n=1$.
 Let $\mb{A}:=\mb{A}^1$. As in \cite[1.4.2]{Laumon-transf-Fourier},
 consider the following three types of {\em irreducible} overholonomic
 $F$-modules $\E$ on $\mb{A}$:
 \begin{itemize}
 \item[(T1)] Let $s$ be a closed point of $\mb{A}$, and
	     $i_s\colon s\hookrightarrow\mb{A}$. Then
	     $\ms{E}\cong i_{s+}V$ where $V$ is an irreducible $K$-vector
	     space with Frobenius.

 \item[(T2)] There exists a simple overholonomic $F$-module $\E'$ on
	     $\mb{A}'$ of type (T1) such that
	     $\E\cong\ms{F}_\pi(\E')$.

 \item[(T3)] We have $\E\cong j_{!+}\FF$ where
	     $j\colon\mb{A}\setminus S\hookrightarrow\mb{A}$ with $S$
	     a finite subset of $\mb{A}$ and $\FF$ is an irreducible
	     $F$-isocrystal which is not of type (T2).
 \end{itemize}
 These $F$-modules are irreducible, since Fourier transform is
 involutive by \ref{Fourier}.\ref{Fourier2} and by Proposition
 \ref{BorelVII.10.5}. Moreover, irreducible objects can be classified in
 the above three types by Theorem \ref{intemedstructhem} (i).
 By definition, $F$-modules of type (T1) are
 transformed under Fourier transform to (T2) and (T2) to (T1). Thus
 $F$-modules of type (T3) are transformed under Fourier transform to
 (T3).

 An $F$-module $\E$ is said to be {\em geometrically constant} if there
 exists an isomorphism $\E\cong\mc{O}_{\mb{A}}^{\oplus r}$
 for some $r$ as $\D$-modules {\em without Frobenius structure}.
 We may check that $\FF_{\pi}(i_{0+}(V))$ is geometrically
 constant. Moreover, for $s\neq0$, $\FF_{\pi}(i_{s+}(V))$ is ramified at
 infinity. In other words, an $F$-isocrystal $\E$ on $\mb{A}$ which is
 not geometrically constant and
 unramified at $\infty$ (i.e. is $\O$-coherent around $\infty$) is of
 type (T3).
\end{empt}

\begin{lem}
 \label{mainlemmaforestcur}
 Let $p\colon \mb{A}:=\mb{A}^1\rightarrow\mr{Spec}(k)$ be the structural
 morphism, $S$ be a set of closed points in $\mb{A}$, and
 $j\colon U:=\mb{A}\setminus S\hookrightarrow\mb{A}$ be the
 canonical inclusion. Let $\ms{E}$ be an overconvergent $F$-isocrystal
 purely of $\iota$-weight $w$ on $U$. We assume moreover that
 $\ms{E}$ is irreducible, unramified at $\infty$, and not geometrically
 constant. Then for any eigenvalue $\alpha$ of
 the Frobenius acting on $H^0(p_+j_+\ms{E})$, we have the
 estimation $|\iota(\alpha)|\geq q^{w/2}$.
\end{lem}
\begin{proof}
 Since we may change $K$ by its totally ramified extension by Lemma
 \ref{finiteextmixok}, we may assume and fix a root $\pi$ of the
 equation $x^{p-1}+p=0$ in $K$. By \cite[Proposition 4.1.6 (ii)]{AM},
 $0'\in\mb{A}'$ is the only
 singularity of $\mc{F}_{\pi}(\ms{E})$. By assumption, $\ms{E}$ is a
 simple $F$-module of type (T3). This is showing that $\ms{F}_\pi(\ms{E})$
 is of type (T3) as well, and thus there exists an irreducible
 overconvergent $F$-isocrystal $\ms{E}'$ on $\mb{A}'\setminus\{0'\}$
 such that $\ms{F}_\pi(j_{!+}\ms{E}[1])\cong j'_{!+}(\ms{E}')$ where
 $j'\colon\mb{A}'\setminus\{0'\}\hookrightarrow\mb{A}'$.
 By duality, it suffices to show that any eigenvalue of
 $\HH^0(p_!j_!\ms{E})$ is of weight $\leq w$, and by Theorem
 \ref{crewfini} (i), we are reduced to showing that
 $\HH^0(p_!j_{!+}\ms{E})$ is of weight $\leq w$.

 In this proof, for the notation being compatible with \cite{AM}, we
 introduce the following: for a smooth curve
 $q\colon X\rightarrow\mr{Spec}(k')$ where $k'$ is a field over $k$ and
 a overholonomic $F$-module $\FF$ on $X$,
 we denote $\HH^{i}(q_!\FF)$ by $H^{i+1}_{\mr{rig},c}(X,\FF)(1)$. Using
 this notation, we need to show that
 $H^1_{\mr{rig},c}(\mb{A},j_{!+}\ms{E})(1)$ is of weight $\leq w$. By
 the same argument as \cite[6.2.8]{AM} (here we need the assumption that this is unramified at $\infty$), we have an exact sequence of
 Deligne modules
 \begin{equation*}
  0\rightarrow H^1_{\mr{rig},c}(\mb{A}_{\overline{k}},j_{!+}\ms{E})
   \rightarrow\Psi_{0'}((j_{!+}\ms{E})^\wedge|_{S_{0'}})(-2)\rightarrow
   (K^{\mr{ur}}\otimes_K\ms{E}|_{s_\infty})(-1)\rightarrow
   \underbrace{H^2_{\mr{rig},c}(\mb{A}_{\overline{k}},j_{!+}\ms{E})}
   _{=0}\rightarrow0.
 \end{equation*}
 Moreover, we have
 \begin{align}
  \label{isomofmon}
  \tag{($\star$)}
  \bigl(\Psi_{0'}((j_{!+}\ms{E})^\wedge)(-2)\bigr)^{I=1,N=0}\cong
   \HH^{-1}\bigl(i_{0'}^+\,j'_{!+}(\ms{E}'(-2))\bigr)
  \cong \HH^0\bigl(p_!\,j_{!+}\ms{E}(-1)\bigr)
  \notag
  \cong H^1_{\mr{rig},c}(\mb{A}_{\overline{k}},j_{!+}\ms{E})
 \end{align}
 where $I$ denotes the inertia group at $0'$. The first isomorphism is
 deduced by Theorem \ref{compcrewloc} (iii).

 Now, {\em assume $\ms{E}'$ is $\iota$-pure of weight $w'\in\mb{R}$}. Then
 Theorem \ref{crewfini} (ii) can be applied to get
 \begin{equation*}
  \Psi_{0'}(\ms{E}')(-2)/\bigl(\Psi_{0'}(\ms{E}')(-2)\bigr)^{I=1,N=0}
   \cong
   \bigl(K^{\mr{ur}}\otimes_K\ms{E}|_{s_\infty}\bigr)(-1)
 \end{equation*}
 is $\iota$-pure of weight $w'+4$. Applying Theorem
 \ref{crewfini} (i) to $\ms{E}$ and $\mb{D}(\ms{E})$,
 $(K^{\mr{ur}}\otimes_K\ms{E}|_{s_\infty})$ is $\iota$-pure of
 $\iota$-weight $w+1$, which implies that $w'=w-1$. Applying Theorem
 \ref{crewfini} (i) to the isomorphisms \ref{isomofmon}, we get the
 desired upper bound.

 It remains to show that $\mathcal{E}'$ is $\iota$-pure. Using Theorem
 \ref{crewrealpure}, it suffices to find an $\iota$-real overconvergent
 $F$-isocrystal $\ms{G}'$ on $\mb{A}'\setminus\{0'\}$ of
 which $\ms{E}'$ is a constituent. For this, we just follow exactly the
 same line as the last half part of \cite[(4.4)]{Laumon-transf-Fourier},
 and we omit the detail here.
\end{proof}

\begin{thm}
\label{thm-f+j!+pure}
 Let $f\colon X\rightarrow\mr{Spec}(k)$ be a smooth curve, and
 $\E$ be an object of $F\text{-}D^{\mr{b}}_{\geq w}(X/K)$. Then $f_+\E$
 is of weight $\geq w$.
\end{thm}
\begin{proof}
 By d\'{e}vissage, we may assume $\E$ to be an overconvergent
 $F$-isocrystal on $X$.
 Let $Z$ be a closed subscheme of $X$. By Proposition \ref{lem-2.2.8},
 $\mb{R}\underline{\Gamma}^\dag_Z(\E)$ is $\iota$-mixed of weight $\geq
 w$. Thus by considering the localization triangle, we may replace $X$
 by $X\setminus Z$, and shrink $X$. Since, by shrinking $X$, there
 exists a non-constant morphism $g\colon
 X\xrightarrow{g}X'\subset\mb{P}^1_k$ such that $g$ is finite \'{e}tale,
 we may assume that $X\subset\mb{P}^1$. Since it
 suffices to show the theorem after taking a finite extension of $k$, we
 may assume that $\infty\in\mb{P}^1$ is contained in $X$. Let
 $\widetilde{\E}$ be the push-forward of $\E$ to $\mb{P}^1$.
 We put $j\colon\mb{A}^1\hookrightarrow\mb{P}^1$ and
 $i\colon\{\infty\}\hookrightarrow\mb{P}^1$. Consider the following
 exact sequence:
 \begin{equation*}
  \mb{R}\underline{\Gamma}^\dag_{\{\infty\}}(\E)
   \rightarrow
  \widetilde{\E}\rightarrow
  j_+(\widetilde{\E}\Vert_{\mb{A}^1})\xrightarrow{+1}.
 \end{equation*}
 Since $\mb{R}\underline{\Gamma}^\dag_{\{\infty\}}(\E)$ is $\iota$-mixed
 of weight $\geq w$, it is reduced to showing that
 $p_+(\widetilde{\E}\Vert_{\mb{A}^1})$ is of weight $\geq w$ where
 $p\colon\mb{A}^1\rightarrow\mr{Spec}(k)$ is the structural
 morphism. We may assume $\E$ to be irreducible. When
 $\E$ is not geometrically constant, then the theorem follows by Lemma
 \ref{mainlemmaforestcur}. Otherwise, the verification is easy.
\end{proof}

\begin{rem*}
 In fact, this theorem is not completely new, and by Remark
 \ref{defpurerem}, it follows directly from the main theorem of
 \cite{kedlaya-weilII}. However, the proof is independent.
\end{rem*}

\section{Monodromy filtration of a convergent log-isocrystal}
Throughout this section, we consider situation (B) in Notation and
convention. Before starting the section, let us fix notation.

\subsection{Notation}
\label{chap3}
Until \S\ref{subsectionFrob3.5}, we consider the following situation unless
otherwise stated.

\begin{empt}
\label{3.1.1}
Let $\phi \colon \X \to \widehat{\mathbb{A}} ^d _\V$ be an \'etale morphism 
of affine formal $\V$-schemes 
and let $t _1,\dots , t _d$ be the corresponding local coordinates, i.e. 
the image of the coordinates $x _1,\dots, x _d$ 
under the homomorphism of complete $\V$-algebras
$\V \{©¢x _1,\dots, x _d \}
\to 
\Gamma (\X ,\O _\X)$
given by $\phi$.
For any $i= 1,\dots ,d$, we put $\ZZ _i
 = V (t _i)$. 
 We denote by  $A := \Gamma (\X ,\O _\X)$, $B := A / t _1 A$, so $\ZZ _1 : = \Spf B$. 
 Fix $1\leq r \leq d$, and we denote by $\mathfrak{D}$ the strict
 normal crossing divisor of $\X$ whose irreducible components are that of $\ZZ
 _2\cup  \dots\cup  \ZZ _r$ and $\mathfrak{D}=\emptyset$ if $r=1$. Put $\ZZ
 := \ZZ _1 \cup \mathfrak{D}$.  We get a strict normal crossing divisor
 of $\ZZ _1$ defined by $\mathfrak{D} _1 := \bigcup _{i=2} ^{r} \ZZ _1
 \cap \ZZ _i$. We put $\Y := \X \setminus \ZZ$, and let $j\colon  \Y
 \subset \X$ be the open immersion, $u\colon  (\X, \ZZ) \to \X$, $u _1
 \colon (\X, \ZZ) \to (\X, \mathfrak{D})$,
 $v\colon  (\X, \mathfrak{D}) \to \X$, $i_1\colon \ZZ_1\hookrightarrow\X$,
 $i _1 ^{\#}\colon (\ZZ _1, \mathfrak{D} _1) \to (\X, \mathfrak{D})$,
 $w _1\colon  (\ZZ _1, \mathfrak{D} _1) \to \ZZ _1$ be the canonical
 morphisms of {\em log schemes}. We sometimes denote
 $\X ^\flat:=  (\X, \ZZ) $, $\X ^{\#}:=(\X, \mathfrak{D})$, 
 $\ZZ _1 ^{\#}:= (\ZZ _1, \mathfrak{D} _1)$ for simplicity.
\end{empt}

\begin{empt}
\label{paragraph3.1.2}
Let $\overline{t} _2, \dots, \overline{t} _d$ be the global sections of $\O _{\ZZ _1}$
induced by $t _2, \dots, t _d$ and let 
$\phi _1\colon \ZZ _1 \to \widehat{\mb{A}}^{d-1} _\V$ be the \'etale morphism induced by the local 
coordinates $\overline{t} _2, \dots, \overline{t} _d$.
The homomorphism of $\V$-algebras
$\V \{©¢x _1,\dots, x _d \}
\to 
B \{ x _1 \}$
given by $x _1 \mapsto x _1$ and 
$x _i \mapsto \overline{t} _i$ for $i \geq 2$
induces the morphism of smooth formal $\V$-schemes
$id \times \phi _1\colon 
\widehat{\mathbb{A}} ^1 _\V  \times \ZZ _1 \to \widehat{\mathbb{A}} ^d _\V$.
The homomorphism of $B$-algebras
$B  \{ x _1 \}
\to 
B$
given by $x _1 \mapsto 0$ gives a closed immersion of 
smooth formal $\V$-schemes denoted by 
$i ''_1\colon 
\ZZ _1 \hookrightarrow \widehat{\mathbb{A}} ^1 _\V \times \ZZ _1$.
By putting
 $\X':=(\widehat{\mathbb{A}} ^1  \times \ZZ _1)\times_{\widehat{\mb{A}}^n}\X$,
 we get the closed immersion 
 $i' _1 \colon \ZZ _1 \hookrightarrow 
 \X '$ making commutative the diagram
  \begin{equation}
  \label{diagphiprime}
 \xymatrix @R=0,3cm {
 {\ZZ _1}
 \ar@/^0,25cm/[rr] ^-{i'' _1}
 \ar@{.>} [r] _-{i' _1}
 \ar[rd] _-{i _1}
 &
  {\X'}
  \ar[d] ^-{f}
  \ar[r] _-{\phi '} 
  \ar@{}[rd] ^-{}|\square
  &
  {\widehat{\mathbb{A}} ^1  \times \ZZ _1}
  \ar[d] ^-{id \times \phi _1}
    \\
  &
  \X
 \ar[r]
   \ar[r] _-{\phi} 
  &
  \widehat{\mb{A}}^d}
 \end{equation}
 where the square is cartesian with etale morphisms and
 where $f$ and $\phi'$ are the canonical projections. 
  Putting $\ZZ ' _1:= f ^{-1} \ZZ _1= \ZZ _1 \times _{\X}  \X'$,
 we get a section $( id, i ' _1) \colon \ZZ _1 \hookrightarrow \ZZ' _1$
 of the etale projection $ \ZZ ' _1\to \ZZ _1$. 
  Hence, by using SGA1 Exp.\ I Corollary 5.3. , 
  by shrinking $\X'$ but not $\X$, we
 may assume that $\mc{Z}_1\xrightarrow{\sim}\mc{Z}_1\times_{\X}\X'$.
 Let $g_1\colon\X'\rightarrow\mc{Z}_1$ be the canonical morphism, 
 i.e. the composition of $\phi'$ and the projection 
 $\widehat{\mathbb{A}} ^1  \times \ZZ _1 \to \ZZ _1$.
 By construction, $i '_1$ is  a section of $g _1$. 
 
  We put 
 $\mathfrak{D}':= f ^{-1}(\mathfrak{D})$, 
 $\ZZ':= f ^{-1}(\ZZ)$,
 $\X ^{\prime \flat}:=  (\X', \ZZ')$, 
 $\X ^{\prime \#}:=(\X', \mathfrak{D}')$.
    We denote by
  $f ^\# \colon \X ^{\prime \#} \to \X ^{ \#}$,
  $f ^\flat \colon \X ^{\prime \flat} \to \X ^{ \flat}$,
  $u ' _1 \colon \X ^{\prime \flat} \to \X ^{\prime \#} $,
 $g _1 ^\flat \colon \X ^{\prime \flat} \to \ZZ _1 ^{\#}$,
 $g _1 ^{\#}\colon \X ^{\prime \#} \to \ZZ _1 ^{\#}$, 
 $i _1 ^{\prime \#}\colon  \ZZ _1 ^{\#} \to \X ^{\prime \#} $,
 $v '\colon \X ^{\prime \#} \to \X '$ 
 the associated morphisms.
 We summarize the  notation and the above properties in the following commutative diagrams:
 \begin{equation}
  \label{suretsec}
   \xymatrix @R=0,4cm{
   {\ZZ _1} 
   \ar[r] _-{i _1 '}\ar@{=}[d] ^-{}\ar@{}[dr]|\square&
   {\X '} 
   \ar@/_0.5pc/[l] _-{g _1 }\ar[d] ^-{f}\ar@{}[dr]|\square&
   {\X ^{\prime \sharp}}
   \ar[l] ^-{v '}\ar[d] ^-{f ^{\sharp}}\ar@{}[dr]|\square
   \ar@/^0.5pc/[r] ^-{g _1 ^{ \sharp}}&
   {\ZZ _1 ^{\sharp}}
   \ar[l] ^-{i _1 ^{\prime \sharp}}\ar@{=}[d]
   \\ 
  {\ZZ _1}\ar[r] ^-{i _1}& 
   {\X}&
   {\X ^{\sharp}}\ar[l] ^-{v }&
   {\ZZ _1^{\sharp},}\ar[l] _-{i _1 ^{\sharp}}
   \ar@/^0,3cm/[lll] ^-{w _1}
   }\qquad
   \xymatrix @R=0,4cm{
   {\X ^{\prime \sharp}}
   \ar[d] _-{f ^{\sharp}}\ar@{}[dr]|\square&
   {\X ^{\prime \flat}}
   \ar[l] _-{u _1 '}\ar[d] ^-{f ^{\flat}}\\ 
  {\X ^{\sharp}}&
   {\X ^{\flat}}\ar[l] ^-{u _1 }
      }
 \end{equation}
 where squares (including the below ``square'' containing $w _1$) are cartesian, where the commutativity above means that
 $i ' _1$ is a section of $g _1$ 
 and 
 $i ^{\prime \#}_1$ is a section of $g _1 ^\#$.

 We note that $\mathfrak{D}'=
 g _1 ^{-1}(\mathfrak{D} _1)$.
 Indeed, since 
 $\mathfrak{D}= \phi ^{-1} (V (x _2\dots x _d))$, then using the commutativity of the square
 \ref{diagphiprime} we get 
 $\mathfrak{D}'= f ^{-1}\phi ^{-1} (V (x _2\dots x _d))
 =
 \phi ^{\prime -1} (id \times \phi _1) ^{-1} (V (x _2\dots x _d))
 =
 \phi ^{\prime -1} (\widehat{\mathbb{A}} ^1  \times \mathfrak{D} _1)
 =
 g _1 ^{-1}(\mathfrak{D} _1)$.
Moreover, since $i ' _1$ is a section of $g _1$, 
then this implies
$\mathfrak{D}'=g _1 ^{ -1}
 \circ i _1 ^{\prime -1} (\mathfrak{D}')$.

\end{empt}

\begin{empt}
[Canonical Frobenius]
\label{canonicalFrob}
Let $F _{\widehat{\A} ^{d} _{\V}}\colon  \widehat{\A} ^{d} _{\V} \to \widehat{\A} ^{d} _{\V}$ and
 $F _{\widehat{\A} ^{d-1} _{\V}} \colon  \widehat{\A} ^{d-1} _{\V} \to \widehat{\A} ^{d-1} _{\V}$ be
 the canonical ($s$-th) Frobenius endomorphisms given by the formula
 $P(\underline{x}) \mapsto P
 (\underline{x}^{q})$ for $\underline{x}=(x _1,\dots, x _d)$ or $\underline{x}=(x _1,\dots, x _{d-1})$
 respectively.
 Since the morphism $\phi \colon \X \to \widehat{\A} ^{d} _{\V}$ (resp.\ $\phi _1\colon \ZZ _1 \to
 \widehat{\A} ^{d-1} _{\V}$) 
 is \'{e}tale, the 
 $s$th Frobenius endomorphism of $X$ (resp. of $Z _1$) has a unique lifting
 $F_{\X}\colon \X \to \X$, (resp.\ $F _{\ZZ_1}\colon  \ZZ_1 \to \ZZ_1$)
 which commutes with the canonical Frobenius endomorphism 
 of $\widehat{\A} ^{d} _{\V}$ (resp. $\widehat{\A} ^{d-1} _{\V}$).
From $F_{\X}$ and $F _{\ZZ_1}$ we get canonically
the factorisations 
$F _{\X ^\flat} \colon \X ^\flat \to \X ^\flat $, 
$F _{\ZZ _1 ^\#}\colon  \ZZ _1 ^\# \to \ZZ_1 ^\#$.
 Since the local coordinates $t _1, \dots, t _d$ are fixed 
 we can call abusively these Frobenius endomorphisms ``canonical''.
\end{empt}

\begin{empt}
[Log-$\nabla$-modules,  logarithmic arithmetic $\D$-modules and convergent log isocrystals]
\label{Log-isoc-arith-Dmod}
We recall here the following facts. 
We denote by $\D _{\X^\flat}$ the usual sheaf of differential operators of $\X^\flat/\Spf \V$.  
Let $\D ^\dag_{\X^\flat}$ be the sheaf of differential operators of finite level and infinite order 
 on the formal log-scheme $\X^\flat$, which is a kind of weak completion of $\D _{\X^\flat}$ as $\O _{\X}$-ring (e.g. see 
\cite[4.4]{caro_log-iso-hol} for the explicit construction).
We denote by $\sp \colon \X _K \to \X$ the specialisation morphism
from $\X _K$, the rigid analytic space of $\X$, to $\X$.

Let $\FF$ be an $\O _{\X,\Q}$-module. 
A structure of 
$\D _{\X^\flat,\Q}$-module on $\FF$ extending that of 
$\O _{\X,\Q}$-module correspond bijectively to an integrable log connection 
$\FF \to \FF \otimes _{\O _{\X}} \Omega ^{1 } _{\X ^\flat/\Spf \V}$.
Hence, similarly than in \cite[4.4.2]{Be1}, since $\X$ is affine (otherwise we have to replace ``projective'' by ``locally projective''), 
we check that the functors
$\sp _*$ and $\sp ^*$ induce quasi-inverse equivalences between the category of 
$\D _{\X^\flat,\Q}$-modules which are projective of
 finite type over $\O _{\X,\Q}$ and the  category of 
 log-$\nabla$-modules on $\X _K$ with respect to $t _1,\dots, t _r\in \Gamma (\X _K, \O _{\X _K})$  
 (relative to $\mathrm{Max} (K)$), where  the last category is defined in \cite[2.3.7]{kedlaya-semistableI}. 
On the last category, Kedlaya defined in \cite[2.3.9]{kedlaya-semistableI}
the notion of residue of $\nabla$ along $V(t _i)$, which can be translated similarly for 
 a $\D _{\X^\flat,\Q}$-module, projective of
 finite type over $\O _{\X,\Q}$ (e.g. see \cite[1.1]{caro-Tsuzuki} for the definition of the residues
 or read the beginning of \ref{def-monod-filpre}). 
From \cite[3.2.14]{kedlaya-semistableI}, 
the category of log-$\nabla$-modules on $\X _K$ with respect to $t _1,\dots, t _r\in \Gamma (\X _K, \O _{\X _K})$  
with nilpotent residues is abelian. 
Then so is the category of $\D _{\X^\flat,\Q}$-modules, projective of
 finite type over $\O _{\X,\Q}$ and with nilpotent residues. 
 Kedlaya defined in \cite[6.3.1]{kedlaya-semistableI}, the notion 
 of a {\it convergent} log-$\nabla$-modules on $\X _K$ with respect to $t _1,\dots, t _r\in \Gamma (\X _K, \O _{\X _K})$  
with nilpotent residues.  
As in Theorem \cite[4.15]{caro_log-iso-hol} (the coherence as $\D _{\X^\flat,\Q}$-module is useless in the definition
since this is a consequence of the convergence condition and of the finiteness condition as $\O _{\X,\Q}$-module)
and with the remark \cite[4.10]{caro_log-iso-hol}, 
we get that
for any  $\D _{\X^\flat,\Q}$-module $\G$, projective of
 finite type over $\O _{\X,\Q}$ and with nilpotent residues,
 $\sp ^* (\G)$ is convergent if and only if its structure of $\D _{\X^\flat,\Q}$-module extends to a structure of 
 coherent $\D ^\dag _{\X^\flat,\Q}$-module. 
 
To sum up, the functors $\sp _*$ and $\sp ^*$ induce quasi-inverse equivalences
from the category of coherent $\D ^\dag _{\X^\flat,\Q}$-modules,  
projective of
 finite type over $\O _{\X,\Q}$ and with nilpotent residues
to  the category of convergent log-$\nabla$-modules on $\X _K$ with respect to $t _1,\dots, t _r\in \Gamma (\X _K, \O _{\X _K})$  
with nilpotent residues. 
Since $\D ^\dag _{\X^\flat,\Q}$-coherence is stable under 
kernels and cokernels, these categories are abelian subcategories of the abelian category of 
log-$\nabla$-modules on $\X _K$ with respect to $t _1,\dots, t _r\in \Gamma (\X _K, \O _{\X _K})$  
with nilpotent residues. 

We denote by 
$\mathrm{Isoc} (\X ^{\flat}/K)$ the category of coherent $\D ^\dag _{\X^\flat,\Q}$-modules,  
projective of
 finite type over $\O _{\X,\Q}$.
  As in \cite[4.9]{caro_log-iso-hol}, 
 we call these objects convergent log-isocrystal on $\X ^{\flat}/K$.
We remark that this category is not abelian (e.g. 
consider the cokernel of the morphism $\alpha _{1,\FF}$ of 
the paragraph
\ref{nota-kercoker-alpha}).

We denote by 
$\mathrm{Isoc} ^0 (\X ^{\flat}/K)$ the subcategory of 
$\mathrm{Isoc} (\X ^{\flat}/K)$ whose objects have with nilpotent residues. 
The category $\mathrm{Isoc} ^0 (\X ^{\flat}/K)$ is abelian. 
 
 \end{empt} 
 
 \begin{empt}
[Overconvergent and convergent $F$-isocrystals]
\label{F-isoc-alg-nalg}
By functoriality, 
$\D _{\X ^\flat \underset{F _{\X ^\flat}}{\longrightarrow} \X ^\flat } := F _\X ^* \D _{\X ^\flat }$ is a 
$(\D _{ \X ^\flat }, \D _{ \X ^\flat })$-bimodule, where $F _\X ^*\colon  F _\X ^{-1}\O _{\X}\to  \O _{\X}$ is the Frobenius extension
induced by $F _\X \colon \X \to \X$.
We get the functor
$F _{\X ^\flat} ^{! _{\mathrm{alg}}}
\colon 
D ^{-} (\D _{\X ^\flat ,\Q}) 
\to 
D ^{-} (\D _{\X ^\flat ,\Q}) $, 
by posing 
$F _{\X ^\flat} ^{! _{\mathrm{alg}}}
(\G) :=
\D _{\X ^\flat \underset{F _{\X ^\flat}}{\longrightarrow} \X ^\flat ,\Q} 
\otimes ^{\L} _{\D _{\X ^\flat ,\Q}} \G $,
for any $\G \in D ^{-} (\D _{\X ^\flat ,\Q}) $.

Similarly, replacing $\D$ by $\D ^\dag$, we get the functor 
$F _{\X ^\flat} ^{! }
\colon 
D ^{\mathrm{b}} _\mathrm{coh} (\D ^\dag _{\X ^\flat ,\Q}) 
\to 
D ^{\mathrm{b}} _\mathrm{coh} (\D ^\dag _{\X ^\flat ,\Q}) $, 
by posing 
$F _{\X ^\flat} ^{!}
(\G) :=
\D ^\dag _{\X ^\flat \underset{F _{\X ^\flat}}{\longrightarrow} \X ^\flat ,\Q} 
\otimes ^{\L} _{\D ^\dag _{\X ^\flat ,\Q}} \G $,
for any $\G \in D ^{\mathrm{b}} _\mathrm{coh}(\D ^\dag _{\X ^\flat ,\Q}) $.

For any 
$\G \in \mathrm{Isoc} (\X ^\flat/K)$,
since 
 $\D ^\dag _{\X ^\flat ,\Q} \otimes _{\D _{\X ^\flat ,\Q}} \G \riso 
 \G$ (see \cite[4.14]{caro_log-iso-hol}), 
by flatness and preservation of the notion of Kedlaya's convergence 
we check that
$F _{\X ^\flat} ^{! _{\mathrm{alg}}} (\G) \riso 
F _{\X ^\flat} ^{! } (\G) \riso F _{\X} ^{*} (\G)$, where the last isomorphism is an isomorphism as $\O _{\X,\Q}$-modules. 
Hence, the functors $F _{\X ^\flat} ^{! _{\mathrm{alg}}}$ and $F _{\X ^\flat} ^{!}$ factor through 
$F _{\X ^\flat} ^{! _{\mathrm{alg}}}
\colon 
\mathrm{Isoc} (\X ^\flat/K)
\to 
\mathrm{Isoc} (\X ^\flat/K)$
and
$F _{\X ^\flat} ^{!}
\colon 
\mathrm{Isoc} (\X ^\flat/K)
\to 
\mathrm{Isoc} (\X ^\flat/K)$, 
which we will simply denote by 
$F _{\X} ^*$.

 We denote by 
$F\text{-}\mathrm{Isoc} (\X ^{\flat}/K)$ the  category of coherent $F$-$\D ^\dag _{\X^\flat,\Q}$-modules,  
projective of
 finite type over $\O _{\X,\Q}$. In other words,
 an object of $F\text{-}\mathrm{Isoc} (\X ^{\flat}/K)$ is the data of an object $\FF$ of 
 $\mathrm{Isoc} (\X ^{\flat}/K)$ and an isomorphism 
 $\phi _\FF\colon F _{\X} ^* \FF \riso \FF$ in $\mathrm{Isoc} (\X ^{\flat}/K)$.
 For simplicity, if there is no risk of confusion, we denote  $(\FF ,\phi _\FF)$ by $\FF$. 
 The Frobenius structure implies 
the  nilpotence of the residues.  This yields that 
$F\text{-}\mathrm{Isoc} (\X ^{\flat}/K)$ is also an abelian category. 
Hence, we get the forgetful functor of abelian categories
$\omega\colon F\text{-}\mathrm{Isoc} (\X ^{\flat}/K) \to \mathrm{Isoc} ^0(\X ^{\flat}/K)$.

We denote by $(F\text{-})\mathrm{Isoc} ^{\dag \dag}(\X ,Z/K)$ the
abelian category of coherent ($F$-)$\D ^\dag _{\X} (\hdag Z)
_{\Q}$-modules, $\O _{\X} (\hdag Z) _{\Q}$-coherent. In other words,
 an object of $(F\text{-})\mathrm{Isoc} ^{\dag \dag}(\X ,Z/K)$  is the
data of a coherent $\D ^\dag _{\X} (\hdag Z) _{\Q}$-module $\E$
coherent over $\O _{\X} (\hdag Z) _{\Q}$ (resp. and an isomorphism 
$\phi _\E\colon F _{\X} ^* \E \riso \E$ of $\D ^\dag _{\X} (\hdag Z)
_{\Q}$-modules). Berthelot proved  that $\sp _*$ induces an equivalence
from the category of ($F$-)isocrystal on $Y$ overconvergent along $Z$
  (i.e. the category of coherent $j ^{\dag} \O _{\X _K}$-modules endowed
  with an overconvergent connection along $Z$)
 to 
that of coherent ($F$-)$\D ^\dag _{\X} (\hdag Z) _{\Q}$-modules,  
coherent over $\O _{\X} (\hdag Z) _{\Q}$ 
 (see \cite[4.6.7]{Be2} and \cite[2.2.12]{caro_courbe-nouveau}).
 Via this equivalence of categories, the constant coefficients
  correspond each other, i.e. 
  $\sp _* (j ^{\dag} \O _{\X _K}) = \O _{\X} (\hdag Z) _{\Q}$. 

We denote by $ (\hdag Z) \colon F\text{-}\mathrm{Isoc} (\X ^{\flat}/K)
\to 
F\text{-}\mathrm{Isoc} ^{\dag \dag}(\X ,Z/K)$ the exact functor defined 
by posing
$(\hdag Z)   (\FF) := 
\D ^\dag _{\X} (\hdag Z) _{\Q} \otimes _{\D ^\dag _{\X^\flat,\Q}} \FF$
for any $\FF \in  F\text{-}\mathrm{Isoc} (\X ^{\flat}/K)$
(indeed,  by 
\cite[2.2.1]{caro-Tsuzuki}, $\D ^\dag _{\X} (\hdag Z) _{\Q}= \D ^\dag _{\X^\flat} (\hdag Z) _{\Q}$). 
We also put $\FF(\hdag Z) := (\hdag Z)   (\FF)$.
 \end{empt}

\begin{empt}
\label{dfnFandE}
In the rest of the section, we fix $(\FF ,\phi _\FF)$  an object of the abelian category
$F\text{-}\mathrm{Isoc} (\X ^{\flat}/K)$, i.e., 
$\FF$ is a coherent $\D ^\dag_{\X^\flat,\Q}$-module, projective of
 finite type over $\O _{\X,\Q}$ with nilpotent residues
 and  $\phi _{\FF}\colon F _{\X} ^{*} (\FF) \riso \FF$ 
is an isomorphism of $\D ^\dag_{\X^\flat,\Q}$-modules.
 We defined 
 \cite[5,9]{caro_log-iso-hol} the direct image by $u\colon \X ^{\flat}\to \X$ of $\FF$
 by posing 
 $$u _+ (\FF) := \left ( \D ^\dag_{\X,\Q} \otimes _{\O _{\X}} \O _\X (\ZZ) \right ) \otimes ^{\L } _{\D ^\dag_{\X^\flat,\Q}} \FF$$
 (see also the notation  \cite[5.3 and 5.8]{caro_log-iso-hol}). 
 From Theorem \cite[2.2.9]{caro-Tsuzuki} (see the example \cite[2.2.10]{caro-Tsuzuki}), 
 we get $ u  _+ (\FF)\riso \FF (\hdag Z)$.
  We put $\E = u _+ (\FF)$ 
  and we denote by 
 $\phi _{\E}\colon F_{\X} ^{*} (\E) \riso \E$ the isomorphism induced by 
 $\phi _\FF$. 
 In Kedlaya's terminology of \cite{kedlaya-semistableI}, 
  $\E$ is associated (via the equivalence of categories $\sp _*$)
   to a unipotent overconvergent isocrystals with nilpotent residues. From the equivalence of categories of
 \cite[6.4.5]{kedlaya-semistableI}, 
 we remark that the isomorphism
 $\phi _{\FF}$ is determined by $\phi _\E$, i.e. is the unique one satisfying 
 $(\hdag Z)( \phi _{\FF}) = \phi _{\E}$.
\end{empt}

\subsection{Monodromy filtration of a convergent log-isocrystal with nilpotent residues}
We keep the notation of \S\ref{chap3}.

\begin{empt}
\label{nota-FisocZ1dual}
As in \ref{F-isoc-alg-nalg}, 
we denote by 
$F\text{-}\mathrm{Isoc} (\ZZ _1 ^\#/K)$ the abelian category of coherent $F$-$\D ^\dag _{\ZZ _1 ^\#,\Q}$-modules,  
projective of
 finite type over $\O _{\ZZ _1,\Q}$ and with nilpotent residues. 
If $V$ is an object of $F\text{-}\mathrm{Isoc} (\ZZ _1 ^\#/K)$, 
the underlying $\D ^\dag _{\ZZ _1 ^\#,\Q}$-module is still denoted by $V$ and
the Frobenius structure is denoted by $\phi _V$.
For any integer $n$, for any $V\in F\text{-}\mathrm{Isoc} (\ZZ _1 ^\#/K)$,
we put $V (n):= (V, q ^{-n}\phi _V)$.

a) Let $V,V' \in F\text{-}\mathrm{Isoc} (\ZZ _1 ^\#/K)$. 
The object 
$\mathcal{H} om _{\O _{\ZZ _1,\Q}} (V ,V')$ of $\mathrm{Isoc} (\ZZ _1 ^\#/K)$
is canonically endowed with a Frobenius structure via the isomorphisms
$$ F _{\ZZ _1} ^{*}\mathcal{H} om _{\O _{\ZZ _1,\Q}} (V ,V')
\riso 
\mathcal{H} om _{\O _{\ZZ _1,\Q}} ( F _{\ZZ _1} ^{*}V , F _{\ZZ _1} ^{*}V')
\riso 
\mathcal{H} om _{\O _{\ZZ _1,\Q}} (V ,V'),$$
where the first isomorphism comes from the fact that 
$V$ is $\O _{\ZZ _1,\Q}$-projective of finite type, 
where the last isomorphism is induced by
$\phi _V ^{-1}\colon V \riso F _{\ZZ _1} ^* (V)$
and 
$\phi _{V '}\colon F _{\ZZ _1} ^* (V')\riso V '$.
Hence, 
$\mathcal{H} om _{\O _{\ZZ _1,\Q}} (V ,V')$ 
is also canonically an object of 
$ F\text{-}\mathrm{Isoc} (\ZZ _1 ^\#/K)$.
In particular, 
$V ^\vee := \mathcal{H} om _{\O _{\ZZ _1,\Q}} (V , \O _{\ZZ _1,\Q}) \in F\text{-}\mathrm{Isoc} (\ZZ _1 ^\#/K)$. 
For any morphism 
$N \colon V \to V'$ of $F\text{-}\mathrm{Isoc} (\ZZ _1 ^\#/K)$, 
we denote by 
$N ^\vee\colon  V ^{\prime \vee} \to V ^\vee$ the morphism 
of $F\text{-}\mathrm{Isoc} (\ZZ _1 ^\#/K)$ defined by $\phi \mapsto \phi \circ N $.
If $\iota\colon W \hookrightarrow V$ is a monomorphism of $F\text{-}\mathrm{Isoc} (\ZZ _1 ^\#/K)$, 
then we denote by 
$W ^{\bot}$ the kernel of $\iota ^\vee$.

b) In the same way, we denote by
$V \otimes _{\O _{\ZZ _1,\Q}} V '\in F\text{-}\mathrm{Isoc} (\ZZ _1 ^\#/K)$
the 	object $V \otimes _{\O _{\ZZ _1,\Q}} V '\in\mathrm{Isoc} (\ZZ _1 ^\#/K)$
whose Frobenius structure is canonically induced by 
$\phi _V$ and $\phi _{V'}$.

c) For any integers $n, n'$, we have
$\mathcal{H} om _{\O _{\ZZ _1,\Q}} (V(n) ,V'(n'))=
\mathcal{H} om _{\O _{\ZZ _1,\Q}} (V ,V') (n'-n)$,
$(V (n) ) ^{\vee}= V ^{\vee} (-n)$, 
$V (n)\otimes _{\O _{\ZZ _1,\Q}} V ' (n')
=
(V \otimes _{\O _{\ZZ _1,\Q}} V ' )(n+n')$.

\end{empt}

\begin{dfn}
 We define the category $\mathfrak{C}$ as follows. An object $(V,N)$ of $\mathfrak{C}$  is the data of 
 an object $V$ of $F\text{-}\mathrm{Isoc} (\ZZ _1 ^\#/K)$ and a morphism
 $N \colon V(1)\to V$ of $F\text{-}\mathrm{Isoc} (\ZZ _1 ^\#/K)$
 which is a nilpotent morphism in 
 $\mathrm{Isoc} (\ZZ _1 ^\#/K)$.
 A morphism $(V,N) \to (V',N')$ is an
 $F\text{-}\D ^{\dag} _{\ZZ _1 ^\# ,\Q}$-linear morphism $V \to V'$
 commuting with the nilpotent endomorphisms. We denote by $\mathfrak{C}
 ^\mathrm{0}$ the full subcategory of $\mathfrak{C}$ of objects $(V,N)$
 such that $N=0$. We check easily that both categories $\mathfrak{C}$
 and $\mathfrak{C} ^\mathrm{0}$ are abelian.
\end{dfn}

\begin{dfn}
\label{dfndualfrakC}
Let $(V,N),(V',N')  \in \mathfrak{C}$.
We define the following operations in $\mathfrak{C}$:

 \begin{enumerate}
 
   \item  The {\em internal Hom} of $(V,N)$ and $(V',N')$ in
	$\mathfrak{C}$ is 
	$\mathcal{H} om _{\O _{\ZZ _1,\Q}} \bigl((V,N), (V',N')\bigr):=
	(\mathcal{H} om _{\O _{\ZZ _1,\Q}} (V , V'), N''))$, where for any $\phi \in \mathcal{H} om _{\O _{\ZZ _1,\Q}} (V , V')
	(1)$, we put $N''(\phi):=  N' \circ  \phi
	-\phi \circ N (-1)$.
 
  \item  In particular, the {\em dual of $(V,N)$} in $\mathfrak{C}$ is 
		$(V,N)	^\vee: =(V ^\vee ,- N  (-1) ^\vee)$.

  \item The {\em tensor product} of $(V,N)$ and $(V',N')$ in
	$\mathfrak{C}$ is $(V,N)\otimes  _{\O _{\ZZ _1,\Q}} (V',N'):=(V
	\otimes _{\O _{\ZZ _1,\Q}} V ' , N \otimes\mr{id} +\mr{id}
	\otimes N')$.

 \end{enumerate}
\end{dfn}

\begin{rem}
 Let $(V,N),(V',N')  \in \mathfrak{C}$. Since $V$ is a projective $\O
 _{\ZZ _1,\Q}$-module the canonical $\O _{\ZZ _1,\Q}$-linear morphism $
 V ^\vee \otimes _{\O _{\ZZ _1,\Q}} V '\to \mathcal{H} om _{\O _{\ZZ
 _1,\Q}} (V , V')$ is an isomorphism.
 We can show that this isomorphism commutes with the canonical nilpotent
 endomorphisms, and induces an isomorphism in $\mathfrak{C}$.
\end{rem}

\begin{lem}
\label{lemm-KerdualcokerN}
 Let $(V,N)\in \mathfrak{C}$. 
 Then, the evaluation homomorphism induces the canonical isomorphism:
 $\mathrm{ev} \colon (V,N) \riso ( (V , N ) ^\vee ) ^\vee$.
 Moreover, we get the canonical isomorphisms of $\mathfrak{C}
 ^\mathrm{0}$:
 \begin{equation}
  \label{KerdualcokerN}
  (\ker N (-1)) ^\vee \riso \coker\bigl(N (-1) ^\vee \bigr),
  \qquad
   (\coker N ) ^\vee \riso \ker (N ^\vee ).
 \end{equation}
\end{lem}

\begin{proof}
The proof is straightforward.
\end{proof}

\begin{dfn}
 [Monodromy filtration]
 \label{monodromyfiltration}
 Let $(V,N)\in \mathfrak{C}$. 
Copying the proof of \cite[1.6.1]{deligne-weil-II}, we check that there
 exists a unique finite increasing filtration $M$ on $V$ such that
 $NM_i\subset M_{i-2}(-1)$, and $N^k$ induces an isomorphism
 $\gr^M_k(V)\riso \gr^M_{-k}(V)(-k)$. We call this filtration the
 {\em monodromy filtration} and it will usually be denoted by $M$.
 \end{dfn}

\begin{rem}
 Let $\omega\colon F\text{-}\mathrm{Isoc} (\ZZ _1 ^\#/K) \to \mathrm{Isoc} ^0 (\ZZ _1 ^\#/K)$ 
be the forgetful functor. 
Let $\mathfrak{E}$ be the category defined in the
 same way as $\mathfrak{C}$ but without Frobenius structure 
 (i.e., we replace the abelian category 
$ F\text{-}\mathrm{Isoc} (\ZZ _1 ^\#/K) $
by $\mathrm{Isoc} ^0 (\ZZ _1 ^\#/K)$) 
 and let
 $\omega\colon \mathfrak{C} \to \mathfrak{E}$ be the forgetful functor
 defined by $\omega (V,N) = (\omega (V), \omega (N) )$. For any $(V,N)
 \in \mathfrak{C}$, from  \cite[1.6.1]{deligne-weil-II}, 
 we have a monodromy filtration on $\omega(V,N)$
 denoted by $M'$. Then by the uniqueness property of monodromy
 filtration we get that $\omega(M_i)\cong M'_i$.
\end{rem}

\begin{empt}
 Let $(V,N),(V',N') ,(V'',N'') \in \mathfrak{C}$. As in
 \cite[1.6.9]{deligne-weil-II}, with the notation \ref{nota-FisocZ1dual}, we put
 \begin{equation*}
  M _i \bigl((V,N) ^\vee\bigr):= M _{-i-1} (V,N) ^{\bot},\qquad
   M _i\bigl((V',N')\otimes _{\O _{\ZZ _1,\Q}} (V'',N'')\bigr) :=\sum _{i' +i'' =i} M _{i'}
   (V',N') \otimes _{\O _{\ZZ _1,\Q}} M _{i''} (V'',N'').
 \end{equation*}
From the formulas
 \begin{gather*}
  \gr _i ^{M} \bigl((V,N) ^\vee \bigr)=
  M _{-i-1} (V,N) ^{\bot}/ M _{-i} (V,N) ^{\bot}\riso 
  (M _{-i} (V,N) / M _{-i-1} (V,N))^{\vee}=
  \gr _{-i} ^{M}  (V,N),
  \\
   \gr _i ^{M}\bigl((V',N')\otimes _{\O _{\ZZ _1,\Q}}(V'',N'')\bigr) =\sum _{i' +i'' =i}
   \gr _{i'}^{M} (V',N') \otimes_{\O _{\ZZ _1,\Q}} \gr _{i''} ^{M} (V'',N''),
 \end{gather*}
 we check that these filtrations on $(V,N) ^\vee$ and
 $(V',N')\otimes _{\O _{\ZZ _1,\Q}}(V'',N'')$ are the monodromy filtrations
 (which justifies the notation).
 
\end{empt}

\begin{empt}
[Remarks on the induced filtration by inclusion and quotient on respectively
$\ker N$ and $\mathrm{coker} N$]
 \label{twistFrobstr}
  Let $(V,N)\in \mathfrak{C}$.
By construction of the monodromy filtration (see the proof of \cite[1.6.1]{deligne-weil-II} for the explicit construction),
we check by induction on the smallest integer $d \in \N$ such that $N ^{d+1}=0$
that $\ker N (-1) \subset M _0 (V,N)$.
In particular, denoting (by abuse of notation) by $M _i( \ker N (-1) )$ the induced filtration on $\ker N (-1) $
(from the monodromy filtration of $(V,N)$ by inclusion),
we get  $\gr ^{M} _i\bigl(
 \ker N (-1)\bigr) =0$ for $i >0$. 
 
 Moreover, since we have in $(V, N) ^\vee$ the inclusion
 $ M _{-1} ( (V, N) ^\vee)= 
 M _0 (V,N) ^{\bot} \subset  \ker N (-1) ^{\bot}$, 
 the image of $ M _{-1} ( (V, N) ^\vee)$
 throught
 $V ^\vee \to \bigl( \ker  N (-1)  \bigr) ^{\vee}
 \underset{\ref{KerdualcokerN}}{\riso} 
 \coker\bigl(N (-1)^\vee \bigr)
 =
  \coker (N _{V ^\vee}) $, 
 where $N _{V ^\vee}= - N (-1)^\vee$ is the the nilpotent morphism 
 of $(V, N) ^\vee$, is null. 
 Denoting (by abuse of notation) by $M _i \coker (N _{V ^\vee})$ the induced filtration 
 on $\coker (N _{V ^\vee})$ 
 (from the monodromy filtration of $V ^\vee$ by quotient), we get that 
 $M _{-1} \coker (N _{V ^\vee})=0$.
Using the biduality isomorphism of lemma
 \ref{lemm-KerdualcokerN}, 
 denoting by $M _i \coker (N )$ the induced filtration 
 on $\coker (N )$
 (from the monodromy filtration of $V$ by quotient), we get that 
 $M _{-1} \coker (N )=0$.
In particular,  
 we get 
 $\gr ^{M} _i \coker N=0$ if $i <0$.
\end{empt}

\begin{empt}
[The functor $i _1 ^*$]
Let $i ^\flat _1 \colon \ZZ _1 ^\# \hookrightarrow \X ^\flat$
be the closed immersion. 
By functoriality, 
$\D _{\ZZ _{1} ^\# \to \X ^\flat } := i _1 ^* \D _{\X ^\flat }$ is a 
$(\D _{\ZZ _{1} ^\# }, \D _{ \X ^\flat })$-bimodule, where $i _1 ^*:= \O _{\ZZ _1} \otimes _{i _1 ^{-1} \O _{\X}}-$.
We get the functor
$i _1 ^{\flat !}\colon 
D ^{-} (\D _{\X ^\flat ,\Q}) 
\to 
D ^{-} (\D _{\ZZ _{1,\Q} ^\# }) $, 
by posing 
$i _1 ^{\flat !} (\G) :=
\D _{\ZZ _{1} ^\# \to \X ^\flat ,\Q}
\otimes ^{\L} _{\D _{\X ^\flat ,\Q}} \G [-1]$,
for any $\G \in D ^{-} (\D _{\X ^\flat ,\Q}) $.
Similarly (see \ref{F-isoc-alg-nalg}), 
we have the functor 
$
F _{\ZZ ^\# _1} ^!
 \colon 
D ^{-} (\D _{\ZZ _{1,\Q} ^\# })
\to 
D ^{-} (\D _{\ZZ _{1,\Q} ^\# })$
(resp. 
$
F _{\X ^\flat} ^!
 \colon 
D ^{-} (\D _{\X ^\flat ,\Q})
\to 
D ^{-} (\D _{\X ^\flat ,\Q})$.
By transitivity (in this context of $\D$-modules and not $\D ^\dag$-modules,
we do not have to deal with the highly technical notion of Berthelot's quasi-coherence as in 
\cite{Beintro2}), we get the isomorphism
$ i _1 ^{\flat !} \circ F _{\X ^\flat} ^!
\riso 
F _{\ZZ ^\# _1} ^! \circ i _1 ^{\flat !}$ 
of functors
$D ^{-} (\D _{\X ^\flat ,\Q}) 
\to 
D ^{-} (\D _{\ZZ _{1,\Q} ^\# }) $.

By flatness and preservation of the notion of Kedlaya's convergence (as explained in 
\ref{Log-isoc-arith-Dmod}), the functor $i _1 ^{\flat !} [1]$ factors through 
$i _1 ^{\flat !} [1] \colon 
\mathrm{Isoc} (\X ^\flat/K)
\to 
\mathrm{Isoc} (\ZZ _1 ^\#/K)$.
We denote by $i _1 ^*$ this functor. 
Similarly (see \ref{F-isoc-alg-nalg}), the factorizations
$F _{\ZZ ^\# _1} ^! \colon \mathrm{Isoc} (\ZZ _1 ^\#/K)
\to \mathrm{Isoc} (\ZZ _1 ^\#/K)$ and 
$F _{\X ^\flat} ^!
\colon 
\mathrm{Isoc} (\X ^\flat/K)
\to 
\mathrm{Isoc} (\X ^\flat/K)$, 
are denoted by
$F _{\ZZ _1} ^*$ and $F _{\X} ^*$. 
Hence, the functor $i _1 ^*$ commutes canonically with Frobenius, 
i.e. we get the factorisation:
\begin{equation}
\label{i_1^*}
i _1 ^* \colon 
F \text{-}\mathrm{Isoc} (\X ^\flat/K)
\to 
F \text{-}\mathrm{Isoc} (\ZZ _1 ^\#/K)
\end{equation}

We remark that 
$i _1 ^{*}$ is equal to the cokernel of the multiplication by
 $t _1$, i.e. $i _1 ^* (\G) = \G /t _1 \G$
 for any $\G \in F \text{-}\mathrm{Isoc} (\X ^\flat/K)$.

\end{empt}

\begin{empt}
[Monodromy filtrations coming from $\FF$]
\label{def-monod-filpre}
 We put $\H:=i _1 ^* (\FF) \in F\text{-}\mathrm{Isoc} (\ZZ _1 ^\# /K)$ (recall that $\FF$ is defined in \ref{dfnFandE}
 and the functor $i_1^*$ is defined in 
\ref{i_1^*}). 
 The action of $t _1 \partial
 _1$ on $\FF$ induces the residue morphism $N _{1,\FF}\colon \H \to \H$,
 which is a $\D ^{\dag}_{\ZZ _1 ^\# ,\Q}$-linear homomorphism. 
 Let $f \in \O_{\X }$ and $m \in \FF $ so that $f \otimes m \in \O _{\X} \otimes _{F_{\X },\O _\X} \FF
 =
 F_{\X }
 ^* (\FF)\subset F_{\X} ^* \bigl(\FF (\hdag Z)\bigr)$, 
 where $F_{\X} \colon \X \to \X$ is the canonical Frobenius defined in \ref{canonicalFrob}.
For any $1\leq i
 \leq d$, we have
 \begin{equation}
  \label{2.3.4.1}
   t _{i} \partial _{i} (f \otimes m) =
   t _{i} \partial _{i} (f)  \otimes m
   + f \otimes q t _{i} \partial_{i}(m)
 \end{equation}
 in $F_{\X } ^* (\FF)$. 
 Now, we define the following Frobenius structure on $\H$ via the composition 
 \begin{equation}
 \label{phiH}
 \phi _{\H}\colon F^*_{\ZZ_1}\H
 \riso i _1 ^* F ^* _{\X} (\FF) \underset{i _1 ^* \phi _{\FF}}{\riso} i _1 ^*  (\FF) =\H
 \end{equation}

 Using the computation \ref{2.3.4.1}, we check that the left square of the following diagram is
 commutative:
 \begin{equation}
  \label{prephi2-iso-comp}
   \xymatrix@C=40pt @R=0,3cm {
   F^*_{\ZZ_1}\H
   \ar[d] _{q F^*_{\ZZ_1}(N _{1,\FF})}\ar[r] ^-{\sim}&
   i _1 ^* F ^* _{\X} (\FF) 
   \ar[r] ^-{ i _1 ^* (\phi_{\FF})}\ar[d] _{N _{1,F ^* _{\X}\FF}}&
   i _1 ^*  (\FF)\ar[d] ^{N _{1,\FF}}\\
  F^*_{\ZZ_1}\H\ar[r] _-{\sim}&
   i _1 ^* F ^* _{\X} (\FF)
   \ar[r] ^-{ i _1 ^* (\phi_{\FF})}&
   i _1 ^*  (\FF).
   }
 \end{equation}
 Since the right square is commutative by functoriality, the diagram
 \ref{prephi2-iso-comp} is commutative.
  The
 commutativity of the diagram \ref{prephi2-iso-comp} means that the homomorphism 
 $N _{1,\FF}$ induces the homomorphism 
$ N _{1,\FF}
 \colon\H(1)\rightarrow\H$ 
 in $ F\text{-}\mathrm{Isoc} (\ZZ _1 ^\# /K)$.
 Since $\FF$ has nilpotent residues, 
 then by definition 
 $N _{1,\FF}$ is nilpotent. 
In other words,
we have checked  $(\H,N _{1,\FF})\in
 \mathfrak{C}$. 
 As in \ref{monodromyfiltration}, 
 we denote by $(M _i (\H,N _{1,\FF}) ) _{i\in \Z}$ the monodromy filtration on 
 $(\H,N _{1,\FF})$.
 
\end{empt}

\subsection{Comparison with Crew's Frobenius structure}

\begin{empt}
[Notation: curve case]
 \label{notacrewFrobdef}
 Using the notation of \S\ref{chap3}, suppose in this paragraph
 that $X$ is
 of dimension $1$ and that $Z$ is $k$-rational. 
 In this context, we simply denote $t _1$ by $t$ and $\partial _1$ by $\partial$. 
 In that case, from \cite[4.3.3]{kedlaya-semistableI}, 
 we get 
 $A ^1 _K [0,1) \riso ]Z[ _{\X} $, where
 $A ^1 _K [0,1)$ is the analytic subspace of the analytic line $\A ^{1,\mathrm{an}} _K$ of the elements
 $t\in K$ such that $| t| <1$ (see the notation \cite[3.1.2]{kedlaya-semistableI}).
 We put $\O ^{\mathrm{an}} _K:= \Gamma (  A ^1 _K [0,1), \O _{A ^1 _K [0,1)})$.
 Hence, $\O ^{\mathrm{an}} _K$ is the subring of $K [[t]]$ of power series in $t$ convergent for $|t| <1$.
 We denote by $\RR _K $ 
 the Robba ring relative to $K$.
 We recall this Robba ring $\RR _K $ is the ring of formal Laurent series $f=\sum _{n\in \Z} a _n t ^n$
 in $t$ with coefficients in $K$, convergent in some  thin annulus $r < |t| <1$ for some real number $0<r<1$ (depending on $f$). 
 We remark that $\RR _K $ is equal to 
 $\Gamma (]Z[ _{\X}, j ^{\dag} \O _{\X _K}) $.
  
 We put $F := \Gamma (\X, \FF)$, 
 $F _{]Z[}:= \Gamma\bigl( ]Z[ _{\X} , \sp ^{*} (\E)\bigr)=
\O ^{\mathrm{an}} _K \otimes _{A _K} F $,
$E := \Gamma (\X, \E)$ 
and 
$E _{]Z[}:= \Gamma\bigl( ]Z[ _{\X} , \sp ^{*} (\E)\bigr)=
\RR _K \otimes _{A [ \frac{1}{t}] ^\dag _K} E$.
 We have the canonical Frobenius endomorphisms 
 $F ^* \colon \O ^{\mathrm{an}} _K  \to \O ^{\mathrm{an}} _K$ 
 (resp. $F ^* \colon \RR _K  \to \RR _K $ )
 given by  
 $\sum _{n\in \N} a _n t ^n\mapsto \sum _{n\in \N} a _n t ^{qn}$
(resp.  $\sum _{n\in \Z} a _n t ^n\mapsto \sum _{n\in \Z} a _n t ^{qn}$).
We put 
$F ^* (F _{]Z[}) := 
\O ^{\mathrm{an}} _K \otimes _{F, \O ^{\mathrm{an}} _K} F _{]Z[}$, 
$F ^* (E _{]Z[}) :=  
\RR _K  \otimes _{F, \RR _K } E _{]Z[}$.
 We set
 $F _{]Z[}  ^{\nabla ^{\infty}}:= \bigcup _{n
 >0}\ker\bigl( (t \partial) ^n \colon  F _{]Z[} \to F _{]Z[}\bigr)$ and
 $E _{]Z[}  ^{\nabla ^{\infty}}:= \cup _{n >0}\ker ( (t  \partial) ^n
 \colon  E _{]Z[} \to E _{]Z[})$.
 We remark that the $K$-linear endomorphism induced by $t \partial $ on $E _{]Z[}  ^{\nabla ^{\infty}}$
 (resp. $F _{]Z[}  ^{\nabla ^{\infty}})$ is nilpotent. 
The monomorphism of $K$-vector spaces 
$E _{]Z[}  ^{\nabla ^{\infty}}\subset  E _{]Z[}$
induces by extension the canonical morphism
 $\mathrm{can}\colon \RR _K
 \otimes _{K} E _{]Z[}  ^{\nabla ^{\infty}}\to E _{]Z[}$. 
 Similarly we get the canonical morphism
 $\mathrm{can}\colon \O ^{\mathrm{an}} _K
 \otimes _{K} F _{]Z[}  ^{\nabla ^{\infty}}\to F _{]Z[}$. 
 We denote by $i _1 ^{*}$ the cokernel of the multiplication by
 $t $.
 To simplify notation, we denote by $\phi$ the Frobenius structure on
 $F$, $F _{]Z[}$, $E _{]Z[}$, which are induced by extension from that
 on $F$.

\end{empt}

\begin{empt}
[Unipotence and consequences]
 We keep the hypotheses and notation of the paragraph
 \ref{notacrewFrobdef}. 
 Following 
\cite[3.6.9]{kedlaya-semistableI}, $E _{]Z[}$ is unipotent. 
This implies that the canonical morphism 
$\mathrm{can}\colon \RR _K
 \otimes _{K} E _{]Z[}  ^{\nabla ^{\infty}}\to E _{]Z[}$
 (e.g. see \cite[4.1]{matsuda-katz}). In particular, 
 $E _{]Z[}  ^{\nabla ^{\infty}}$ is a $K$-vector space of dimension equal to the rank of $E _{]Z[}$ 
 as $ \RR _K$-module. In fact, from the nilpotent endomorphism induced by $t \partial$ on 
 $E _{]Z[}  ^{\nabla ^{\infty}}$, we get a canonical connection on  $\RR _K
 \otimes _{K} E _{]Z[}  ^{\nabla ^{\infty}}$ and this canonical isomorphism 
 $\mathrm{can}\colon \RR _K
 \otimes _{K} E _{]Z[}  ^{\nabla ^{\infty}}\riso E _{]Z[}$ is in fact an isomorphism of 
$(\RR _K, \nabla)$-modules.
 Following \cite[3.6.2]{kedlaya-semistableI}
 (this is more explicitely written in \cite[2.12]{Shiho-logextension} and also more general),
 since $F _{]Z[}$ is log-convergent then 
 $F _{]Z[}$ is unipotent.
Hence, the canonical morphism $\mathrm{can}\colon \O ^{\mathrm{an}} _K
 \otimes _{K} F _{]Z[}  ^{\nabla ^{\infty}}\to F _{]Z[}$ is an
 isomorphism.
 Since $A [ \frac{1}{t}] ^\dag _K \otimes _{A _K} F\riso E$, we check that the canonical morphism
$\RR _K  \otimes _{\O ^{\mathrm{an}} _K} F _{]Z[} \to E _{]Z[}$ is an isomorphism.
This yields that the dimension of the $K$-vector space $F _{]Z[}  ^{\nabla ^{\infty}}$ is equal to that of 
$E _{]Z[}  ^{\nabla ^{\infty}}$. 
Hence, the canonical inclusion $F _{]Z[}  ^{\nabla ^{\infty}} \subset
 E _{]Z[}  ^{\nabla ^{\infty}}$ is a  $K$-isomorphism.

\end{empt}

\begin{empt}[Crew's Frobenius structure on $ E _{]Z[}  ^{\nabla ^{\infty}}$]
 \label{crewFrobdef}
 We keep the hypotheses and notation of the paragraph
 \ref{notacrewFrobdef}. 
 Since the canonical Frobenius endomorphism $F ^* \colon
 \RR _K  \to \RR _K $ (the local coordinate $t $ is fixed) is
 $K$-linear, we have $F ^* (\RR _K \otimes _{K}
 E _{]Z[}  ^{\nabla ^{\infty}} )\riso \RR _K \otimes _{K} F ^* (E _{]Z[}
 ^{\nabla ^{\infty}} )= \RR _K \otimes _{K} E _{]Z[}  ^{\nabla
 ^{\infty}} $. Similarly, we get $F ^* (\O ^{\mathrm{an}} _K \otimes
 _{K} F _{]Z[}  ^{\nabla ^{\infty}} )\riso \O ^{\mathrm{an}} _K  \otimes
 _{K} F _{]Z[}  ^{\nabla ^{\infty}} $. We define $\psi _\eta$ and $\psi
 $ so that the corresponding squares of the diagram
\begin{equation}
\label{diag-comp-Frob-Crew}
\xymatrix @R=0,3cm @C=0,3cm{
&
{F _{]Z[}  ^{\nabla ^{\infty}} } 
\ar@{.>}[dd] ^-(0.25){\exists ! \psi _\eta} |\hole
\ar[rr] ^-{}
\ar[ld] ^-{}
&&
{\O ^{\mathrm{an}} _K  \otimes _{K} F _{]Z[}  ^{\nabla ^{\infty}} } 
\ar@{.>}[dd] ^-(0.25){\sim } _-(0.25){\exists ! \psi} |\hole
\ar[rr] _-{F ^* (\mathrm{can})} ^-{\sim}
\ar[ld] ^-{}
&& 
{F ^* (F _{]Z[} )}
\ar[dd] ^-{\sim } _-{\phi}
\ar[rr] _-{\pi _1^*}
\ar[ld] ^-{}
&& 
{F _{]Z[} /t F _{]Z[}}
\ar[dd] ^-{\sim } _-{i _1 ^{*}(\phi)}
\\ 
{E _{]Z[}  ^{\nabla ^{\infty}} } 
\ar@{.>}[dd] ^-(0.25){\exists ! \psi _\eta}
\ar[rr] ^-{}
&&
{\RR _K \otimes _{K} E _{]Z[}  ^{\nabla ^{\infty}} } 
\ar@{.>}[dd] ^-(0.25){\sim } _-(0.25){\exists ! \psi}
\ar[rr] _-(0.25){F ^* (\mathrm{can})} ^-(0.25){\sim}
&& 
{F ^* (E _{]Z[} )}
\ar[dd] ^-(0.25){\sim } _-(0.25){\phi}
\\ 
&
{F _{]Z[}  ^{\nabla ^{\infty}} } 
\ar[rr] ^-{} |(0.425) \hole
\ar[ld] ^-{}
&&
{\O ^{\mathrm{an}} _K  \otimes _{K} F _{]Z[}  ^{\nabla ^{\infty}} } 
\ar[rr]  |(0.55) \hole _-(0.25){\mathrm{can}} ^-(0.25){\sim}
\ar[ld] ^-{}
&&
{F _{]Z[}} 
\ar[rr] ^-{\pi _1 ^*}
\ar[ld] ^-{}
&& 
{F _{]Z[} /t F _{]Z[} }
\\
{E _{]Z[}  ^{\nabla ^{\infty}} } 
\ar[rr] ^-{}
&&
{\RR _K \otimes _{K} E _{]Z[}  ^{\nabla ^{\infty}} } 
\ar[rr] ^-{\mathrm{can}} _-{\sim}
&&
{E _{]Z[},} 
} 
\end{equation}
 where $\pi ^* _1$ is the canonical projection, are commutative (the existence and unicity of $\psi$ 
 is obvious, for that of $\psi _\eta$ we use the fact that $\psi$ commute with connections). 
 The isomorphism $ \psi _\eta \colon E _{]Z[}
 ^{\nabla ^{\infty}} \riso  E _{]Z[}  ^{\nabla ^{\infty}}$ that we get
 is Crew's Frobenius action defined in \cite[\S10]{crewfini}.
 The isomorphism 
 $ \psi _\eta \colon F _{]Z[} ^{\nabla ^{\infty}} \riso  F _{]Z[}  ^{\nabla ^{\infty}}$
 (resp. $i _1 ^{*}(\phi)\colon F _{]Z[} /t F _{]Z[} \riso F _{]Z[} /t F _{]Z[}$)
will be by definition the Frobenius structure on 
$F _{]Z[} ^{\nabla ^{\infty}}$
(resp. $F _{]Z[} ^{\nabla ^{\infty}}$) 
induced canonically from that of $\FF$.

\end{empt}

\begin{lem}
[Comparison with Crew's Frobenius structure]
\label{crewFrobcomp}
 We keep the hypotheses
 of  \ref{notacrewFrobdef}. 
  We have a canonical  $K$-isomorphism of the form
 $$ E _{]Z[}  ^{\nabla ^{\infty}}
 \riso F /t F$$ 
which  commutes with Frobenius (Crew's Frobenius structure of \ref{crewFrobdef} for the left term
and that defined in \ref{phiH}
for the right term) 
and with the (nilpotent) monodromy operation
 induced by $t \partial $.
 In particular, the monodromy filtrations correspond.
\end{lem}
   
\begin{proof}
 From the commutativity of the left square of the diagram \ref{diag-comp-Frob-Crew}, 
the isomorphism $F _{]Z[}  ^{\nabla ^{\infty}} \riso
 E _{]Z[}  ^{\nabla ^{\infty}}$
 commutes with Frobenius structures (denoted there by $\psi _\eta$) and with the action of $t \partial$. 
 We compute that the composition of the horizontal
 arrows $F _{]Z[}  ^{\nabla ^{\infty}} \to  F
 _{]Z[} /t F _{]Z[}$ of the diagram \ref{diag-comp-Frob-Crew} is a
 $K$-isomorphism with commutes with the actions induced by $t \partial$, 
 and from the commutativity of
 the diagram \ref{diag-comp-Frob-Crew}, this isomorphism commutes with Frobenius.
 On the other hand, by applying the functor $i ^* _1$ to $F \to F _{]Z[}$,
 we get an isomorphism commuting with Frobenius and with the actions induced by $t \partial$.
Hence we get by composition the canonical $K$-isomorphism $ E _{]Z[}  ^{\nabla ^{\infty}}
 \riso F /t F$ satisfying the required properties. 
 
\end{proof}

The next Theorem is an easy consequence of 
Crew's Theorem \cite[Theorem 10.8]{crewfini} and can be seen as a translation in the context of arithmetic $\D$-modules.
\begin{thm}[Crew]
  \label{crewhigher}
  With the notation of \S\ref{chap3}, assume that $\E$ is
  $\iota$-pure of weight $w$ and that $\E$ comes from
  an overconvergent $F$-isocrystals on $Y$. 
Then   $(\hdag D _1) (\gr ^{M} _i (\H , N _{1,\FF}))\in F\text{-}\mathrm{Isoc} ^{\dag \dag}(\ZZ_1 , D _1/K)$ 
  is $\iota$-pure of weight $w +1+i$, 
  where
    $(\hdag D _1) :=
 \D ^\dag _{\ZZ _1 } (\hdag D _1) _{\Q} \otimes _{\D ^\dag _{\ZZ _1 ^\#,\Q}} -
 \colon 
 F\text{-}\mathrm{Isoc} (\ZZ _1 ^\# /K)
\to 
F\text{-}\mathrm{Isoc} ^{\dag \dag}(\ZZ_1 , D _1/K)$
is the functor defined in \ref{F-isoc-alg-nalg}
and 
where $M$ is the monodromy filtration as defined in \ref{def-monod-filpre}.
  
 \end{thm}

 \begin{proof}
 Let $x$ be a closed point of $\ZZ _{1} \setminus \mathfrak{D} _1$, $k
 (x)$ its residue field and $\V _x$ a finite \'{e}tale $\V$-algebra
 which is a lifting of $k (x)$. We denote by $i _x\colon  \Spf \V _x
 \hookrightarrow \ZZ _1$ a lifting of the induced closed immersion by
 $x$. We denote $\Spf \V _x$ by $\{x \}$. We have to check that $i ^{!}
 _x\gr ^{M} _i (\H , N _{1,\FF}) (\hdag D _1)$ is $\iota$-pure of weight
 $w +1+i$.
 Up to a change of basis of $\X$ by the extension $\V \to \V _x$, we can
 suppose $k = k(x)$ by Lemma \ref{finiteextmixok}. Moreover, we can
  suppose that  $\mathfrak{D} _1$ is
  empty (i.e $\ZZ = \ZZ _1$) and $\ZZ \cap V\bigl( t _2 - t _2 (x), \dots, t _d - t _d
  (x)\bigr) = \{ x \}$. 
  Let $\X _1 :=
 \Spf A_1:= V\bigl( t _2 - t _2 (x), \dots, t _d - t _d (x)\bigr)$, 
 $\Y _1 := \X_1 \setminus \ZZ$, $\alpha\colon  \X_1 \hookrightarrow \X$,
 $\X _1 ^{\#}:= (\X_1, \{x \})$,
 $\alpha ^\# \colon \X _1 ^{\#} \hookrightarrow \X ^\#$, 
 $\beta \colon \{x\} \hookrightarrow \ZZ $, 
 $\FF_1:=
 \alpha ^{\#*} (\FF) \in F\text{-}\mathrm{Isoc} (\X _1 ^{\#}/K)$, 
 $\H _1:= \FF_1/t _1 \FF _1 $, 
 $\E _1:= \FF _1(\hdag \{ x\})$. Since $t _1 \partial _1$
 commutes with $\alpha ^\#$, 
 we get 
 $\beta ^* \gr ^{M} _i (\H , N   _{1,\FF})
 \riso \gr ^{M} _i (\H _1, N   _{1,\FF _1})$, 
 where the filtrations are the monodromy filtrations.  
 Hence, 
 we reduce to the curve case, i.e. $\X
 _1=\X$. Then, by using the comparison between our setting and Crew's one of the paragraph \ref{crewFrobcomp} and the remarks of
  \ref{defpurerem},
  we apply \cite[Theorem 10.8]{crewfini}.
 \end{proof}

\begin{coro}
 \label{rem-ker-coker-N}
 With the notation of \ref{crewhigher}, 
 the object
$(\hdag D _1) (\ker(N _{1, \FF})(-1)) $ 
 (resp.\ $(\hdag D _1) (\coker(N _{1, \FF}))$) 
 of 
 $ F\text{-}\mathrm{Isoc} ^{\dag \dag}(\ZZ_1 , D _1/K)$
 is
 $\iota$-mixed of weight $\leq w +1$ (resp.\ $\iota$-mixed of weight
 $\geq w +1$).

\end{coro}

\begin{proof}
Since $(\hdag D _1) $ is exact, 
 the  monodromy filtration of $(\H,N _{1,\FF})\in \mathfrak{C}$
induces a filtration on 
$(\hdag D _1) (\H)$ and then (as subobject and respectively quotient) 
on 
$\bigl(\ker N _{1,\FF} (-1)\bigr) (\hdag
 D _1)$ and 
 $(\coker N _{1,\FF}) (\hdag D _1)$. 
Denoting these filtrations by $M$, 
with the Remark \ref{twistFrobstr},
we get $\gr ^{M} _i\bigl(\bigl(\ker N _{1,\FF} (-1)\bigr) (\hdag
 D _1)\bigr) =0$ for $i >0$
 and 
$\gr ^{M} _i  (\coker N _{1,\FF} (\hdag D _1))=0$ if $i <0$. 
We conclude the proof by using  Theorem \ref{crewhigher}.
\end{proof}

\subsection{Relations between a convergent log-isocrystal and its
  associated overconvergent isocrystal}
\label{section3.3}
In this subsection, we keep the notation of \S\ref{chap3} and we
denote by $j:=(\star,\mr{id},\mr{id})\colon(Y,X,\X) \to (X,X,\X)$ the
canonical morphism of frames.
The aim of this subsection is to check Theorem
\ref{mixedisoals}. 

\begin{ntn}
[Dual functors, holonomicity]
\label{nota-dual}
Let $\PP$ be a smooth formal scheme over $\V$ and $\mathfrak{E}$ be a strict normal crossing divisor of $\PP$ relatively to $\V$.
We denote by  $\PP ^\#:= (\PP, \mathfrak{E})$ the corresponding log smooth formal scheme over $\V$ whose log structure is induced
by $\mathfrak{E}$.
Let $T$ be a divisor of the special fiber of $\PP$. 
We denote by 
$D _{\mathrm{parf}}
 ( \D ^\dag _{\PP ^\#} (\hdag T) _\Q) $ the derived category of perfect complexes of 
left  $\D ^\dag _{\PP ^\#} (\hdag T) _\Q$-modules.
When $T$ is empty, we have
$D _{\mathrm{parf}}
 ( \D ^\dag _{\PP ^\#, \Q}) 
=
D _{\mathrm{coh}} ^{\mathrm{b}}
 ( \D ^\dag _{\PP ^\#, \Q}) 
 $
 (see \cite[5.5]{caro_log-iso-hol}).
 When $T$ is not empty, we do not know so far if we have such equality. 

a) Following \cite[5.17]{caro_log-iso-hol}, we have the dual functor
$$\DD _{\PP ^\#, T}\colon 
D _{\mathrm{parf}}
 ( \D ^\dag _{\PP ^\#} (\hdag T) _\Q) 
\to 
D  _{\mathrm{parf}}
( \D ^\dag _{\PP ^\#} (\hdag T) _\Q)$$
defined by posing,
$ \DD _{\PP ^\#, T} (\G) := 
\R \mathcal{H} om _{\D ^\dag _{\PP ^\#} (\hdag T) _\Q}
( \G, \D ^\dag _{\PP ^\#} (\hdag T) _\Q \otimes _{\O _\PP} \omega ^{-1} _{\PP ^\#})
[d _P]$, for any 
$\G \in D _{\mathrm{parf}}
 ( \D ^\dag _{\PP ^\#} (\hdag T) _\Q)$.
 When $T$ is empty, we remove it from the notation. 

 As in \cite[5.19]{caro_log-iso-hol}, 
 a $ \D ^\dag _{\PP ^\#} (\hdag T) _\Q$-holonomic module 
 is by definition a 
coherent $ \D ^\dag _{\PP ^\#} (\hdag T) _\Q$-module $\mathcal{A}$
 such that, for any $i \not =0$, we have
 $\mathcal{H} ^i \DD _{\PP ^\#, T} (\mathcal{A}) =0$.
 For example, 
 let $\G$ be a log-isocrystal on 
$\PP ^\#$ overconvergent along $T$.
Then from \cite[5.17]{caro_log-iso-hol} we have
$\G \in D _{\mathrm{parf}}
 ( \D ^\dag _{\PP ^\#} (\hdag T) _\Q) $. 
If we denote by 
$\G ^\vee: = \mathcal{H} om _{\O _{\PP} (\hdag T) _\Q}
(\G, \O _{\PP} (\hdag T) _\Q)$, then 
from \cite[5.21]{caro_log-iso-hol} we have the isomorphism
\begin{equation}
\label{dualvsdualvee}
\DD _{\PP ^\#, T} (\G) 
\riso 
\G ^\vee.
\end{equation}
The isomorphism \ref{dualvsdualvee} implies
that $\G$ is $ \D ^\dag _{\PP ^\#} (\hdag T) _\Q$-holonomic.
From the isomorphism \ref{dualvsdualvee},
we also get 
the following factorizations 
$\DD _{\PP ^\#} \colon \mathrm{Isoc} ( \PP ^\#/K)
\to \mathrm{Isoc} ( \PP ^\#/K)$
and 
$\DD _{\PP ^\#} \colon \mathrm{Isoc} ^0 ( \PP ^\#/K)
\to \mathrm{Isoc} ^0 ( \PP ^\#/K)$.

b) If there is no risk of confusion,
we put $\DD _{T}:= \DD _{\PP ^\#, T}$
and we define the functor 
$(\hdag T)   \colon 
D _{\mathrm{parf}}
 ( \D ^\dag _{\PP ^\#\Q} )
 \to 
 D _{\mathrm{parf}}
 ( \D ^\dag _{\PP ^\#} (\hdag T) _\Q) $
 by setting
$(\hdag T) (\G) := 
\D ^\dag _{\PP ^\#} (\hdag T) _\Q\otimes _{ \D ^\dag _{\PP ^\#\Q}} \G$
for any $\G \in D _{\mathrm{parf}}
 ( \D ^\dag _{\PP ^\#\Q} )$.
We recall that dual functors commute with the extensions (see \cite[I.4.4]{virrion}), i.e.
\begin{equation}
\label{dual-ext}
\DD _{T} \circ (\hdag T) \riso 
(\hdag T) \circ  \DD 
\end{equation}
as functors 
$D _{\mathrm{parf}}
 ( \D ^\dag _{\PP ^\#\Q} )
 \to 
 D _{\mathrm{parf}}
 ( \D ^\dag _{\PP ^\#} (\hdag T) _\Q) $.

\end{ntn}

\begin{empt}
[Holonomicity, relative duality isomorphism of $i _1 ^\#$]
\label{holo-rel-dual}
Tensoring by $\O _{\X} (\mathfrak{D}) := \mathcal{H} om _{\O _{\X}} ( \omega _{\X }, \omega _{\X^\#})$, 
we have the evaluation isomorphism
$\omega _{\X } \otimes _{\O _\X} \O _{\X} (\mathfrak{D})
 \riso 
 \omega _{\X ^\#}$. 
Moreover, using the projection isomorphism, 
since 
$i _1 ^* \O _{\X _1} (\mathfrak{D} ) = \O _{\ZZ _1} (\mathfrak{D} _1)$, 
we get similarly
$i _{1*}\omega _{\ZZ _1}\otimes _{\O _\X} \O _{\X} (\mathfrak{D})
\riso
i _{1*} (\omega _{\ZZ _1} \otimes _{\O _{\ZZ _1}} \O _{\ZZ _1} (\mathfrak{D} _1))
\riso 
i _{1*}\omega _{\ZZ _1 ^\#}$.
Hence, by applying the tensor product by 
$\O _{\X} (\mathfrak{D}) $ over $\O _{\X}$ 
to the well known canonical isomorphism
$i _{1*}\omega _{\ZZ _1} \riso\R \mathcal{H} om _{\O _\X}
(i _{1*}\O  _{\ZZ _1} , \omega _{\X})$ (use the transitivity of the extraordinary inverse image as 
$\O $-module as defined by  Grothendieck in \cite{HaRD}), we get 
$\omega _{\ZZ _1 ^\#} \riso \R \mathcal{H} om _{\O _\X}
(i _{1*}\O  _{\ZZ _1} , \omega _{\X ^\#})$. 
Hence we can copy the construction of \cite[I.5.3.2]{Mebkhout} and then we obtain, for any 
$\G \in D _{\mathrm{parf}}
 ( \D ^\dag _{\PP ^\#} (\hdag T) _\Q) $, 
the isomorphism
\begin{equation}
\label{rel-dual-isom}
i _{1+} ^\#  (\DD _{\ZZ _{1} ^\#} (\G) )
\riso 
\DD _{\X ^\#} (i _{1+} ^\# (\G)), 
\end{equation}
which is called the relative duality isomorphism.
Notice that when the closed immersion is not exact, the isomorphism
\ref{rel-dual-isom}
is not true anymore (e.g. consider the isomorphism \ref{5.24.(ii)logisoc} which is a kind 
of twisted relative duality isomorphism).
Since the functor 
$ i _{1+} ^\# $ is exact on the category of coherent 
$\D ^\dag _{\ZZ _{1} ^\#,\Q}$-modules, 
the isomorphism 
\ref{rel-dual-isom} implies that if 
$\G$ is a holonomic $\D ^\dag _{\ZZ _{1} ^\#,\Q}$-module, then 
$ i _{1+} ^\# (\G)$ is $\D ^\dag _{\X ^\#,\Q}$-holonomic. 
\end{empt}

\begin{empt}
Using \cite[2.2.1]{caro-Tsuzuki}, 
we get $\D ^\dag _{\X} (\hdag Z) _{\Q}= \D ^\dag _{\X^\flat} (\hdag Z) _{\Q}$.
Hence, the functors  $\DD _{\X ^\flat, Z}$ and $\DD _{\X, Z}$ are equal and we 
can simply denote it  by $\DD _Z$ without risk of confusion.
 Since 
 $\FF, \DD _{\X^\flat} (\FF) \in \mathrm{Isoc} ^0 ( \X ^\flat/K)$ (see \ref{nota-dual}.a) for the last one),
 from \cite[2.2.9]{caro-Tsuzuki}, we
 get $u _+ (\FF) \riso (\hdag Z) (\FF)$, and 
 $u _+ (\DD _{\X^\flat} (\FF)) \riso (\hdag Z) (\DD _{\X^\flat} (\FF))$. Hence, we have the isomorphisms
 \begin{equation}
  \label{2.2.9bis}
   u _+\circ  \DD _{\X ^\flat} (\FF)
   \riso
   (\hdag Z) \circ \DD _{\X ^\flat} (\FF)
\underset{\ref{dual-ext}}{\riso}  
   \DD _{Z} \circ (\hdag Z) (\FF)
   \riso 
   \DD _{Z} \circ u _+ (\FF) ,
 \end{equation}
 which are the identity over $\Y$. 
 By applying Theorem
 \cite[5.24.(ii)]{caro_log-iso-hol} to 
$ \FF (-\ZZ)\in \mathrm{Isoc}  ( \X ^\flat/K)$ (we do not need in this theorem the nilpotence of the residue), since $\FF = \FF (-\ZZ) (\ZZ)$, 
we get the isomorphism
 \begin{equation}
  \label{5.24.(ii)logisoc}
   u _! (\FF) \riso u _{+} (\FF (-\ZZ)),
 \end{equation}
 which is the biduality isomorphism over $\Y$. Then we get the 
 isomorphism
 \begin{equation}
  \label{j!Eu!F}
   j_! (\E)=\DD\circ \DD _{Z} \circ u _+ (\FF)
   \underset{\ref{2.2.9bis}}{\riso}
   \DD\circ  u _+\circ  \DD _{\X ^\flat} (\FF)
   =
   u _{!} (\FF)
   \underset{\ref{5.24.(ii)logisoc}}{\riso}
   u _+( \FF (-\ZZ)),
 \end{equation}
 which is the biduality isomorphism over $\Y$.
\end{empt}

\begin{empt}
 \label{blabla-u_1!+}
By copying word for word their proof, 
the results of \cite{caro_log-iso-hol}  can be extended 
from the context of  the morphism $u \colon \X ^{\flat}\to \X$ to the context
of the morphism $u _1 \colon \X ^{\flat}\to \X ^\#$, i.e. to the case of a morphism 
 which partially forget the log structure.
 For the reader, we display here these results that we will need
 and explain briefly how they are checked.

i)  First, we follow the notation of \cite[5.1]{caro_log-iso-hol} :
 We get a left $\D ^\dag _{\X ^{\flat}, \Q}$-module (and not a 
 $\D ^\dag _{\X ^{\#}, \Q}$-module) by posing
  $\O _{\X} (\ZZ _1) : = \mathcal{H} om _{\O _{\X}} (
 \omega _{\X ^\#} ,  \omega _{\X ^\flat})$.
 For any integer $n \in \Z$, for any left $\D ^\dag _{\X ^{\flat}, \Q}$-module $\G$, 
 we set
 $\O _{\X} (n \ZZ _1):= \O _{\X} (\ZZ _1) ^{\otimes n}$ 
( where $\otimes n$ means that we tensorise $n$times as $\O _{\ZZ _1}$-module) 
and $\G (n\ZZ _1):= \G \otimes _{\O _{\X} } \O _{\X} (n \ZZ _1)$.
Let $\Mod  ( \D ^\dag  _{ \X ^\flat ,\Q})$ be the category of 
left $\D ^\dag  _{ \X ^\flat ,\Q}$-modules. 
Since $ \O _{\X} (n \ZZ _1)$ is a free $ \O _{\X} $-module of rank one, 
then the induced functors
$(n \ZZ _1)\colon 
\Mod  ( \D ^\dag  _{ \X ^\flat ,\Q})
\to 
\Mod  ( \D ^\dag  _{ \X ^\flat ,\Q})$,
given by 
$\G \mapsto \G (n\ZZ _1)$,
is exact.
We get the canonical isomorphism
$ \omega _{\X ^\flat} \otimes _{\O _\X} ( \G \otimes _{\O _{\X}}\omega _{\X ^\#} ^{-1} )
\riso
 \G (\ZZ _1)$ 
 (as in \cite[5.2.4]{caro_log-iso-hol}).
 Hence, as in \cite[5.8]{caro_log-iso-hol}, we get a 
$ (\D ^\dag  _{ \X ^\# }, \D ^\dag  _{ \X ^\flat})$-bimodule by setting 
 $\D ^\dag  _{ \X ^\# \leftarrow  \X ^\flat}:= 
  \omega _{\X ^\flat} \overset{\mathrm{d}}{\otimes} _{\O _\X} ( \D ^\dag  _{ \X ^\# }\otimes _{\O _{\X}}\omega _{\X ^\#} ^{-1} )
  \riso \D ^\dag  _{ \X ^\# } \otimes  _{\O _\X} \O _{\X} (\ZZ _1)$,
  where the symbol `$\mathrm{d}$' means that we take the right structure of right $ \D ^\dag  _{ \X ^\# }$-module of the right 
  $ \D ^\dag  _{ \X ^\# }$-bimodule
  $ \D ^\dag  _{ \X ^\# }\otimes _{\O _{\X}}\omega _{\X ^\#} ^{-1} $ and the isomorphism is induced by the evaluation map
  (in fact, in \cite[5.8]{caro_log-iso-hol} we had chosen 
  $\D ^\dag  _{ \X ^\# \leftarrow  \X ^\flat}:= \D ^\dag  _{ \X ^\# } \otimes  _{\O _\X} \O _{\X} (\ZZ _1)$
  but here we stick with a more standard notation).
 Then, we get the functor
$u _{1+}\colon  D ^{\mathrm{b}} _{\mathrm{coh}} ( \D ^\dag  _{ \X ^\flat ,\Q}) 
\to 
D ^{\mathrm{b}} _{\mathrm{coh}} ( \D ^\dag  _{ \X ^\# ,\Q}) $ defined by posing, 
 for any $\G \in D ^{\mathrm{b}} _{\mathrm{coh}} ( \D ^\dag  _{ \X ^\flat ,\Q}) $, 
 $u _{1+}  := 
 \D ^\dag  _{ \X ^\# \leftarrow  \X ^\flat, \Q} 
 \otimes ^\L _{\D ^\dag  _{ \X ^\flat, \Q}}
 \G$.

ii)  Similarly than \cite[5.12]{caro_log-iso-hol},
we put 
$ u _{1!} : = \DD _{\X ^\#} \circ u _{1+} \circ \DD _{\X ^\flat}
\colon  D ^{\mathrm{b}} _{\mathrm{coh}} ( \D ^\dag  _{ \X ^\flat ,\Q}) 
\to 
D ^{\mathrm{b}} _{\mathrm{coh}} ( \D ^\dag  _{ \X ^\# ,\Q})$. 
Using some transposition isomorphims for $\X ^\#$ and $\X ^\flat$ (i.e. 
\cite[1.18 and 1.19]{caro_log-iso-hol}), we check as for Proposition  \cite[5.14]{caro_log-iso-hol}
(see  the notation \cite[5.3]{caro_log-iso-hol})
the isomorphism 
\begin{equation}
\label{5.3-log-iso-holbis}
u _{1!} (\G ) \riso \D ^\dag _{ \X ^\#, \Q}  \otimes ^\L _{\D ^\dag _{ \X ^\flat, \Q}} \G
\end{equation}
 for any $\G \in D ^{\mathrm{b}} _{\mathrm{coh}} ( \D ^\dag  _{ \X ^\flat ,\Q}) $.
 Moreover, 
as in   \cite[5.16]{caro_log-iso-hol}, 
we have the canonical isomorphism 
\begin{equation}
\label{5.24.(ii)caro_log-iso-hol}
\iota _{1,\G}\colon u _{1!} (\G)
 \riso u _{1+} (\G (-\ZZ _1))
\end{equation}
which is the biduality isomorphism
 outside $Z _1$. The main point of the proof of the isomorphism \ref{5.24.(ii)caro_log-iso-hol}
 is to check the canonical isomorphism of right $\D ^{(0)} _{\X ^\flat}$-modules of the form
 $( \omega _{\X ^\#} \otimes _{\O _\X} \G ^{\flat}) 
 \otimes _{\D ^{(0)} _{\X ^\flat}} \G ^{\#}
 \riso
 ( \omega _{\X ^\#} \otimes _{\O _\X} \G ^{\#}) 
 \otimes _{\D ^{(0)} _{\X ^\flat}} \G ^{\flat}
 $,
 for any 
 $\D ^{(0)} _{\X ^\flat}$-bimodule
$ \G ^{\#}$
and any 
left $\D ^{(0)} _{\X ^\#}$-module 
$\G ^\flat$. The check of this last isomorphism is the same as that of
\cite[3.7.1]{caro_log-iso-hol} (more precisely, every results of \cite[3]{caro_log-iso-hol}
can be extended in this relative context : replace everywhere respectively 
$\X$ by $\X ^\#$ and $\X ^\#$ by $\X ^\flat$).

iii) Let $\G \in \mathrm{Isoc} (\X ^\flat/K)$. 
 As in \cite[2.6]{caro_log-iso-hol}, we define the following first Spencer sequence associated with $\G$ :
 \begin{equation}
 \label{firstSpencer}
 0 
 \to 
 \D  _{ \X ^\flat, \Q} \otimes _{\O _{\X}} \wedge ^d \mathcal{T} _{\X ^\flat}\otimes _{\O _{\X}}\G 
 \to 
 \dots
 \to 
 \D  _{ \X ^\flat, \Q} \otimes _{\O _{\X}} \wedge ^1 \mathcal{T} _{\X ^\flat} \otimes _{\O _{\X}}\G 
 \to 
  \D  _{ \X ^\flat, \Q} \otimes _{\O _{\X}}\G 
  \to 
  \G
  \to 
  0,
 \end{equation}
where 
$\mathcal{T} _{\X ^\flat}$ 
is the tangent space of $\X ^\flat/\V$ (recall that
$  \D  _{ \X ^\flat, \Q} =   \D  ^{(0)} _{ \X ^\flat} \otimes _{\O _\X} \O _{\X, \Q}$).
We denote by $Sp ^\bullet _{\D  _{ \X ^\flat, \Q} } (\G)$ the sequence 
\ref{firstSpencer}.
As in \cite[2.8]{caro_log-iso-hol}, we check that the sequence 
$Sp ^\bullet _{\D  _{ \X ^\flat, \Q} } (\G)$ is exact. 
Moreover, since $\G$ is a flat $\O _{\X,\Q}$-module,
copying the proof of \cite[2.10]{caro_log-iso-hol}, 
we get the exactness of the sequence
$\D  _{ \X ^\#, \Q}  \otimes _{\D  _{ \X ^\flat, \Q} } Sp ^\bullet _{\D  _{ \X ^\flat, \Q} } (\G)$.
Since the extension 
$\D  _{ \X ^\#, \Q}  \to \D ^\dag _{ \X ^\#, \Q} $ is flat, 
we get the exactness of the sequence
$\D  ^\dag _{ \X ^\#, \Q}  \otimes _{\D  _{ \X ^\flat, \Q} } Sp ^\bullet _{\D  _{ \X ^\flat, \Q} } (\G)$.
Since 
$Sp ^\bullet _{\D  _{ \X ^\flat, \Q} } (\G)$ gives a flat resolution of 
$\G$ by left $\D  _{ \X ^\flat, \Q} $-modules, 
this implies that 
the canonical morphism
$\D ^\dag  _{ \X ^\#, \Q}  \otimes ^\L _{\D  _{ \X ^\flat, \Q} }  \G
\to 
\D ^\dag _{ \X ^\#, \Q}  \otimes _{\D  _{ \X ^\flat, \Q} } \G$
is an isomorphism.
From \cite[4.14]{caro_log-iso-hol}, 
the canonical morphism
$\G 
\to 
\D ^\dag _{ \X ^\flat, \Q}  \otimes _{\D  _{ \X ^\flat, \Q} } \G \liso \D ^\dag _{ \X ^\flat, \Q}  \otimes ^\L _{\D  _{ \X ^\flat, \Q} } \G$
is an isomorphism (the last isomorphism
comes from the flatness of 
$\D  _{ \X ^\flat, \Q}  \to \D ^\dag _{ \X ^\flat, \Q} $).
This yields the last canonical isomorphism:
$u _{1!} (\G )
\underset{\ref{5.3-log-iso-holbis}}{\riso}
\D ^\dag _{ \X ^\#, \Q}  \otimes ^\L _{\D ^\dag _{ \X ^\flat, \Q}} \G
\riso 
\D ^\dag _{ \X ^\#, \Q}  \otimes _{\D ^\dag _{ \X ^\flat, \Q}} \G$.
 With \ref{5.24.(ii)caro_log-iso-hol}, this implies
 \begin{equation}
 \label{u_1+exact}
 u _{1+} (\G ) \riso \D ^\dag _{ \X ^\#, \Q}  \otimes _{\D ^\dag _{ \X ^\flat, \Q}} \G (\ZZ _1).
 \end{equation}
 
iv)  Finally, 
 from \cite[3.5.6.2]{caro-stab-u!R-Gamma}, 
 for any $\G \in \mathrm{Isoc} ^0 (\X ^\flat/K)$ (not only  $\mathrm{Isoc} (\X ^\flat/K)$), 
 we have the isomorphism
\begin{equation}
\label{2.2.3caro-Tsuzuki}
 \rho _{1,\G}\colon u _{1+} (\G)\riso (\hdag Z _1) (\G ),
\end{equation}
which is a generalization of \cite[2.2.9]{caro-Tsuzuki} that we will need.
 
\end{empt}

\begin{empt}
 \label{nota-kercoker-alpha}
We denote
 by $\alpha _{1,\FF}\colon \FF (-\ZZ _1)\hookrightarrow \FF$ the
 canonical monomorphism of $\mathrm{Isoc} (\X ^\flat/K)$. 
 From \ref{u_1+exact}, for any $\G \in  \mathrm{Isoc} (\X ^\flat/K)$, 
$ \mathcal{H} ^{i} u _{1+} (\G)=0$  ($i\neq0$).
 Thus, $\alpha_{1,\FF}$
 induces the homomorphism  of coherent $\D ^\dag _{\X ^{\#}, \Q}$-modules
 $\beta _{1,\FF} := u _{1+} (\alpha _{1,\FF}) \colon 
 u _{1+} (\FF (-\ZZ _1))
 \to 
 u _{1+} (\FF)$.
 In fact, with \ref{5.24.(ii)caro_log-iso-hol} and \ref{dualvsdualvee},
 we check that the morphism
 $\beta _{1,\FF}$ is a morphism of $\D ^\dag _{\X ^{\#}, \Q}$-holonomic modules.
 \end{empt}

\begin{empt}
[Cone of $\beta _{1,\FF}$]
\label{coneBeta}
Since the module $\FF (-\ZZ _1)\in \mathrm{Isoc} (\X ^\flat/K) $ and its exponents and their differences
are non Liouville, then from \cite[3.5.5.1]{caro-stab-u!R-Gamma}, 
$ i _{1} ^{\# !}  (u _{1+} (\FF (-\ZZ _1)))$
is isomorphic to a complex of objects in $\mathrm{Isoc} (\ZZ _{1}©¢^\#/K)$
and in particular belongs to 
$D ^\mathrm{b} _{\mathrm{coh}} ( \D ^\dag _{\ZZ _{1} ^{\#},\Q})$.
In our special context, since the residues of $\FF (-\ZZ _1)$ induces by $t _i \partial _i$ for $i\geq 2$ are nilpotent,
then so are that of the isocrystals of the complex
$ i _{1} ^{\prime \# !}  (u _{1+} (\FF (-\ZZ _1)))$ (see 
\cite[1.1.22]{caro-Tsuzuki} which is the main argument of the check of 
\cite[3.5.5.1]{caro-stab-u!R-Gamma}).
Moreover, with \cite[3.5.5.2]{caro-stab-u!R-Gamma}, we have the exact triangle 
of $D ^\mathrm{b} _{\mathrm{coh}} ( \D ^\dag _{\X ^{\#},\Q})$: 
\begin{equation}
\label{caro-stab-u!R-Gamma-ET-FZ1}
i _{1+} ^{\#}  i _{1} ^{\# !} 
(u _{1+} (\FF (-\ZZ _1)))
\to 
u _{1+} (\FF (-\ZZ _1))
\to 
(\hdag Z _1) \circ u _{1+} (\FF (-\ZZ _1))
\to 
+1
\end{equation}
Since  $\D ^\dag _{\X ^\#} (\hdag Z _1) _{\Q}= \D ^\dag _{\X^\flat} (\hdag Z _1) _{\Q}$ (use \cite[2.2.1]{caro-Tsuzuki}),
we get : 
 \begin{gather}
     \label{u_1hadgZ1}
  (\hdag Z _1) \circ u _{1+} (\FF (-\ZZ _1) ):=
  \D ^\dag _{\X ^\#} (\hdag Z _1) _{\Q} \otimes _{\D ^\dag _{\X ^\#,\Q}} u _{1+} (\FF (-\ZZ _1) )
\underset{\ref{u_1+exact}}{\riso}
    \D ^\dag _{\X ^\flat} (\hdag Z _1) _{\Q} \otimes _{\D ^\dag _{\X ^\flat,\Q}} \FF 
    =: (\hdag Z _1) (\FF ).
 \end{gather}
Since 
$\FF \in \mathrm{Isoc} ^0 (\X ^\flat/K) $,
from  
\ref{2.2.3caro-Tsuzuki},
we have
$ \rho _{1,\FF}\colon u _{1+} (\FF)\riso (\hdag Z _1) (\FF )$.
As in \ref{u_1hadgZ1}, this yields that the canonical morphism 
$u _{1+} (\FF ) \to  (\hdag Z _1) \circ u _{1+} (\FF)$ is an isomorphism.
Hence, from \ref{caro-stab-u!R-Gamma-ET-FZ1}, we get the exact triangle of 
$D ^\mathrm{b} _{\mathrm{coh}} ( \D ^\dag _{\X ^{\#},\Q})$:
\begin{equation}
\label{caro-stab-u!R-Gamma-ET-FZ1beta}
i _{1+} ^{\#}  i _{1} ^{\# !} 
(u _{1+} (\FF (-\ZZ _1)))
\to 
u _{1+} (\FF (-\ZZ _1))
\overset{\beta _{1,\FF}}{\longrightarrow}
u _{1+} (\FF)
\to 
+1
\end{equation}

Since $\mathrm{Isoc} ^0 (\ZZ _{1} ^{\#}/K)$ is an abelian category and 
since  
$i _{1} ^{\# !} 
(u _{1+} (\FF (-\ZZ _1)))$ is isomorphic to a complex of 
$\mathrm{Isoc} ^0 (\ZZ _{1} ^{\#}/K)$,
from \ref{caro-stab-u!R-Gamma-ET-FZ1beta}, 
we get that 
$i_1 ^{\# !} (\ker (\beta _{1,\FF}))$
and 
$  i_1 ^{\# !} (\coker (\beta _{1,\FF}))$ 
are objects of the abelian category $\mathrm{Isoc} ^0 (\ZZ _{1} ^{\#}/K)$.

Since $i_{1+} ^{\# } $ preserves the holonomicity (see \ref{holo-rel-dual}),
we  get from
\ref{caro-stab-u!R-Gamma-ET-FZ1beta}
the long exact sequence of holonomic 
$\D ^\dag _{\X ^\#,\Q}$-modules
\begin{equation}
\label{exsqcaro-stab-u!R-Gamma-ET-FZ1beta}
0 
\to 
\ker (\beta _{1,\FF})
\to 
u _{1+} (\FF (-\ZZ _1))
\overset{\beta _{1,\FF}}{\longrightarrow}
u _{1+} (\FF)
\to 
\coker (\beta _{1,\FF})
\to
0.
\end{equation}
Using the definition of holonomicity, 
we check that the image of $\beta _{1,\FF}$ is also holonomic.
(Remark that this is still not known 
in the logarithmic context that a coherent submodule of a $\D ^\dag _{\X ^\#,\Q}$-holonomic
module is $\D ^\dag _{\X ^\#,\Q}$-holonomic (but this is probably true). So a priori we do need
the holonomicity of $\ker (\beta _{1,\FF})$ and 
$\coker (\beta _{1,\FF})$ to 
check the holonomicity of the image of $\beta _{1,\FF}$.)
Hence, by applying the functor
$\mathcal{H} ^0 \DD _{\X ^\#}$ to the long exact sequence
\ref{exsqcaro-stab-u!R-Gamma-ET-FZ1beta},
we get another long exact sequence which can be translated by the isomorphisms
\begin{equation}
\label{logdualkercoker}
\DD _{\X ^\#} \coker (\beta _{1,\FF})
 \riso 
\ker  (\DD _{\X ^\#} (\beta _{1,\FF})),
\
\DD _{\X ^\#} \ker (\beta _{1,\FF})
 \riso 
\coker  (\DD _{\X ^\#} (\beta _{1,\FF}))
\end{equation}
where we remove  $\mathcal{H} ^0$ in the notation
since modules are holonomic.

\end{empt}

\begin{empt}
\label{desc-g1+}

i) As $g _1 ^\# \colon \X ^{\prime \#} \to \ZZ _1 ^\#$ is log smooth of relative dimension $1$,
using \cite[IV.3.2.3]{Ogus-Logbook}, 
we check that
$\Omega ^1 _{\X ^{\prime \#}/\ZZ _1 ^\#} =
\omega  _{\X ^{\prime \#}/\ZZ _1 ^\#} 
\riso 
\omega  _{\X ^{\prime \#}}\otimes _{g _{1} ^{-1} \O _{\ZZ _1}} g _{1} ^{-1} (\omega  _{\ZZ _1 ^\#} ^{-1})$.
We denote by 
$\delta _{\X ^{\prime \#}}\colon 
\omega  _{\X ^{\prime \#}}\otimes _{\O _{\X '}}  \widehat{\D} ^{(0)} _{\X ^{\prime \#}}
\riso 
\omega  _{\X ^{\prime \#}}\otimes _{\O _{\X '}}  \widehat{\D} ^{(0)} _{\X ^{\prime \#}}$
the canonical involution which exchanges both right $\widehat{\D} ^{(0)} _{\X ^{\prime \#}}$-module structures
(this comes from \cite[1.19]{caro_log-iso-hol} by completion) :
$\omega _{\X ^{\prime \#}} \otimes _{\O _{\X '}} 
\widehat{\D} ^{(0)} _{\X ^{\prime \#}}
\riso
\omega _{\X ^{\prime \#}} \otimes _{\O _{\X '}} 
\widehat{\D} ^{(0)} _{\X ^{\prime \#}}$.
The canonical morphism of left $\widehat{\D} ^{(0)} _{\X ^{\prime \#}}$-module of the form
$\widehat{\D} ^{(0)} _{\X ^{\prime \#}}
\to
g _{1} ^{*} (\widehat{\D} ^{(0)} _{\ZZ _1 ^\#})$
induces the morphism of left
$(\widehat{\D} ^{(0)} _{\X ^{\prime \#}}, g _{1} ^{-1} \widehat{\D} ^{(0)} _{\ZZ _1 ^\#} )$-bimodules
of the form
$\widehat{\D} ^{(0)} _{\X ^{\prime \#}}\otimes _{g _{1} ^{-1} \O _{\ZZ _1}} g _{1} ^{-1} (\omega  _{\ZZ _1 ^\#} ^{-1})
\to
g _{1 \mathrm{g}} ^{*} (\widehat{\D} ^{(0)} _{\ZZ _1 ^\#} \otimes _{\O _{\ZZ _1 }} \omega _{\ZZ _1 ^\#} ^{-1})$,
where the symbol `$ \mathrm{g}$' means that we choose the left structure of 
right $\widehat{\D} ^{(0)} _{\ZZ _1 ^\#} $-module of 
$\widehat{\D} ^{(0)} _{\ZZ _1 ^\#} \otimes _{\O _{\ZZ _1 }} \omega _{\ZZ _1 ^\#} ^{-1}$. 
We denote by 
$\beta _{ \ZZ _1 ^\#} \colon 
\widehat{\D} ^{(0)} _{\ZZ _1 ^\#} \otimes _{\O _{\ZZ _1 }} \omega _{\ZZ _1 ^\#} ^{-1}
\riso 
\widehat{\D} ^{(0)} _{\ZZ _1 ^\#} \otimes _{\O _{\ZZ _1 }} \omega _{\ZZ _1 ^\#} ^{-1}$
the involution isomorphism 
which exchanges both left $\widehat{\D} ^{(0)} _{\X ^{\prime \#}}$-module structures.
Applying the functor 
$\omega _{\X ^{\prime \#}} \otimes _{\O _{\X '}}  -$ to these morphisms of left 
$(\widehat{\D} ^{(0)} _{\X ^{\prime \#}}, g _{1} ^{-1} \widehat{\D} ^{(0)} _{\ZZ _1 ^\#} )$-bimodules 
we get the morphism of $(g _{1} ^{-1} \widehat{\D} ^{(0)} _{\ZZ _1 ^\#} ,\widehat{\D} ^{(0)} _{\X ^{\prime \#}})$-bimodules
of the form:
\begin{equation}
\label{DeRham-g_1-leftarrowpre1}
\omega _{\X ^{\prime \#}} \otimes _{\O _{\X '}}  
( \widehat{\D} ^{(0)} _{\X ^{\prime \#}}\otimes _{g _{1} ^{-1} \O _{\ZZ _1}} g _{1} ^{-1} (\omega  _{\ZZ _1 ^\#} ^{-1}))
\to
\omega _{\X ^{\prime \#}} \otimes _{\O _{\X '}}  
g _{1 \mathrm{g}} ^{*} (\widehat{\D} ^{(0)} _{\ZZ _1 ^\#} \otimes _{\O _{\ZZ _1 }} \omega _{\ZZ _1 ^\#} ^{-1})
\underset{\beta _{ \ZZ _1 ^\#}}{\riso} 
g _{1 \mathrm{d}} ^{*} (\widehat{\D} ^{(0)} _{\ZZ _1 ^\#} \otimes _{\O _{\ZZ _1 }} \omega _{\ZZ _1 ^\#} ^{-1})
 \otimes _{\O _{\X '}}  
 \omega _{\X ^{\prime \#}},
\end{equation}
where the symbol `$ \mathrm{d}$'  means that we choose the right structure of 
right $\widehat{\D} ^{(0)} _{\ZZ _1 ^\#} $-module of 
$\widehat{\D} ^{(0)} _{\ZZ _1 ^\#} \otimes _{\O _{\ZZ _1 }} \omega _{\ZZ _1 ^\#} ^{-1}$. 
Hence, by composing with the canonical involution we get
\begin{equation}
\label{DeRham-g_1-leftarrowpre2}
\Omega ^1 _{\X ^{\prime \#}/\ZZ _1 ^\#} \otimes _{\O _{\X '}} 
\widehat{\D} ^{(0)} _{\X ^{\prime \#}}
 \riso 
g _{1} ^{-1} \omega  _{\ZZ _1 ^\#} ^{-1}  \otimes _{g _{1} ^{-1} \O _{\ZZ _1}} 
(\omega _{\X ^{\prime \#}} \otimes _{\O _{\X '}}  
\widehat{\D} ^{(0)} _{\X ^{\prime \#}})
\underset{\delta _{\X ^{\prime \#}} \otimes id }{\riso}
(\omega _{\X ^{\prime \#}} \otimes _{\O _{\X '}}  
 \widehat{\D} ^{(0)} _{\X ^{\prime \#}})
 \otimes _{g _{1} ^{-1} \O _{\ZZ _1}} g _{1} ^{-1} \omega  _{\ZZ _1 ^\#} ^{-1},
\end{equation}
where the canonical right $\widehat{\D} ^{(0)} _{\X ^{\prime \#}}$-module structure on 
$\Omega ^1 _{\X ^{\prime \#}/\ZZ _1 ^\#} \otimes _{\O _{\X '}} 
\widehat{\D} ^{(0)} _{\X ^{\prime \#}}$ 
is the multiplication at the right on $\widehat{\D} ^{(0)} _{\X ^{\prime \#}}$ in the tensor product. 
As in \cite[5.1.1]{these_montagnon}, which can be extended to the relative case,
(see also \cite[2.4.6.2]{Beintro2} for the non logarithmic relative case), we check that 
the canonical morphism of right $\widehat{\D} ^{(0)} _{\X ^{\prime \#}}$-modules 
\begin{equation}
\label{DeRham-g_1-leftarrow}
\Omega ^1 _{\X ^{\prime \#}/\ZZ _1 ^\#} \otimes _{\O _{\X '}} 
\widehat{\D} ^{(0)} _{\X ^{\prime \#}}
\to 
g _{1 \mathrm{d}} ^{*} (\widehat{\D} ^{(0)} _{\ZZ _1 ^\#} \otimes _{\O _{\ZZ _1 }} \omega _{\ZZ _1 ^\#} ^{-1})
 \otimes _{\O _{\X '}}  
 \omega _{\X ^{\prime \#}} 
 =:
 \widehat{\D} ^{(0)} _{\ZZ _1 ^\# \leftarrow \X ^{\prime \# }}
\end{equation}
which is the composition of \ref{DeRham-g_1-leftarrowpre2} with \ref{DeRham-g_1-leftarrowpre1},
induces the exact sequence of right $\widehat{\D} ^{(0)} _{\X ^{\prime \#}}$-modules
\begin{equation}
\label{DeRham-g_1}
0
\to 
\widehat{\D} ^{(0)} _{\X ^{\prime \#}}
\to 
\Omega ^1 _{\X ^{\prime \#}/\ZZ _1 ^\#} \otimes _{\O _{\X '}} 
\widehat{\D} ^{(0)} _{\X ^{\prime \#}}
\to 
\widehat{\D} ^{(0)} _{\ZZ _1 ^\# \leftarrow \X ^{\prime \# }}
\to 
0,
\end{equation}
where the first part of the exact sequence is the relative to $g _1 ^\#$ de Rham complex of $\widehat{\D} ^{(0)} _{\X ^{\prime \#}}$.
As for \cite[4.3.3.(ii)]{Beintro2} (to deduce the right case from the left case, 
we use the equivalences \cite[1.16]{caro_log-iso-hol}), 
since $g _1 ^\#$ is log smooth, then 
$\widehat{\D} ^{(0)} _{\ZZ _1 ^\# \leftarrow \X ^{\prime \# },\Q}
\otimes _{\widehat{\D} ^{(0)} _{\X ^{\prime \#},\Q}} 
\D ^{\dag} _{\X ^{\prime \#},\Q} 
\riso 
\D ^{\dag} _{\ZZ _1 ^\# \leftarrow \X ^{\prime \# },\Q}:=
\omega _{\X ^{\prime \#}}  \otimes _{\O _{\X '}}
g _{1 \mathrm{d}} ^{*} (\D ^{\dag} _{\ZZ _1 ^\#,\Q}  \otimes _{\O _{\ZZ _1 }} \omega _{\ZZ _1 ^\#} ^{-1})$.
Hence, since $g _{1*}$ is exact, 
since the extension
$\widehat{\D} ^{(0)} _{\X ^{\prime \#}} 
\to \D ^{\dag} _{\X ^{\prime \#},\Q} $ is flat, 
using the exactness of \ref{DeRham-g_1}
we get 
for any coherent left $\D ^{\dag} _{\X ^{\prime \#},\Q} $-module $\G ^\#$
\begin{equation}
\label{g1dashdescr}
g _{1+} ^\# (\G^\#) :=
\R g _{1 *} (\D ^{\dag} _{\ZZ _1 ^\# \leftarrow \X ^{\prime \# },\Q} \otimes ^\L _{\D ^{\dag} _{\X ^{\prime \#},\Q} } \G ^\#)
\riso 
g _{1*} [  \G ^\#\to  \Omega ^1 _{\X ^{\prime \#}/\ZZ _1 ^\#} \otimes _{\O _{\X '}}  \G ^\#] [1],
\end{equation}
where $ [  \G ^\#\to  \Omega ^1 _{\X ^{\prime \#}/\ZZ _1 ^\#} \otimes _{\O _{\X '}}  \G ^\#] [1]$ means that 
the position of $\G ^\#$ is $-1$.

ii) Similarly, for any coherent left $\D ^{\dag} _{\X ^{\prime \flat},\Q} $-module $\G ^\flat$, we check the isomorphism
\begin{equation}
\label{g1flatdescr}
g _{1+} ^\flat (\G ^\flat) :=
\R g _{1 *} (\D ^{\dag} _{\ZZ _1 ^\# \leftarrow \X ^{\prime \flat },\Q} \otimes ^\L _{\D ^{\dag} _{\X ^{\prime \flat},\Q} } \G ^\flat)
\riso 
g _{1*} [  \G ^\flat\to  \Omega ^1 _{\X ^{\prime \flat}/\ZZ _1 ^\#} \otimes _{\O _{\X '}}  \G ^\flat ][1].
\end{equation}
By identifying $ \Omega ^1 _{\X ^{\prime \#}/\ZZ _1 ^\#} $ (which is free with $d t _1$ as a basis)
(resp. $ \Omega ^1 _{\X ^{\prime \flat}/\ZZ _1 ^\#} $, which is free with $d log t _1$ as a basis)
with $\O _{\X '}$, 
we get that $g _{1+} ^\# (\G^\#) [-1]$
is 
the complex 
$g _{1*} \G^\# \overset{\partial _1}{\longrightarrow} g _{1*} \G^\#$
(resp. 
$g _{1+} ^\flat (\G^\flat) [-1]$
is $g _{1*} \G^\flat \overset{t _1 \partial _1}{\longrightarrow} g _{1*} \G^\flat$).

\end{empt}

\begin{rem}
Beware that the sequence 
\ref{DeRham-g_1} is not anymore exact if we replace the level $0$ by any level $m$ (the proof is a computation 
involving Koszul complexes, the computation only works at level $0$). But remark that since the extension 
$\widehat{\D} ^{(0)} _{\X ^{\prime \#},\Q} \to \widehat{\D} ^{(m)} _{\X ^{\prime \#},\Q}$ is flat, 
since
$\widehat{\D} ^{(0)} _{\ZZ _1 ^\# \leftarrow \X ^{\prime \# },\Q}
\otimes _{\widehat{\D} ^{(0)} _{\X ^{\prime \#}},\Q} 
\widehat{\D} ^{(m)} _{\X ^{\prime \#},\Q} 
\riso
\widehat{\D} ^{(m)} _{\ZZ _1 ^\# \leftarrow \X ^{\prime \# },\Q}:=
g _{1 \mathrm{d}} ^{*} (\widehat{\D} ^{(m)} _{\ZZ _1 ^\#} \otimes _{\O _{\ZZ _1 }} \omega _{\ZZ _1 ^\#} ^{-1})
 \otimes _{\O _{\X '}}  
 \omega _{\X ^{\prime \#}} $,
then this is true after applying the functor $-\otimes _\Z \Q$. 
\end{rem}

\begin{lem}
 \label{lemm-square-outideD}
Let $\PP$ be a smooth formal scheme, $\T$ be a strict normal crossing divisor on $\PP$, 
$\PP ^\# := (\PP, \T)$, $H$ be a divisor of $P$.
 Consider the following diagram of coherent $\D ^\dag _{\PP ^\#,
 \Q}$-modules:
 \begin{equation}
 \label{lemm-square-outideD-diag}
  \xymatrix@R=0,3cm{
   {\M _1}\ar[r] ^-{}\ar[d] ^-{}& 
   {\M _1 '}\ar[d] ^-{}\\ 
  {\M _2}\ar[r] ^-{}& 
   {\M _2 '.}
   }
 \end{equation}
 We suppose that the canonical morphism $\M ' _2 \to \M' _2 (\hdag H):=
 \D ^\dag _{\PP ^\#} (\hdag H)  _\Q \otimes _{\D ^\dag _{\PP ^\#,\Q}}
 \M$ is an isomorphism. Then the diagram is commutative if and only if
 it is so after restricting to $\PP\setminus H$.
\end{lem}
\begin{proof}
Suppose the diagram commutative after restricting to $\PP':= \PP\setminus H$.
From the diagram \ref{lemm-square-outideD-diag}, we get two morphisms
$f,g\colon \M _1 \to \M '_2$ and we have to check that $f-g=0$. 
Let $\NN'$ be the image of $f-g$.
Since by hypothesis $\NN ' (\hdag H) | \PP' = \NN ' | \PP'=0$, 
since  $\NN ' (\hdag H)$ is $ \D ^\dag _{\PP ^\#} (\hdag H)  _\Q$-coherent,
from  \cite[4.8]{caro_log-iso-hol} we get
$\NN ' (\hdag H)=0$.
Since the composition
$\NN ' \to \NN ' (\hdag H) \to \M '_2 (\hdag H) \liso \M '_2$ is injective,
then so is $\NN ' \to \NN ' (\hdag H)$.
\end{proof}

\begin{prop}
 \label{dual-beta}
 (i) The following diagram of holonomic $\D ^\dag _{\X ^{\#},
 \Q}$-modules
 \begin{equation*}
   \xymatrix @R=20pt @C=2cm {
   {\DD _{\X ^{\#}} u _{1+} \DD _{\X ^{\flat}} (\FF)
   =u _{1!}  (\FF)} 
   \ar[rr] ^-{\iota _{1,\FF}} _-{\sim}
   \ar[d] _-{\DD _{\X ^{\#}} (\beta _{1,\DD _{\X ^{\flat}}\FF} )}
   & &
   {u _{1+} (\FF (-\ZZ _1))} 
   \ar[d] ^-{\beta _{1,\FF} }
   \\ 
  {\DD _{\X ^{\#}} \circ u _{1+}\circ (-\ZZ _1) \circ \DD _{\X ^{\flat}}(\FF)} 
   \ar[r] ^-{\sim} _-{\DD _{\X ^{\#}} (\iota _{1,\DD _{\X^\flat}(\FF)})}
   &
   {   \DD _{\X ^{\#}} \circ u _{1!} \circ \DD _{\X ^{\flat}}(\FF)}
   \ar[r] ^-{\sim}_-{\star}
   &
   { u _{1+} (\FF)} 
   }
 \end{equation*}
 where $\beta _{1,\FF}$ 
 (resp.  $\iota _{1,\FF}$) 
 is defined in \ref{nota-kercoker-alpha} 
 (resp. \ref{5.24.(ii)caro_log-iso-hol}) and similarly if we replace $\FF$ by $\DD _{\X^\flat}(\FF)$, 
 where $\star$ is induced by functoriality using the biduality
 isomorphisms $\DD _{\X ^{\flat}} \circ \DD _{\X ^{\flat}}\riso\mr{id}$
 and $\DD _{\X ^{\#}} \circ \DD _{\X ^{\#}}\riso \mr{id}$,
 is commutative.

 (ii) We have the isomorphisms of coherent $\D ^\dag _{\X ^{\#},
 \Q}$-modules with support in $\ZZ _1$:
  \begin{equation}
   \label{nota-kercoker-epsilon-iso}
   \ker (\beta _{1,\FF})
   \riso
   \DD _{\X ^{\#}} (\coker (\beta _{1, \DD _{\X ^{\flat}} (\FF)})),\qquad
   \coker (\beta _{1,\FF})
   \riso
   \DD _{\X ^{\#}} (\ker (\beta _{1, \DD _{\X ^{\flat} }(\FF)})).
  \end{equation}
\end{prop}

\begin{proof}
 Let us show (i). By Lemma \ref{2.2.3caro-Tsuzuki}, we have the
 isomorphism $u _{1+}(\FF) \riso \FF (\hdag Z _1)$. Hence, 
 since the diagram is
 commutative outside of $Z_1$, the commutativity follows by Lemma
 \ref{lemm-square-outideD}. The second part of the proposition is
 straightforward from the first one and \ref{logdualkercoker}.
\end{proof}

\begin{lem}
 \label{kercoker-epsilon}
 With the notation \ref{nota-kercoker-alpha}, 
we have the canonical isomorphisms of 
$\mathrm{Isoc} ^0 (\ZZ _{1} ^{\#}/K)$ 
 \begin{equation}
  \label{ker-epsilon-exseq}
  i_1 ^{\# !} (\ker (\beta _{1,\FF}))\riso\ker (N _{1,\FF}),\qquad
  i_1 ^{\# !} (\coker (\beta _{1,\FF}))\riso\coker (N _{1,\FF}).
 \end{equation}
\end{lem}

\begin{proof}
1) First we reduce to the case where $f=\mathrm{id}$. It is straightforward that 
$\ker (N _{1,\FF})$ and $\coker (N _{1,\FF})$ are objects of 
$\mathrm{Isoc} ^0 (\ZZ _{1} ^{\#}/K)$. Concerning the left terms of
\ref{ker-epsilon-exseq}, this has been checked in \ref{coneBeta}.
Now, let us check the isomorphisms. 
From the relative duality isomorphism satisfied by $i _{1} ^\#$ (see \ref{rel-dual-isom})
and the logarithmic version of Kashiwara Theorem (cf.\
 \cite[5.3.6]{caro-stab-sys-ind-surcoh}), for any coherent
 $\D^\dag_{\X ^\#,\Q}$-module $\G$ with support in $\ZZ_1 $, we get the
 isomorphism $\DD_{\ZZ_1 ^\#} i_1^{ \#!}\G\cong i_1^{ \#!}\DD_{\X ^\#} \G$.
 Thus, by duality, using  Lemma \ref{lemm-KerdualcokerN}, 
 \ref{nota-kercoker-epsilon-iso} and 
 \ref{dualvsdualvee}, it is sufficient to check the second
 isomorphism. 
 Since $f$ is log-\'etale, then 
 $ f ^{\flat !} \colon \mathrm{Isoc}  (\X ^{ \flat}/K)
 \to \mathrm{Isoc} (\X ^{\prime \flat}/K)$ 
 (resp.  $ f ^{\flat !} \colon \mathrm{Isoc} ^0   (\X ^{\flat}/K)
 \to \mathrm{Isoc} ^0 (\X ^{\prime \flat}/K)$) 
 is an exact functor. 
Put $\FF ' := f ^{\flat !} (\FF)$, 
$\alpha '_{1,\FF'} := f ^{\flat !} (\alpha _{1,\FF})\colon 
\FF '(-\ZZ _1)
 \hookrightarrow 
\FF '$ the canonical inclusion, 
$\beta '_{1,\FF'} := u '_{1+} (\alpha '_{1,\FF'}) \colon 
 u ' _{1+} (\FF '(-\ZZ _1))
 \to 
 u '_{1+} (\FF')$.
 From \cite[5.4.6.1]{caro-stab-sys-ind-surcoh},
we get $ u ' _{1+}  \circ  f ^{\flat !} \riso  f ^{\# !} \circ u _{1+}$
for coherent complexes.
Hence, $ f ^{\# !} (\beta _{1,\FF}) \riso \beta '_{1,\FF'}$. 
This yields
\begin{gather}
 i _{1} ^{\prime \# !}  \coker (\beta '_{1,\FF '}))
 \riso  i _{1} ^{\prime \# !} f ^{\# !} \coker (\beta _{1,\FF}) 
 \riso 
 i _{1} ^{ \# !} \coker (\beta _{1,\FF}).
\end{gather}
 By definition, we remark 
 $N _{1,\FF}=N _{1,\FF'}$.
So, we are reduced to the case where
 $f=\mr{id}$.

2) Suppose $f =\mathrm{id}$. Since $i _{1} ^{\#}$ is an exact closed immersion, 
then from the logarithmic version of Berthelot-Kashiwara's theorem (cf.\
 \cite[5.3.6]{caro-stab-sys-ind-surcoh}) 
 the functors $i _{1} ^{ \# !}$ and
$ i _{1+} ^{ \# }$ induce
quasi-inverse exact functors between 
the category of coherent $\D ^\dag _{\X ^{ \#}, \Q}$-modules with support in $\ZZ _1$
 and that of coherent $\D ^\dag _{\ZZ _1 ^{\#}, \Q}$-modules. 
Since $g ^{\#} _{1}\circ i _{1+} ^{ \#} = \mr{id}$, this yields
$g ^{\#} _{1+} (\coker (\beta _{1,\FF}))
 \riso  i _{1} ^{\# !} \coker (\beta _{1,\FF})$, which implies
 $g ^{\#} _{1+} (\coker (\beta _{1,\FF})) \riso  \mathcal{H} ^{0}  g ^{\#} _{1+} (\coker (\beta _{1,\FF}))$.
 From  \ref{g1dashdescr}, the functor 
 $\mathcal{H} ^{0} g ^\#_{1+}$ is right exact and then 
 $\mathcal{H} ^{0} g ^\#_{1+} (\coker (\beta _{1,\FF}))
   \riso
   \coker  (\mathcal{H} ^{0} g ^\#_{1+} (\beta _{1,\FF}))$.
 
 From \ref{u_1+exact}, 
 $\mathcal{H} ^i  u _{1+}(\G)=0$
 for any $i\neq0$ and $\G \in \mathrm{Isoc}(\X ^{\flat}/K)$.
 Since $g ^\# _1 \circ u  _1 = g_1 ^{\flat}$, this implies
 \begin{equation}
   \label{ker-epsilon-exseqf=id1}
  (\mathcal{H} ^{0} g ^\#_{1+} )(\beta _{1,\FF})=
   (\mathcal{H} ^{0} g ^\#_{1+} )\circ u _{1+ } (\alpha _{1,\FF})\riso
   \mathcal{H} ^{0} g _{1+} ^{\flat} (\alpha _{1,\FF}).
 \end{equation}
Since 
the functor 
$\mathcal{H} ^{0} g _{1+} ^{\flat}$ is right exact
(see \ref{g1flatdescr}),
we get the second isomorphism:
$  \coker
 (\mathcal{H} ^{0} g ^\#_{1+} )(\beta _{1,\FF})
 \riso
 \coker
 ( \mathcal{H} ^{0} g _{1+} ^{\flat} (\alpha _{1,\FF}))
 \riso
 \mathcal{H} ^{0} g_{1+} ^{\flat} \coker  (\alpha _{1,\FF})$
 Since $\coker  (\alpha _{1,\FF}) =\FF / \FF (-\ZZ _1)$
 and $\mathcal{H} ^{0} g_{1+} ^{\flat}\riso g _{1*} \coker (t_1 \partial _1 \cdot )$
 (see the description after \ref{g1flatdescr}), 
 we get
 \begin{equation}
    \label{ker-epsilon-exseqf=id2}
 \mathcal{H} ^{0} g_{1+} ^{\flat} \coker  (\alpha _{1,\FF}) \riso\coker(N_{1,\FF}).
 \end{equation}
By composing all the isomorphisms, 
the lemma follows.
\end{proof}

\begin{ntn}
 \label{nota-(-D)}
 Moreover, we get a left $ \D ^\dag  _{ \X ^\#}$-module  
 (and then a left $ \D ^\dag  _{ \X ^\flat}$-module), by posing 
   $\O _{\X} (\mathfrak{D}) : = \mathcal{H} om _{\O _{\X}} (
 \omega _{\X } ,  \omega _{\X ^\#})$.
As in \ref{blabla-u_1!+}, 
 for any integer $n \in \Z$, for any 
$\G ^\# \in  D ^{\mathrm{b}} ( \D ^\dag  _{ \X ^\# ,\Q})$,
$\G ^\flat \in  D ^{\mathrm{b}}  ( \D ^\dag  _{ \X ^\flat ,\Q})$,
 we set
 $\O _{\X} (n \mathfrak{D}):= \O _{\X} (\mathfrak{D}) ^{\otimes n}$ 
( where $\otimes n$ means that we tensorise $n$times as $\O _{\X}$-module), 
$\G ^\# (n\mathfrak{D}):= \G ^\# \otimes _{\O _{\X}} \O _{\X} (n \mathfrak{D})$
and $\G ^\flat (n\mathfrak{D}):= \G ^\flat \otimes _{\O _{\X} } \O _{\X} (n \mathfrak{D})$.
Since $ \O _{\X} (n \mathfrak{D})$ is a free $ \O _{\X} $-module of rank one, 
then the induced functors
$(n \mathfrak{D})\colon 
\Mod  ( \D ^\dag  _{ \X ^\# ,\Q})
\to 
\Mod  ( \D ^\dag  _{ \X ^\# ,\Q})$
and 
$(n \mathfrak{D})\colon 
\Mod  ( \D ^\dag  _{ \X ^\flat ,\Q})
\to 
\Mod  ( \D ^\dag  _{ \X ^\flat ,\Q})$,
given respectively by 
$\G ^\#  \mapsto \G ^\# (n\mathfrak{D})$
and 
$\G ^\flat \mapsto \G ^\flat (n\mathfrak{D})$,
are exact. 
We remark that the functor $(\ZZ _1)\colon 
\Mod  ( \D ^\dag  _{ \X ^\flat ,\Q})
\to 
\Mod  ( \D ^\dag  _{ \X ^\flat ,\Q})$ defined in 
\ref{blabla-u_1!+}
commutes with 
$(\mathfrak{D})$
and more precisely we have the isomorphism
$(\mathfrak{D}) \circ (\ZZ _1)
\riso
(\ZZ _1) \circ (\mathfrak{D}) 
\riso 
(\ZZ) $
of functors 
$\Mod  ( \D ^\dag  _{ \X ^\flat ,\Q})
\to 
\Mod  ( \D ^\dag  _{ \X ^\flat ,\Q})$,
where 
$\O _{\X} (\ZZ) : = \mathcal{H} om _{\O _{\X}} (
 \omega _{\X } ,  \omega _{\X ^\flat})$
 and $(\ZZ)  := \O _{\X} (\ZZ)  \otimes _{\O _{\X} }-$.

Similarly, we have the functor
$(n\mathfrak{D} _1) \colon 
 \Mod  ( \D ^\dag  _{ \ZZ _1 ^\# ,\Q})
\to 
\Mod  ( \D ^\dag  _{  \ZZ _1 ^\# ,\Q})$.
\end{ntn}

\begin{lem}
\label{com-i1!(-D)}
With the notation \ref{nota-(-D)}, 
we have the isomorphism 
$(-\mathfrak{D} _1) \circ i_1^{\#!}
 \riso 
 i_1^{\# !}\circ (-\mathfrak{D} )$
 of functors 
 $D ^{\mathrm{b}} _{\mathrm{coh}} ( \D ^\dag  _{ \X ^\# ,\Q}) 
 \to 
 D ^{\mathrm{b}} ( \D ^\dag  _{ \ZZ _1 ^\# ,\Q}) $,
 where 
 $ i_1^{\# !} \colon 
 D ^{\mathrm{b}} _{\mathrm{coh}} ( \D ^\dag  _{ \X ^\# ,\Q}) 
 \to 
 D ^{\mathrm{b}} ( \D ^\dag  _{ \ZZ _1 ^\# ,\Q}) $
 is the extraordinary inverse image
 (see the definition \cite[5.1.2.1]{caro-stab-sys-ind-surcoh}).
\end{lem}

\begin{proof}
Using $\D ^\dag  _{ \X ^\# ,\Q}$-flat resolutions, we reduce to check 
the isomorphism 
$(-\mathfrak{D} _1) \circ i_1^{\#*}
 \riso 
 i_1^{\# *}\circ (-\mathfrak{D} )$
 of functors 
$\coh  ( \D ^\dag  _{ \X ^\# ,\Q}) \to \Mod ( \D ^\dag  _{ \ZZ _1 ^\# ,\Q}) $,
where 
$ i_1^{\# *} \colon \coh  ( \D ^\dag  _{ \X ^\# ,\Q}) \to \Mod ( \D ^\dag  _{ \ZZ _1 ^\# ,\Q}) $
is the canonical factorisation induced by the functor $ i_1^{\# *} $.

\end{proof}

\begin{ntn}
 \label{nota-epsilon-theta1}
 With the notation \ref{nota-kercoker-alpha} and \ref{nota-(-D)},
let
 $\epsilon _{1,\FF}:=u  _+\circ (-\mathfrak{D})(\alpha_{1,\FF})\colon u
 _{+} (\FF  (-\ZZ )) \to u _{+}  (\FF (-\mathfrak{D}))$ and $\theta
 _{1,\E}\colon \DD \circ \DD _{Z} (\E)\to (\hdag Z _1) \circ \DD \circ
 \DD _{Z} (\E) 
 \underset{\ref{dual-ext}}{\riso}
\DD _{Z _1}\circ \DD _{Z} (\E)$ be the canonical
 homomorphisms. Both homomorphisms are isomorphisms outside $Z _1$. 
\end{ntn}

\begin{lem}
  \label{j!Eu!Fbis-lem}
We keep the notation of \ref{nota-epsilon-theta1}. 

\begin{enumerate}
\item We have the isomorphism 
 \begin{equation}
   \label{u1+comm-D}
 u _{1+}\circ (-\mathfrak{D}) \riso (-\mathfrak{D}) \circ  u _{1+}
 \end{equation}
of  functors 
 $\mathrm{Isoc} (\X ^\flat/K)\to \coh (\D ^\dag _{\X ^\# ,\Q})$.

\item We have the isomorphism 
 \begin{equation}
  \label{j!Eu!Fbis}
   \DD _{Z_1} 
   \circ \DD _{Z} \circ u _+ (\FF)
   \riso
   u _{+}  (\FF (-\mathfrak{D})),
 \end{equation}
 which coincides with \ref{j!Eu!F} outside $Z _1$.
\end{enumerate}

\end{lem}

\begin{proof}
1) For any $r\geq i >1$, for any section $P $ of $\D ^\dag _{ \X ^\#, \Q} $,
 there exists a section $Q$ of $\D ^\dag _{ \X ^\#, \Q} $ such that $t _i P= Q t _i$ (notice this is not any more true for $t_1$).
 Using the description \ref{u_1+exact} of the functor
 $u _{1+}$, we get the first assertion.

2) By applying $(\hdag Z _1)$ to \ref{j!Eu!F}, we obtain the last isomorphism
 \begin{equation}
\label{j!Eu!Fbis-iso0}
   \DD _{Z_1} 
   \circ \DD _{Z} \circ u _+ (\FF)
\underset{\ref{dual-ext}}{\riso}
   (\hdag Z _1) \circ \DD 
   \circ \DD _{Z} \circ u _+ (\FF)
\underset{\ref{j!Eu!F}}{\riso}
   (\hdag Z _1) 
   (u _+( \FF (-\ZZ))).
 \end{equation}
We have
$(\hdag Z _1)\circ (-\ZZ) (\FF) 
\riso 
(\hdag Z _1)\circ (-\ZZ _1) \circ (-\mathfrak{D}) (\FF) 
  \riso
  (\hdag Z _1)\circ (-\mathfrak{D}) (\FF) $.
Since the functor $(-\mathfrak{D})$ commutes canonically with 
$u_{1+}$ (see \ref{u1+comm-D})
and with $(\hdag Z _1)$, by applying $(-\mathfrak{D})$ to the isomorphism 
 $\rho _{1,\FF}$ (cf.\ Lemma \ref{2.2.3caro-Tsuzuki}), 
 we get 
 $  (\hdag Z _1)\circ (-\mathfrak{D}) (\FF) 
 \liso 
 u _{1+}\circ (-\mathfrak{D})  (\FF)$. 
By composition, we get 
\begin{equation}
\label{j!Eu!Fbis-iso1}
(\hdag Z _1)\circ (-\ZZ) (\FF) 
\riso 
 u _{1+}\circ (-\mathfrak{D})  (\FF).
\end{equation}
Applying $v_+$ to \ref{j!Eu!Fbis-iso1}, we
 obtain the second isomorphism:
 \begin{equation}
  \label{j!Eu!Fbis-pre1}
  (\hdag Z _1) \circ u _+ \circ (-\ZZ)  (\FF)
   \riso 
   v _+ \circ (\hdag Z _1)\circ (-\ZZ)
   (\FF) 
\underset{\ref{j!Eu!Fbis-iso1}}{\riso} 
   v _+ \circ u _{1+}\circ
   (-\mathfrak{D})  (\FF)
   \riso 
   u _{+}\circ (-\mathfrak{D})  (\FF) ,
 \end{equation}
 where the first isomorphism is induced by $ (\hdag Z _1) \circ u _+ \riso v _+
 \circ (\hdag Z _1)$ from \cite[2.2.2]{caro-Tsuzuki} and the last one is a consequence of the transitivity 
 of the direct images. This homomorphism  \ref{j!Eu!Fbis-pre1}
 is the identity outside $Z _1$.
By composition of \ref{j!Eu!Fbis-iso0} with \ref{j!Eu!Fbis-pre1}, 
we get the desired isomorphism, 
 which coincides with \ref{j!Eu!F} outside $Z _1$.

\end{proof}

\begin{prop}
 (i) With the notation \ref{nota-epsilon-theta1}, 
 the following diagram is commutative:
  \begin{equation}
   \label{cube-E-F-bis}
    \xymatrix@R=10pt{
    {j _! (\E)=\DD \circ \DD _{Z} (\E)} 
    \ar[r] ^-{\theta _{1,\E}}
    \ar[d] ^-{\sim} _-{\ref{j!Eu!F}}& 
    {\DD _{Z _1}\circ \DD _{Z} (\E)} 
    \ar[d] _-{\sim} ^-{\ref{j!Eu!Fbis}}\\ 
   {u _{+} (\FF  (-\ZZ )) } 
    \ar[r] ^-{\epsilon _{1,\FF}}& 
    {u _{+}  (\FF (-\mathfrak{D})).}
    }
  \end{equation}

 (ii) We have the canonical isomorphisms:
 \begin{equation}
  \label{esp1=theta1}
   \mathcal{H} ^{\dag 0} _{Z _1}\bigl(j _! (\E)\bigr)\riso \ker (\theta
   _{1,\E})
   \riso \ker  ( \epsilon _{1,\FF}),
   \qquad
   \mathcal{H} ^{\dag 1} _{Z _1}\bigl(j _! (\E)\bigr)
   \riso  \coker (\theta _{1,\E} ) \riso \coker  ( \epsilon _{1,\FF}).
 \end{equation}
\end{prop}
\begin{proof}
 Thanks to Lemma \ref{lemm-square-outideD}, it is sufficient to check
 the commutativity of the square \ref{cube-E-F-bis} outside $Z _1$.
 Since the horizontal morphisms of the square \ref{cube-E-F-bis} are the identity outside $Z _1$,
 this commutativity is a consequence of the second statement of \ref{j!Eu!Fbis-lem}. 
 Let us check the part (ii) of the proposition. 
 Since $(\hdag Z _1) \circ j _{!} (\E) \riso \DD _{Z _1}\circ \DD _{Z} (\E)$,
the exact triangle of localisation of $j _! (\E)$ with respect to the divisor $Z _1$ can be written in the form of 
 \begin{equation}
 \label{loctrianj!E}
 \R \underline{\Gamma} ^\dag _{Z _1} j _! (\E)
 \to 
j _! (\E)
\underset{\theta _{1,\E}}{\longrightarrow} 
\DD _{Z _1}\circ \DD _{Z} (\E)
\to 
 \R \underline{\Gamma} ^\dag _{Z _1} j _! (\E) [1].
 \end{equation}
From \ref{loctrianj!E}, we obtain both first isomorphisms.
We get both second isomorphisms 
from the part (i) of the proposition.
\end{proof}

\begin{prop}
 \label{w1+kerprop}
With the notation \ref{nota-epsilon-theta1}, we have the canonical isomorphisms:
 \begin{equation}
  \label{w1+kercoker}
  i_1^!\ker  ( \epsilon _{1,\FF})
  \riso
  w _{1+}\bigl( \ker (N _{1,\FF})  (-\mathfrak{D} _1) \bigr),
  \qquad
  i_1^!  \coker( \epsilon _{1,\FF})\riso
  w _{1+}\bigl( \coker (N _{1,\FF})  (-\mathfrak{D} _1) \bigr).
 \end{equation}
\end{prop}

\begin{proof}
Using \ref{u1+comm-D}, with the notation \ref{nota-kercoker-alpha},
we get $u _{1+}\circ (-\mathfrak{D}) (\alpha _{1,\FF}) \riso (-\mathfrak{D}) (\beta _{1,\FF}) $.
 Thus, since the functor 
 $(-\mathfrak{D}) \colon 
 \coh (\D ^\dag _{\X ^\# ,\Q}) \to \coh (\D ^\dag _{\X ^\# ,\Q})$ is exact,
 we get the exact sequence
 \begin{equation}
  \label{es1-w1+kercoker}
   \tag{($\star$)}
   0\to\ker (\beta   _{1,\FF} )(-\mathfrak{D})\to u _{1+} ( \FF (-\ZZ))
   \xrightarrow{u _{1+}\circ (-\mathfrak{D}) (\alpha _{1,\FF})}
   u _{1+} ( \FF (-\mathfrak{D}))\to\coker (\beta   _{1,\FF} )(-\mathfrak{D})
   \to 0.
 \end{equation}
Via the logarithmic version of
 Berthelot-Kashiwara's theorem \cite[5.3.6]{caro-stab-sys-ind-surcoh}),
 we get the first isomorphism 
\begin{equation}
 \label{iso1-w1+kercoker}
  \tag{($\star\star$)}
  \ker (\beta _{1,\FF})(-\mathfrak{D} )
  \liso
  i^{\#}_{1+} \circ  i^{\#!}_1  (\ker (\beta _{1,\FF})(-\mathfrak{D} ))
\underset{\ref{com-i1!(-D)}}{\riso} 
  i^{\#}_{1+} \bigl( (-\mathfrak{D} _1) \circ i^{\#!}_1  (\ker (\beta
  _{1,\FF}))\bigr).
\end{equation}
From Lemma \ref{kercoker-epsilon} and with the stability of 
$\mathrm{Isoc}  (\ZZ _{1} ^{\#}/K)$ under the $(-\mathfrak{D} _1)$, 
we get $(-\mathfrak{D} _1) \circ i^{\#!}_1
 (\ker (\beta _{1,\FF})) \in \mathrm{Isoc}  (\ZZ _{1} ^{\#}/K)$. 
 From \cite[5.24]{caro_log-iso-hol}, 
the functor
 $\mathcal{H} ^0 w _{1+}$ is acyclic over $\mathrm{Isoc}  (\ZZ _{1} ^{\#}/K)$
 (i.e.\ $\H^iw_{1+}(\G) = 0$ for any $\G \in \mathrm{Isoc}  (\ZZ _{1} ^{\#}/K)$).
 Since the functor $i^{\#}_{1+}$ (resp. $i _{1+}$) is exact over $\mathrm{Isoc}  (\ZZ _{1} ^{\#}/K)$
 (resp. $\coh ( \D ^\dag _{\ZZ _1  ,\Q})$), since we have by transitivity of direct images 
 $v _+ \circ i^{\#}_{1+} \riso i _{1+} \circ w _{1+}$, 
with isomorphism ($\star\star$) we obtain  that
 $\ker (\beta _{1,\FF})(-\mathfrak{D} )$ is acyclic for 
 $\mathcal{H} ^0 v _{+}$.
 In the same way, $\coker (\beta _{1,\FF})(-\mathfrak{D} )$ is acyclic
 for $\mathcal{H} ^0 v _{+}$.
 Since  $\FF (-\ZZ)$ and $ \FF (-\D)$ belong to 
 $\mathrm{Isoc}  (\X ^{\flat}/K)$,
 then 
 from \ref{u_1+exact} (resp.  \cite[5.24]{caro_log-iso-hol}),
 they are
$\mathcal{H} ^0 u  _{1+}$ acyclic 
(resp. $\mathcal{H} ^0 u  _{+}$ acyclic).
Since $v _+ \circ u _{1+} \riso u _+$,   
 then this implies that
 $u _{1+} ( \FF (-\ZZ))$ and $u _{1+} ( \FF (-\D))$ are
 $\mathcal{H} ^0 v _+$ acyclic. 
 Thus, by applying the functor $\mathcal{H} ^0 v _{+}$ to
 \ref{es1-w1+kercoker}, we obtain the exact sequence:
 \begin{equation}
 \label{w1+kerpropEnd1}
   0\to
   v _+ ( \ker (\beta   _{1,\FF} )(-\mathfrak{D}))
   \to
   u _{+} ( \FF (-\ZZ))
   \xrightarrow{\epsilon _{1,\FF}}u _{+} ( \FF (-\mathfrak{D}))
   \to v _+ ( \coker (\beta   _{1,\FF} )(-\mathfrak{D}))\to 0.
 \end{equation}
 Using 
  $v _+ \circ i^{\#}_{1+} \riso i _{1+} \circ w _{1+}$,
  we get the first isomorphism:
 \begin{equation}
  \label{w1+kerpropEnd2}
  v _+\bigl( \ker (\beta   _{1,\FF} )(-\mathfrak{D})\bigr)
   \xrightarrow[v _+ \text{\ref{iso1-w1+kercoker}}]{\sim}
   i _{1+}\circ   w _{1+} \circ (-\mathfrak{D} _1) \circ   i^{\#!}_1
   \bigl(\ker (\beta   _{1,\FF})\bigr)
   \underset{\ref{ker-epsilon-exseq}}{\riso}
   i _{1+}\circ   w _{1+} \circ (-\mathfrak{D} _1)\bigl(\ker (N
   _{1,\FF})\bigr).
 \end{equation}
 Considering \ref{w1+kerpropEnd1} and \ref{w1+kerpropEnd2},
 we get  $ \ker  ( \epsilon _{1,\FF})
  \riso
  i_{1+} w _{1+}\bigl( \ker (N _{1,\FF})  (-\mathfrak{D} _1) \bigr)$. 
  Hence, using Berthelot-Kashiwara theorem, we get the first isomorphism. 
By replacing ``$\ker$'' by ``$\coker$'' in the proof above, we get the second isomorphism, and we
 conclude the proof.
\end{proof}

\begin{rem}
\label{rem3.4.10-3.4.17}
With the notation \ref{nota-kercoker-alpha} and \ref{nota-epsilon-theta1}, we remark that
the homomorphisms $\epsilon _{1,\FF}$ and $\beta _{1,\FF}$ are equal
 outside of $D$, and the isomorphisms in Lemma
 \ref{kercoker-epsilon} and Lemma \ref{w1+kerprop} are the same outside
 of $D _1$.
 In particular, when $f = \mathrm{id}$ and $\D $ is empty (i.e. $r=1$),
 then the isomorphism of Lemma \ref{w1+kerprop}
 is simply the one induced by composition from 
\ref{ker-epsilon-exseqf=id1}
 and 
\ref{ker-epsilon-exseqf=id2}.
  
\end{rem}

\begin{thm}
 \label{mixedisoals}
 We put $\Y _1 := \ZZ _1 \setminus \mathfrak{D} _1$, and let
 $j_1:=(\star,\mr{id},\mr{id})\colon (Y_1,Z_1,\ZZ_1) \to(Z_1, Z
 _1,\ZZ_1)$ be the canonical morphism of frames.
We have the isomorphisms of 
coherent $\D ^\dag _{\ZZ _1  ,\Q}$-modules of the form:
 \begin{equation}
  \label{u+mixedstable-iso1}
   i^!_1\mathcal{H} ^{\dag 1} _{Z _1} (j _! (\E))
   \riso j_{1!} \circ   (\hdag D _1)
   \bigl(\coker N_{1,\FF}\bigr),
   \qquad
   i^!_1\mathcal{H} ^{\dag 0} _{Z _1} (j_! (\E))\riso 
   j_{1!} \circ (\hdag D _1)  \bigl(\ker N_{1,\FF} \bigr).
 \end{equation}
 \end{thm}

\begin{proof}
The composition of  $i^!_1 \ref{esp1=theta1}$ with  \ref{w1+kercoker} 
is of the form
\begin{equation}
\label{WhenDempty}
 i^!_1\mathcal{H} ^{\dag 1} _{Z _1} (j _! (\E))
   \riso w _{1+}\bigl( \coker N_{1,\FF} (-\mathfrak{D} _1) \bigr).
\end{equation}
 Since $\coker N_{1,\FF} \in \mathrm{Isoc} ^0 (\ZZ _1 ^\#/K)$, 
as  \ref{j!Eu!F} and \ref{2.2.3caro-Tsuzuki}, we get
\begin{equation*}
 w _{1+}\bigl( \coker N_{1,\FF} (-\mathfrak{D} _1) \bigr)
 \underset{\ref{j!Eu!F}}{\riso}
  j _{1!} \circ w _{1+} \left (\coker N_{1,\FF} \right )
\underset{\ref{2.2.3caro-Tsuzuki}}{\riso}
  j _{1!} \circ   (\hdag D _1)  \left (\coker N_{1,\FF}
  \right ).
\end{equation*}

By composition, we obtain the left isomorphism of \ref{u+mixedstable-iso1}.
We get similarly the right isomorphism of \ref{u+mixedstable-iso1} replacing everywhere cokernels by kernels. 
\end{proof}

\subsection{An isomorphism compatible with Frobenius}
\label{subsectionFrob3.5}
The goal of this subsection is  Theorem \ref{iso3314-Frob}
which show the compatibility with Frobenius
of the main isomorphisms \ref{u+mixedstable-iso1} of the previous subsection.

\begin{empt}
 In this subsection (except for Theorem \ref{iso3314-Frob}), 
 we keep the following notation and hypotheses :
 Let $\phi \colon \X \to \widehat{\mathbb{A}} ^d _\V$ be an \'etale morphism 
of affine formal $\V$-schemes 
and let $t _1,\dots , t _d$ be the corresponding local coordinates
and $\partial  _1 , \dots, \partial  _d$ be the corresponding derivations.
We set 
$A := \Gamma (\X, \O _\X)$,
 $B := A / t _1 A$ and $\ZZ _1 : = \Spf  B$. 
 We suppose that there exists a smooth morphism
 $g _1\colon \X \to \ZZ _1$ such that 
 $g _1 \circ i _1=\id$ and 
 the $\ZZ _1$-morphism of the form
 $h\colon \X \to \widehat{\A} ^{1} _{\ZZ _1}$
 induced by $t _1$ is \'etale.
Let $\overline{t} _2, \dots, \overline{t} _d$ be the global sections of $\O _{\ZZ _1}$
induced by $t _2, \dots, t _d$ and let 
$\phi _1\colon \ZZ _1 \to \widehat{\mb{A}}^{d-1} _\V$ be the \'etale morphism induced by the local 
coordinates $\overline{t} _2, \dots, \overline{t} _d$.
We keep the notation the canonical absolute Frobenius 
$F _{\widehat{\A} ^{d} _{\V}}$, 
$F _{\widehat{\A} ^{d-1} _{\V}}$,
$F _\X $, $F _{\ZZ _1}$ of the paragraph
\ref{canonicalFrob}.
 We set  
 $F_{\X/\ZZ _1}= (F _\X , g _1) \colon \X \to \X ' := \X \times _{\ZZ _1,F _{\ZZ _1}} \ZZ _1$.
 By abuse of notation, we also simply denote by $F _{\ZZ _1}$ 
 the morphisms 
 $\id \times F _{\ZZ _1} \colon 
 \X' \to \X$ and 
 $\id \times F _{\ZZ _1} \colon \widehat{\A} ^{1} _{\ZZ _1} \to \widehat{\A} ^{1} _{\ZZ _1}$ induced by $F_{\ZZ_1}$. 
 Similarly, we denote 
 by 
 $\phi _1$ the morphism $\id \times \phi _1 \colon 
\widehat{\A} ^{1} _{\ZZ _1}  \to  \widehat{\A} ^{d} _\V$.
 Hence, to sum-up, we get the commutative diagrams:
 \begin{equation}
 \label{g-icomm}
  \xymatrix @ R= 0,4cm{
  {\ZZ _1}\ar@{=}[d]\ar@{^{(}->}[r] ^-{i _1}& 
  {\X}\ar[d] ^-{F _{\X/\ZZ _1}}\ar@{->>}[r] ^-{g _1}&
  {\ZZ _1}\ar@{=}[d]
  \\ 
 {\ZZ _1}\ar[d] _-{F _{\ZZ _1}}\ar@{^{(}->}[r] ^-{i _1}
  \ar@{}[rd]|\square&
  {\X'}\ar[d] ^-{F _{\ZZ _1}}\ar@{->>}[r] ^-{g _1}\ar@{}[rd]|\square&
  {\ZZ _1}\ar[d] ^-{F _{\ZZ _1}}
  \\
 {\ZZ _1}\ar@{^{(}->}[r] ^-{i _1}&
  {\X}\ar@{->>}[r] ^-{g _1}&
  {\ZZ _1,}
  }
  \qquad
  \xymatrix @ R= 0,4cm{
  {\X'}
  \ar@{}[rd]|\square
  \ar[d] _-{F _{\ZZ _1}}
  \ar[r] ^-{h'}
  &
  {\widehat{\A} ^{1} _{\ZZ _1} }
  \ar[d] ^-{F _{\ZZ _1}}
  \ar[r] ^-{\phi _1}
  &
  {\widehat{\A} ^{d} _\V }
  \ar[d] ^-{F _{\widehat{\A} ^{d-1} _\V}}
  \\
 {\X}
 \ar[r] ^-{h}
 \ar[rd] ^-{g _1}
 &
  {\widehat{\A} ^{1} _{\ZZ _1} }
  \ar[r] ^-{\phi _1}
  \ar[d] ^-{\pi _1}
  \ar@{}[rd]|\square
  &
  {\widehat{\A} ^{d} _\V }
    \ar[d] ^-{\pi _1}
  \\
  &
  {\ZZ _1}
  \ar[r] ^-{\phi _1}
  &
  {\widehat{\A} ^{d-1} _\V,}
  }
 \end{equation}
 where the symbol $\square$ means the cartesianity of the corresponding square,
 where $\pi _1\colon \widehat{\A} ^{d} _\V \to \widehat{\A} ^{d-1} _\V$ is the projection on the last 
 $d-1$ coordinates and
 $h' \colon \X '\to \widehat{\A} ^{1} _{\ZZ _1}$ is the canonical projection 
 (via the isomorphism $\X' \riso \X \times _{h, \widehat{\A} ^{1} _{\ZZ _1}, F _{\ZZ _1}} \widehat{\A} ^{1} _{\ZZ _1}$). 
 We denote by $t '_1, \dots , t '_d$ the
 local coordinate of $\X'$ corresponding to the \'etale map $\phi _1 \circ h'$
 and by $\partial ' _1 , \dots, \partial ' _d$ the corresponding derivations.  
 Hence, we compute that the image of $t _1$ (resp. $t _2,\dots, t _d$)
 under $F _{\ZZ _1} ^* \colon \Gamma (\X, \O _\X) \to \Gamma (\X', \O _{\X'})$ 
 is $t ' _1$ (resp. $t ^{\prime q} _2,\dots, t ^{\prime q} _d$).
 On the other hand, 
 the image of $t '_1$ (resp. $t '_2,\dots, t '_d$)
 under $F _{\X /\ZZ _1} ^* \colon  \Gamma (\X', \O _{\X'}) \to \Gamma (\X, \O _\X)$ 
 is $t ^q _1$ (resp. $t _2,\dots, t _d$).
Finally, we remark that, via $h$ (resp. $h'$), $t _1$ (resp.\ $t' _1$)
is a relative local coordinate of $\X/\ZZ _1$ (resp.\ $\X'/\ZZ _1$). 
In particular, 
$\Omega^1 _{\X /\ZZ _1}$ is $\O _{\X}$-free with $d t _1$ as a basis,
$\Omega^1 _{\X '/\ZZ _1}$ is $\O _{\X'}$-free with $d t '_1$ as a basis,
$\omega _{\X}$ is $\O _{\X}$-free with $d t _1 \wedge \dots \wedge d t _d$ as a basis,
$\omega _{\X ^\flat}$ is $\O _{\X}$-free with $d log t _1 \wedge d t _2 \wedge \dots \wedge d t _d$ as a basis.
Hence, since the local coordinates are fixed, 
to make some computations, 
we will sometimes identify 
$\Omega^1 _{\X /\ZZ _1}$,
(resp. $\omega _{\X}$,
resp. $\omega _{\X ^\flat}$)
with $\O _{\X}$
and 
$\Omega^1 _{\X '/\ZZ _1}$
with 
$\O _{\X'}$.
\end{empt}

\begin{empt}
As in \ref{dfnFandE}, 
let $\X ^{\flat}:= (\X, \ZZ _1)$, $u \colon \X ^\flat \to \X$, 
$(\FF ,\phi _\FF)\in F\text{-}\mathrm{Isoc} (\X ^{\flat}/K)$, 
$(\E , \phi _\E) := u _+ (\FF ,\phi _\FF)$.
 We put $\E' : = F _{\ZZ _1} ^* (\E)$, $\FF' : = F _{\ZZ _1} ^* (\FF)$. 
 We get the commutative diagram:
 \begin{equation}
  \label{pre-Frob-direct-image}
   \xymatrix @R=0,3cm {
   {F _{\X /\ZZ_1} ^{*} \E'}\ar[r] ^-{\sim}&
   {F _\X ^* (\E) }\ar[r] ^-{\sim} _-{\phi _{\E}}&
   {\E}\\ 
  {F _{\X /\ZZ_1} ^{*} \FF'}\ar[r] ^-{\sim}\ar@{^{(}->}[u] ^-{}&
   {F _\X ^* (\FF) }\ar[r] ^-{\sim} _-{\phi _{\FF}}\ar@{^{(}->}[u] ^-{}&
   {\FF.}\ar@{^{(}->}[u] ^-{}
   }
 \end{equation}
\end{empt}

\begin{empt}
[Relative de Rham complex and Frobenius]
Since de Rham complexes commute with base changes, 
for any $\G \in D ^\mathrm{b} _{\mathrm{coh}} ( \D ^\dag _{\X,\Q})$, 
we have 
$F _{\ZZ _1} ^* g _{1*} \bigl(\Omega _{\X /\ZZ _1} ^\bullet
   \otimes _{\O _{\X}} \G  \bigr)
   \riso
g _{1*}   F _{\ZZ _1} ^* \bigl(\Omega _{\X /\ZZ _1} ^\bullet
   \otimes _{\O _{\X}} \G  \bigr)
   \riso  
   g _{1*} \bigl( \Omega _{\X ' /\ZZ _1} ^\bullet
   \otimes _{\O _{\X'}} F _{\ZZ _1} ^*  \G   \bigr)$ (recall that $\X$ is affine and then $g _{1*}$ is exact).
   Moreover, as in \cite[4.3.5]{Be2}, we check the canonical morphism
     induced by functoriality under $F _{\X /\ZZ _1}\colon \X \to \X'$ between de Rham complexes
 \begin{equation}
 \label{4.3.5-Be2}
 \Omega _{\X ' /\ZZ _1} ^\bullet
   \otimes _{\O _{\X'}}  \G ' 
   \to 
\Omega _{\X /\ZZ _1} ^\bullet
   \otimes _{\O _{\X}}  F _{\X /\ZZ _1} ^* \G ' 
 \end{equation}
is a quasi-isomorphism for any $\G' \in D ^\mathrm{b} _{\mathrm{coh}} ( \D ^\dag _{\X',\Q})$.
Hence,
we get the isomorphism
\begin{equation}
\label{FrobReldeRham}
F _{\ZZ _1} ^* g _{1*} \bigl(\Omega _{\X /\ZZ _1} ^\bullet
   \otimes _{\O _{\X}} \E  \bigr)
   \riso 
   g _{1*} \bigl( \Omega _{\X ' /\ZZ _1} ^\bullet
   \otimes _{\O _{\X'}} \E '   \bigr)
   \riso 
   g _{1*} \bigl(
   \Omega _{\X /\ZZ _1} ^\bullet
   \otimes _{\O _{\X}}  F _{\X} ^*  \E \bigr)
   \underset{\phi _\E}{\riso}
   g _{1*} \bigl(
   \Omega _{\X /\ZZ _1} ^\bullet
   \otimes _{\O _{\X}} \E \bigr)
\end{equation}
  whose composition is the Frobenius structure of 
$g _{1*} \bigl(\Omega _{\X /\ZZ _1} ^\bullet
   \otimes _{\O _{\X}} \E  \bigr)$
   induced by that of $\E$.
\end{empt}

\begin{empt}
From \cite[3.15]{Abe-Frob-Poincare-dual}, we have the canonical isomorphism
 compatible with Frobenius:
 \begin{equation}
  \label{g+Omega}
   g  _{1 +} (\E) \riso g _{1*} \bigl(\Omega _{\X /\ZZ _1} ^\bullet
   \otimes _{\O _{\X}} \E  \bigr) [1] (1),
 \end{equation}
 where Frobenius structure of the right term is detailed in \ref{FrobReldeRham}.
 By identifying $\Omega^1 _{\X /\ZZ _1}$ (via the basis $d t _1$ as $\O _{\X}$-module) with $\O _{\X}$ and applying the
 functor $\mathcal{H} ^{0} $ to the right term of 
\ref{g+Omega}, we get a Frobenius structure on 
$\coker _{\partial _{1} } (\E)$ (since $\X $ is affine, by abuse of notation we sometimes forget to write $g _{1*}$)
from that of $g _{1*} \bigl(\Omega _{\X /\ZZ _1} ^\bullet
   \otimes _{\O _{\X}} \E  \bigr) [1]$. 
 Hence, applying the
 functor $\mathcal{H} ^{0} $ to 
\ref{g+Omega}, 
 we get the isomorphism compatible with Frobenius
 \begin{equation}
  \label{g+OmegaH0}
   \mathcal{H} ^{0} g  _{1 +} (\E) \riso
   \coker _{\partial _{1} } (\E) (1).
 \end{equation}
 Let us describe the Frobenius structure on $\coker _{\partial _{1} } (\E)$.
Since $ F _{\X /\ZZ_1 *} =\id$,  we get the adjunction morphism
$ \E' \to F _{\X /\ZZ_1} ^{*} \E'$. With the isomorphism $F _{\X /\ZZ_1} ^{*} \E'  \riso F _{\X } ^{*} \E$ of 
\ref{pre-Frob-direct-image},
this yields $\mathrm{adj}\colon
 \E' \to  F _{\X } ^{*} \E$. By identifying $\Omega^1 _{\X' /\ZZ _1}$ with $\O _{\X'}$
 and $\Omega^1 _{\X /\ZZ _1}$ with $\O _{\X}$, the canonical
 quasi-isomorphism $\Omega _{\X '/\ZZ _1} ^\bullet \otimes _{\O _{\X'}}
 \E '\to \Omega _{\X /\ZZ _1} ^\bullet \otimes _{\O _{\X}} F _\X ^{*}(\E )$ (see \ref{4.3.5-Be2})
 is given by the commutative diagram
\begin{equation}
 \label{adj-cokerE'E}
  \xymatrix @R=0,3cm  @C=2cm{
  {\E'}\ar[r] ^-{\mathrm{adj}}\ar[d] ^-{\partial _{1} '}& 
  {F _{\X } ^{*} \E}\ar[d] ^-{\partial _{1} }\\ 
 {\E'}\ar[r] ^-{q t _1 ^{q-1}\cdot\mathrm{adj}}& 
  {F _{\X } ^{*} \E.}
  }
\end{equation}
 The maps $q t _1 ^{q-1}\cdot\mathrm{adj}$ induces the isomorphism $\coker
 _{\partial _{1} '} (\E') \riso \coker _{\partial _{1} } (F _{\X} ^{*}
 \E)$ given by $[x'] \mapsto [q t _1 ^{q-1}\otimes x']$. Hence we get
 \begin{equation}
  \label{defi-FrobcokerE}
   \xymatrix @R=0,3cm {
   {F _{\ZZ _1} ^{*} \coker _{\partial_{1}} (\E)}\ar[r] ^-{\sim}& 
   {\coker _{\partial _{1} '} (\E')}\ar[r] ^-{\sim}&
   {\coker _{\partial _{1} } (F _{\X} ^{*} \E)}\ar[r] ^-{\sim} _-{\phi _\E}&
   {\coker _{\partial _{1} } (\E),}
   }
 \end{equation}
 where the first homomorphism is an isomorphism because $F _{\ZZ _1}$ is
 flat. 
 We notice that the composition 
 \ref{defi-FrobcokerE} is indeed by construction the Frobenius structure of
 $\coker _{\partial_{1}} (\E)$ which comes from that of 
 $g _{1*} \bigl(\Omega _{\X /\ZZ _1} ^\bullet
   \otimes _{\O _{\X}} \E  \bigr)$
given in
 \ref{FrobReldeRham}. 
 
\end{empt}

\begin{empt}
As in \ref{g1flatdescr}, 
we have the canonical isomorphism (without Frobenius structure):
\begin{equation}
 \label{gflat+Omega}
  g ^{\flat} _{1 +} (\FF) \riso 
  g _{1*} (\Omega _{\X ^{\flat}/\ZZ _1} ^\bullet \otimes _{\O _{\X}} \FF
  )[1].
\end{equation}
 By identifying 
 $\Omega _{\X ^{\flat} /\ZZ _1}$ with $\O _{\X}$ and by applying the
 functor $\mathcal{H} ^{0}$, we get the isomorphism
 \begin{equation}
  \label{gflat+OmegaH0}
   \mathcal{H} ^{0} g ^{\flat}  _{1 +} (\FF) \riso
   \coker _{t _{1}\partial _{1} } (\FF).
 \end{equation}
 Now, let us define a Frobenius structure on $\coker _{t _{1} \partial
 _{1} } (\FF)$ as follows (which will be a log analogue of that 
 on $\coker _{\partial _{1} } (\E)$ given in \ref{defi-FrobcokerE}) : we denote by $\mathrm{adj}\colon \FF' \to F _{\X
 /\ZZ_1} ^{*} \FF'  \riso F _{\X } ^{*} \FF$
 the canonical homomorphism. By \ref{2.3.4.1}, we can check that the
 diagram
 \begin{equation}
  \label{adj-cokerF'F}
   \xymatrix @R=0,3cm  @ C=2cm{
   {\FF'}\ar[r] ^-{\mathrm{adj}}\ar[d] ^-{t _{1} '\partial _{1} '}& 
   {F _{\X } ^{*} \FF}\ar[d] ^-{t _{1} \partial _{1} }\\ 
  {\FF'}\ar[r] ^-{q \cdot \mathrm{adj}}& 
   {F _{\X } ^{*} \FF}
   }
 \end{equation}
 is commutative. Thus the map $q\cdot \mathrm{adj}$
 induces the canonical map $\coker _{t _{1} '\partial _{1} '} (\FF')\to
 \coker _{t _{1} \partial _{1} } (F _{\X } ^{*} \FF)$. We get the
 Frobenius structure on $\coker _{t _{1} \partial _{1} } (\FF)$ via the
 homomorphisms:
 \begin{equation}
  \label{defi-cokerF}
   F _{\ZZ _1}  ^{*} \coker _{t _{1} \partial_{1}} (\FF)
   \riso
   \coker _{t _{1} '\partial _{1} '} (\FF')
   \xrightarrow[q\cdot \mathrm{adj}]{\sim}
   \coker _{t _{1} \partial _{1} } (F _{\X } ^{*} \FF)
   \xrightarrow[\coker _{t _{1} \partial _{1} }(\phi)]{\sim}
   \coker _{t _{1} \partial _{1} } (\FF).
 \end{equation}
\end{empt}

\begin{empt}
We have the isomorphisms
 \begin{equation}
  \label{preiso-lemm1Prop-hdagZ'1w1+kercoker}
   g _{1*} (\Omega _{\X ^{\flat}/\ZZ _1} ^\bullet \otimes _{\O _{\X}} \FF  [1])
   \underset{\ref{gflat+Omega}}{\liso}
   g ^{\flat}  _{1 +} (\FF)
   \liso
   g _{1 +} (u _{+}(\FF))
=
   g _{1 +} (\E)
   \underset{\ref{g+Omega}}{\riso}
   g _{1*} (\Omega _{\X /\ZZ _1} ^\bullet \otimes _{\O _{\X}} \E  [1]).
 \end{equation}
\end{empt}

\begin{lem}
 \label{lemm1Prop-hdagZ'1w1+kercoker}
 By applying $\mathcal{H} ^{0}$ to the isomorphism
 \ref{preiso-lemm1Prop-hdagZ'1w1+kercoker}, we obtain the isomorphism
 $\coker _{t _{1}\partial _{1} } (\FF)\riso\coker _{\partial _{1} } (\E)$. 
 This isomorphism commutes with Frobenius structures defined respectively 
 in \ref{defi-cokerF} and \ref{defi-FrobcokerE}.
\end{lem}
\begin{proof}
1)  Let us recall the definition of the isomorphisms appearing in
 \ref{preiso-lemm1Prop-hdagZ'1w1+kercoker}. 
 
 a)   
 The isomorphism $ g ^{\flat}  _{1 +}(\FF)\liso g _{1 +} \circ u _{+}$
is by definition (notice that here $u ^{-1}= \id$ and $u _* =\id$ and then we do not need to use 
a projection formula) : 
\begin{equation}
\label{iso-trans-g1u+}
 g ^{\flat}  _{1 +} (\FF)
 =
 g _{1*} (  \D^\dag_{\ZZ _1 \leftarrow\X ^{\flat}}  
 \otimes ^{\L} _{ \D^\dag_{\X ^{\flat}} } \FF )
 \riso 
 g _{1*} (  \D^\dag_{\ZZ _1 \leftarrow\X }
\otimes  ^{\L}_{\D^\dag_{\X } }(\D^\dag_{\X \leftarrow \X ^{\flat}}
  \otimes ^{\L} _{ \D^\dag_{\X ^{\flat}} } \FF ))
 \riso 
   g _{1 +} (u _{+}(\FF)),
\end{equation}
where the first isomorphism
comes from 
the canonical isomorphism 
\begin{equation}
\label{iso-trans-g1u+pre}
\D^\dag_{\ZZ _1 \leftarrow\X }
\otimes  ^{\L}_{\D^\dag_{\X } }\D^\dag_{\X \leftarrow \X ^{\flat}}
\riso 
 \D^\dag_{\ZZ _1 \leftarrow\X ^{\flat}} .
\end{equation}

b) 
From \ref{DeRham-g_1}, we get the exact sequence 
 \begin{equation}
  \label{lemm1Propexs}
   \tag{($\star$)}
   0 \to \D^\dag_{\X ^{(\flat)} ,\Q}\to 
   \Omega^1 _{\X ^{(\flat)}/\ZZ _1} \otimes _{\O _{\X}} 
   \D^\dag_{\X   ^{(\flat)},\Q}
   \to
   \D^\dag_{\ZZ _1 \leftarrow\X ^{(\flat)},\Q}  \to 0.
 \end{equation}
 The construction of  the left arrow of 
 \ref{lemm1Propexs} is detailed just above \ref{DeRham-g_1-leftarrow}
 and the right arrow is given by the canonical connection of $\D^\dag_{\X ^{(\flat)},\Q}$.
As for \ref{g1flatdescr},
by applying the functor 
 $g _{1 *} (- \otimes ^\L _{\D ^{\dag} _{\X ,\Q} } \E )$
 (resp.  $g _{1 *} (- \otimes ^\L _{\D ^{\dag} _{\X ^{\flat},\Q} } \G )$)
 to the exact sequence \ref{lemm1Propexs}
 we get the isomorphism 
 \ref{g+Omega} (resp.\ \ref{gflat+Omega}).

2) We now compute the isomorphism
$\coker _{t _{1}\partial _{1} } (\FF)\riso\coker _{\partial _{1} } (\E)$.
We denote by 
$\delta _{\X ^\flat} \colon 
\omega _{\X ^{\flat}} \otimes _{\O _{\X}} \D^\dag_{\X ^{\flat},\Q}
\riso
\omega _{\X ^{\flat}} \otimes _{\O _{\X}} \D^\dag_{\X ^{\flat},\Q}$
as usual the involution isomorphism  which exchanges both structure of 
right $\D^\dag_{\X ^{\flat},\Q}$-modules
(this comes from \cite[1.19]{caro_log-iso-hol} by completion and passage to the inductive limite).
We get the isomorphism of right
$\D^\dag_{\X ^{\flat},\Q}$-modules:
\begin{equation}
\label{alphapre1}
\Omega^1 _{\X ^{\flat}/\ZZ _1} \otimes _{\O _{\X}} \D^\dag_{\X ^{\flat},\Q}
\riso
g _{1} ^{-1} \omega  _{\ZZ _1} ^{-1}  \otimes _{g _{1} ^{-1} \O _{\ZZ _1}} 
(\omega _{\X ^{\flat}} \otimes _{\O _{\X}} \D^\dag_{\X ^{\flat},\Q})
\underset{\delta _{\X ^\flat}}{\riso} 
(\omega _{\X ^{\flat}} \otimes _{\O _{\X}} \D^\dag_{\X ^{\flat},\Q})
 \otimes _{g _{1} ^{-1} \O _{\ZZ _1}} g _{1} ^{-1} \omega  _{\ZZ _1} ^{-1} ,
\end{equation}
where the structure of right 
$\D^\dag_{\X ^{\flat},\Q}$-module of the right term of \ref{alphapre1}
is that coming from the twisted structure.
From $\D^\dag_{\X ^{\flat},\Q} \subset \D^\dag_{\X ,\Q}$, we get 
\begin{equation}
\label{alphapre11}
(\omega _{\X ^{\flat}} \otimes _{\O _{\X}} \D^\dag_{\X ^{\flat},\Q})
 \otimes _{g _{1} ^{-1} \O _{\ZZ _1}} g _{1} ^{-1} \omega  _{\ZZ _1} ^{-1} 
 \to
(\omega _{\X ^{\flat}} \otimes _{\O _{\X}} \D^\dag_{\X ,\Q})
 \otimes _{g _{1} ^{-1} \O _{\ZZ _1}} g _{1} ^{-1} \omega  _{\ZZ _1} ^{-1}
\end{equation}
 The involution isomorphism $\delta _\X$ induces
 the isomorphism 
 of $\D^\dag_{\X ,\Q}$-bimodules of the form
 $ \omega _{\X } ^{-1} \otimes _{\O _{\X} }(\delta _{\X }) 
 \colon 
\D^\dag_{\X ,\Q}
\riso
\omega _{\X } \otimes _{\O _{\X}} \D^\dag_{\X ,\Q}
\otimes _{\O _{\X}} 
 \omega _{\X } ^{-1}$, where both structure of
 $\D^\dag_{\X ,\Q}$-module of the right term are the twisted ones. 
This yields the first isomorphism of right
$\D^\dag_{\X ^{\flat},\Q}$-modules:
\begin{gather}
\notag
(\omega _{\X ^{\flat}} \otimes _{\O _{\X}} \D^\dag_{\X ,\Q})
 \otimes _{g _{1} ^{-1} \O _{\ZZ _1}} g _{1} ^{-1} \omega  _{\ZZ _1} ^{-1}
\underset{\delta _{\X}}{\riso} 
 g _{1} ^{-1} \omega  _{\ZZ _1} ^{-1} \otimes _{g _{1} ^{-1} \O _{\ZZ _1}} 
\otimes _{\O _{\X}}  \omega _{\X } \otimes _{\O _{\X}} 
( \D^\dag_{\X ,\Q} \otimes _{\O _{\X}} 
 \omega _{\X } ^{-1} \otimes _{\O _{\X}} 
 \omega _{\X ^{\flat}}) 
 \\
 \label{alphapre2}
 \riso
 \Omega^1 _{\X /\ZZ _1} \otimes _{\O _{\X}} 
 ( \D^\dag_{\X ,\Q} \otimes _{\O _{\X}} 
 \omega _{\X } ^{-1} \otimes _{\O _{\X}} 
 \omega _{\X ^{\flat}}) 
 =
 \Omega^1 _{\X /\ZZ _1} \otimes _{\O _{\X}} 
   \D^\dag_{\X \leftarrow \X ^{\flat},\Q}
\end{gather}
   Hence, we get by composition of \ref{alphapre1}, \ref{alphapre11} and \ref{alphapre2}
   the morphism of right
$\D^\dag_{\X ^{\flat},\Q}$-modules:
$\alpha \colon  \Omega^1 _{\X ^{\flat}/\ZZ _1} \otimes _{\O _{\X}} \D^\dag_{\X ^\flat ,\Q }
  \to 
\Omega^1 _{\X /\ZZ _1} \otimes _{\O _{\X}} 
   \D^\dag_{\X \leftarrow \X ^{\flat},\Q}
$. 
Via
$\D^\dag_{\X \leftarrow \X ^{\flat},\Q} \riso \D^\dag_{\X,\Q} \otimes _{\O  _{\X}} \O  _{\X} (\ZZ _1) 
\subset \D^\dag_{\X} (\hdag Z _1) _{\Q} $, 
we will identify
$\D^\dag_{\X \leftarrow \X ^{\flat},\Q} $ 
with the sub $(\D^\dag_{\X, \Q} , \D^\dag_{\X ^\flat, \Q} )$-bimodule
$\D^\dag_{\X,\Q} \frac{1}{t _1}$
of $\D^\dag_{\X} (\hdag Z _1) _{\Q}$.
The inclusion 
$ \D^\dag_{\X ^{\flat} ,\Q} \subset 
\D^\dag_{\X ^\flat} (\hdag Z _1) _{\Q} 
=
\D^\dag_{\X} (\hdag Z _1) _{\Q} $ has the factorisation
$\iota \colon  \D^\dag_{\X ^{\flat} ,\Q} \subset 
\D^\dag_{\X \leftarrow \X ^{\flat},\Q} $.
Consider now the following diagram of right $\D^\dag_{\X ^{\flat} ,\Q}$-modules
 \begin{equation}
 \label{lemm1Prop-hdagZ'1w1+kercoker-diag1}
   \xymatrix @R=0,3cm {
   {0}
   \ar[r] ^-{}
   & 
   {\D^\dag_{\X \leftarrow \X ^{\flat},\Q}}
   \ar[r] ^-{\nabla}
   &
   {\Omega^1 _{\X /\ZZ _1} \otimes _{\O _{\X}} 
   \D^\dag_{\X \leftarrow \X ^{\flat},\Q}}
   \ar[r] ^-{}
   &
   { \D^\dag_{\ZZ _1 \leftarrow\X ,\Q}
\otimes _{\D^\dag_{\X ,\Q} }\D^\dag_{\X \leftarrow \X ^{\flat},\Q}}
   \ar[r] ^-{}
      \ar[d] ^-{\sim} _-{\ref{iso-trans-g1u+pre}}
      &
   {0}
   \\ 
  {0}\ar[r] ^-{}
  & 
   { \D^\dag_{\X ^{\flat} ,\Q}}
   \ar[r] ^-{\nabla ^{\flat}}
   \ar[u] ^-{\iota}
   &
   {\Omega^1 _{\X ^{\flat}/\ZZ _1} \otimes _{\O _{\X}} \D^\dag_{\X ^{\flat},\Q}}
   \ar[r] ^-{}
   \ar[u] ^-{\alpha}
   &
   {\D^\dag_{\ZZ _1 \leftarrow\X ^{\flat},\Q} }
   \ar[r] ^-{}
   &
   {0}
   }
 \end{equation}
 where the sequence of the bottom (resp. of the top) is  the respective case of \ref{lemm1Propexs} 
 (resp. is the image of \ref{lemm1Propexs} under the functor 
 $-\otimes _{\D^\dag_{\X ,\Q} }\D^\dag_{\X \leftarrow \X ^{\flat},\Q}$).
 The exactness of the sequence of the bottom has already been checked. 
 Concerning that of the top of \ref{lemm1Prop-hdagZ'1w1+kercoker-diag1}, 
 using the vanishing property given by 
 \ref{iso-trans-g1u+pre},
 this comes from that of \ref{lemm1Propexs}. 
By identifying
$\D^\dag_{\X \leftarrow \X ^{\flat},\Q} $ 
with the sub $(\D^\dag_{\X, \Q} , \D^\dag_{\X ^\flat, \Q} )$-bimodule
$\D^\dag_{\X,\Q} \frac{1}{t _1}$
of $\D^\dag_{\X} (\hdag Z _1) _{\Q}$, we
compute that $\alpha$ is such that
$d log t _1 \otimes 1 \mapsto d t _1 \otimes \frac{1}{t _1}$
and is determined by $\D^\dag_{\X ^{\flat} ,\Q}$-linearity by this formula.
Moreover, the left arrow of  \ref{lemm1Propexs} is the morphism of right 
$\D^\dag_{\X   ^{(\flat)},\Q}$-modules which satisfies
$d t _1 \otimes 1 \mapsto  g _1 ^* (1) \otimes (dt _1 \wedge \dots \wedge d t _d )\otimes (dt _2 \wedge \dots \wedge  d t _d  ) ^*$
(resp. 
$d log t _1 \otimes 1 \mapsto  g _1 ^* (1) \otimes (d log t _1 \wedge d t _2 \wedge \dots \wedge d t _d) \otimes (dt _2 \wedge \dots \wedge  d t _d  ) ^*$.
Finally, we compute that 
the isomorphism \ref{iso-trans-g1u+pre} sends 
$(g _1 ^* (1) \otimes (dt _1 \wedge \dots \wedge d t _d )\otimes (dt _2 \wedge \dots \wedge  d t _d  ) ^* ) \otimes \frac{1}{t _1}$
to 
$g _1 ^* (1) \otimes (dlog t _1 \wedge d t _2 \wedge \dots \wedge d t _d) \otimes (dt _2 \wedge \dots \wedge  d t _d  ) ^*$.
 Hence, we compute that the diagram \ref{lemm1Prop-hdagZ'1w1+kercoker-diag1}
 is commutative.

The commutativity of diagram of modules \ref{lemm1Prop-hdagZ'1w1+kercoker-diag1}
can be formulated as follows: we have the commutative diagram of complexes:
\begin{equation}
 \label{diagofcomplexes1}
   \xymatrix @R=0,3cm {
   {[\D^\dag_{\X \leftarrow \X ^{\flat},\Q} 
   \overset{\nabla}{\longrightarrow} 
   \Omega^1 _{\X /\ZZ _1} \otimes _{\O _{\X}} 
   \D^\dag_{\X \leftarrow \X ^{\flat},\Q}]}
   \ar[r] ^-{}
   &
   { \D^\dag_{\ZZ _1 \leftarrow\X ,\Q}
\otimes _{\D^\dag_{\X ,\Q} }
\D^\dag_{\X \leftarrow \X ^{\flat},\Q}}
   \\ 
   {[\D^\dag_{\X ^{\flat} ,\Q}
      \overset{\nabla^{\flat}}{\longrightarrow}
    \Omega^1 _{\X ^{\flat}/\ZZ _1} \otimes  _{\O _{\X}} \D^\dag_{\X ^{\flat},\Q}]  }
   {}
   \ar[r] ^-{}
   \ar[u] 
   &
   {\D^\dag_{\ZZ _1 \leftarrow\X ^{\flat},\Q} ,}
     \ar[u] ^-{\sim} _-{\ref{iso-trans-g1u+pre}}
    }
 \end{equation}
 where the left arrow is given by $\iota$ and $\alpha$ as above.
By applying the functor $g _{1*}( -\otimes ^{\L} _{\D^\dag_{\X ^{\flat}} } \FF)$ to the diagram 
\ref{diagofcomplexes1},
since 
$u _+ ( \FF) = \D^\dag_{\X \leftarrow \X ^{\flat},\Q} 
\otimes ^{\L} _{\D^\dag_{\X ^{\flat},\Q}} \FF$,
 we get the commutative diagram of complexes:
 \begin{equation}
 \label{diag-complexes}
\xymatrix @R=0,3cm {
{g _{1*} [  u _+ ( \FF)
   \overset{\nabla}{\longrightarrow} 
\Omega _{\X /\ZZ _1} \otimes _{\O _{\X}} u _+ ( \FF) ]} 
\ar[r] ^-{\sim} _-{\ref{g+Omega}}
&
{g _{1*}(\D^\dag_{\ZZ _1 \leftarrow\X }  
\otimes ^{\L}_{\D^\dag_{\X } }
u _+ ( \FF) )} 
 \\ 
{g _{1*} [\FF
      \overset{\nabla^{\flat}}{\longrightarrow}
      \Omega _{\X ^{\flat}/\ZZ _1} \otimes _{\O _{\X}}\FF ]} 
\ar[r] ^-{\sim} _-{\ref{gflat+Omega}}
\ar[u] ^-{}
&
{g _{1*}(\D^\dag_{\ZZ _1 \leftarrow\X ^{\flat}} \otimes ^{\L} _{\D^\dag_{\X ^{\flat}} } \FF),} 
\ar[u] ^-{\sim} _-{\ref{iso-trans-g1u+}}
}
 \end{equation}
 where the left morphism is in fact (by acyclicity) the morphism of complexes which is the image under 
 the functor 
$g _{1*}( -\otimes  _{\D^\dag_{\X ^{\flat}} } \FF)$
of the left morphism of \ref{diagofcomplexes1}. 
 From the commutativity of \ref{diag-complexes}, by definition of the isomorphism \ref{preiso-lemm1Prop-hdagZ'1w1+kercoker},
 we get that the morphism of complexes of the left of \ref{diag-complexes} is precisely 
 the isomorphism \ref{preiso-lemm1Prop-hdagZ'1w1+kercoker}.
 
 Now, 
 let us describe the left morphism of \ref{diag-complexes}.
By forgetting to mention $g _{1*}$,
this left morphism of \ref{preiso-lemm1Prop-hdagZ'1w1+kercoker} is given by two (compatible with $\nabla$ and $\nabla ^\flat$) morphisms, 
one is of the form
$\FF \to \E$ and the other one is of the form
$\Omega _{\X ^{\flat}/\ZZ _1} \otimes _{\O _{\X}}\FF 
\to 
\Omega _{\X /\ZZ _1} \otimes _{\O _{\X}} \E$ (recall that $u _+ (\FF) =\E$).
By using the canonical isomorphism
$\FF (\hdag Z _1) \riso u _+ (\FF) $, 
the first  morphism 
is given by the inclusion $\FF \hookrightarrow \FF (\hdag Z _1) \riso u _+ (\FF) =\E$ (indeed, the morphism is induced by the inclusion $\iota$).
By identifying $ \Omega ^1 _{\X /\ZZ _1} $ 
(resp. $ \Omega ^1 _{\X ^{\flat}/\ZZ _1} $)
with $\O _{\X}$ via the basis $d t _1$ (resp. via tha basis $d log t _1$),
we compute that  
the second morphism
is given by the multiplication by $\frac{1}{t _1}\colon \FF \hookrightarrow \FF (\hdag Z _1) \riso u _+ (\FF) =\E$
(see the computation describing $\alpha$ above).
This yields that the left morphism of \ref{diag-complexes} can be identified with the left square 
 of the diagram:
 \begin{equation}
 \label{constructionisocokernels}
  \xymatrix @R=0,3cm {
  {\E }\ar[r] ^-{\partial _1}&
  {\E}\ar[r] ^-{}&
  {\coker _{\partial _{1} } (\E) }\ar[r] ^-{}&
  {0}\\ 
  {\FF}\ar[r] ^-{t_1\partial _{1}}\ar[u] ^-{} _-{}& 
   {\FF}\ar[u] ^-{} _-{\frac{1}{t _1}}\ar[r] ^-{}& 
   {\coker _{t_1\partial _{1} } (\FF)}\ar[r] ^-{}
   \ar[u] ^-{\sim} _-{\frac{1}{t _1}}&
   {0,}
   }
 \end{equation}
  The right (vertical) morphism of the diagram \ref{constructionisocokernels} is then equal to the isomorphism
  of the proposition \ref{lemm1Prop-hdagZ'1w1+kercoker}.
  
  3) Finally we check the compatibility with Frobenius.
Consider the diagram
 \begin{equation}
 \label{lemm1Prop-hdagZ'1w1+kercokerEnd}
   \xymatrix @R=0,3cm {
   {F _{\ZZ _1} ^{*} \coker _{\partial_{1}} (\E)}\ar[r] ^-{\sim}& 
   {\coker _{\partial _{1} '} (\E')}\ar[r] ^-{\sim}&
   {\coker _{\partial _{1} } (F _{\X} ^{*} \E)}\ar[r] ^-{\sim} _-{\phi _\E}&
   {\coker _{\partial _{1} } (\E)}\\
  {F _{\ZZ _1} ^{*} \coker _{t _1\partial_{1}} (\FF)}
   \ar[r] ^-{\sim}\ar[u] ^-{} _-{\frac{1}{t _1}}& 
   {\coker _{t' _1\partial _{1} '} (\FF')}
   \ar[r] ^-{\sim}\ar[u] ^-{} _-{\frac{1}{t' _1}}&
   {\coker _{t _1\partial _{1} } (F _{\X } ^{*} \FF)} 
   \ar[r] ^-{\sim} _-{\phi _\FF }\ar[u] ^-{} _-{\frac{1}{t _1}}&
   {\coker _{t _1\partial _{1} } (\FF),}
   \ar[u] ^-{} _-{\frac{1}{t _1}}
   }
 \end{equation}
 where the composition of the top
 is the isomorphism $F _{\ZZ _1} ^{*} \coker _{\partial_{1}} (\E) \to \coker _{\partial _{1} } (\E)$
 is the Frobenius structure of $\coker _{\partial _{1} } (\E)$ as in \ref{defi-FrobcokerE}
 and the isomorphism below
 $F _{\ZZ _1} ^{*} \coker _{t _1\partial_{1}} (\FF)
 \riso 
 \coker _{t _1\partial _{1} } (\FF)$
 is the Frobenius 
as defined in \ref{defi-cokerF}.
 By using the commutativity of 
 \ref{pre-Frob-direct-image}
 and the functoriality in $\FF$ of the diagram \ref{constructionisocokernels}
 (i.e. replacing $\FF$ by a morphism $\FF _1 \to \FF _2$ and 
 $\E$ by $u _+ (\FF _1) \to u _{+} (\FF _2)$ we get a similar diagram in three dimension),
 we get the commutativity of the right
 square of \ref{lemm1Prop-hdagZ'1w1+kercokerEnd}.
 The commutativity of the left  square of \ref{lemm1Prop-hdagZ'1w1+kercokerEnd}
 is obvious. 
 Since $F _{\X /\ZZ _1}^* \colon 
 t ' _1 \mapsto t _1 ^q$, using the computation of the horizontal morphisms of the middle square
 (see respectively just after \ref{adj-cokerE'E} and \ref{adj-cokerF'F}), 
 we get the commutativity of the middle square. 
Hence, the diagram  \ref{lemm1Prop-hdagZ'1w1+kercokerEnd} is commutative,
which finishes the proof. 

\end{proof}

\begin{empt}
\label{cokerFrob-morp}
Since  $N _{1,\FF}\colon\H(1)\rightarrow\H$ commutes with Frobenius, 
we get Frobenius structures on $\ker N _{1, \FF}$
and $\coker N _{1, \FF}$.
Moreover, by applying the
functor $\coker _{t _{1} \partial _{1} }$ to the surjection
$\FF \twoheadrightarrow \FF / \FF (-\ZZ _1)$, we get the morphism
$\vartheta _\FF \colon\coker _{t _{1} \partial _{1} } (\FF) (1)\to\coker N _{1, \FF}$.
\end{empt}

\begin{lem}
 \label{lemm2Prop-hdagZ'1w1+kercoker}
 The canonical morphism
 $\vartheta _\FF\colon \coker _{t _{1} \partial _{1} } (\FF) (1)\to\coker N _{1, \FF}$ defined in \ref{cokerFrob-morp}
 is compatible with Frobenius (for the respective Frobenius structures
 defined in \ref{defi-cokerF} and \ref{cokerFrob-morp}).
\end{lem}
\begin{proof}
 Consider the diagram:
 \begin{equation*}
   \xymatrix @R=0,3cm {
   {F _{\ZZ _1} ^{*} \coker _{t _1\partial_{1}} (\FF)} 
   \ar[r] ^-{\sim}\ar[d] ^-{qF _{\ZZ _1} ^{*} \vartheta _\FF} _-{}&
   {\coker _{t' _1\partial _{1} '} (\FF')} 
   \ar[r] ^-{\sim}_-{q\cdot \mathrm{adj}}
   &
   {\coker _{t _1\partial _{1} } (F _{\X } ^{*} \FF)} 
   \ar[r] ^-{\sim} _-{\phi _{\FF}}
   \ar[d] ^-{\vartheta _{F _{\X } ^{*} \FF}} _-{}
   &
   {\coker _{t _1\partial _{1} } (\FF)}
   \ar[d] ^-{\vartheta _{\FF}} _-{}
   \\
  {F _{\ZZ _1} ^{*} \coker N _{1, \FF}} 
   \ar[rr] ^-{\sim}
   && 
   {\coker  (N _{1, F _{\X} ^{*} \FF})} 
   \ar[r] ^-{\sim} _-{\phi _{\H}}&
   {\coker N _{1, \FF},}
   }
 \end{equation*}
 where the composition of the top horizontal morphisms is the Frobenius
 structure as defined in \ref{defi-cokerF}
 and where the composition of the bottom horizontal morphisms is the Frobenius
 structure of $\coker  (N _{1,  \FF})$. 
The right square is
 commutative by functoriality of $\FF \mapsto \vartheta _\FF$ (and by definition of 
 $\phi _\H$ given in \ref{phiH}).
 We check the commutativity of the left triangle by computation.
 \end{proof}

\begin{lem}
 \label{lemm-square-surjFrob}
 Consider the commutative square of coherent $\D ^\dag _{\ZZ _1,
 \Q}$-modules
 \begin{equation*}
   \xymatrix@R=0,3cm{
   {\E _1}\ar@{->>}[r] ^-{a _1}\ar[d] ^-{f}&
   {\E _1 '}\ar[d] ^-{f'}\\ 
  {\E _2}\ar[r] ^-{a _2}& 
   {\E _2 '},
   }
 \end{equation*}
 whose homomorphism on the top is surjective. We suppose these modules
 have a Frobenius structure such that $f, a _1, a _2$ commute with
 Frobenius. Then so is $f'$.
\end{lem}
\begin{proof}
It is sufficient to check that two morphisms
of the form $F _{\ZZ _1} ^* \E ' _1 \to \E ' _2$ 
which are equal after composition with 
$F _{\ZZ _1} ^* (a _1) \colon F _{\ZZ _1} ^*\E  _1 \to F _{\ZZ _1} ^* \E ' _1$
are equal. This property is satisfied because 
$F _{\ZZ _1} ^* (a _1)$ is surjective. 
\end{proof}

\begin{prop}
 \label{Prop-hdagZ'1w1+kercoker}
 The isomorphisms
 \begin{equation*}
   \ker(N _{1, \FF})\riso g _{1+} \ker  (\theta _{\E}),
   \qquad
   \coker(N _{1, \FF})\riso g _{1+} \coker  (\theta _{\E}),
 \end{equation*}
 which are constructed by composing \ref{w1+kercoker} and
 \ref{esp1=theta1}, commute with Frobenius.
\end{prop}
\begin{proof}
 First, let us prove the compatibility with Frobenius of the second isomorphism.
 Consider the following diagram:
 \begin{equation*}
   \xymatrix @R=0,3cm {
   { \mathcal{H} ^{0} g _{1 +} j_! (\E) } 
   \ar@*{[|<1pt>]}[r] ^-{\theta _{\E}}\ar[d] ^-{\sim } _-{\ref{j!Eu!F}}&
   { \mathcal{H} ^{0} g _{1 +} \E}
   \ar@*{[|<1pt>]}[r] ^-{}\ar@*{[|<1pt>]}@{=}[d] ^-{}&
   { \mathcal{H} ^{0} g _{1 +} \coker \theta _\E}
   \ar[r] ^-{}\ar[d] ^-{\sim} _-{ \ref{esp1=theta1}}&
   {0}\\
  { \mathcal{H} ^{0} g  _{1 +} \circ u _{+}( \FF(-\ZZ _1)) } 
   \ar[r] ^-{\epsilon _{\FF}}\ar[d] ^-{\sim}& 
   {\mathcal{H} ^{0} g  _{1 +} \circ u _{+} (\FF) } 
   \ar[r] ^-{}\ar@*{[|<1pt>]}[d] ^-{\sim}& 
   {\mathcal{H} ^{0} g  _{1 +} (\coker  ( \epsilon _{1,\FF})) } 
   \ar[r] ^-{}\ar[d] ^-{\sim} _-{ \ref{w1+kercoker}}&
   {0,}\\
  {\coker _{t _1 \partial _1} ( \FF(-\ZZ _1)) } 
   \ar[r] ^-{\alpha _{\FF}}& 
   {\coker _{t _1 \partial _1} (\FF) (1)} 
   \ar@*{[|<1pt>]}[r] ^-{}&
   {\coker N _{1, \FF},} 
   \ar[r] ^-{}&
   {0}
   }
 \end{equation*}
 where  the left above square is (modulo the biduality isomorphism $\DD _{\X}\circ \DD _{\X} (\E) \riso \E$) 
 the image under the functor $ \mathcal{H}^{0} g _{1+}$ of the square \ref{cube-E-F-bis},
where both vertical morphisms of the bottom left square comes
from the canonical isomorphisms
$\mathcal{H} ^{0} g  _{1 +} \circ u _{+} \riso 
\mathcal{H} ^{0} g  ^\flat _{1 +} 
\underset{\ref{gflat+OmegaH0}}{\riso}
\coker _{t _1 \partial _1}$. 
Both left squares are commutative by definition or fonctoriality.
The right square of the top is commutative by definition of the isomorphism
\ref{esp1=theta1}.
Using the remark \ref{rem3.4.10-3.4.17}, 
we check the commutativity of the 
the bottom right square.
 Hence, the diagram is commutative.

 Since $\X$ is affine, the functors $\coker _{t _1
 \partial _1} $ and $\mathcal{H} ^{0} g_{1 +}$ are right exact, so
 the horizontal sequences are exact. Now, the horizontal morphisms on
 the top commute with Frobenius. The composition of the middle vertical
 homomorphisms $\mathcal{H} ^{0} g _{1 +} \E\riso\coker _{t _1 \partial
 _1} (\FF) (1)$ is compatible with Frobenius by \ref{g+OmegaH0} and Lemma
 \ref{lemm1Prop-hdagZ'1w1+kercoker}, and the right horizontal
 homomorphism on the bottom as well by Lemma
 \ref{lemm2Prop-hdagZ'1w1+kercoker}.
 We conclude the proof by using Lemma \ref{lemm-square-surjFrob}.

Finally we get the compatibility with Frobenius of the first isomorphism of \ref{Prop-hdagZ'1w1+kercoker} 
from that of the second one
 via the isomorphisms:
 \begin{align*}
  \DD _{\X } i _{1+} \ker N _{1, \FF}& \riso
  i _{1+}  \DD _{\ZZ _1} \ker N _{1, \FF}
  \underset{\cite[3.12]{Abe-Frob-Poincare-dual}}{\riso}
  i _{1+} \bigl(\ker N _{1, \FF} \bigr) ^{\vee} (-d +1)
  \underset{\ref{KerdualcokerN}}{\riso}
  i _{1+}\coker (N _{1, \FF} ^\vee)  (-d +1)\\
  &
  \underset{\ref{dfndualfrakC}}{\riso}
 i _{1+} \coker(N _{1, \FF ^\vee})(-d)
  \riso\coker  (\theta _{\E^{\vee}}) (-d)
    \underset{\cite[3.12]{Abe-Frob-Poincare-dual}}{\riso}
  \coker  (\theta _{\DD _{\X,Z}\E})
  \underset{\ref{i!+bidual}.1}{\riso}
  \DD _{\X} \ker(\theta  _{\E}.)
 \end{align*}
\end{proof}

\begin{lem}
 \label{lem-outD1}
 Let $(Y,X,\PP)$ be a frame, and
 $j=(\star,\mr{id},\mr{id})\colon(Y,X,\PP)\rightarrow(X,X,\PP)$ be a
 morphism of frames.
 Let $\G _1\in F\text{-}\mr{Ovhol}(X,\PP/K)$, $\G _2 \in
 F\text{-}\mr{Ovhol}(Y,\PP/K)$, and $\phi \colon \G _1 \riso
 j_!(\G _2)$ be an isomorphism in $\mr{Ovhol}(X,\PP/K)$.
 Let $\mc{U}$ be an open formal subscheme of
 $\PP$ containing $Y$. If $\phi|_{\mc{U}}$ commutes with the Frobenius
 structures, then so is $\phi$.
\end{lem}
\begin{proof}
 Let us show that $\phi^{-1}$ is compatible with Frobenius. By taking
 the adjoint, it suffices to show that the homomorphism
 $j^!(\phi^{-1})\colon\G_2\rightarrow j^!\G_1$ in $\mr{Ovhol}(Y,\PP)$ is
 compatible with Frobenius. The lemma follows since the restriction
 functor $\mr{Ovhol}(Y,\PP)\rightarrow\mr{Ovhol}(Y,\mc{U})$ is
 faithful by Remark \ref{rem-faithful}.
\end{proof}

\begin{thm}
\label{iso3314-Frob}
 We consider the situation in paragraph \ref{mixedisoals}, and $D_1$ is
 not empty anymore.
 The isomorphisms \ref{u+mixedstable-iso1} are compatible with
 Frobenius.
\end{thm}
 
\begin{proof}
 From Lemma \ref{lem-outD1}, it is enough to check it outside $D _1$,
 and we can suppose $D _1$ to be empty. By the same reason given at the
 first step of the proof of Lemma  \ref{kercoker-epsilon}, we can
 suppose $f =\mr{id}$. 
 Since by construction, outside of $D _1$, the isomorphisms of
 \ref{u+mixedstable-iso1} are (modulo the biduality isomorphism) that constructed in 
\ref{WhenDempty}, i.e. 
are equal to that of Proposition
 \ref{Prop-hdagZ'1w1+kercoker}, 
 then the claim follows by this proposition \ref{Prop-hdagZ'1w1+kercoker}.
\end{proof}

\subsection{Stability of mixedness for a unipotent $F$-isocrystal}
\label{section35}

Let $(Y,X,\PP)$ be a frame such that $X$ is smooth and there exists a
 strict normal crossing  divisor $Z$ of $X$ such that $Y:=X\setminus Z$. 
Let $\{Z_i\}_{i\in I}$ be
 the set of irreducible components of $Z$, $r := \# I$. 
 We can suppose $I = \{ 1,\dots, r\}$. For $J\subset I$, we put
 $Z_J:=\bigcap_{i\in J}Z_i$. Then we define
 \begin{equation*}
  Z_{(0)}:=X,\qquad
  Z_{(k)}:=\bigcup_{\#J=k}Z_J,\qquad
   Z^{\circ}_{(k)}:=Z_{(k)}\setminus Z_{(k+1)}.
 \end{equation*}
 Then $\{Z^\circ_{(k)}\}_{0\leq k\leq r}$ is a smooth stratification of
 $X$. This stratification is denoted by $\mr{Strat}_Z(X)$.

\begin{lem}
[Gluing isomorphisms]
\label{2.1.5Be2}
We suppose $p ^{m}> e /p-1$, 
where $e$ is the absolute ramification index of $\V$. 
Let $f,g,h\colon X \to Y$ be a morphism of log-smooth log-schemes over $V/\pi ^{i+1}\V$
and $f _0,g_0,h _0\colon X _0 \to Y _0$ the induced morphism of log-schemes over $V/\pi ^{1}\V$.
Let $\FF$ be a left $\D ^{(m)} _{Y}$-module.
If $f _0=  g _0=h _0$,
then there exists a canonical morphism of $\D ^{(m)} _{X}$-modules 
\begin{equation}
\notag
\tau _{f,g}\colon 
g  ^* \FF \riso f ^{*} \FF,
\end{equation}
such that $\tau _{f,f}=\mathrm{Id}$
and 
$\tau _{f,h}= \tau _{f,g} \circ \tau _{g,h}$.
\end{lem}

\begin{proof}
Since  $p ^{m}> e /p-1$, $\mathfrak{a}:= \pi \O _T$ is endowed with a canonical $m$-PD-structure which is 
$m$-PD-nilpotente (see \cite[A.4]{Be2}).
Hence, we can copy the proof of \cite[2.1.5]{Be2}.
\end{proof}

\begin{empt}
\label{nota-alphaetc}
Let $(\PP _\alpha) _{\alpha \in \Lambda}$ be a cover of 
$\PP$ by affine open sets $\PP _\alpha $ endowed with local coordinates 
$t ^\alpha _1 ,\dots, t ^\alpha _n$ such that 
$\smash{\overline{t}} ^\alpha _1, \dots, \smash{\overline{t}} ^\alpha  _d$
are local coordinates of $X _\alpha := X \cap P _\alpha$  
(where 
$\smash{\overline{t}} ^\alpha _i$ is the image of $t ^\alpha _i$ via 
$\Gamma (\PP, \O _{\PP _\alpha}) 
\to
\Gamma (X _\alpha , \O _{X _\alpha })$)
and such that
$Z _i \cap X_\alpha = V ( \smash{\overline{t}} ^\alpha  _i)$. 

We set $\PP _{\alpha \beta}:= \PP _\alpha \cap \PP _\beta$,
$\PP _{\alpha \beta \gamma}:= \PP _\alpha \cap \PP _\beta \cap \PP _\gamma$,
$X _\alpha := X \cap P _\alpha$,
$X_{\alpha \beta } := X _\alpha \cap X _\beta$ et
$X_{\alpha \beta \gamma } := X _\alpha \cap X _\beta \cap X _\gamma $.
We denote by $Y _\alpha $ the open of  $X _\alpha$ complementary to $Z$,
$Y _{\alpha \beta} := Y _\alpha \cap Y _\beta$,
$Y _{\alpha \beta \gamma} := Y _\alpha \cap Y _\beta \cap Y _\gamma $,
$j _\alpha$ : $ Y _\alpha
\hookrightarrow X _\alpha$,
$j _{\alpha \beta} $ :
$Y _{\alpha \beta}  \hookrightarrow  X _{\alpha \beta}$
and
$j _{\alpha \beta \gamma} $ :
$Y _{\alpha \beta \gamma }  \hookrightarrow  X _{\alpha \beta \gamma} $
the canonical open immersions.

For any 3uple $(\alpha, \, \beta,\, \gamma)\in \Lambda ^3$, fix 
$\X _\alpha$ (resp. $\X _{\alpha \beta}$, $\X _{\alpha \beta \gamma}$)
some smooth formal $\V$-schemes lifting $X _\alpha$
(resp. $X _{\alpha \beta}$, $X _{\alpha \beta \gamma}$).
Choose etale morphisms 
$\phi _\alpha \colon \X _{\alpha} \to \widehat{\A} ^d _\V$,
$\phi _{\alpha,\beta} \colon \X _{\alpha,\beta} \to \widehat{\A} ^d _\V$,
$\phi _{\alpha,\beta,\gamma} \colon \X _{\alpha,\beta,\gamma} \to \widehat{\A} ^d _\V$.
We denote by 
$\ZZ _\alpha := \phi _{\alpha} ^{-1} (V (x _1\cdots x _d))$,
$\ZZ _{\alpha,\beta} := \phi _{\alpha,\beta} ^{-1} (V (x _1\cdots x _d))$,
$\ZZ _{\alpha,\beta,\gamma} := \phi _{\alpha,\beta,\gamma} ^{-1} (V (x _1\cdots x _d))$,
where $x _1,\dots, x _r$ are the $r$ first coordinates of $\widehat{\A} ^d _\V$.
We can suppose that $Z _\alpha:= X _\alpha \cap Z$,
 $Z _{\alpha,\beta} = X _{\alpha,\beta} \cap Z$
 and 
$Z _{\alpha,\beta,\gamma}= X _{\alpha,\beta,\gamma} \cap Z$. 
We set
$\X _\alpha ^\flat:= (\X _\alpha , \ZZ _\alpha)$,
$\X _{\alpha \beta} ^\flat
:=
(\X _{\alpha \beta}, \ZZ _{\alpha \beta})$, 
$\X _{\alpha \beta \gamma} ^\flat
:=
(\X _{\alpha \beta \gamma}, \ZZ _{\alpha \beta,\gamma})$.

Using \cite[3.11]{Kato-logFontaine-Illusie},
there exists
a morphism of formal $\V$-schemes of the form
$p _1 ^{\alpha \beta \flat}$ :
$\X ^\flat   _{\alpha \beta} \rightarrow \X ^\flat  _{\alpha}$
(resp. $p _2 ^{\alpha \beta\flat}$ :
$\X ^\flat   _{\alpha \beta} \rightarrow \X ^\flat  _{\beta}$)
which is a lifting of
$X ^\flat   _{\alpha \beta} \rightarrow X ^\flat  _{\alpha}$
(resp. $X ^\flat   _{\alpha \beta} \rightarrow X ^\flat  _{\beta}$).
Similarly, for any $(\alpha,\,\beta,\,\gamma )\in \Lambda ^3$, fix some lifting 
$p _{12} ^{\alpha \beta \gamma\flat}\colon\X ^\flat   _{\alpha \beta \gamma} \rightarrow \X ^\flat   _{\alpha \beta} $,
$p _{23} ^{\alpha \beta \gamma\flat}\colon\X ^\flat   _{\alpha \beta \gamma} \rightarrow \X ^\flat   _{\beta \gamma} $,
$p _{13} ^{\alpha \beta \gamma\flat}\colon\X ^\flat   _{\alpha \beta \gamma} \rightarrow \X ^\flat   _{\alpha \gamma} $,
$p _1 ^{\alpha \beta \gamma\flat}\colon\X ^\flat   _{\alpha \beta \gamma} \rightarrow \X ^\flat   _{\alpha} $,
$p _2 ^{\alpha \beta \gamma\flat}\colon\X ^\flat   _{\alpha \beta \gamma} \rightarrow \X ^\flat   _{\beta} $,
$p _3 ^{\alpha \beta \gamma\flat}\colon\X ^\flat   _{\alpha \beta \gamma} \rightarrow \X ^\flat   _{\gamma} $.
By removing the symbol ``$\flat$'' we mean the corresponding formal scheme or morphism of formal schemes.

\end{empt}

\begin{dfn}
\label{defindonnederecol}
We keep notation \ref{nota-alphaetc}.
We define the category $F\text{-}\mathrm{Isoc} ((\X  _\alpha ^\flat )_{\alpha \in \Lambda}/K)$ as follows:

- an object is a family $(\FF _\alpha) _{\alpha \in \Lambda}$
of coherent  $\smash{\D} ^{\dag} _{\X ^\flat _{\alpha},\Q}$-modules
which are locally projective of finite type as $\O_{\X _{\alpha} ,\Q}$-module and endowed with 
a glueing data $ (\vartheta _{\alpha\beta}) _{\alpha ,\beta \in \Lambda}$,
i.e. the data for any $\alpha,\,\beta \in \Lambda$ of a
$\smash{\D} ^{\dag} _{\X _{\alpha \beta} ^\flat,\Q}$-linear isomorphism
$ \vartheta _{  \alpha \beta} \ : \  p _2  ^{\alpha \beta \flat!} (\FF _{\beta}) \riso p  _1 ^{\alpha \beta \flat!} (\FF _{\alpha}),$
satisfying the cocycle condition:
$\vartheta _{13} ^{\alpha \beta \gamma }=
\vartheta _{12} ^{\alpha \beta \gamma }
\circ
\vartheta _{23} ^{\alpha \beta \gamma }$,
where $\vartheta _{12} ^{\alpha \beta \gamma }$, $\vartheta _{23} ^{\alpha \beta \gamma }$
and $\vartheta _{13} ^{\alpha \beta \gamma }$ 
are defined via the commutative squares
\small
\begin{equation}
  \label{diag1-defindonnederecol}
\xymatrix  @R=0,3cm {
{  p _{12} ^{\alpha \beta \gamma \flat!} p  _2 ^{\alpha \beta \flat!}  (\FF _\beta )}
\ar[r] ^-{\tau} _-{\sim}
\ar[d] ^-{p _{12} ^{\alpha \beta \gamma \flat!} (\vartheta _{\alpha \beta})} _-{\sim}
&
{p _2 ^{\alpha \beta \gamma \flat!}  (\FF _\beta )}
\ar@{.>}[d] ^-{\vartheta _{12} ^{\alpha \beta \gamma }}
\\
{ p _{12} ^{\alpha \beta \gamma \flat!}  p  _1 ^{\alpha \beta \flat!}  (\FF _\alpha)}
\ar[r]^{\tau} _-{\sim}
&
{p _1 ^{\alpha \beta \gamma \flat!}(\FF _\alpha),}
}
%
%
\xymatrix  @R=0,3cm {
{  p _{23} ^{\alpha \beta \gamma \flat!} p  _2 ^{\beta \gamma\flat!}  (\FF _\gamma )}
\ar[r] ^-{\tau} _-{\sim}
\ar[d] ^-{p _{23} ^{\alpha \beta \gamma\flat !} (\vartheta _{ \beta \gamma})} _-{\sim}
&
{p _3 ^{\alpha \beta \gamma\flat!}  (\FF _\gamma )}
\ar@{.>}[d] ^-{\vartheta _{23} ^{\alpha \beta \gamma }}
\\
{ p _{23} ^{\alpha \beta \gamma \flat!}  p  _1 ^{ \beta \gamma \flat!}  (\FF _\beta)}
\ar[r]^{\tau} _-{\sim}
&
{p _2 ^{\alpha \beta \gamma\flat!}(\FF _\beta),}
}
%
%
\xymatrix  @R=0,3cm {
{  p _{13} ^{\alpha \beta \gamma \flat!} p  _2 ^{\alpha \gamma \flat!}  (\FF _\gamma )}
\ar[r] ^-{\tau} _-{\sim}
\ar[d] ^-{p _{13} ^{\alpha \beta \gamma \flat!} (\vartheta _{\alpha \gamma})} _-{\sim}
&
{p _3 ^{\alpha \beta \gamma\flat!}  (\FF _\gamma )}
\ar@{.>}[d]^{\vartheta _{13} ^{\alpha \beta \gamma }}
\\
{ p _{13} ^{\alpha \beta \gamma\flat !}  p  _1 ^{\alpha \gamma\flat !}  (\FF _\alpha)}
\ar[r]^{\tau} _-{\sim}
&
{p _1 ^{\alpha \beta \gamma\flat!}(\FF _\alpha),}
}
\end{equation}
where the glueing isomorphisms denoted by $\tau $ are defined from \ref{2.1.5Be2} by completion, tensorising by $\Q$ and inductive limits.
\normalsize

- a morphism
$((\FF _{\alpha})_{\alpha \in \Lambda},\, (\vartheta _{\alpha\beta}) _{\alpha ,\beta \in \Lambda})
\rightarrow
((\FF ' _{\alpha})_{\alpha \in \Lambda},\, (\vartheta '_{\alpha\beta}) _{\alpha ,\beta \in \Lambda})$
is a familly of morphisms $f _\alpha$ : $\FF _\alpha \rightarrow \FF '_\alpha$
commuting with glueing datas, i.e., such that the following diagrams are commutative : 
\begin{equation}
  \label{diag2-defindonnederecol}
\xymatrix  @R=0,3cm {
{ p _2  ^{\alpha \beta \flat !} (\FF _{\beta}) }
\ar[d] ^-{p _2  ^{\alpha \beta \flat!} (f _{\beta}) }
\ar[r] ^-{\vartheta _{\alpha\beta}} _-{\sim}
&
{  p  _1 ^{\alpha \beta \flat!} (\FF _{\alpha}) }
\ar[d] ^-{p  _1 ^{\alpha \beta \flat!} (f _{\alpha})}
\\
{p _2  ^{\alpha \beta \flat!} (\FF '_{\beta})  }
\ar[r]^{\vartheta '_{\alpha\beta}} _-{\sim}
&
{ p  _1 ^{\alpha \beta \flat !} (\FF '_{\alpha})  .}
}
\end{equation}

\end{dfn}

The category $F\text{-}\mathrm{Isoc} ((\X  _\alpha ^\flat )_{\alpha \in \Lambda}/K)$
is a logarithmic analogue of the category
$F\text{-}\mathrm{Isoc} ^{\dag\dag} ((\X  _\alpha )_{\alpha \in \Lambda},\, Z/K)$
constructed in \cite{caro-construction} which we recall below for the reader :

\begin{dfn}
\label{defindonnederecol2}
We keep notation \ref{nota-alphaetc}.
The category $F\text{-}\mathrm{Isoc} ^{\dag\dag} ((\X  _\alpha )_{\alpha \in \Lambda},\, Z/K)$ 
is defined as follows:

- an object is a family $(\E _\alpha) _{\alpha \in \Lambda}$
of coherent  $\smash{\D} ^{\dag} _{\X _{\alpha} } (\hdag Z  \cap X _{\alpha}) _{\Q}$-modules
which are $\O_{\X _{\alpha} } (\hdag Z  \cap X _{\alpha}) _{\Q}$-coherent and 
endowed with glueing data $ (\theta _{\alpha\beta}) _{\alpha ,\beta \in \Lambda}$,
i.e. the data for any $\alpha,\,\beta \in \Lambda$ of a
$\smash{\D} ^{\dag} _{\X _{\alpha \beta} }(\hdag Z  \cap X _{\alpha \beta}) _{ \Q}$-linear
isomorphism
$ \theta _{  \alpha \beta} \ : \  p _2  ^{\alpha \beta !} (\E _{\beta}) \riso p  _1 ^{\alpha \beta !} (\E _{\alpha}),$
satisfying the cocycle condition:
$\theta _{13} ^{\alpha \beta \gamma }=
\theta _{12} ^{\alpha \beta \gamma }
\circ
\theta _{23} ^{\alpha \beta \gamma }$,
where $\theta _{12} ^{\alpha \beta \gamma }$, $\theta _{23} ^{\alpha \beta \gamma }$
and $\theta _{13} ^{\alpha \beta \gamma }$ 
are defined via the commutative squares
\small
\begin{equation}
  \label{diag1-defindonnederecol2}
\xymatrix  @R=0,3cm {
{  p _{12} ^{\alpha \beta \gamma !} p  _2 ^{\alpha \beta !}  (\E _\beta )}
\ar[r] ^-{\tau} _-{\sim}
\ar[d] ^-{p _{12} ^{\alpha \beta \gamma !} (\theta _{\alpha \beta})} _-{\sim}
&
{p _2 ^{\alpha \beta \gamma!}  (\E _\beta )}
\ar@{.>}[d] ^-{\theta _{12} ^{\alpha \beta \gamma }}
\\
{ p _{12} ^{\alpha \beta \gamma !}  p  _1 ^{\alpha \beta !}  (\E _\alpha)}
\ar[r]^{\tau} _-{\sim}
&
{p _1 ^{\alpha \beta \gamma!}(\E _\alpha),}
}
%
%
\xymatrix  @R=0,3cm {
{  p _{23} ^{\alpha \beta \gamma !} p  _2 ^{\beta \gamma!}  (\E _\gamma )}
\ar[r] ^-{\tau} _-{\sim}
\ar[d] ^-{p _{23} ^{\alpha \beta \gamma !} (\theta _{ \beta \gamma})} _-{\sim}
&
{p _3 ^{\alpha \beta \gamma!}  (\E _\gamma )}
\ar@{.>}[d] ^-{\theta _{23} ^{\alpha \beta \gamma }}
\\
{ p _{23} ^{\alpha \beta \gamma !}  p  _1 ^{ \beta \gamma !}  (\E _\beta)}
\ar[r]^{\tau} _-{\sim}
&
{p _2 ^{\alpha \beta \gamma!}(\E _\beta),}
}
%
%
\xymatrix  @R=0,3cm {
{  p _{13} ^{\alpha \beta \gamma !} p  _2 ^{\alpha \gamma !}  (\E _\gamma )}
\ar[r] ^-{\tau} _-{\sim}
\ar[d] ^-{p _{13} ^{\alpha \beta \gamma !} (\theta _{\alpha \gamma})} _-{\sim}
&
{p _3 ^{\alpha \beta \gamma!}  (\E _\gamma )}
\ar@{.>}[d]^{\theta _{13} ^{\alpha \beta \gamma }}
\\
{ p _{13} ^{\alpha \beta \gamma !}  p  _1 ^{\alpha \gamma !}  (\E _\alpha)}
\ar[r]^{\tau} _-{\sim}
&
{p _1 ^{\alpha \beta \gamma!}(\E _\alpha),}
}
\end{equation}
where the glueing isomorphisms denoted by $\tau $ are defined as in \cite[2.1.5]{Be2}.
\normalsize

- a morphism
$((\E _{\alpha})_{\alpha \in \Lambda},\, (\theta _{\alpha\beta}) _{\alpha ,\beta \in \Lambda})
\rightarrow
((\E ' _{\alpha})_{\alpha \in \Lambda},\, (\theta '_{\alpha\beta}) _{\alpha ,\beta \in \Lambda})$
is a familly of morphisms $f _\alpha$ : $\E _\alpha \rightarrow \E '_\alpha$
commuting with glueing datas.
\end{dfn}

\begin{empt}
[Unipotent overcoherent $F$-isocrystal]
\label{unipotent-res-glue}
We keep notation \ref{nota-alphaetc}.
The category $F\text{-}\mr{Isoc}^{\dag\dag}(Y,\PP/K)$
is canonically isomorphic to 
$F\text{-}\mathrm{Isoc} ^{\dag\dag} ((\X  _\alpha )_{\alpha \in \Lambda},\, Z/K)$
(see \cite{caro-construction}).
Since in this subsection the choices are fixed, 
by abuse of notation we will simply denote 
$F\text{-}\mathrm{Isoc} ((\X  _\alpha ^\flat )_{\alpha \in \Lambda}/K)$
by 
$F\text{-}\mathrm{Isoc} (X ^\flat /K)$.

From 
\cite[6.4.5]{kedlaya-semistableI},
the functor 
$(\hdag  X _\alpha \cap Z) \colon 
F\text{-}\mathrm{Isoc}  (\X  _\alpha ^\flat/K)
\to 
F\text{-}\mathrm{Isoc} ^{\dag\dag} (\X  _\alpha, Z\cap \X  _\alpha/K)$
is fully faithful. 
Hence, 
we get the fully faithful functor 
$ F\text{-}\mathrm{Isoc} (X ^\flat /K)
\to 
F\text{-}\mathrm{Isoc} ^{\dag\dag} ((\X  _\alpha )_{\alpha \in \Lambda},\, Z/K)$
which  is defined 
by sending 
$((\FF _{\alpha})_{\alpha \in \Lambda},\, (\vartheta _{\alpha\beta}) _{\alpha ,\beta \in \Lambda})$
to 
$((\FF _{\alpha} (\hdag X _\alpha \cap Z)_{\alpha \in \Lambda},\, (\theta _{\alpha\beta}) _{\alpha ,\beta \in \Lambda})$
where 
$\theta _{\alpha\beta}$ are the isomorphisms induced canonically by 
$\vartheta _{\alpha\beta}$.
By composing this functor with the glueing equivalence
$F\text{-}\mathrm{Isoc} ^{\dag\dag} ((\X  _\alpha )_{\alpha \in \Lambda},\, Z/K)
\cong
F\text{-}\mr{Isoc}^{\dag\dag}(Y,\PP/K)$,
we get a fully faithful functor which will be 
denoted by 
\begin{equation}
\label{hdagZglue}
(\hdag Z)\colon
 F\text{-}\mathrm{Isoc} (X ^\flat /K)
\to 
F\text{-}\mr{Isoc}^{\dag\dag}(Y,\PP/K).
\end{equation}
To justify this notation, notice that
when $\PP _\alpha = \PP$ for any $\alpha$ then the functor \ref{hdagZglue}
is equal to the usual functor
$(\hdag Z)$ as in the lifted case \ref{F-isoc-alg-nalg}. 
The essencial image of this functor 
$(\hdag Z)$ are by definition the 
unipotent (overcoherent) $F$-isocrystals. 
\end{empt}

\begin{empt}
[Residue morphism, monodromy filtration]
\label{monodromy-smoothcase}
We keep the notation \ref{unipotent-res-glue}. 
By gluing \ref{i_1^*}, we get the factorisation
\begin{equation}
\label{i_1^*glue}
i _1 ^* \colon 
F \text{-}\mathrm{Isoc} (X ^\flat/K)
\to 
F \text{-}\mathrm{Isoc} (Z _1 ^\#/K),
\end{equation}
where $i _1 \colon Z _1 \hookrightarrow X$ is the closed immersion
and 
$F \text{-}\mathrm{Isoc} (Z _1 ^\#/K)$ is constructed similarly.

Let $\FF \in F \text{-}\mathrm{Isoc} (X ^\flat/K)$ and 
$\H := i _1 ^* (\FF)$. 
 We note that in the situation of \S\ref{chap3}, the nilpotent
 homomorphism $N_{1,\FF}$ induced by the action of $t _1 \partial _1$ on $\H$
 does not depend on the choice of local coordinates $t _1, \dots, t _d$ such that $Z _1= V (t_1)$.
 Thus, by gluing, we may construct as in \ref{def-monod-filpre} 
 the nilpotent homomorphism
$ N _{1,\FF} \colon\H(1)\rightarrow\H$ 
 in $ F\text{-}\mathrm{Isoc} (Z _1 ^\# /K)$,
 the residue morphism of $\FF$ corresponding to $Z _1$.
As in \ref{monodromyfiltration}, 
we check that there
 exists a unique finite increasing filtration $M$ on $\H$ such that
 $N _{1,\FF}M_i\subset M_{i-2}(-1)$, and $N^k _{1,\FF}$ induces an isomorphism
 $\gr^M_k(\H)\riso \gr^M_{-k}(\H)(-k)$. We call this filtration the
 {\em monodromy filtration} of $(\H, N _{1,\FF})$.
 \end{empt}

\begin{empt}
 \label{nota-u+mixedstable}
 We define the category $\mr{ULNM}(Y,\PP/K)$ to be the full
 subcategory of $F\text{-}\mr{Isoc}^{\dag\dag}(Y,\PP/K)$ consisting of
 $\E$ such that: 1.\ $\E$ is unipotent,
 and 2.\ $\E$ is the restriction of an $\iota$-mixed overconvergent $F$-isocrystal on
 $Y$. 
 \end{empt}

\begin{thm}
 \label{u+mixedstable}
 Let $j=(\star,\mr{id},\mr{id})\colon(Y,X,\PP)\rightarrow(X,X,\PP)$ be
 the morphism of frames.

 (i) Suppose $X$ proper. 
 For any $\E\in\mr{ULNM}(Y,\PP/K)$, then $j_{!} (\E)\Vert_W$ is an
 $\iota$-mixed $F$-isocrystal for any stratum $W$ of the stratification
 $\mathrm{Strat}_{Z} (X)$. In particular, $j_!(\E)$ is an $\iota$-mixed
 $F$-complex.

 (ii) Suppose $Z$ smooth. If $\E$ is $\iota$-pure of weight $w$ 
then $j _{!+} (\E)$ is $\iota$-pure of
 weight $w$.
\end{thm}

\begin{proof}
Using \cite[3.2.20]{kedlaya-semistableI},
we can assume by d\'{e}vissage
that $\E$ is
 $\iota$-pure of weight $w$.
Since $\E$ is 
 a unipotent $F$-isocrystal, 
 there exists a (unique up to isomorphism)  
 an object $\FF$ of  $F\text{-}\mathrm{Isoc} (X ^\flat /K)$
 such that
$\E \riso (\hdag Z)(\FF)$ (see \ref{hdagZglue}).

 Let us start to show (i). We proceed by induction on $r$. For $r=0$, we
have $\ZZ=\emptyset$ and this case is obvious.
Suppose now that $r \geq 1$ and 
 the proposition is checked for $r'<r$.

a)  Let $D _1 := Z _2 \cup \dots \cup Z _r$, $Y _1:= Z _1 \setminus D
 _1$, and $j _1 \colon Y _1 \to Z _1$ be the open immersion. 
From \ref{monodromy-smoothcase}, 
we get an element
$\H := i _1 ^* (\FF)$ of $ F\text{-}\mathrm{Isoc} (Z _1 ^\# /K)$
and a nilpotent homomorphism
$ N _{1,\FF} \colon\H(1)\rightarrow\H$ 
 in $ F\text{-}\mathrm{Isoc} (Z _1 ^\# /K)$.
 Hence, we get the objects
 $\FF _1:= \ker  N_{1,\FF}$
 and $\FF _2:= \coker N_{1,\FF}$ of
 $F\text{-}\mr{Isoc} (Z _1 ^\#/K)$.
The monodromy filtration on $\H$ induces canonically some filtration on  
$\FF _1$ and $\FF _2$.
 Since the functor 
$(\hdag D _1)\colon
 F\text{-}\mathrm{Isoc} (Z _1 ^\flat /K)
\to 
F\text{-}\mr{Isoc}^{\dag\dag}(Y _1,\PP/K)$
(defined as in \ref{hdagZglue})
is exact, 
the  filtrations on $\FF _1$ and $\FF _2$ induce canonically a filtration on the unipotent $F$-isocrystals
 $\E _1:= (\hdag D _1) (\FF _1)$
 and $\E _2:=(\hdag D _1) (\FF _2)$ of
 $F\text{-}\mr{Isoc}^{\dag\dag}(Y _1,\PP/K)$, 
 which will be denoted by  $M$.
 Then $\gr^M_i(\E_1)$,
 $\gr^M_i(\E_2)$ are $\iota$-pure. Indeed, by Remark
 \ref{rem-m-out-div}, the verification is local, and we may reduce to
 the situation of Theorem \ref{crewhigher}. Then by this theorem the
 purity follows.
 Hence, by induction hypothesis, $j _{1!} (\E _1)\Vert_W$ and $j _{1!}
 (\E _2)\Vert_W$ are in $D_\mr{isoc,\mr{m}}^{\mr{b}}(W,\PP/K)$ for any
 stratum $W$ of $\mathrm{Strat}_{D_1} (Z _1)$.

b)  On the other hand, for any stratum $W$ of $\mr{Strat}_{D_1}(Z_1)$, we
 have isomorphisms for any $k\in\mb{Z}$
 \begin{equation}
 \label{isom-j1EVertW}
  \HH^k\bigl(\H ^{\dag 0} _{Z _1} (j _! (\E))\Vert_W\bigr)
   \cong \HH^k\bigl(j_{1!}(\E_1)\Vert_W\bigr),\qquad
  \HH^k\bigl(\H ^{\dag 1} _{Z _1} (j _! (\E))\Vert_W\bigr)
   \cong \HH^k\bigl(j_{1!}(\E_2)\Vert_W\bigr).
 \end{equation}
 Indeed, first notice that the right sides of both isomorphisms are objects of 
 $F\text{-}\mr{Isoc}^{\dag\dag}(W,\PP/K)$.
Moreover, the fact that an object
 of 
 $F\text{-}\mathrm{Ovhol}(W,\PP/K)$
 is an object of 
 $F\text{-}\mr{Isoc}^{\dag\dag}(W,\PP/K)$
 is local in $\PP$. 
 Hence, by using Kedlaya's fully faithfulness theorems 
\cite[5.2.1]{kedlaya-semistableI} (in fact we only need \cite[4.1.1]{tsumono} since $Y$ is smooth) and  \cite{kedlaya_full_faithfull},
the check of the isomorphisms 
\ref{isom-j1EVertW} is local.
Hence,  we may assume that
 we are in the situation of \S\ref{chap3}. Thus, by Theorem
 \ref{iso3314-Frob}, the isomorphisms follows. 
 
 c) Combining a) and b), 
we get that $\H ^{\dag k} _{Z _1} (j _! (\E))\Vert_W$
 ($k=0,1$) is in $D^{\mr{b}}_{\mr{isoc},\mr{m}}(W,\PP/K)$. Since we have
 the triangle
 \begin{equation*}
  \H ^{\dag 0} _{Z _1} (j _! (\E))\Vert_W\rightarrow
   \R\underline{\Gamma}^\dag_{Z_1}(j_!(\E))\Vert_W\rightarrow
   \H ^{\dag 1} _{Z _1} (j _! (\E))\Vert_W\xrightarrow{+},
 \end{equation*}
 we conclude by Lemma \ref{extensisocmix} that $j_!(\E)\Vert_W\cong
 \R\underline{\Gamma} ^\dag _{Z_1}\bigl(j_{!} (\E)\bigr)\Vert_W$ is in
 $D^{\mr{b}}_{\mr{isoc},\mr{m}}(W,\PP/K)$ for any stratum $W$ in
 $\mr{Strat}_{D_1}(Z_1)$.

 By replacing $Z_1$ by $Z_2,\dots , Z _r$ and $D_1$ by $D_i:=\bigcup_{j\neq
 i} Z_j$, we get the same result. By construction, any stratum of
 $\mr{Strat}_Z(X)$ except for $Z^{\circ}_{(0)}=Y$ is a disjoint
 union of strata in $\mr{Strat}_{D_i}(Z_i)$. Thus the theorem follows.

 Let us show (ii). Since the verification is local, we may assume that
 we are in the situation of \S\ref{chap3} and since 
 $Z$ is smooth we can suppose $\ZZ = \ZZ _1$. Consider the exact
 sequence:
 \begin{equation*}
  0\to j _{!+} (\E)\to j _+ (\E)\to
   \mathcal{H} ^{\dag 1} _{Z } (j _! (\E))\to 0.
 \end{equation*}
From \ref{iso3314-Frob}, 
we get the isomorphism $\mathcal{H} ^{\dag 1} _{Z } (j _! (\E))
   \riso 
   i _{1+}\bigl(\coker N_{1,\FF}\bigr)$ compatible with Frobenius.
 Using the Corollory \ref{rem-ker-coker-N}, 
this yields that $\mathcal{H} ^{\dag 1} _{Z
 } (j _! (\E)) $ is $\iota$-mixed of weight $\geq w +1$. 
Usinf by Lemma \ref{wj+},  we get that $j _+ (\E)$ is $\iota$-mixed of weight $\geq w$.
This implies by Proposition \ref{lem-2.2.8}.\ref{lem-2.2.3(4)} (use also the above exact sequence)
 that $j _{!+} (\E)$ is $\iota$-mixed of weight $\geq w$. By
 Corollary \ref{i!+bidual} and Lemma
 \ref{lem-isco-stabwf!+}.\ref{lem-isco-stabwf!+(3)}, we get that $j
 _{!+} (\E)$ is $\iota$-mixed of weight $\leq w$.
\end{proof}

\begin{rem}
 The second condition in the definition of $\mr{ULNM}(Y,\PP/K)$ is used to
 apply Theorem \ref{crewhigher}. If we remove this condition, we do not
 know if Theorem \ref{u+mixedstable} holds or not.
\end{rem}

\section{Theory of weights}
Throughout this section, we consider situation (B) in Notation and
convention.

\subsection{Stability of mixedness}

\begin{empt}
 Let $\star\in\{\emptyset,\leq w,\geq w\}$. Let $Y$ be a realizable
 variety and $\E \in F\text{-}D ^{\mathrm{b}}_{\mathrm{ovhol}}(Y/K)$.
 We say that {\em $\E$ satisfies the condition} $\SQ{\star}$ if the
 following holds:
 \begin{quote}
  $\SQ{\star}$\quad
  For any integer $i$, any subquotient of $\HH^i(\E)$ is $\iota$-mixed
  of weight $\star+i$.
 \end{quote}
\end{empt}

\begin{thm}
 \label{3.3.1WeilII}
 Let $f\colon X  \to Y$ be a morphism of realizable varieties.

 (i) For any $\E\in F\text{-}D ^{\mathrm{b}}_{\geq w}(X/K)$, we have $f
 _+ (\E)\in F\text{-}D ^{\mathrm{b}}_{\geq w}(Y/K)$.

 (ii) The dual functor $\DD_X$ exchanges $F\text{-}D ^{\mathrm{b}}_{\geq
 w}(X/K)$ and $F\text{-}D ^{\mathrm{b}}_{\leq -w}(X/K)$.

 (iii) The condition $\SQ{}$ holds for any $F$-complex in $F\text{-}D
 ^{\mathrm{b}}_{\mr{m}}(X/K)$.
\end{thm}

\begin{rem*}
 See Theorem \ref{sqpropcompl} for more complete results for (iii).
\end{rem*}

\begin{proof}
\noindent
{\bf 0)} Preliminaries. 

 a) {\it D\'{e}vissage in $\E$}:
 Let $\E' \to \E \to \E''\xrightarrow{+}$ be an exact triangle of
 $F\text{-}D ^{\mathrm{b}}_{\geq w}(X/K)$. By Lemma \ref{lem-2.2.8},
 if the part (i) and (ii) of the theorem holds for $\E'$ and $\E''$, then so does for $\E$.

 b) {\it D\'{e}vissage in $X$}:
 Let $j\colon U\hookrightarrow X$ be an open immersion and $i\colon
 X\setminus U\hookrightarrow X$ be the complement. Then if the theorem
 holds for $f \circ j$ and $f \circ i$ then so does for $f$. This
 follows from Lemma \ref{lem-2.2.8}.

 c) {\it D\'{e}vissage in $Y$}:
 Let $V$ be a open subvariety of $Y$ and $W:= Y\setminus V$. If the
 theorem holds for $f' \colon f ^{-1} (V) \to V$ and $f'' \colon f
 ^{-1} (W) \to W$ then so does for $f$. Indeed, by Lemma \ref{wj+}, the
 theorem holds for the morphisms $f^{-1}(V)\xrightarrow{f'} V
 \hookrightarrow Y$ and $f^{-1}(W)\xrightarrow{f''} W\hookrightarrow Y$,
 and by b) the claim follows.

 d) Let $\E'\rightarrow\E\rightarrow\E''\xrightarrow{+}$ be a triangle,
 and assume that $\SQ{}$ holds for $\E'$ and $\E''$. Then $\SQ{}$ holds
 for $\E$ as well.
 \medskip

 \noindent
 {\bf 1)} By Lemma \ref{wj+}, the theorem holds when $X$ is of dimension
 $0$ and we may assume that
 $\dim(Y)\leq\dim(X)$. We proceed by induction on the dimension of
 $X$. Assume the theorem holds for $\dim(X)<n$. We will show the theorem
 for $\dim(X)=n$. Let $\E\in F\text{-}D ^{\mathrm{b}}_{\geq w}(X/K)$.
 \medskip

 \noindent
 {\bf 2)} 
 a) Let us show (i) in the case where $f$ is quasi-finite.
 If $f$ is an immersion, then by Lemma \ref{wj+} the theorem holds.
 Hence, by using EGA IV, Theorem 8.12.6, we may assume that $f$ is
 finite.  By using 0.a) and 0.b), we may assume that $X$ is smooth,
 $\E \in  F\text{-}\mr{Isoc}^{\dag\dag}(X/K)$ and it is $\iota$-pure of
 weight $w$. Moreover, thanks to 0.c) and the induction hypothesis, we
 may further suppose that $Y$ is  smooth and that $f _{+} (\E) \in
 F\text{-}D ^{\mathrm{b}}_{\mathrm{isoc}}(Y/K) $.
 Then it is sufficient to check that for every  closed point $y$ of $Y$, 
 $i _{y} ^{!} \circ f _{+} (\E)$ is $\iota$-pure of weight $w$ where
 $i_y\colon\{y\}\hookrightarrow Y$ is the closed immersion as usual.
 Hence, by using a base change theorem \ref{summpropre}, we reduce
 to the case where $\dim X =\dim Y=0$, and the theorem follows.
 \medskip

 b) Assume $f$ is a universal homeomorphism. From 2.a) and Proposition
 \ref{pullbackstability}, $f _+$ and $f ^{!}$ preserve the mixedness
 $\geq w$. Thus, by Proposition \ref{homeo-univ}, the functors $f _{+}$
 and $f ^{!}$ induce canonical equivalences of categories between
 $F\text{-}D ^{\mathrm{b}}_{\geq w}(X/K)$ and
 $F\text{-}D ^{\mathrm{b}}_{\geq w}(Y/K)$.
 \medskip

 \noindent
 {\bf 3)} Let us show (ii) and (iii).

 a) Let us show (ii) for $\dim(X)=n$.
 We may assume $X$ to be proper. By
 d\'{e}vissage, it suffices to show that for an open affine smooth subscheme
 $j\colon U\hookrightarrow X$ and an $\iota$-mixed overconvergent
 $F$-isocrystal $\E$ on $U$, $j_!(\E)$ is $\iota$-mixed. By using
 Kedlaya's semistable reduction theorem \cite{kedlaya-semistableIV},
 there exists a generically
 finite \'{e}tale morphism $g\colon X'\rightarrow X$ such that
 $X'$ is projective smooth, $Z':=X'\setminus g^{-1}(U)$ is a
 strictly normal crossing divisor, and $g_U^*(\E)$ is a unipotent
 $F$-isocrystal where $g_U\colon g^{-1}(U)\rightarrow U$. There exists
 an open subscheme $V\subset U$ such that $g_V\colon
 V':=g^{-1}(V)\rightarrow V$ is finite \'{e}tale. For $\star\in\{U,V\}$,
 let $j_{\star}\colon\star\hookrightarrow X$, $j'_{\star}\colon
 g^{-1}(\star)\hookrightarrow X'$.
 Now, for a $\iota$-mixed $F$-module $\FF$ on $X'$, we claim that
 $g_+(\FF)$ is $\iota$-mixed. Indeed, consider the
 triangle $\R\underline{\Gamma}^\dag_{X'\setminus
 V'}(\FF)\rightarrow\FF\rightarrow j'_{V+}\FF\Vert_{V'}
 \xrightarrow{+}$. By induction hypothesis,
 $g_+\bigl(\R\underline{\Gamma}^\dag_{X'\setminus V'}(\FF)\bigr)$ is
 $\iota$-mixed, and since $g_V$ is finite \'{e}tale, by 2.a),
 $g_+j'_{V+}\FF\Vert_{V'}$ is $\iota$-mixed as well, and the claim
 follows.

 By Theorem \ref{u+mixedstable}, $j'_{U!}(g^*_U\E)$ is $\iota$-mixed. By using the induction hypothesis,
 this is equivalent to the property that $j'_{V!}(g^*_U (\E) \Vert V ')$ is $\iota$-mixed as
 well. Now, by the claim, so is $g_+\,j'_{V!}(g^*_U (\E) \Vert V ')$. Since
 it contains $j_{V!} (\E \Vert V)$ as a direct factor, then $j_{V!} (\E \Vert V)$ is
 $\iota$-mixed, and using the induction hypothesis again, we get that 
 $j_!(\E)$ is $\iota$-mixed.
 \medskip

 b) Let us show (iii) for $\dim(X)=n$. For $\E$, there exists an open
 subscheme $j\colon U\hookrightarrow X$ such that $U$ is smooth and
 the complement of a divisor in $X$ and $\E\Vert_U$ is an isocrystal. By
 considering the localization triangle, 0.d), and induction
 hypothesis, it suffices to show $\SQ{}$ for $j_+(\FF)$ where $\FF$ is
 an irreducible overconvergent $F$-isocrystal on $U$. Then, by
 \ref{cokertheta-div}, we get the exact sequence
 \begin{equation*}
  0\rightarrow j_{!+}(\FF)\rightarrow j_{+}(\FF)\rightarrow
   \R^1\underline{\Gamma}^\dag_Z(j_!(\FF))\rightarrow0.
 \end{equation*}
 We have already shown at 2.a) that $j_+(\FF)$ and 
 $j_!(\FF)$ are $\iota$-mixed. Thus,
 $\R\underline{\Gamma}^\dag_Z(j_!(\FF))$ is $\iota$-mixed, and
 $\R^1\underline{\Gamma}^\dag_Z(j_!(\FF))$ is $\iota$-mixed by the
 induction hypothesis on $\SQ{}$. Thus, by 0.d),
 $j_{!+}(\FF)$ is $\iota$-mixed. Since $j_{!+}(\FF)$ is irreducible by
 Proposition \ref{BorelVII.10.5}, this satisfies $\SQ{}$. Thus, $\SQ{}$
 for $j_+(\FF)$ follows by 0.d).
 \medskip

 \noindent
  {\bf 4)} It remains to check (i). By using 0), 2.a), and Noether's normalization theorem, we may
 assume that $f$ is a smooth morphism of relative dimension $1$, $X$ is
integral smooth, and $\E \in F\text{-}\mathrm{Isoc} ^{\dag \dag } (X/K)$. By
 0.c), we may assume that $Y$ is integral, affine and smooth.
We choose $h \colon \overline{X} \to Y$ 
a proper morphism of realizable varieties, 
an open immersion $j\colon X \hookrightarrow \overline{X} $
such that $f = h \circ j$ and $X$ is dense in $\overline{X} $.
 By using Kedlaya's semistable reduction theorem, 
 there exist a
 projective surjective morphism $g\colon \overline{X}' \to \overline{X}$
 such that the induced morphism
 $g'\colon X':=g ^{-1} (X) \to X$ is generically finite \'{e}tale, $\overline{X}'$
 is quasi-projective, integral, smooth, $D ':= \overline{X}' \setminus X' $ is a strict
 normal crossing divisor of $\overline{X}'$ and $g ^{\prime *} (\E)$ is a unipotent along $D'$
 overconvergent $F$-isocrystal. 
 Let $U$ be an open dense subvariety of $X$ 
 such that the induced morphism
 $U':=g ^{-1} (U) \to U$ is finite \'{e}tale.
 Let $\mc{K}$ be the function field of $Y$.
Let  $\overline{X}' _{\mc{K}}$, 
 and $U' _{\mc{K}}$ be the fibers of the generic element of $Y$.
 Since $U'_\mc{K}$ is a smooth curve over $\Spec(\mc{K})$, we have its
 canonical smooth compactification $C'_\mc{K}$ over $\mc{K}$.
 Since $\overline{X}' _{\mc{K}}$ is projective over $\mc{K}$, there exists a unique
finite morphism $\alpha_\mc{K}\colon C'_\mc{K} \to
 \overline{X}'_\mc{K}$, which is nothing but the normalization morphism,
 that extends the inclusion $U'_\mc{K}\to\overline{X}'_\mc{K}$.
 By using EGA IV, Theorem 8.8.2 and shrinking $Y$, we may assume that
 $C'_\mc{K}$ comes from a variety $C'$ which is projective smooth
 of relative dimension $1$ over $Y$.  
By using EGA IV, Theorems 8.8.2, 8.10.5, 17.7.8,  
by shrinking $Y$ further,
  we can assume that $\alpha _\mc{K}$ (resp.\ the open immersion
 $U'_\mc{K} \hookrightarrow C'_\mc{K}$) comes from a finite morphism
 $\alpha \colon C' \to \overline{X}'$
 (resp.\ an open immersion $j'\colon U'
 \hookrightarrow C'$) such that $\alpha \circ j'$ is the canonical open immersion.
 Since $\alpha$ is a finite morphism of smooth integral varieties and since 
 $\alpha \circ j'$ is an open immersion then $\alpha$ is an isomorphism. 
 Since $\E $ is a direct factor of
 $g' _{+} g'^* (\E )$ (because $g$ is generically finite and etale and $\E$ is an isocrystal), 
 then $f _{+}(\E )$ is a direct factor of $(f \circ g ') _{+} g'^* (\E)$. 
 Hence, we can suppose
 that there exists the following commutative diagram of varieties
 \begin{equation*}
  \xymatrix@C=30pt@R=15pt{
   D\ar@{^{(}->}[r]^i\ar[dr] &\overline{X}\ar[d]_{h}&
   X\ar@{_{(}->}[l]_j\ar[dl]^f\\
  &Y&
   }
 \end{equation*}
 where $h$ is projective smooth purely of relative dimension $1$, $j$ is
 an open immersion, 
 $D := \overline{X} \setminus X $ is a strict
 normal crossing divisor of $\overline{X}$ with
 $i$ the induced closed immersion, $\E$ is $\iota$-pure of weight $w$
 having unipotent monodromy along
 $D$. Now, there exists a finite radicial extension $\mc{L}$ of
 $\mc{K}$ such that the morphism
 $D_{\mc{L}}\rightarrow\mr{Spec}(\mc{L})$ is finite \'{e}tale. Let
 $a\colon\widetilde{Y}\rightarrow Y$ be the normalization
 of $Y$ in $\mc{L}$. By 2.b), we may take the the pull-back by $a$, and
 shrinking $Y$ further, we may assume moreover that $h\circ i$ is finite
 \'{e}tale. By shrinking $Y$ once again, we may assume
 $h_+(j_{!+}(\E))\in F\text{-}D^{\mathrm{b}}_{\mathrm{isoc}}(Y/K)$.
 \medskip

 \noindent
 {\bf 5)} Let us finish the proof of (i). By 3.a), we know that
 $j_!(\E)$ is $\iota$-mixed. Thus,
 $\R\underline{\Gamma}^\dag_D(j_!(\E))$ is $\iota$-mixed as
 well, and since $\SQ{}$ holds for this $F$-complex, we get that
 $\R^1\underline{\Gamma}^\dag_D(j_!(\E))$ is $\iota$-mixed.
 Thus, we get that $j_{!+}(\E)$ is $\iota$-mixed. Moreover, by Theorem
 \ref{u+mixedstable} (ii), $j_{!+}(\E)$ is $\iota$-pure of weight $w$.
 Now, let us show that $h_+(j_{!+}(\E))$ is $\iota$-pure of weight $w$.
 For this, it suffices to show that for an
 $\iota$-pure $F$-module $\FF$ of weight $w$ on $\overline{X}$ such that
 $h_+(\FF)\in F\text{-}D ^{\mathrm{b}}_{\mathrm{isoc}}(Y/K)$, $h_+(\FF)$
 is $\iota$-pure of weight $w$.
 Let $y$ be a closed point of $Y$, and
 $i_y\colon\{y\}\hookrightarrow Y$ be the closed immersion. We put
 $i'_y\colon h^{-1}(y)\hookrightarrow\overline{X}$ the closed immersion,
 and $h'_y\colon h^{-1}(y)\rightarrow\{y\}$ the proper smooth
 curve. By base change theorem \ref{summpropre}, we have
 $i_y^!\,h_+(\FF)\cong h'_{y+}\,i'^!_y(\FF)$. Since
 $i'^!_y(\FF)$ is $\iota$-mixed of weight $\geq w$, by Proposition
 \ref{pullbackstability}, $i_y^!\,h_+(\FF)$ is $\iota$-mixed of weight
 $\geq w$. Since $h$ is proper, $h_!\cong h_+$ by paragraph
 \ref{summpropre}.
 By 3.a),
 since $\DD_{\overline{X}}(\FF)$ is $\iota$-pure of weight $-w$, we get
 that $i_y^+\,h_+(\FF)$ is $\iota$-mixed of weight $\leq w$, and we
 conclude that $h_+(\FF)$ is $\iota$-pure of weight $w$.

 Finally, consider the exact sequence 
 \begin{equation*}
  0\rightarrow j_{!+}\E\rightarrow j_{+}\E\rightarrow
   \HH^1\R\underline{\Gamma} ^\dag _{D} (j _{!} (\E))
   \to0.
 \end{equation*}
 Since $j_{!+}(\E)$ and $j_+(\E)$ are $\iota$-mixed of weight $\geq w$,
 so is $\HH^1\R\underline{\Gamma} ^\dag _{D} (j _{!} (\E))$ by
 Lemma \ref{lem-2.2.8}. By the finite \'{e}taleness of $h\circ i$,
 $h_+\bigl(\HH^1\R\underline{\Gamma}^\dag _{D} (j _{!} (\E))\bigr)$ is
 $\iota$-mixed of weight $\geq w$.
 Since we showed that $h_+(j_{!+}(\E))$ is $\iota$-pure of weight
 $w$, we get that $f_+(\E)\cong h_+(j_+(\E))$ is $\iota$-mixed of weight
 $\geq w$, which concludes the proof.
\end{proof}

\begin{empt}
 \label{sumupWeilII}
 Let us sum up the results.
 \begin{thm*}
  Let $f\colon X\to Y$ be a morphism of realizable varieties. We get:
  \begin{enumerate}
   \item $f _!$ and $f ^+$ preserve $D^{\mr{b}}_{\leq w}$;

   \item $f ^!$ and $f _+$ preserve $D^{\mr{b}}_{\geq w}$;

   \item $\widetilde{\otimes}_Y$ sends $D^{\mr{b}}_{\geq w}\times
	 D^{\mr{b}}_{\geq w'}$ to $D^{\mr{b}}_{\geq w+w'}$, and
	 $\otimes_Y$ sends $D^{\mr{b}}_{\leq w}\times
	 D^{\mr{b}}_{\leq w'}$ to $D^{\mr{b}}_{\leq w+w'}$;

   \item $\DD _{Y}$ exchanges $D^{\mr{b}}_{\leq w}$ and $D^{\mr{b}}_{\geq
	 -w}$.
  \end{enumerate}
 \end{thm*}
\end{empt}

\subsection{Purity results}
In this subsection, we prove the purity stability for intermediate
extension. With this purity, we complete the theory of weights in
$p$-adic cohomology theory. For the proof of Theorem
\ref{thm-purity-fourier}, we follow the idea of Kiehl-Weissauer
\cite{KW-Weilconjecture}.

\begin{lem}
\label{lem-purity-fourier}
 Let $X$ be a realizable variety, and $\bigl\{X_t\bigr\}_{t\in I}$ be a
 set of divisors of $X$ such that $X_t\cap X_{t'}=\emptyset$ for any
 $t\neq t'$ in $I$. Let $i_t\colon X_t\hookrightarrow X$ be the closed
 immersion. Then for any overholonomic $F$-module $\E$ over $X$, the $F$-complex
 $i^!_t\E\,[1]$ is in fact a $F$-module (or equivalently $\HH^0i^!_t\E=0$)
 for any but finitely many $t\in I$.
\end{lem}
\begin{proof}
 We have an injection $i_{t+}\HH^0(i^!_t\E)\hookrightarrow \E$. Since
 $i_{t+}\HH^0(i^!_t\E)$ is supported on $X_t$, different $t\in I$ give
 different $F$-submodules of $\E$. Since there are finitely many
 constituents of an overholonomic $F$-module, we conclude that
 $\HH^0(i^!_t\E)=0$ except for finitely many $t$ in $I$. In this case,
 since $X_t$ is assumed to be a divisor, $i^!_t\E$ is concentrated at
 degree $1$, and the lemma follows.
\end{proof}

\begin{thm}
 \label{thm-purity-fourier}
 Let $X$ be a realizable variety, and $i\colon Z\hookrightarrow X$ be a
 closed subvariety whose complement is denoted by $U$. Let $\E$
 be an overholonomic $F$-module on $X$ which is $\iota$-mixed  and such that
 $\HH^0 i^!\E=0$. If $\E\Vert_U$ is $\iota$-mixed of weight $\geq w$, so
 is $\E$.
\end{thm}

\begin{proof}
 Since the verification is local, we may assume that $X$ is affine. Take
 a closed embedding $X\hookrightarrow\mb{A} _k^n$ for some $n$. The
 assumption being stable under closed push-forward, we assume $X$ to be
 $\mb{A} _k^n$ from now on. Now, we will show the theorem using the
 induction on $n$. Assume that the theorem holds for $n<N$. 

 Let us, first, reduce to the case $\dim(Z)=0$. Choose a linear
 projection $X=\mb{A}_k^N\rightarrow T:=\mb{A} _k^1$. For $t\in|T|$, we
 denote $i_t\colon X_t:=X\times_T\{t\}\hookrightarrow X$,
 the induced closed immersion. If there is no risk of confusion, we
 denote by $i\colon Z_t:=Z\times_T\{t\}\hookrightarrow X _t$ and 
 $i_t\colon Z_t:=Z\times_T\{t\}\hookrightarrow Z$ the induced closed
 immersions. By Lemma \ref{lem-purity-fourier}, there exists a
 finite set $B\subset |T|$ of closed points  such that, for any $t \not
 \in B$, we have $\HH^0 i^!_t\E=0$. By using Lemma
 \ref{lem-purity-fourier} once again, by enlarging $B$ if necessary, we
 have $\HH^0i^!_t\,(\HH^{1}i^!\E)=0$ for any $t \not \in B$. Since the
 assumption shows that $\HH^{0}i^!\E=0$, we get $\HH^{1}(i^!_t\,i^!\E)=0$
 for any $t \not \in B$. Since $i^!_t\,i^! \riso i^!\,i^! _t$, this
 implies that $\HH^0 i^! \HH^1 i^!_t\E=0$ for any $t \not \in B$. 
 Since
 $\E\Vert_U$ is $\iota$-mixed of weight $\geq w$, then $\HH^1
 i_t^!(\E)\Vert_{U _t}$
 is $\iota$-mixed of weight $\geq w+1$, where $U _t
 :=U\times_T\{t\}$. Hence, by induction, we get that $\HH^1 i^!_t\E$  is
 $\iota$-mixed of weight $\geq w+1$, for any $t \not \in B$. By applying
 the same argument for the other projection $\mb{A}
 _k^N\rightarrow\mb{A} _k^1$, we can check there exists a closed
 subscheme $Z'\subset Z$ of dimension $0$ such that $\E$ is
 $\iota$-mixed of weight $\geq w$ on $X\setminus Z'$. By the
 transitivity and left exactness of $i^!$, we may replace $Z$ by $Z'$.

 Now, let us finish the proof by treating the case $\dim(Z)=0$. Since
 the weight do not change by extension of $K$ by Lemma
 \ref{finiteextmixok}, we may assume $\pi\in K$ to use geometric Fourier
 transform.
 We put $T _{\pi} (\E):= \ms{F}_\pi(\E) [2]$ and $\mc{J} _\pi := \mc{K}
 _\pi[2]$, which makes the description slightly simpler because the
 weight of $\mc{J}_\pi$ is $0$. Let $x\in|\mb{A} _k^N|$, and $k (x)$ be
 its residue field, $\iota_x\colon\Spec k(x)\hookrightarrow\mb{A}^N$,
 $\alpha _x\colon
 \mb{A} ^N_{x}:=\mb{A}^{N}\times_{\Spec k}\Spec
 k(x)\hookrightarrow\mb{A}^{2N}$, $i _x \colon Z _x :=Z \times_{\Spec
 k}\Spec k(x)\hookrightarrow  \mb{A} ^N_{x}$ be the closed immersions,
 and $\E _{x}:= \alpha _x ^{!} \circ p _{1} ^{!} (\E)$
 and $\rho _{x}\colon \Spec k(x) \to \Spec k$ the induced morphism.
 Since $p _{1}  \circ \alpha _x $ is finite and \'{e}tale, 
 $\HH^0 i _x^!\E _x=0$ and $\E _x$ is $\iota$-mixed of weight $\geq
 w$ outside $Z _x$. For a realizable variety $f\colon
 Y\rightarrow\mr{Spec}(k)$ and an overholonomic $F$-complex $\FF$ over
 $Y$, we denote $\HH^m(f_+(\FF))$ by $H^m(Y,\FF)$.
 By \ref{otimes-comm-u!} and the base change,
 we get:
 \begin{align}
  \label{thm-purity-fourier-iso1}
  \notag
   H^m\bigl(\mb{A} _{x}^N , \E _x \widetilde{\otimes}
   \alpha _x ^! \mc{J}_{\pi}\bigr) 
   &\riso 
   H^m\bigl(\mb{A} _{x}^N , \alpha _x ^! ( p _{1} ^{!}(\E)
   \widetilde{\otimes} \mc{J}_{\pi}) \bigr)\\
   &\riso
   \HH^m\bigl(\rho _{x,+} \circ \iota _x ^! \circ  p _{2+}( p _{1}
  ^{!}(\E) \widetilde{\otimes} \mc{J}_{\pi})\bigr)
  \cong
   \HH^m\bigl(\rho _{x,+}\circ \iota _x ^! (T _{\pi} (\E))\bigr).
 \end{align}
 Since $T _{\pi} (\E)[-N] $ is an $F$-module by
 \ref{Fourier}.\ref{Fourier1}, we get $ H^m\bigl(\mb{A} _{x}^N , \E _x
 \widetilde{\otimes} \alpha _x ^! \mc{J}_{\pi}) =0$ for any integer $m
 \not \in [-N,0]$ by \ref{thm-purity-fourier-iso1}. Moreover since $i
 ^{!} _{x} (\E _x \widetilde{\otimes} \alpha _x ^! \mc{J}_{\pi}) \riso i
 ^{!} _{x} (\E _x )\widetilde{\otimes} i ^{!} _{x}  \circ \alpha _x ^!
 \mc{J}_{\pi}$, $i ^{!} _{x}  \circ \alpha _x ^! \mc{J}_{\pi}$ is a
 $F$-module, and $\HH^0 i _x^!\E _x=0$, we get that $H^m\bigl(Z _x,i
 _x^!(\E _x\widetilde{\otimes} \alpha_x^! \mc{J}_{\pi})\bigr)=0$ for any
 integer $m \not \in [1,N]$. Now, the exact sequence
 \begin{equation*}
  \rightarrow H^m\bigl(Z _x,i _x^!(\E _x\widetilde{\otimes} \alpha_x^!
   \mc{J}_{\pi})\bigr)
   \rightarrow
   H^m\bigl(\mb{A} _{x}^N,\E_x \widetilde{\otimes} \alpha_x^!
   \mc{J}_{\pi}\bigr)
   \rightarrow
   H^m\bigl(\mb{A} _{x}^N\setminus Z _x,\E _x \widetilde{\otimes}
   \alpha_x^! \mc{J}_{\pi}\bigr)
   \rightarrow
 \end{equation*}
 implies that
 \begin{equation*}
  H^m\bigl(\mb{A} _{x}^N,\E _x\widetilde{\otimes} \alpha _x
   ^!\mc{J}_{\pi}\bigr)
   \begin{cases}
    \cong H^m\bigl(\mb{A} _{x}^N\setminus Z _x,\E_x\widetilde{\otimes}
    \alpha _x^! \mc{J}_{\pi}\bigr)
    &\mbox{for $-N\leq m<0$}\\
    \hookrightarrow H^{0}\bigl(\mb{A} _{x} ^N \setminus Z
    _x,\E_x\widetilde{\otimes}
    \alpha _x ^!\mc{J}_{\pi}\bigr)
    &\mbox{for $m=0$}\\
   = 0&\mbox{otherwise}.
   \end{cases}
 \end{equation*}
 By Theorem \ref{sumupWeilII}, $H^m\bigl(\mb{A} _{x}^N\setminus Z _x,\E
 _x\widetilde{\otimes}
 \alpha _x^!\mc{J}_{\pi}\bigr)$ is $\iota$-mixed of weight $\geq w+m$ (recall
 that $\mc{J}_\pi$ is $\iota$-pure of weight $0$), and thus from
 \ref{thm-purity-fourier-iso1}, $\iota _x ^! (T _{\pi} (\E))$ is of
 weight $\geq w$.
 Hence, $T _{\pi} (\E)$ is $\iota$-mixed of weight $\geq w$. By Theorem
 \ref{sumupWeilII} and \ref{Fourier}.\ref{Fourier2}, we note that $\FF\in
 D^{\mr{b}}_{\geq w}(\mb{A}^N)$ if and only if $T _{\pi} (\FF)\in
 D^{\mr{b}}_{\geq w}(\mb{A}^N)$, and the theorem follows.
\end{proof}

\begin{thm}
 \label{sqpropcompl}
 Let $X$ be a realizable variety, $\E \in F\text{-}D
 ^{\mathrm{b}}_{\mathrm{ovhol}}(X/K)$, and
 $\star\in\bigl\{\emptyset,\leq w,\geq w\bigr\}$. If $\E$ is
 $\iota$-mixed of weight $\star$, then $\E$ satisfies the condition
 $\SQ{\star}$.
\end{thm}
\begin{proof}
 By Remark \ref{rem-m-geqw} (i) and duality, it is sufficient to
 consider the case where $\star$ is $\geq w$.
 We prove the theorem by induction on the dimension of $X$. The case
 where $\dim(X) =0$ is obvious. If we have a distinguished triangle
 $\E'\rightarrow\E\rightarrow\E''\xrightarrow{+}$ in $F\text{-}D
 ^{\mathrm{b}}_{\mathrm{ovhol}}(X/K)$ such that $\SQ{\geq w}$ holds for
 $\E'$ and $\E''$, then $\E$ satisfies $\SQ{\geq w}$ as well. 
 Hence, by d\'{e}vissage, it suffices to show that for any smooth
 irreducible open subscheme $j\colon U\hookrightarrow X$ such that $j$
 is affine and for any irreducible overconvergent $F$-isocrystal $\E$ on
 $U$ which is $\iota$-pure of weight $w$, the $F$-module $j_+\E$
 satisfies $\SQ{\geq w}$.
 For this, let $\FF$ be an overholonomic $F$-module on $X$ defined by
 the exact sequence $0\rightarrow j_{!+}\E\rightarrow
 j_{+}\E\rightarrow\mc{F}\rightarrow0$. Let $Z:=X\setminus U$.
 By Theorem \ref{sumupWeilII}, the $F$-complex $\R \underline{\Gamma} ^\dag
 _{Z} (j _{!} (\E))$ is $\iota$-mixed.
Since 
$\R \underline{\Gamma} ^\dag _{Z} (j _{+} (\E))\bigr)=0$, one gets
 $\FF \riso \HH ^{1} \bigl(\R \underline{\Gamma} ^\dag _{Z} (j
 _{!} (\E))\bigr)$, and the induction hypothesis shows that this
 satisfies the $\SQ{}$ property and {\it a fortiori} is $\iota$-mixed.
 Hence  $ j_{!+}\E$ is $\iota$-mixed. From Theorem \ref{thm-purity-fourier}, 
 By Corollary \ref{i!+bidual} and Proposition \ref{BorelVII.10.5}, we
 get that $j_{!+}\E$ is irreducible and $\iota$-pure of weight $w$, and
 hence satisfies $\SQ{\geq w}$.
 Since  $\FF \riso \HH ^{1}\bigl(\R \underline{\Gamma} ^\dag
 _{Z} (j _{!+} (\E))\bigr)$, $\FF$ is $\iota$-mixed of weight $\geq
 w+1$. Thus, by induction hypothesis, $\FF$ satisfies $\SQ{\geq w+1}$
 and {\it a fortiori} $\SQ{\geq w}$. Thus the theorem follows.
\end{proof}

We see easily that the theorem leads to:

\begin{cor}
 \label{purityinterme}
 Let $j \colon U \hookrightarrow X$ be an open immersion of realizable
 varieties. The intermediate extension $j_{!+}$ sends an $\iota$-pure
 $F$-module on $U$ of weight $w$ to an $\iota$-pure $F$-module on $X$ of
 weight $w$.
\end{cor}

\subsection{Applications}
This subsection is devoted to exhibit a few applications of the
theory of weights. Let $*$ be either $\D^\dag_{X,\Q}$ or
$F\text{-}\D^\dag_{X,\Q}$. We denote by $\mr{Hom}_*$ and
$\mr{Ext}^1_{*}$ the Hom group and Yoneda's Ext group of the abelian
category of overholonomic $*$-modules respectively.

\begin{thm}[Semi-simplicity of pure $F$-modules]
 Let $X$ be a realizable variety, and $\E$ be an $\iota$-pure $F$-module
 in $F\text{-}\mr{Ovhol}(X/K)$. Then $\E$ is semi-simple in
 $\mr{Ovhol}(X/K)$ (not in $F\text{-}\mr{Ovhol}(X/K)$).
\end{thm}
\begin{proof}
 Let $\FF\subset \E$ be the maximal semi-simple submodule
 in $\mr{Ovhol}(X/K)$. Then we see that $\FF$ is stable under the
 Frobenius structure of $\E$, and the inclusion $\FF\subset\E$ is in
 fact in $F\text{-}\mr{Ovhol}(X/K)$. By Theorem \ref{sqpropcompl}, $\FF$
 and $\E/\FF$ are $\iota$-pure as well. This gives us
 $[\E]\in\mr{Ext}^1_{F\text{-}\DdagQ{X}}(\E/\FF,\FF)$. Now, by using
 Proposition \ref{AppendixB3}, Theorem \ref{sumupWeilII},
 and Remark \ref{Extiscoh}, we have
 $\mr{Ext}^1_{\DdagQ{X}}(\E/\FF,\FF)^F=0$. Since the canonical
 homomorphism
 $\alpha\colon\mr{Ext}^1_{F\text{-}\DdagQ{X}}(\E/\FF,\FF)
 \rightarrow\mr{Ext}^1_{\DdagQ{X}}(\E/\FF,\FF)$ factors through
 $\mr{Ext}^1_{\DdagQ{X}}(\E/\FF,\FF)^F\rightarrow
 \mr{Ext}^1_{\DdagQ{X}}(\E/\FF,\FF)$, we obtain $\alpha([\E])=0$, which
 concludes the proof.
\end{proof}

\begin{lem}
 \label{basiclemonExt}
 Let $X$ be a realizable variety, and $\E,\FF$ be objects in
 $F\text{-}\mr{Ovhol}(X/K)$. Then, we have the following short exact
 sequence:
 \begin{equation*}
  0\rightarrow\mr{Hom}_{\DdagQ{X}}(\E,\FF)_F\rightarrow
   \mr{Ext}^1_{F\text{-}\DdagQ{X}}(\E,\FF)\rightarrow
   \mr{Ext}^1_{\DdagQ{X}}(\E,\FF)^F\rightarrow0.
 \end{equation*}
\end{lem}
\begin{proof}
 Let $0\rightarrow\FF\xrightarrow{\alpha}\mc{G}
 \xrightarrow{\beta}\E\rightarrow0$ be an
 extension of $\DdagQ{X}$-modules, and denote by
 $[\mc{G}]\in\mr{Ext}^1_{\DdagQ{X}}(\E,\FF)$ the class defined by
 the extension. Let $\phi_{\FF}\colon\FF\rightarrow\FF$ be the Frobenius
 structure of $\FF$, and the same for $\phi_{\E}$. Then the class
 $F^*[\mc{G}]$ can be written as
 \begin{equation*}
  0\rightarrow\FF\xrightarrow{\alpha\circ \phi_{\FF}}
   \mc{G}\xrightarrow{\phi_{\E}^{-1}\circ\beta}\E
   \rightarrow0.
 \end{equation*}
 The condition $F^*[\mc{G}]=[\mc{G}]$ means that there exists a dotted
 arrow making the diagram commutative:
 \begin{equation*}
  \xymatrix{
   0\ar[r]&\FF\ar[d]^\sim_{\phi_{\FF}}\ar[r]&\mc{G}\ar[r]\ar@{.>}[d]&
   \E\ar[r]\ar[d]^{\phi_\E}_\sim&0\\
   0\ar[r]&\FF\ar[r]&\mc{G}\ar[r]&\E\ar[r]&0.
   }
 \end{equation*}
 Thus the surjectivity of the homomorphism
 $\mr{Ext}^1_{F\text{-}\DdagQ{X}}(\E,\FF)\rightarrow
   \mr{Ext}^1_{\DdagQ{X}}(\E,\FF)^F$ follows. Now, assume that
 the image of $[(\mc{G},\phi_{\mc{G}})]$ is $0$. This means that
 $\mc{G}$ is split. Then the computation of the kernel is
 standard. (See \cite[proof of 5.1.2]{BBD} for example.)
\end{proof}

\begin{lem}
 \label{basicvanExt}
 Let $X$ be a realizable variety, and $\E,\FF$ be objects in
 $F\text{-}\mr{Ovhol}(X/K)$. Assume that $\E$, $\FF$ are irreducible
 $\iota$-pure $F$-modules such that $\mr{wt}(\FF)>\mr{wt}(\E)$. Then
 $\mr{Ext}^1_{F\text{-}\DdagQ{X}}(\E,\FF)=0$.
\end{lem}

\begin{proof}
 It suffices to show that the two outer modules appearing in the short
 exact sequence of Lemma \ref{basiclemonExt} vanish. By Proposition
 \ref{AppendixB3} and Theorem \ref{sumupWeilII}, we get the claim.
\end{proof}

\begin{thm}[Weight filtration]
 Let $X$  be a realizable variety, and $\E$ be an object of
 $F\text{-}\mr{Ovhol}(X/K)$ which is $\iota$-mixed. Then there exists a
 unique increasing filtration $W$ on $\E$ such that $\gr_i^W(\E)$ is
 purely of $\iota$-weight $i$. Any homomorphism of 
 $\iota$-mixed $F$-modules
 on $X$ is strictly compatible with the filtration.
\end{thm}
\begin{proof}
Using Lemma \ref{basicvanExt}, we can follow the proof of
 \cite[5.3.6]{BBD}.
\end{proof}

\begin{empt}
 Let $X$ be a realizable variety, and $\E, \FF \in F\text{-}D
 ^{\mathrm{b}}_{\mathrm{ovhol}}(X/K)$. By definition, we have the
 equality:
 \begin{equation*}
  \mathrm{Hom} _{F\text{-}D
   ^{\mathrm{b}}_{\mathrm{ovhol}}(X/K)} (\E, \FF)=
   \mathrm{Hom} _{D ^{\mathrm{b}}_{\mathrm{ovhol}}(X/K)} (\E,
   \FF) ^{F}.
 \end{equation*}
 By Theorem \ref{sumupWeilII} and Proposition \ref{AppendixB3}, we get
 that if $\E\in F\text{-}D ^{\mathrm{b}}_{\leq w }(X/K)$ and $\FF \in
 F\text{-}D ^{\mathrm{b}}_{> w }(X/K)$ then 
 \begin{equation}
  \label{AppendixB4}
   \mathrm{Hom} _{F\text{-}D
   ^{\mathrm{b}}_{\mathrm{ovhol}}(X/K)} (\E, \FF) =0.
 \end{equation}
\end{empt}

\begin{thm}[Semi-simplicity]
 \label{ssforcompl}
 Let $X$ be a realizable variety, and $\E$ be an $\iota$-pure overholonomic $F$-complex
 on $X$. Then $\E$ is isomorphic, in $D
 ^{\mathrm{b}}_{\mathrm{ovhol}}(X/K)$, to
 $\bigoplus_{n\in\mb{Z}}\HH^n(\E)[n]$.
\end{thm}
\begin{proof}
Using \ref{AppendixB4}, the proof is essentially the same as that of
 \cite[5.4.5, 5.4.6]{BBD}.
\end{proof}

\begin{rem}
 For a scheme $X$ over a field, denote by $D^{\mr{b}}_{\mr{c}}(X)$ the
 bounded derived category of constructible
 $\overline{\mb{Q}}_{\ell}$-complexes where $\ell$ is a prime number
 different from $p$. Let $X$ be a variety over the
 finite field $k$. As pointed out in Remark \ref{prop-nota-Dsurhol}, the
 category $(F\text{-})D^{\mr{b}}_{\mr{ovhol}}(X/K)$ is an analogue of the
 category of Weil complexes
 $(F\text{-})D^{\mr{b}}_{\mr{c}}(X\otimes_k\overline{k})$, and should not
 be regarded as an analogue of the derived category
 $D^{\mr{b}}_{c}(X)$.
 In the $\ell$-adic setting \cite[5.4.5]{BBD}, the corresponding theorem
 to Theorem \ref{ssforcompl} is stated for complexes in
 $D^{\mr{b}}_{\mr{c}}(X)$, so there are slight differences in the
 formulation arising from the lack of the category
 corresponding to $D^{\mr{b}}_{\mr{c}}(X)$ in our theory. However, since
 there exists the factorization
 \begin{equation*}
  D^{\mr{b}}_{\mr{c}}(X)\rightarrow
   F\text{-}D^{\mr{b}}_{\mr{c}}(X\otimes_k\overline{k})\rightarrow
   D^{\mr{b}}_{\mr{c}}(X\otimes_k\overline{k}),
 \end{equation*}
 if the true derived category corresponding to $D^{\mr{b}}_{\mr{c}}(X)$
 in the $p$-adic cohomology theory is constructed, Theorem
 \ref{ssforcompl} should not lose any information.
\end{rem}

\newcommand{\cc}[1]{[\![#1]\!]}

\begin{empt}
 We conclude this paper by showing an application to $\ell$-independence
 results of $L$-functions. To do this, we need to introduce, for a
 realizable variety $X$ over a finite field $k$, the category
 $F\text{-}D^{\mr{b}}_{\mr{ovhol}}(X/\overline{\mb{Q}}_p)$. The
 construction is essentially the same as \cite[7.3]{AM}. Let $L$ be a
 finite extension of the maximal unramified extension $K^{\mr{ur}}$. An
 object of $F\text{-}D^{\mr{b}}_{\mr{ovhol}}(X/L)$ is a pair
 $(\E,\rho)$ where $\E\in
 F\text{-}D^{\mr{b}}_{\mr{ovhol}}(X/K^{\mr{ur}})$ and $\rho\colon
 L\rightarrow\mr{End}_{F\text{-}D^{\mr{b}}_{\mr{ovhol}}
 (X/K^{\mr{ur}})}(\E)$
 is a homomorphism of $K^{\mr{ur}}$-algebras. By taking the cohomology
 functor $\HH^*$, $\rho$ induces a homomorphism
 $\HH^*(\rho)\colon L\rightarrow\mr{End}_{F\text{-}\mr{Ovhol}
 (X/K^{\mr{ur}})}(\HH^*(\E))$. The ``heart'' is denoted by
 $F\text{-}\mr{Ovhol}(X/L)$.

 For finite extensions $L'/L/K^{\mr{ur}}$, there exists a scalar
 extension functor $F\text{-}D^{\mr{b}}_{\mr{ovhol}}(X/L)\rightarrow
 F\text{-}D^{\mr{b}}_{\mr{ovhol}}(X/L')$. This is ``t-exact'' in the
 obvious sense, and induces a functor
 $F\text{-}\mr{Ovhol}(X/L)\rightarrow F\text{-}\mr{Ovhol}(X/L')$. By
 taking the limit over finite extensions $L$ of $K^{\mr{ur}}$, we have
 the categories
 $F\text{-}D^{\mr{b}}_{\mr{ovhol}}(X/\overline{\mb{Q}}_p)$
 and
 $F\text{-}\mr{Ovhol}(X/\overline{\mb{Q}}_p)$. We notice that
 $F\text{-}\mr{Ovhol}(\mr{Spec}(k)/\overline{\mb{Q}}_p)$ is nothing but
 the category of pairs $(V,\phi)$ where $V$ is a finite dimensional
 $\overline{\mb{Q}}_p$-vector space, and $\phi$ is an automorphism of
 $V$.

 Since $\iota$-weights remain to be the same after scalar extension,
 Theorem \ref{sumupWeilII}, Theorem \ref{sqpropcompl}, Corollary
 \ref{purityinterme} remain to be true even after replacing $K$ by
 $\overline{\mb{Q}}_p$.
\end{empt}

\begin{rem*}
 Just to state Theorem \ref{FujGabbind}, we do not need to introduce
 this generalized category. However, in the proof, we need to extend the
 scalar to $\overline{\mb{Q}}_p$, in order to produce sufficiently many
 sheaves $\mc{L}_{\rho}$ using the notation of \cite[\S3]{Fujiwara-Gabber}.
\end{rem*}

\begin{empt}
 Let $X$ be a realizable variety over $k$. Let $\E$ be an object in
 $F\text{-}\mr{Ovhol}(X/\overline{\mb{Q}}_p)$. For a closed point $x$ of
 $X$, let $\rho_x\colon\mr{Spec}(k(x))\rightarrow\mr{Spec}(k)$ be
 the structural morphism. We define the {\em local $L$-factor at $x$} to
 be
 \begin{equation*}
  L_x(\E,t):=\det_{\overline{\mb{Q}}_p}\bigl
   (1-t^{\deg(x)}\cdot F;\, \rho_{x+}\circ
   i_x^+(\E)\bigr)^{-1/\deg(x)},
 \end{equation*}
 and the {\em global $L$-function}\footnote{
 There are several ways of normalizations of global $L$-function. Ours
 is the same as that of \cite[7.2.3]{AM}.} to be
 \begin{equation*}
  L(X,\E,t):=\prod_{x\in |X|}L_x(\E,t).
 \end{equation*}
 Let $K(X,\overline{\mb{Q}}_p)$ be the Grothendieck category of
 $F\text{-}\mr{Ovhol}(X/\overline{\mb{Q}}_p)$. By additivity, the
 definition of local $L$-factor extends to a homomorphism
 \begin{equation*}
  K(X,\overline{\mb{Q}}_p)\rightarrow\prod_{x\in|X|}
   \bigl(1+t\cdot\overline{\mb{Q}}_p\cc{t}\bigr)^{\times}.
 \end{equation*}
\end{empt}

\begin{dfn}
 Let $\ell$ be a prime number different from $p$, and $E$ be a field of
 characteristic $0$. We fix embeddings
 $\overline{\mb{Q}}_p\hookleftarrow E\hookrightarrow
 \overline{\mb{Q}}_\ell$, where the first one is denoted by $\iota_p$
 and the second by $\iota_\ell$. A pair
 $(\E,\mathscr{F})\in K(X,\overline{\mb{Q}}_p)\times
 K(X,\overline{\mb{Q}}_\ell)$ is said to be an {\em $E$-compatible
 system} if for any closed point $x$ of $X$, we have
 \begin{equation*}
  \iota_p\bigl(L_x(\E,t)\bigr)=
   \iota_{\ell}\bigl(L_x(\mathscr{F},t)\bigr)\in E\cc{t}.
 \end{equation*}
\end{dfn}

\begin{thm}[Fujiwara-Gabber's $\ell$-independence]
 \label{FujGabbind}
 Compatible systems are stable under $f_+$, $f^!$, $\DD$,
 $\otimes$. Moreover, let $j\colon U\hookrightarrow X$ be an immersion
 of realizable varieties, and $(\E,\mathscr{F})$ be an $E$-compatible
 system such that $\E$ (and hence $\mathscr{F}$) is $\iota_p$-pure. Then
 $(j_{!+}(\E),j_{!*}(\mathscr{F}))$ is also an $E$-compatible system.
\end{thm}
\begin{proof}
 The proof is exactly the same as \cite{Fujiwara-Gabber}. In the proof,
 we need to use Lefschetz fixed point theorem.
 Since we are using ``the extended scalar'', we need to check the
 theorem in this situation. For the verification, it is reduced to the
 overconvergent $F$-isocrystal case by the same argument as
 \cite{caro_devissge_surcoh}. We repeat the d\'{e}vissage argument of
 \cite[6.2]{E-LS} to reduce to the case where the fixed point is empty,
 and in this case we repeat \cite[5.4]{E-LS}. We leave the detailed
 verifications to the reader.
\end{proof}

\begin{empt}
 Let $\mc{K}$ be a function field over $k$ (i.e. a field of transcendental degree $1$ over $k$). We define
 \begin{equation*}
  F\text{-}\mr{Isoc}^{\dag}(\mc{K}/K):=
   \indlim_{U}\,F\text{-}\mr{Isoc}^{\dag}(U/K)
 \end{equation*}
 where the inductive limit runs over smooth curves over $k$ whose
 function field is $\mc{K}$.

 Now, take $E$ in $F\text{-}\mr{Isoc}^\dag(\mc{K}/K)$. By definition,
 there exists a smooth curve $U$ whose function field is $\mc{K}$, and
 $\widetilde{E}$ be an overconvergent $F$-isocrystal on $U$. Let
 $j\colon U\hookrightarrow X$ be the smooth compactification. We define the
 {\em Hasse-Weil $L$-function} by
 \begin{equation*}
  L_{\mr{HW}}(\mc{K},E,t):=L\bigl(X,
   j_{!+}(\mr{sp}_+(\widetilde{E})),t\bigr)
   \in E\cc{t}.
 \end{equation*}
From \ref{transitivityi!+},
we check that this Hasse-Weil $L$-function does not depend on the choice of
 $U$ and of the smooth compactification $X$ of $U$. 
 By Theorem \ref{compcrewloc} (iii), we may write the $L$-function
 as
 \begin{equation*}
  L_{\mr{HW}}(\mc{K},E,t)
   =\prod_{x\in|X|}\det_K\bigl(1-\varphi_x\cdot
   t^{\deg(x)};\Psi_x(\widetilde{E})^{N=0,I=1}\bigr)^{-1/\deg(x)},
 \end{equation*}
 where we used the notation of paragraph \ref{crewfini}.
 In particular, our Hasse-Weil $L$-function is nothing but
 $L_{\mr{WD}}(X,\widetilde{E}(1),t)$ in \cite[4.3.10]{Marmora-fact-eps}
 by Theorem \ref{compcrewloc} (i).

 On the other hand, by Lefschetz trace formula and self-duality
 \ref{i!+bidual}, we get the functional equation
 \begin{equation*}
  L_{\mr{HW}}(\mc{K},E,t)=\varepsilon(\mc{K},E)\cdot t^{g(\mc{K})-2}
   \cdot L_{\mr{HW}}(\mc{K},E^{\vee},q\cdot t^{-1})
 \end{equation*}
 where $g(\mc{K})$ denotes the genus of the function field $\mc{K}$, and
 $\varepsilon(\mc{K},E)$ is some number in $\overline{\mb{Q}}_p$ (cf.\
 \cite[7.2.3]{AM}). This is a generalization of
 \cite[4.3.11]{Marmora-fact-eps} to $F$-isocrystals of arbitrary ranks.

 Now, let $X$ be a realizable variety over $\mc{K}$. Then there exists a
 model $f\colon\X\rightarrow U$ where $U$ is a smooth curve whose function
 field is $\mc{K}$, $\X$ is a realizable variety, and the generic fiber
 of $f$ is $X$. We denote by $p\colon\X\rightarrow\mr{Spec}(k)$ the
 structural morphism. We define the cohomology $\H^m(X/K)$ in
 $F\text{-}\mr{Isoc}^{\dag}(\mc{K}/K)$ to be the one represented by
 $\HH^m\bigl(f_+p^+(K)\bigr)$.
 We can check easily
 that the construction does not depend on auxiliary choices.
 We end this paper with the following $p$-adic interpretation of
 Hasse-Weil $L$-function over function fields, which follows directly
 from Theorem \ref{FujGabbind}:
 \begin{cor}
  \label{anstoTri}
  Let $f\colon X\rightarrow\mr{Spec}(\mc{K})$ be a realizable proper
  smooth scheme (e.g.\ projective smooth variety). Then we
  have
  \begin{equation*}
   L_{\mr{HW}}(\mc{K},H^m(\overline{X}),t)=
    L_{\mr{HW}}(\mc{K},\H^m(X/K),t),
  \end{equation*}
  where the left hand side is the Hasse-Weil $L$-function associated to
  the $\ell$-adic representation
  $H^m(\overline{X}):=H^m(X\otimes_{\mc{K}}\mc{K}^{\mr{sep}},\mb{Q}_\ell)$
  of $\mr{Gal}(\mc{K}^{\mr{sep}}/\mc{K})$.
 \end{cor}
\end{empt}

\numberwithin{prop}{section}
\appendix

\section{Internal homomorphism}
In this appendix, we consider situation (A) in Notation and convention.

\begin{dfn}
 Let $\mathbb{Y}$ be a couple (\ref{defi-(d)plongprop}).
 Let $\E, \FF,\G \in F\text{-}D
 ^{\mathrm{b}}_{\mathrm{ovhol}}(\mathbb{Y}/K)$. We put
 \begin{equation*}
  \label{dfnHom}
   \shom_{\mb{Y}}( \E, \FF) := \DD_{\mb{Y}} ( \E )
   \widetilde{\otimes}_{\mb{Y}}\FF.
 \end{equation*}
\end{dfn}
 From the biduality isomorphism, we get the isomorphisms
 \begin{equation}
  \label{Homdual}
  \shom_{\mb{Y}} ( \E, \FF) \riso
  \shom_{\mb{Y}}\bigl( \DD_{\mb{Y}} ( \FF ), \DD_{\mb{Y}} ( \E
  )\bigr),\qquad
  \shom_{\mb{Y}}\bigl( \E, \shom_{\mb{Y}} ( \FF, \G)\bigr)\riso
  \shom_{\mb{Y}}( \E \otimes \FF,\G).
 \end{equation}
 For any morphism $u \colon \mathbb{Y}' \to \mathbb{Y}$ of couples, by
 the biduality isomorphism and \ref{otimes-comm-u!}, we have
 \begin{equation*}
  \label{Homu!u*}
   u ^{!}\,\shom_{\mb{Y}} ( \E, \FF) \riso
   \shom_{\mb{Y}'}\bigl( u ^{+}(\E),u ^{!}(\FF)\bigr).
 \end{equation*}

\begin{lem}
 \label{lem-B1}
 Let $f\colon P\to S$ be a smooth morphism of noetherian schemes such
 that $p$ is nilpotent in $\mc{O}_S$. For any  $\E ^{(m)} \in D
 ^{\mathrm{b}}_{\mathrm{coh}} (\smash{\D} _{P/S} ^{(m)})$ and $\FF
 ^{(m)} \in D ^{\mathrm{b}}  (\smash{\D} _{P/S} ^{(m)})$, we have the
 canonical isomorphism:
 \begin{equation*}
  f _{+} ^{(m)} \bigl( \DD ^{ (m)} _{P/S}(\E ^{(m)})\otimes ^{\L} _{\O
   _{P} }\FF ^{(m)}\bigr)\,[-d _{P}]
   \riso
   \R f _{*}\bigl(\R \mathcal{H}om _{\smash{\D} _{P/S} ^{(m)}} (\E
   ^{(m)}, \FF ^{(m)})\bigr).
 \end{equation*}
\end{lem}

\begin{proof}
This is the composition of the following isomorphisms:
\begin{gather}
\notag
f _{+} ^{(m)} ( \DD ^{ (m)} _{P/S} (\E ^{(m)})
\otimes ^{\L} _{\O _{P} }
\FF ^{(m)})[-d _{P}]
\xrightarrow[\star]{\sim}
f _{+} ^{(m)} ( 
\R \mathcal{H}om _{\smash{\D} _{P/S} ^{(m)}} (\E ^{(m)}, (\smash{\D} _{P/S} ^{(m)} \otimes _{\O _P} \omega ^{-1} _{P/S})
\otimes ^{\L} _{\O _{P} }
\FF ^{(m)})[-d _{P}])
\\
\notag
=
\R f _{*} (\omega _{P/S} \otimes _{\smash{\D} _{P/S} ^{(m)}} ^{\L}
\R \mathcal{H}om _{\smash{\D} _{P/S} ^{(m)}} (\E ^{(m)}, (\smash{\D} _{P/S} ^{(m)} \otimes _{\O _P} \omega ^{-1} _{P/S})
\otimes ^{\L} _{\O _{P} }
\FF ^{(m)})
)
\\
\notag
\xrightarrow[\star\star]{\sim}
\R f _{*} (\R \mathcal{H}om _{\smash{\D} _{P/S} ^{(m)}} (\E ^{(m)}, 
\omega _{P/S} \otimes _{\smash{\D} _{P/S} ^{(m)}} ^{\L} (\smash{\D} _{P/S} ^{(m)} \otimes _{\O _P} \omega ^{-1} _{P/S})
\otimes ^{\L} _{\O _{P} }
\FF ^{(m)})
)
\\
\notag
\riso
\R f _{*} (\R \mathcal{H}om _{\smash{\D} _{P/S} ^{(m)}} (\E ^{(m)}, \FF ^{(m)})).
\end{gather}
 where $\star$ follows by \cite[2.1.26]{caro_comparaison}, and
 $\star\star$ follows by \cite[2.1.19]{caro_comparaison}.
\end{proof}

\begin{dfn}
\label{Rhom-Ext1-dfn}
 Let $\mathbb{Y}$ be a couple. Choose an l.p.\ frame $(Y,X,\PP,\QQ)$ of
 $\mb{Y}$. Let $\E, \FF \in
 D^{\mathrm{b}}_{\mathrm{ovhol}}(\mathbb{Y}/K)$ and $\E _{\PP}, \FF
 _{\PP}$ be the corresponding objects in $D ^{\mathrm{b}}
 _{\mathrm{ovhol}}(Y, \PP/K)$. Then, using Lemma
 \ref{t-gen-coh-PXTindtPsurhol}, we notice that the complex of abelian
 groups $\R \mathrm{Hom} _{D   (\D ^\dag_{\PP,\Q})} (\E, \FF)$
 does not depend on the choice of the l.p.\ frame $(Y,X,\PP,\QQ)$ of
 $\mb{Y}$. We put
 \begin{equation*}
 \R \mathrm{Hom}_{D ^{\mathrm{b}}_{\mathrm{ovhol}}(\mathbb{Y}/K)} (\E,   \FF)
  :=\R \mathrm{Hom} _{D   (\D ^\dag_{\PP,\Q})} (\E _{\PP}, \FF_{\PP}).
 \end{equation*}
\end{dfn}

\begin{rem}
 \label{Extiscoh}
 Let $\E$, $\FF$ be objects of $\mathrm{Ovhol}(\mathbb{Y}/K)$.
 Then we can check easily that
 \begin{equation*}
  \mathcal{H} ^{1}  \R \mathrm{Hom}_{D ^{\mathrm{b}}_{\mathrm{ovhol}}
 (\mathbb{Y}/K)} (\E,\FF)=\mathrm{Ext} ^1 _{ \D ^{\dag}
 _{\mathbb{Y}/K}} (\E,   \FF),
 \end{equation*}
 where $\mathrm{Ext} ^1 _{ \D ^{\dag}_{\mathbb{Y}/K}} (\E,   \FF)$
 denotes the Yoneda's Ext group in the abelian category
 $\mathrm{Ovhol}(\mathbb{Y}/K)$.
\end{rem}

\begin{prop}
 \label{AppendixB3}
 Let $\mathbb{Y}$ be a couple and
 $a\colon \mathbb{Y} \to \Spec(k)$ the structural morphism which we
 assume to be complete. For any $\E, \FF \in
 F\text{-}D^{\mathrm{b}}_{\mathrm{ovhol}}(\mathbb{Y}/K)$, we have the
 isomorphism
 \begin{equation*}
  a _{+}\, \shom_{\mb{Y}} ( \E, \FF)\riso 
   \R \mathrm{Hom}_{D ^{\mathrm{b}}_{\mathrm{ovhol}}(\mathbb{Y}/K)} (\E,
   \FF).
 \end{equation*}
\end{prop}
\begin{proof}
 Choose an l.p.\ frame $(Y,X,\PP,\QQ)$ of $\mb{Y}$.
 Let $\E, \FF \in F\text{-}D ^{\mathrm{b}}_{\mathrm{ovhol}}(Y, \PP/K)$. 
 Let $f$ be the structural morphism of $\PP$ and  $Z := X \setminus Y$. 
 By definition, we get:
 \begin{gather*}
  a _{+} \shom_{\mb{Y}} ( \E, \FF) \cong
  f _+ \bigl( (\hdag Z) \circ \DD _{\PP} (\E)
  \smash{\overset{\L}{\otimes}}   ^{\dag}
  _{\O  _{\PP,\Q}} \FF\bigr)\,[-d _{P}]\riso 
  f _+  \bigl(\DD _{\PP} (\E)  \smash{\overset{\L}{\otimes}}   ^{\dag}
  _{\O  _{\PP,\Q}} \FF\bigr)\,[-d _{P}],\\
  \R \mathrm{Hom} _{D ^{\mathrm{b}}_{\mathrm{ovhol}}(\mathbb{Y}/K)} (\E,
  \FF)\cong\R \mathrm{Hom} _{D   (\D ^\dag_{\PP,\Q})} (\E, \FF).
 \end{gather*}
 Thus, we are reduced to the case where $Y=X=P$. Let $\E ^{(\bullet)},
 \FF ^{(\bullet)}\in \smash{\underrightarrow{LD}} ^{\mathrm{b}}
 _{\Q,\mathrm{coh}} ( \smash{\widehat{\D}} _{\PP} ^{(\bullet)})$ such
 that $\underrightarrow{\lim} ~ \E ^{(\bullet)} \riso \E$
 and
 $\underrightarrow{\lim} ~ \FF ^{(\bullet)} \riso \FF$.
 We can suppose that there exists an increasing map $\lambda \colon \N
 \to \N$  such that for any integers $m' \geq m$ we have $\lambda (m )
 \geq m$, $\E ^{(m)} \in D ^{\mathrm{b}} _{\mathrm{coh}}
 (\smash{\widehat{\D}} _{\PP} ^{(\lambda (m))})$ and the canonical
 morphism
 \begin{equation*}
  \smash{\widehat{\D}} _{\PP} ^{(\lambda (m'))} \otimes ^\L
   _{\smash{\widehat{\D}} _{\PP} ^{(\lambda (m))}}
   \E ^{(m)} \to \E ^{(m')}
 \end{equation*}
 is an isomorphism. By taking projective limits, we deduce from Lemma
 \ref{lem-B1} the isomorphism
 \begin{equation}
  \tag{($\star$)}
  \label{levmscrHomHom}
   f _{+} ^{(\lambda (m))} ( \DD ^{\lambda (m)} (\E ^{(m)})
   \widehat{\otimes} ^{\L} _{\O _{\PP} }
   \FF ^{(\lambda (m))})[-d _{P}]
   \riso 
   \R \mathrm{Hom} _{D   (\smash{\widehat{\D}} _{\PP} ^{(\lambda (m))})} (\E ^{(m)}, \FF^{(\lambda (m))}) .
 \end{equation}
 Since $\E ^{(m)} \in D ^{\mathrm{b}} _{\mathrm{coh}}
 (\smash{\widehat{\D}} _{\PP} ^{(\lambda (m))})$, we get the
 isomorphisms $\smash{\widehat{\D}} _{\PP} ^{(\lambda (m'))} \otimes ^\L
 _{\smash{\widehat{\D}} _{\PP} ^{(\lambda (m))}}\DD ^{\lambda (m)} (\E
 ^{(m)})\to \DD ^{\lambda (m')} (\E ^{(m')})$ and $\smash{\widehat{\D}}
 _{\PP} ^{(\lambda (m'))} \otimes ^\L _{\smash{\widehat{\D}} _{\PP}
 ^{(\lambda (m))}}\R \mathrm{Hom} _{D   (\smash{\widehat{\D}} _{\PP}
 ^{(\lambda (m))})} (\E ^{(m)}, \FF^{(\lambda (m'))}) \riso \R
 \mathrm{Hom} _{D   (\smash{\widehat{\D}} _{\PP} ^{(\lambda (m'))})} (\E
 ^{(m')}, \FF^{(\lambda (m'))})$. By taking the direct limit to
 \ref{levmscrHomHom} and tensoring by $\Q$, we get the desired
 isomorphism.
\end{proof}

\begin{prop}[Projection formulas]
 Let $u \colon \mathbb{Y} \to \mathbb{Y}'$ a complete morphism of
 couples. For any $\E \in F\text{-}D
 ^{\mathrm{b}}_{\mathrm{ovhol}}(\mathbb{Y}/K)$, $\E '\in F\text{-}D
 ^{\mathrm{b}}_{\mathrm{ovhol}}(\mathbb{Y} '/K)$,
 \begin{equation}
  \label{Kunneth1}
   u _+ \bigl(  \E \widetilde{\otimes} u ^{!}(\E ')\bigr)\riso
   u _{+} ( \E ) \widetilde{\otimes} \E ',\qquad
   u _! \bigl(  \E \otimes u ^{+}(\E ')\bigr)
   \riso
   u _{!} ( \E ) \otimes \E '.
 \end{equation}
\end{prop}
\begin{proof}
 Let $u = (\star,\star,g,\star)\colon  (Y,X,\PP,\QQ)\to
 (Y',X',\PP',\QQ')$ be a morphism of frames for $u$. We put $Z := X
 \setminus Y$ and $Z' := X '\setminus Y'$.
 Let $\E \in F\text{-}D ^{\mathrm{b}}_{\mathrm{ovhol}}(Y, \PP/K)$
 and $\E '\in F\text{-}D ^{\mathrm{b}}_{\mathrm{ovhol}}(Y', \PP'/K)$.
 Then we have
 \begin{equation*}
  u _+\bigl(\E \widetilde{\otimes} u ^{!}(\E ')\bigr)
   \riso
   g _+  \bigl( \E   \smash{\overset{\L}{\otimes}}^{\dag}
   _{\O  _{\PP,\Q}} g ^{!} (\E ')\bigr)\,[-d _{P}],\qquad
   u _{+} ( \E ) \widetilde{\otimes} \E '
   \riso 
   g _+  ( \E   \smash{\overset{\L}{\otimes}}   ^{\dag}
   _{\O  _{\PP',\Q}} \E ')[-d _{P'}].
 \end{equation*}
 Hence, from the projection formula \cite[2.1.4]{caro_surcoherent}, we
 get the first one. By using the biduality isomorphism, this implies
 the second one.
\end{proof}

\begin{prop}[Adjointness Properties]
 \label{adjnessProp}
 Let $u \colon \mathbb{Y} \to \mathbb{Y}'$ a morphism of couples. For
 any $\E \in F\text{-}D ^{\mathrm{b}}_{\mathrm{ovhol}}(\mathbb{Y}/K)$,
 $\E '\in F\text{-}D ^{\mathrm{b}}_{\mathrm{ovhol}}(\mathbb{Y} '/K)$, we
 get the functorial isomorphism:
 \begin{equation*}
  u _{+} \shom_{\mb{Y}}\bigl( \E, u ^{!} (\E ')\bigr)
   \riso
   \shom_{\mb{Y}'} \bigl( u _{!}(\E),\E' \bigr).
 \end{equation*}
\end{prop}
\begin{proof}
 We have
 \begin{equation*}
  u _{+} \shom_{\mb{Y}}\bigl( \E,u ^{!} (\E ')\bigr)
   =
   u _+ \bigl( \DD_{\mb{Y}} ( \E ) \widetilde{\otimes} u ^{!}(\E
   ')\bigr)
   \xrightarrow[\star]{\sim}
   (u _{+} \circ  \DD_{\mb{Y}}) ( \E ) \widetilde{\otimes} \E '
   \riso \shom_{\mb{Y}'}\bigl( u _{!}(\E),\E' \bigr).
 \end{equation*}
 where $\star$ follows by \ref{Kunneth1}.
\end{proof}

\begin{prop}
Let $u \colon \mathbb{Y} \to \mathbb{Y}'$ a morphism of couples. 
For any $\E \in F\text{-}D ^{\mathrm{b}}_{\mathrm{ovhol}}(\mathbb{Y}/K)$,
$\E '\in F\text{-}D ^{\mathrm{b}}_{\mathrm{ovhol}}(\mathbb{Y} '/K)$, 
we get the functorial isomorphism:
 \begin{equation*}
  u _{+} \shom_{\mb{Y}} \bigl(u ^{+} (\E '), \E\bigr)
   \riso
   \shom_{\mb{Y}'} \bigl(\E',  u _{+}(\E) \bigr).
 \end{equation*}
\end{prop}

\begin{proof}
This is a consequence of \ref{Homdual} and Proposition
 \ref{adjnessProp}.
\end{proof}

\bibliographystyle{alpha}

\begin{thebibliography}{BGK{\etalchar{+}}87}


\bibitem[AM11]{AM}
{\scshape T.~Abe {\normalfont \smfandname} A.~Marmora} -- {\og Product formula
  for $p$-adic epsilon factors\fg},  J. Inst. Math. Jussieu (2014).

\bibitem[Abe11]{Abe-Langlands}
{\scshape T.~Abe} -- {\og Langlands program for $p$-adic coefficients and the
  petits camarades conjecture\fg},  (2011).

\bibitem[Abe13]{Abe-Frob-Poincare-dual}
{\scshape T.~Abe} -- {\og {Explicit calculation of Frobenius isomorphisms and
  Poincar{\'e} duality in the theory of arithmetic $\mathcal{D}$-module}\fg},
  \emph{Rend. Sem. Mat. Univ. Padova} (2013).


\bibitem[BBD82]{BBD}
{\scshape A.~A. Be{\u\i}linson, J.~Bernstein {\normalfont \smfandname}
  P.~Deligne} -- {\og Faisceaux pervers\fg}, Analysis and topology on singular
  spaces, {I} ({L}uminy, 1981), Ast{\'e}risque, vol. 100, Soc. Math. France,
  Paris, 1982, p.~5--171.

\bibitem[Ber84]{Ber-CohomologieRigide-Dwork}
{\scshape P.~Berthelot} -- {\og Cohomologie rigide et th\'eorie de {D}work: le
  cas des sommes exponentielles\fg}, \emph{Ast\'erisque} (1984), no.~119-120,
  p.~3, 17--49, $p$-adic cohomology.

\bibitem[Ber96b]{Be1}
\bysame , {\og ${\mathcal{d}}$-modules arithm\'etiques. {I}. {O}p\'erateurs
  diff\'erentiels de niveau fini\fg}, \emph{Ann. Sci. \'Ecole Norm. Sup. (4)}
  \textbf{29} (1996), no.~2, p.~185--272.
  
  \bibitem[Ber00]{Be2}
\bysame , {\og {$\mathcal{D}$}-modules arithm\'etiques. {I}{I}. {D}escente par
  {F}robenius\fg}, \emph{M\'em. Soc. Math. Fr. (N.S.)} (2000), no.~81,
  p.~vi+136.

\bibitem[Ber02]{Beintro2}
\bysame , {\og {Introduction \`a la th\'eorie arithm\'etique des
  {$\mathcal{D}$}-modules}\fg}, \emph{Ast\'erisque} (2002), no.~279, p.~1--80,
  Cohomologies {$p$}-adiques et applications arithm\'etiques, {II}.

\bibitem[BGK{\etalchar{+}}87]{borel}
{\scshape A.~Borel, P.-P. Grivel, B.~Kaup, A.~Haefliger, B.~Malgrange
  {\normalfont \smfandname} F.~Ehlers} -- \emph{Algebraic ${D}$-modules},
  Academic Press Inc., Boston, MA, 1987.


\bibitem[Car04]{caro_surcoherent}
{\scshape D.~Caro} -- {\og {$\mathcal{D}$}-modules arithm{\'e}tiques
  surcoh{\'e}rents. {A}pplication aux fonctions {$L$}\fg}, \emph{Ann. Inst.
  Fourier, Grenoble} \textbf{54} (2004), no.~6, p.~1943--1996.
  
  
\bibitem[Car05]{caro_comparaison}
\bysame , {\og Comparaison des foncteurs duaux des isocristaux
  surconvergents\fg}, \emph{Rend. Sem. Mat. Univ. Padova} \textbf{114} (2005),
  p.~131--211.
  
  \bibitem[Car06]{caro_courbe-nouveau}
\bysame , {\og Fonctions {L} associ{\'e}es aux {$\mathcal{D}$}-modules
  arithm{\'e}tiques. {C}as des courbes\fg}, \emph{Compositio Mathematica}
  \textbf{142} (2006), no.~01, p.~169--206.
  
\bibitem[Car06]{caro_devissge_surcoh}
\bysame ,  {\og {D{\'e}vissages des $F$-complexes de
  $\mathcal{D}$-modules arithm{\'e}tiques en $F$-isocristaux
  surconvergents}\fg}, \emph{Invent. Math.} \textbf{166} (2006), no.~2,
  p.~397--456.

\bibitem[Car07]{caro-2006-surcoh-surcv}
\bysame , {\og {Overconvergent F-isocrystals and differential
  overcoherence}\fg}, \emph{Invent. Math.} \textbf{170} (2007), no.~3,
  p.~507--539.

\bibitem[Car08]{caro-stab-prod-tens}
\bysame , {\og {Sur la stabilit{\'e} par produits tensoriels des $F$-complexes
  de $\D$-modules arithm{\'e}tiques}\fg},  (2008).

\bibitem[Car09a]{caro_log-iso-hol}
\bysame , {\og {Log-isocristaux surconvergents et holonomie.}\fg},
  \emph{Compos. Math.} \textbf{145} (2009), no.~6, p.~1465--1503.

\bibitem[Car09b]{caro_surholonome}
\bysame , {\og $\mathcal{D}$-modules arithm{\'e}tiques surholonomes\fg},
  \emph{Ann. Sci. \'Ecole Norm. Sup. (4)} \textbf{42} (2009), no.~1,
  p.~141--192.
  
\bibitem[Car09c]{caro-construction}
\bysame , 
{\og {$\cal D$-modules arithm{\'e}tiques associ{\'e}s
  aux isocristaux surconvergents. Cas lisse}\fg},
{\em Bull. Soc. Math. Fr.}, 137(4):453--543, 2009.
  
  
  

\bibitem[Car11a]{caro-pleine-fidelite}
\bysame , {\og {Pleine fid{\'e}lit{\'e} sans structure de Frobenius et
  isocristaux partiellement surconvergents}\fg}, \emph{Math. Ann.} \textbf{349}
  (2011), p.~747--805.

\bibitem[Car11b]{caro-image-directe}
\bysame , {\og {Sur la pr{\'e}servation de la surconvergence par l'image
  directe d'un morphisme propre et lisse}\fg}, to appear 
  to Ann. Sci. \'Ecole Norm. Sup.
\bibitem[Car12a]{caro:formal6}
\bysame , {\og The formalism of Grothendieck's six operations in p-adic
  cohomologies\fg},  (2012).

\bibitem[Car12b]{caro-stab-u!R-Gamma}
\bysame , {\og {Sur la pr{\'e}servation de la coh{\'e}rence par image inverse
  extraordinaire par une immersion ferm{\'e}e}\fg}, \emph{ArXiv Mathematics
  e-prints} (2012).

\bibitem[Car12c]{caro-stab-sys-ind-surcoh}
\bysame , {\og {Syst{\`e}mes inductifs surcoh{\'e}rents de
  $\mathcal{D}$-modules arithm{\'e}tiques}\fg}, \emph{ArXiv Mathematics
  e-prints} (2012).

\bibitem[CT12]{caro-Tsuzuki}
{\scshape D.~Caro {\normalfont \smfandname} N.~Tsuzuki} -- {\og
  {{Overholonomicity of overconvergent $F$-isocrystals over smooth
  varieties}}\fg}, \emph{Annals of Math.} (2012).


\bibitem[CM01]{christol-MebkhoutIV}
{\scshape G.~Christol {\normalfont \smfandname} Z.~Mebkhout} -- {\og Sur le
  th\'eor\`eme de l'indice des \'equations diff\'erentielles {$p$}-adiques.
  {IV}\fg}, \emph{Invent. Math.} \textbf{143} (2001), no.~3, p.~629--672.



\bibitem[Cre98]{crewfini}
{\scshape R.~Crew} -- {\og Finiteness theorems for the cohomology of an
  overconvergent isocrystal on a curve\fg}, \emph{Ann. Sci. \'Ecole Norm. Sup.
  (4)} \textbf{31} (1998), no.~6, p.~717--763.

\bibitem[Cre06]{crew-arith-D-mod-curve}
\bysame , {\og Arithmetic {$\mathcal{D}$}-modules on a formal curve\fg},
  \emph{Math. Ann.} \textbf{336} (2006), no.~2, p.~439--448.


\bibitem[Cre12]{Crew-unit-disk}
\bysame , {\og Arithmetic {$\mathcal{D}$}-modules on the unit disk\fg},
  \emph{Compos. Math.} \textbf{148} (2012), no.~1, p.~227--268.


\bibitem[Del80]{deligne-weil-II}
{\scshape P.~Deligne} -- {\og La conjecture de {W}eil. {II}\fg}, \emph{Inst.
  Hautes \'Etudes Sci. Publ. Math.} (1980), no.~52, p.~137--252.

\bibitem[{\'E}LS93]{E-LS}
{\scshape J.-Y. {\'E}tesse {\normalfont \smfandname} B.~Le~Stum} -- {\og
  Fonctions ${L}$ associ\'ees aux ${F}$-isocristaux surconvergents. {I}.
  {I}nterpr\'etation cohomologique\fg}, \emph{Math. Ann.} \textbf{296} (1993),
  no.~3, p.~557--576.

\bibitem[Fuj02]{Fujiwara-Gabber}
{\scshape K.~Fujiwara} -- {\og Independence of {$l$} for intersection
  cohomology (after {G}abber)\fg}, Algebraic geometry 2000, {A}zumino
  ({H}otaka), Adv. Stud. Pure Math., vol.~36, Math. Soc. Japan, Tokyo, 2002,
  p.~145--151.

\bibitem[Har66]{HaRD}
{\scshape R.~Hartshorne} -- \emph{Residues and duality}, Springer-Verlag,
  Berlin, 1966.

\bibitem[Kat89]{Kato-logFontaine-Illusie}
Kazuya Kato.
\newblock Logarithmic structures of {F}ontaine-{I}llusie.
\newblock In {\em Algebraic analysis, geometry, and number theory (Baltimore,
  MD, 1988)}, pages 191--224. Johns Hopkins Univ. Press, Baltimore, MD, 1989.


\bibitem[KM74]{KatzMessing}
{\scshape N.~M. Katz {\normalfont \smfandname} W.~Messing} -- {\og Some
  consequences of the {R}iemann hypothesis for varieties over finite
  fields\fg}, \emph{Invent. Math.} \textbf{23} (1974), p.~73--77.


\bibitem[Ked04]{kedlaya_full_faithfull}
Kiran~S. Kedlaya.
\newblock Full faithfulness for overconvergent {$F$}-isocrystals.
\newblock In {\em Geometric aspects of Dwork theory. Vol. I, II}, pages
  819--835. Walter de Gruyter GmbH \& Co. KG, Berlin, 2004.



\bibitem[Ked06]{kedlaya-weilII}
{\scshape K.~S. Kedlaya} -- {\og Fourier transforms and {$p$}-adic `{W}eil
  {II}'\fg}, \emph{Compos. Math.} \textbf{142} (2006), no.~6, p.~1426--1450.

\bibitem[Ked07]{kedlaya-semistableI}
\bysame , {\og Semistable reduction for overconvergent {$F$}-isocrystals. {I}.
  {U}nipotence and logarithmic extensions\fg}, \emph{Compos. Math.}
  \textbf{143} (2007), no.~5, p.~1164--1212.


\bibitem[Ked11]{kedlaya-semistableIV}
\bysame , {\og Semistable reduction for overconvergent {$F$}-isocrystals, {IV}:
  local semistable reduction at nonmonomial valuations\fg}, \emph{Compos.
  Math.} \textbf{147} (2011), no.~2, p.~467--523.


\bibitem[KW01]{KW-Weilconjecture}
{\scshape R.~Kiehl {\normalfont \smfandname} R.~Weissauer} -- \emph{Weil
  conjectures, perverse sheaves and {$l$}'adic {F}ourier transform}, Ergebnisse
  der Mathematik und ihrer Grenzgebiete. 3. Folge. A Series of Modern Surveys
  in Mathematics [Results in Mathematics and Related Areas. 3rd Series. A
  Series of Modern Surveys in Mathematics], vol.~42, Springer-Verlag, Berlin,
  2001.


\bibitem[Lau87]{Laumon-transf-Fourier}
{\scshape G.~Laumon} -- {\og Transformation de {F}ourier, constantes
  d'\'equations fonctionnelles et conjecture de {W}eil\fg}, \emph{Inst. Hautes
  \'Etudes Sci. Publ. Math.} (1987), no.~65, p.~131--210.


\bibitem[Mar08]{Marmora-fact-eps}
{\scshape A.~Marmora} -- {\og Facteurs epsilon {$p$}-adiques\fg}, \emph{Compos.
  Math.} \textbf{144} (2008), no.~2, p.~439--483.
  
  \bibitem[Mat02]{matsuda-katz}
Shigeki Matsuda.
\newblock Katz correspondence for quasi-unipotent overconvergent isocrystals.
\newblock {\em Compositio Math.}, 134(1):1--34, 2002.

\bibitem[Meb89]{Mebkhout}
Z.~Mebkhout.
\newblock {\em Le formalisme des six op\'erations de {G}rothendieck pour les
  $\mathcal{D} \sb {X}$-modules coh\'erents}.
\newblock Hermann, Paris, 1989.
\newblock With supplementary material by the author and L. Narv\'aez Macarro.


  \bibitem[Mon02]{these_montagnon}
Claude Montagnon.
\newblock {\em {G{\'e}n{\'e}ralisation de la th{\'e}orie arithm{\'e}tique des
  $\mathcal{D}$-modules {\`a} la g{\'e}om{\'e}trie logarithmique}}.
\newblock PhD thesis, Universit{\'e} de {R}ennes {I}, 2002.


  \bibitem[NH97]{Noot-Huyghe-affinite-proj}
{\scshape C.~Noot-Huyghe} -- {\og {${\mathcal{D}}^\dagger$}-affinit\'e de
  l'espace projectif\fg}, \emph{Compositio Math.} \textbf{108} (1997), no.~3,
  p.~277--318, With an appendix by P.\ Berthelot.


\bibitem[NH04]{Noot-Huyghe-fourierI}
\bysame ,  {\og Transformation de {F}ourier des
  {$\mathcal{D}$}-modules arithm\'etiques. {I}\fg}, Geometric aspects of Dwork
  theory. Vol. I, II, Walter de Gruyter GmbH \& Co. KG, Berlin, 2004,
  p.~857--907.
  
\bibitem[Ogu]{Ogus-Logbook}
Arthur Ogus.
\newblock Log book.


\bibitem[Shi10]{Shiho-logextension}
Atsushi Shiho.
\newblock On logarithmic extension of overconvergent isocrystals.
\newblock {\em Math. Ann.}, 348(2):467--512, 2010.


\bibitem[Tsu02]{tsumono}
Nobuo Tsuzuki.
\newblock Morphisms of ${F}$-isocrystals and the finite monodromy theorem for
  unit-root ${F}$-isocrystals.
\newblock {\em Duke Math. J.}, 111(3):385--418, 2002.



\bibitem[Vir00]{virrion}
{\scshape A.~Virrion} -- {\og Dualit\'e locale et holonomie pour les
  {$\mathcal{D}$}-modules arithm\'etiques\fg}, \emph{Bull. Soc. Math. France}
  \textbf{128} (2000), no.~1, p.~1--68.

\bibitem[Vir04]{Vir04}
\bysame , {\og Trace et dualit\'e relative pour les {$\mathcal{D}$}-modules
  arithm\'etiques\fg}, Geometric aspects of Dwork theory. Vol. I, II, Walter de
  Gruyter GmbH \& Co. KG, Berlin, 2004, p.~1039--1112.
\end{thebibliography}

\newcommand{\etalchar}[1]{$^{#1}$}
\providecommand{\bysame}{\leavevmode ---\ }
\providecommand{\og}{``}
\providecommand{\fg}{''}
\providecommand{\smfandname}{et}
\providecommand{\smfedsname}{\'eds.}
\providecommand{\smfedname}{\'ed.}
\providecommand{\smfmastersthesisname}{M\'emoire}
\providecommand{\smfphdthesisname}{Th\`ese}

Tomoyuki Abe:\\
Kavli Institute for the Physics and Mathematics of the Universe\\
The University of Tokyo\\
5-1-5 Kashiwanoha, Kashiwa, Chiba, 277-8583, Japan \\
e-mail: {\tt tomoyuki.abe@ipmu.jp}

\bigskip\noindent
Daniel Caro:\\
Laboratoire de Math\'{e}matiques Nicolas Oresme (LMNO)\\
Universit\'e de Caen, Campus 2\\
14032 Caen Cedex, France\\
e-mail: {\tt  daniel.caro@unicaen.fr}

\end{document}